\documentclass{amsart}
\usepackage{amssymb, enumerate}
\usepackage[dvipsnames]{xcolor}

\usepackage{amssymb}
\usepackage{amsthm}
\usepackage{amsmath}
\usepackage{paralist}
\usepackage{graphics}
\usepackage{epsfig}
\usepackage[colorlinks=true]{hyperref}
\hypersetup{urlcolor=blue, citecolor=red}

\usepackage{graphicx}

\newcommand{\cn}{\color{cyan}}

\newcommand{\dvol}{\,d\operatorname{vol}}
\newcommand\Symb{{\operatorname{Symb}}}
\newcommand\ADN{{Douglis-Nirenberg}}

\newtheorem{theorem}{Theorem}[section]
\newtheorem{lemma}[theorem]{Lemma}
\newtheorem{corollary}[theorem]{Corollary}
\newtheorem{proposition}[theorem]{Proposition}
\newtheorem{problem}[theorem]{Problem}

\theoremstyle{definition}
\newtheorem{definition}[theorem]{Definition}
\newtheorem{example}[theorem]{Example}
\newtheorem{notation}[theorem]{Notation}

\newtheorem{assumption}[theorem]{Assumption}

\theoremstyle{remark}
\newtheorem{remark}[theorem]{Remark}

\usepackage[normalem]{ulem}
\renewcommand\sout{\bgroup\markoverwith
{\textcolor{red}{\rule[0.7ex]{3pt}{1.4pt}}}\ULon}

\newcommand\seq{\, = \,}
\newcommand\ede{\mathrel{\, := \, }}

\newcommand\esssup{\operatorname{ess-sup}}
\newcommand\dist{\operatorname{dist}}
\newcommand\Hom{\operatorname{Hom}}
\newcommand\End{\operatorname{End}}
\newcommand\supp{\operatorname{supp}}

\newcommand\Def{\operatorname{Def}\, }

\newcommand{\CC}{\mathbb{C}}

\newcommand{\II}{\mathbb{I}}

\newcommand{\MM}{\mathbb M}
\newcommand{\NN}{\mathbb N}

\newcommand{\RR}{\mathbb R}

\newcommand{\ZZ}{\mathbb Z}

\newcommand{\maA}{\mathcal A}
\newcommand{\maB}{\mathcal B}
\newcommand{\maC}{\mathcal C}
\newcommand{\maD}{\mathcal D}

\newcommand{\maF}{\mathcal F}

\newcommand{\maL}{\mathcal L}

\newcommand{\maN}{\mathcal N}

\newcommand{\maP}{\mathcal P}
\newcommand{\maQ}{\mathcal Q}
\newcommand{\maR}{\mathcal R}
\newcommand{\maS}{\mathcal S}

\newcommand{\maV}{\mathcal V}
\newcommand{\maW}{\mathcal W}

\newcommand{\CI}{\maC^{\infty}}
\newcommand{\CIc}{\maC^{\infty}_{c}}
\newcommand\pa{\partial}
\newcommand\canDec{M_{0} \cup \big[M' \times (-\infty, 0) \big]}

\newcommand\inv{\operatorname{inv}}
\newcommand\ess{\operatorname{ess}}
\newcommand\iPS[1]{\Psi_{\operatorname{inv}}^{#1}}
\newcommand\iPSsus[3]{\psi_{\operatorname{sus}}^{#1}(#2; #3)}
\newcommand\ePS[1]{\Psi_{\operatorname{ess}}^{#1}}
\newcommand\ePSsus[3]{\Psi_{\operatorname{sus}}^{#1}(#2; #3)}
\newcommand\cl{\operatorname{cl}}
\newcommand\In{{\maR_\infty}}

\newcommand\bop{\boldsymbol T_{\bsnu}}
\newcommand\bopstar{\boldsymbol T_{\bsnu}^{*}}
\newcommand\bss{\boldsymbol s}

\newcommand\bst{\boldsymbol t}
\newcommand\bsu{\boldsymbol u}
\newcommand\bsv{\boldsymbol v}
\newcommand\bsw{\boldsymbol w}
\newcommand\bsnu{\boldsymbol \nu}
\newcommand\bsf{\boldsymbol f}
\newcommand\bsh{\boldsymbol h}
\newcommand\Dnu{\boldsymbol D_{\bsnu}}
\newcommand\bsL{\boldsymbol L}
\newcommand\bsXi{\boldsymbol \Xi}
\newcommand\bsK{\boldsymbol K}
\newcommand\bsS{\boldsymbol S}
\newcommand\bsP{\boldsymbol P}
\newcommand\Defstar{\operatorname{Def}^{*}\, }

\newcommand\tbop{\tilde{\boldsymbol T}_{\bsnu}}
\newcommand\tbopstar{\tilde{\boldsymbol T}_{\bsnu}^{*}}
\newcommand\sgn{\operatorname{sgn}}
\newcommand\loc{\operatorname{loc}}
\newcommand\comp{\operatorname{comp}}

\newcommand\psdinv{\bsXi^{(-1)}}
\newcommand\Smbmn{S^{m}(\RR^{n} \times \RR^{n})}
\newcommand\pv{\operatorname{p.\!\!v.}}

\newcommand\lpar{\rm(}
\newcommand\rpar{\rm)}

\newcommand\JC{\operatorname{\mathfrak L}}
\newcommand\mfkV{\operatorname{\mathfrak V}}
\newcommand\mfkS{\mathfrak S}
\newcommand\mfkM{\mathfrak M}

\newcommand\jap[1]{\langle #1 \rangle}
\renewcommand\div{\operatorname{div}}
\newcommand\cvector[2]{\left(
  \begin{array}{c} {#1} \\ {#2}
  \end{array}
\right)}

\newcommand\newtilde{\widehat }

\numberwithin{equation}{section}

\begin{document}

\title[Well-posedness for a generalized Stokes operator]{Well-posedness
of a generalized Stokes operator on domains with cylindrical ends via layer-potentials}

\author[M. Kohr]{Mirela Kohr}
\address{Faculty of Mathematics and Computer Science,
Babe\c{s}-Bolyai University, 1 M. Kog\u{a}l\-niceanu Str., 400084
Cluj-Napoca, Romania} \email{mkohr@math.ubbcluj.ro}

\author[V. Nistor]{Victor Nistor}
\address{Universit\'{e} de Lorraine, CNRS, IECL, F-57000 Metz, France}
\email{victor.nistor@univ-lorraine.fr}

\author[W.L. Wendland]{Wolfgang L. Wendland}
\address{Institut f\"ur Angewandte Analysis und Numerische Simulation,
Universit\"at Stuttgart, Pfaffenwaldring 57, 70569 Stuttgart,
Germany}
\email{wendland@mathematik.uni-stuttgart.de}

\thanks{M.K. has been partially supported by
  AGC35124/31.10.2018. V.N. has been partially supported by
  ANR-14-CE25-0012-01.}

\date\today


\date{\today}

\dedicatory{In memory of Professor Gabriela Kohr, with deep respect}

\begin{abstract}
  We study the \emph{generalized Stokes operator}
  \begin{equation*} 
    \bsXi \ede \bsXi _{V,V_{0}} \ede
    \left(\begin{array}{ccc}  \bsL + V & \nabla \\
    \nabla^* &  -V_{0}
    \end{array}\right)
  \end{equation*}
  on a \emph{domain with straight cylindrical ends} $\Omega$ using \emph{the method of layer
  potentials} on a larger manifold with straight cylindrical ends $M$, $\Omega \subset M$.
  The operator $\bsXi \ede \bsXi _{V,V_{0}}$ recovers the classical Stokes operator $\bsXi_{0, 0}$
  when the potentials $V$ and $V_{0}$ vanish. Under suitable positivity assumptions
  at infinity on $V$ and $V_{0}$, we prove that $\bsXi$ is Fredholm on
  $L^{2}(M; TM \oplus \CC)$. This allows us
  then to define the single- and double-layer potential operators $\bsS$ and
  $\frac12 + \bsK$. Under further positivity assumptions at infinity, we prove
  that $\bsS$ and $\frac12 + \bsK$ are also Fredholm. Under slightly stronger assumptions
  on $V$ and $V_{0}$ (including non-negativity everywhere,
  i.e., $V, V_{0} \ge 0$), we prove \emph{the invertibility} of
  the operators $\bsXi$, $\bsS$, and $\frac12 + \bsK$. As it is well-known, the invertibility
  of these operators leads to \emph{well-posedness results} for the associated (linear)
  Stokes boundary value problem with Dirichlet boundary conditions on domains with
  straight cylindrical ends. The proofs of these results required us to develop many other
  related tools. In particular, we develop an
  ``algebra tool kit'' to deal with \emph{limit and jump relations of layer potential operators,}
  in general, and on manifolds with straight cylindrical ends, in particular. We do that first on
  $\RR^{n}$, then on (possibly) non-compact manifolds, and then we deal with the specific
  case of manifolds with straight cylindrical ends. We use these results to study the limit and jump
  relations of the layer potential operators associated to our generalized Stokes operator
  $\bsXi  := \bsXi_{V, V_{0}}$ on manifolds with cylindrical ends.
  We also develop Green formulas and energy estimates for our
  generalized Stokes operator $\bsXi$ on manifolds with straight
  cylindrical ends, which requires, in particular, a careful geometric study of the related
  differential operators, such as the deformation operator $\Def$. We also use suitable
  classes of pseudodifferential operators on manifolds with straight cylindrical ends that
  were studied in some previous papers of ours (including ``The Stokes operator on manifolds with
  cylindrical ends,'' J. Diff. Equations, 2024). For completeness, we include a review of the
  needed results on these pseudodifferential operators.
  As an application of all these results,
  we prove the well-posedness result for the Dirichlet problem for the generalized
  Navier-Stokes system with small data on a domain with cylindrical ends.
  We expect our results to have applications to the
  study of Navier-Stokes equations on domains with conical points.
\end{abstract}

\maketitle
\tableofcontents

\section{\cn Introduction}

We prove the \emph{well-posedness of the Dirichlet boundary value problem
for a generalized Stokes system
\begin{equation}\label{eq.bvp.Brinkman.explicit}
  \begin{cases}
    \ \bsL \bsu + V \bsu + \nabla p \seq \bsh \  & \ \mbox{ in } \Omega\\
    \ \, \nabla^* \bsu  - V_{0} p \seq q  \  & \ \mbox{ in } \Omega\\
    \ \ \bsu \seq \bsf  \  & \ \mbox{ on } \pa \Omega\\
  \end{cases}
\end{equation}
for a domain with straight cylindrical ends} $\Omega$ using the method of
\emph{layer potentials.} (Here $\nabla$ is the gradient, $\bsL$ is the deformation
Laplacian, and $V$ and $V_{0}$ are potentials, see the next subsection.
See also Equation \eqref{eq.bvp.Brinkman} for a reformulation of this system
using the generalized Stokes operator $\bsXi := \bsXi_{V, V_{0}}$ \eqref{eq.def.bsXi}.)
In the process, we obtain many other results of independent interest on the generalized
Stokes operator $\bsXi := \bsXi_{V, V_{0}}$, on related differential
operators, and on their associated layer potentials. We carefully develop the needed
concepts and methods and prove in detail our results.

Before stating our main results, let us first introduce the setting of our paper.

\subsection{The setting of our well-posedness results: Domains and manifolds with
straight cylindrical ends}
\label{ssec.intro.setting}
We will work on manifolds $M$ and domains $\Omega$ with straight cylindrical ends
(see the next figure, Figure 1, as well as Figure 2 at the beginning
of Section \ref{sec.NLimits} for pictorial description of our setting).
\begin{figure}[h]
  \label{fig1_no_bdry}
  \centering
  \includegraphics[width=0.4\textwidth]{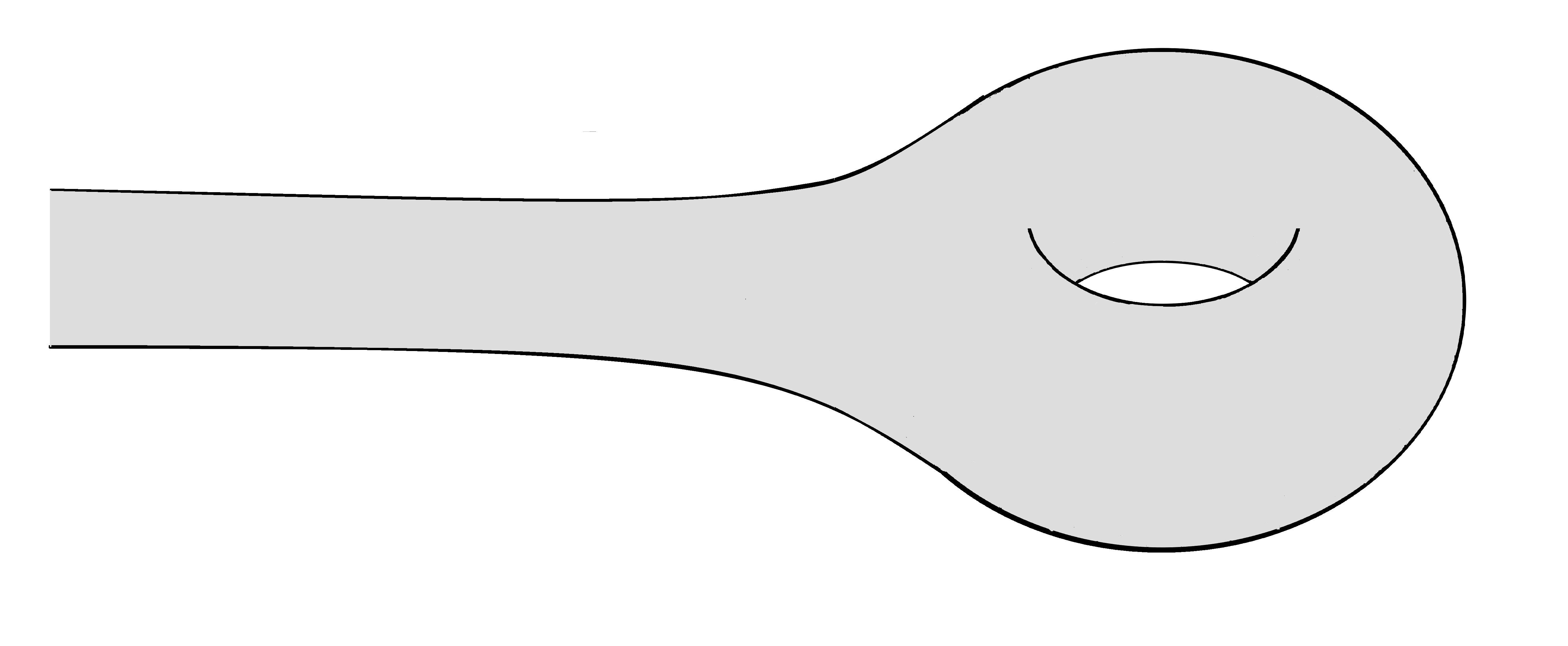}
  \caption{A {\it boundaryless} manifold with straight cylindrical ends.}
\end{figure}

\subsubsection{Manifolds with straight cylindrical ends}
Intuitively, a manifold with straight cylindrical ends is one that, outside
a bounded set, it is a (half-)infinite cylinder $M' \times (-\infty, R]$ with
a product metric, where $M'$ is a closed manifold (i.e. a smooth, compact manifold
without boundary).

More precisely, by a \emph{manifold with straight cylindrical ends} $M$ we will mean
a non-compact Riemannian manifold without boundary (i.e., ``boundaryless'')
isometric to one of the form
\begin{equation}\label{eq.def.cyl.end}
  M \seq M_{0}\cup \big[M'\times (-\infty ,R_{M}]\big]\,,
\end{equation}
where
\begin{itemize}
  \item $(M_{0}, g_{0})$ is a compact Riemannian manifold \emph{with boundary} $\pa M_{0}$,
  \item $M':=\partial M_{0} \neq \emptyset$
  has induced metric $g_{\pa}$ and is identified with $M' \times \{0\}$, and
  \item the
  \emph{cylindrical end} $M'\times (-\infty ,R_{M})$
  is endowed with the product metric $g = g_{\pa} + (dt)^{2}$, $t \in (-\infty ,R_{M})$,
\end{itemize}
(where we followed the conventions in \cite{KMNW-2025,KNW-22, KohrNistor-Stokes}; but see
also \cite{bbendhia21, BNSV22, BBKO, bourgeois23, epsteinP2, EFK, KrejcirikZuazua}, where
manifolds with cylindrical ends arise in the setting of ``wave-guides''). There is no loss
of generality to assume that \emph{$M$ is connected, so we will do that throughout the
paper,} in order to simplify our statements. (Also, we may assume $R_{M} = 0$,
which we will often do.)

\subsubsection{Domains with straight cylindrical ends}
In fact, we are more interested in the \emph{smooth domain with straight cylindrical ends}
$\Omega$, \textbf{on which our boundary value problem \eqref{eq.bvp.Brinkman} is formulated}.
Often $\Omega$ will be contained naturally in a manifold with straight cylindrical ends $M$
in a compatible way, in the sense that
\begin{equation}\label{eq.def.Omega}
  \Omega \cap \big[M'\times (-\infty , R_{\Omega})\big]
  \seq \Omega'\times (-\infty , R_{\Omega})
\end{equation}
(up to an isometry), for some $R_{\Omega} \le R_{M}$, and
its boundary $\Gamma := \pa \Omega$ is a smooth manifold. (So $\Gamma$ will
also be a manifold with straight cylindrical ends, see Figure 2 at the beginning
of Section \ref{sec.NLimits}.)
Again, there is no loss of generality to assume that $\Omega$ is connected,
so we will do that throughout the paper.

We will also assume throughout this paper that \emph{$\Omega$ is on one side
of its boundary,} meaning that $\Gamma$ is the boundary of $\Omega_{-} := M \smallsetminus
\overline{\Omega}$ as well (see Assumption \ref{assumpt.1side}).
This assumption is useful when studying boundary value problems on
$\Omega$ using layer potentials. This implies that $\Omega'$ is also on one side of
its boundary $\Gamma' := \pa \Omega' \subset M'$.

\subsubsection{The Dirichlet boundary value problem for the generalized
Stokes operator $\bsXi := \bsXi_{V,V_{0}}$} Our main interest is, in fact,
in the Dirichlet boundary value problem for a generalized Stokes operator $\bsXi
:= \bsXi_{V, V_{0}}$, this problem is formulated below, see Equation \eqref{eq.bvp.Brinkman}.
The generalized Stokes operators $\bsXi := \bsXi_{V, V_{0}}$ are defined as follows. Let
\begin{equation*}
  \Def : \CI(M; TM) \to \CI(M; T^{*}M \otimes T^{*}M)
\end{equation*}
be the deformation operator (see formula \eqref{eq.def.Def} for the definition of $\Def$ and,
in general, see Section \ref{sec.Green} for more details). Let then $\bsL := 2 \Defstar \Def$
be the ``Deformation Laplacian.'' Let $\End(E)$ be the set of linear maps $E \to E$
(the endomorphisms of $E$). For suitable ``potentials'' $V \in \CI(M; \End(TM))$ and $V_{0}
\in \CI(M)$ that are \emph{translation invariant} in a neighborhood of infinity  (Definition
\ref{def.CIinv}), we consider the operator
\begin{equation} \label{eq.def.bsXi}
  \bsXi \ede \bsXi _{V,V_{0}} \ede
  \left(\begin{array}{ccc}  \bsL + V & \nabla \\
  \nabla^* &  -V_{0}
  \end{array}\right)
  : \CI(M; TM \oplus \CC) \to \CI(M; TM \oplus \CC)\,.
\end{equation}
This operator extends to Sobolev spaces in distribution sense. Let
\begin{equation*}
  U \ede \cvector{\bsu}{p} \in L^2(\Omega ; TM \oplus \CC)\,,
\end{equation*}
so that $\bsu$ is the ``vector part'' of $U$. We then consider
the \emph{Dirichlet boundary value problem, Equation
\eqref{eq.bvp.Brinkman.explicit}, which, in this notation, becomes}
\begin{equation}\label{eq.bvp.Brinkman}
  \begin{cases}
    \bsXi U \ede \bsXi_{V, V_{0}} U \seq 0 & \mbox{ in } \Omega\\
    \bsu \seq \bsh & \mbox{ on } \pa \Omega\,.
  \end{cases}
\end{equation}
When $V$ and $V_{0}$ vanish identically on $M$, this is nothing but \emph{the classical
Dirichlet problem for the Stokes operator.}

\subsection{Main results}
Our main results are for manifolds with straight cylindrical ends $M$ and for
compatible open subsets $\Omega \subset M$ with boundary $\Gamma := \pa \Omega$, as above.
Our main results also hold for closed manifolds (i.e., compact, smooth, boundaryless
manifolds), see \cite{KNW-2025} (some of these results are recalled in Subsection \ref{ssec.ToAdd}).
{\cn Let us thus assume in this introduction that $M = \canDec$ is a manifold with straight
cylindrical ends (see \ref{ssec.intro.setting} or Definition \ref{def.cyl.end}); this will be
the case for the most part of our paper.}

Our \emph{first main
results} are to prove the \emph{invertibility on $\Gamma$ of the single layer and double
layer potential operators $\bsS$ and $\frac12 + \bsK$,} associated to suitable
generalizations $\bsXi := \bsXi_{V, V_{0}}$ of the Stokes operator $\bsXi_{0, 0}$, which
will be defined shortly, where $V$ and $V_{0}$ are suitable non-negative potentials (as in
\ref{ssec.intro.setting}), see Theorems \ref{thm2.S} and \ref{thm2.K}. (We usually denote
the identity operator $\II$ with 1, we thus write $\frac12 + \bsK$ instead of $\frac12 \II + \bsK$.)
We also obtain, for suitable potentials $V$ and $V_{0}$ satisfying slightly weaker assumptions,
that $\bsS$ and $\frac12 + \bsK$ are Fredholm.

In order to define our layer potential operators $\bsS$ and $\frac12 + \bsK$, we first
prove that $\bsXi$ is Fredholm under suitable positivity assumptions at infinity for
$V$ and $V_{0}$. We prove that it ($\bsXi$) is even invertible if the positivity conditions hold
everywhere (Theorem \ref{thm2.main.invert}). We use these results to
define the single and double layer potential operators $\maS_{\rm{ST}}$, $\bsS$,
$\maD_{\rm{ST}}$, and $\bsK$ associated to the Dirichlet boundary value problem
for $\bsXi$. We prove that they satisfy the usual mapping and ``jump'' properties
and that they have the usual principal symbols (Theorem \ref{thm.jump.rel}; but in order
to obtain this result, we develop an extensive theory: energy and representation formulas,
geometric representations and formulas for our operators, the definition and suitable classes
of integral operators, a ``took-kit'' for proving jump
relations, and many other results of independent interest). Then, using the jump
relations and principal symbol properties of the layer potential operators, we prove that
$\bsS$ and $\frac12 + \bsK$ are invertible.
As is well known, this result implies the solvability of the Dirichlet problem
\eqref{eq.bvp.Brinkman} (see Equation \eqref{eq.bvp.Brinkman} for a more
explicit form of this equation), Theorem \ref{thm.main.WP} next, which is the
\emph{main result of this paper.} We write $v \succ 0$ for a non-negative function that does
not vanish identically on any connected component of its domain (see Definition \ref{def.succ}).
The space $\CI_{\inv}(M)$ consists of functions (or sections of a vector bundle in case of
$\CI_{\inv}(M; E)$) that are \emph{translation invariant} in a neighborhood of infinity
(Definition \ref{def.CIinv}).

\begin{theorem}\label{thm.main.WP}
  Let $M = \canDec$ be a manifold with straight cylindrical ends,
  let $\Omega \subset M$ be a compatible smooth domain \lpar i.e., with straight
  cylindrical ends, Equation \eqref{eq.def.Omega}\rpar\ that is on
  one side of its boundary $\Gamma := \pa \Omega \neq \emptyset$
  \lpar Assumption \ref{assumpt.1side}\rpar, and let $m \in \ZZ_{+}$.
  Let $V \in \CI_{\inv}(M; TM)$ and $V_{0} \in \CI_{\inv}(M)$ be such that:
  \begin{enumerate}[\rm (i)]
    \item $V \ge 0$ and $V_{0} \ge 0$ on $M$,
    \item $V_{0} \succ 0$ on $\Omega$,
    \item $\newtilde V_{0} \succ 0$ on $\Omega'$, and
    \item if $\Omega_{0}$ is a connected component of $M'$ contained in $\Omega'$,
    then $\newtilde V \not \equiv 0$ on $\Omega_{0}$.
  \end{enumerate}
  Then, there exists $C_{m} \ge 0$ such that, for any $\bsf \in H^{m+1/2}(\Gamma; TM)$, the 
  homogeneous Dirichlet problem \eqref{eq.Dirichlet.pr} has a unique solution 
  $U = (\bsu \ \ p)^{\top} \in H^{m+1}(\Omega; TM) \oplus H^{m}(\Omega)$ and this solution 
  satisfies
  \begin{equation*}
    \|\bsu\|_{H^{m+1}(\Omega; TM)} + \|p\|_{H^{m}(\Omega)} \le C_{m}
    \|\bsf\|_{H^{m+1/2}(\Gamma; TM)}\,.
  \end{equation*}
\end{theorem}

The proof of this theorem is provided in Subsection \ref{ssec.WP}. At this time,
a proof in general of this theorem using energy (or variational) methods (i.e.,
the Riesz representation formula) seems out of reach, in the sense that it does not
follow from the Riesz representation formula using standard methods. We also obtain
applications of this well-posedness results to the Dirichlet-to-Neumann operator and
to the solvability of the stationary Navier-Stokes system on domains with straight
cylindrical ends of small dimension.

We have outlined here the most important results. However, in this paper, we obtain
many other results of independent interest that lead to the main results (they were
mentioned briefly above). We also include a lot of background material, which is, hence
not entirely new. However, our presentation is original and it gathers many results
that are difficult to find in earlier papers. Moreover, most of our background
material is presented in great detail, so that to make our main result more accessible
(a notable exception is, though, the background material on pseudodifferential
operators, for which many other well-developed works exist).
These other results are discussed in much more detail in the
part on the contents of the paper, which is next.

\subsection{Contents of the paper}
Our paper consists of \ref{sec.NS} sections (including an introduction),
each divided into \emph{subsections}. The assumptions of this paper (especially on
the manifold $M$ on which we work) change from section to section. These assumptions are
always stated at the beginning of that section and are reminded regularly, especially at the
beginning of the subsections and in the statement of the most important results.
(Note, however, that, beginning with Section \ref{sec.psdos}, we work almost exclusively on
manifolds with straight cylindrical ends.) The first six sections are devoted mostly to
background material, but the presentation is original (and, except for the part on
pseudodifferential operators, quite detailed, including details and even results that are hard
to find in the literature).

The first section is this introduction. Section \ref{sec.Green} is organized as follows.
Its first subsection, Subsection \ref{ssec.do}, is devoted to the review of some classical
differential operators that appear in the definition and in the study of the Stokes operator
$\bsXi = \bsXi_{V, V_{0}}$ on a Riemannian manifold (recall that $\bsXi_{V, V_{0}}$ was introduced
in Equation \eqref{eq.def.bsXi}), such as the deformation operator $\Def$ and the ``traction''
$\bop$. In Subsection \ref{ssec.abs.int}, we recall an abstract integration
by parts formula, Proposition \ref{prop.bdry.op}. This general formula is used in
Subsection \ref{ssec.Green1}, where we obtain Green-type formulas for the Stokes
operator $\bsXi$, as well as some representation formulas and energy estimates.
In this section, we work on a general Riemannian manifold $M$.

The third section is devoted to some basic preliminary results such as the concepts
of traces and normal lateral limits, some useful Fourier transforms, the basic definitions
of symbols and pseudodifferential operators. In particular, they explain in detail
the main ``Fourier transform calculation'' that is the prototypical ``jump/limit
result'', see Lemma \ref{lemma.first.jump} and Corollary \ref{cor.first.jump}.
These two results ultimately are responsible and explain all the other jump/limit
results. These results are formulated on $\RR^{n}$.
However, in the last subsection of this section, we work on manifolds with straight
cylindrical ends. In this subsection, we recall the basic definitions
related to manifolds with cylindrical ends and their compatible vector bundles,
differential operators, and Sobolev spaces. Although the results of this section are
classical, our presentation is new. (More references and more related background
material can be found in our book \cite{KMNW-2025}.)

In the forth section, we discuss the jump relations for pseudodifferential
operators $a(x, D)$ at the hyperplane $\{x_{n} = 0\}$ in $\RR^{n}$. This is a long
and technical section in which we perform in great detail all the necessary local
calculations needed to obtain the jump and limit relations that are so important for
the study of layer potentials. We first establish the jump relations for operators of
order $m < -1$ (the ``jump'' in this case is zero, so these are rather ``limit''
results). We prove our results for operators with Kohn-Nirenberg type symbols.
We then establish the jump relations and the other usual results for operators of
order $-1$ with homogeneous, odd principal symbol. The presentation expands
and details the one in our book \cite{KMNW-2025}. For instance, in this paper,
we do not restrict ourselves to classical pseudodifferential operators
(however, for operators of order $-1$, we
give our results directly in the case of an odd symbol, the only one needed in our paper).
Our presentation in this section provides a complete, yet concise, new introduction to
the subject of limit and jump relations for potential operators on a half-space. They are
the first step in the development of an ``algebra tool kit''
to study the limit/jump relations. These
results, together with those in Section \ref{sec.NLimits} provide a short--but complete--introduction
to jump relations on manifolds with cylindrical ends. (Results for manifolds with bounded
geometry can be found in \cite{KNW-2025}.)

The setting of our paper, up to this point, was on $\RR^{n}$ or on a general Riemannian
manifold (except Subsection \ref{ssec.cyl.ends}, where we defined manifolds with straight
cylindrical ends). Beginning with the next section, Section \ref{sec.psdos},
however, we will work almost exclusively on manifolds with straight cylindrical ends.
(Sometimes, we need to consider the particular case of straight cylinders.) In particular,
Section \ref{sec.psdos} is devoted to pseudodifferential operators on manifolds with straight
cylindrical ends. Here we introduce the main two classes of pseudodifferential operators used
in this paper. The differential operators considered in this paper will belong to the
first class, while their parametrices and layer potential operators will belong to
the second of these classes. More precisely, in the first half of this section, we recall the
definition and main properties of the $\inv$-calculus of translation invariant operators in a
neighborhood of infinity, and, in
the second half we recall the definition and the main properties of the $\ess$-calculus introduced
in \cite{KNW-22} (see also \cite{KMNW-2025, KohrNistor-Stokes}). This calculus contains
the $\inv$-calculus, but it has further important properties due to the fact that it contains the parametrices of its operators (i.e., it is spectrally invariant).
Since this material is well-known,
the proofs are short. However, all the necessary definitions and properties are
discussed in great detail. In particular, we introduce the ``limit operator'' $\In(P)$
associated to an operator $P$ in any of our two calculi. The limit operator
$\In(P)$ is a translation invariant operator on the infinite cylinder $M' \times \RR$
associated to our manifold with straight cylindrical ends $M = \canDec$. The concept of
limit operator is crucial for the study of operators on manifolds with cylindrical ends.
We also introduce the indicial family $\widehat P(\tau)$, which is the Fourier transform
of $\In(P)$. See \cite{GrieserBCalc, MazzeoMelroseAsian, MelroseAPS, SchulzeWongBCalc}
for the closely related $b$-calculus.

In the sixth section, we extend the ``jump/limit' results of the forth section to the case
of Riemannian manifolds with straight cylindrical ends. We thus study here the normal lateral
limits on manifolds with straight cylindrical ends. To this end, we develop an ``algebra tool kit''
to study the limit/jump relations. We prove, for example, that normal
lateral limits are compatible with limits at infinity (i.e., with the limit operator
map $\In$), a result that was not published before in a refereed journal (to our knowledge),
see Theorem \ref{thm2.indicial.jump1}. This result is key in the study of the layer potentials
in the remaining sections. In order to define the normal lateral limits on
manifolds, we need to develop some geometric tools, including the normal geodesics,
and to make some corresponding geometric assumptions, which we prove are satisfied for
our manifolds with cylindrical ends. Our generalized Stokes operator $\bsXi$ is a
differential operator of order two, but is not elliptic in the usual
sense. In order to use the theory of elliptic operators, we recall the definition
of \ADN-operators and ellipticity in this setting in the last subsection of section.
The normal lateral limit results extend right away to \ADN-operators. The content
of this section is motivated by the need to establish the \ADN-ellipticity of the
Stokes operator $\bsXi$. Our results extend those in \cite{KMNW-2025, KNW-2025}.

The seventh section is devoted to the study of some of the main technical results on a
cylinder. We begin by studying the form of our differential operators ($\bsXi$,
$\Def$, and $\nabla$) on a cylinder $\mfkS \times \RR$ (which is a
particular case of a manifold with cylindrical ends), where $\mfkS$ is a boundaryless, smooth,
compact manifold (i.e., a
closed manifold). Then we extend the results of
Section \ref{sec.Green} on Green formulas and energy estimates to the cylinder $\mfkS \times \RR$.
We also extend the Green formulas and the energy estimates of Section \ref{sec.Green}
to the indicial family $\widehat{\bsXi}(\tau)$ of $\bsXi$.
As an application, we obtain criteria for the invertibility of the translation invariant,
generalized Stokes operator $\bsXi$
on the cylinder $\mfkS \times \RR$, under suitable positivity assumptions on $V$ and $V_{0}$
(Assumption \ref{assumpt.VV0weak}).

In Section \ref{sec.8}, we begin the study of the main operator of this paper, namely,
the generalized Stokes operator $\bsXi := \bsXi_{V, V_{0}}$ (see Equation \eqref{eq.def.bsXi})
\emph{acting on a manifold with straight cylindrical ends.}
(Recall that, if $V$ and $V_{0}$ vanish, then $\bsXi$ becomes the usual Stokes operator.)
We first obtain \emph{Fredholm and invertibility results} for
the Stokes operator $\bsXi$ under suitable positivity conditions on the potentials
$V$ and $V_{0}$ (Assumption \ref{assumpt.VV0weak} and, respectively, Assumption
\ref{assumpt.VV0}). These results are interesting in themselves. Moreover, they play a crucial
role in our paper, as these Fredholm and invertibility results are used then (in this
section) to define the various layer potentials for the Stokes operator.
In the last subsection of this section, we prove the jump and limit
properties of these potentials. (These results were proved for the usual Stokes
operator on Lipschitz domains in \cite{M-W}, but our method is different.) The proof of
the jumpt/limit relations is based on the ``algebra tool kit'' developed for
this purpose in the fourth and sixth sections of this paper.

In the ninth section, we specialize again to the case of a cylinder $M = \mfkS \times \RR$
with an invariant metric and study the Fourier transforms (or indicial operators)
of the corresponding layer potential operators. In combination with the Green formulas
on cylinders, this allows us to establish the invertibility of the main layer
potential operators $\bsS$ and $\frac12 + \bsK$ on a cylinder. This result is interesting in
itself and is then used to study these operators on a general manifold with straight
cylindrical ends.

The tenth section is devoted to Fredholmness and invertibility of the Stokes layer potential
operators when $M$ is a general manifold with straight cylindrical ends, depending on the
properties of our potentials $V$ and $V_{0}$. Thus, under suitable positivity assumptions on
$V$ and $V_{0}$, we prove the
invertibility of the single layer potential operator $\bsS$ and of the operator $\frac{1}{2} +\bsK$,
where $\bsK$ is the corresponding double layer potential operator, on Sobolev spaces on the
boundary of a domain with straight cylindrical ends. These are some of the main results of
this paper and rely on the results developed in the previous sections. (The presentation
in this paper is such as to lead as quickly as possible to this invertibility result, while
preserving clarity and completeness, but sacrificing occasionally the generality.)

In the eleventh section, we prove the main well-posedness result of this paper, that is,
the well-posedness of the Dirichlet problem for the generalized
Stokes system, Theorem \ref{thm.main.WP}, assuming that
$\Omega$ is a smooth domain with straight cylindrical ends in a connected manifold with straight
cylindrical ends $M$ and suitable positivity assumptions on $V$ and $V_{0}$. Our proof uses
layer potentials. This well-posedness result is useful to define the Dirichlet-to-Neumann
operator for the generalized Stokes system and to prove that the conormal derivative of the
Stokes double layer potential has no jump across the boundary of the domain $\Omega $, Theorem
\ref{jump-conormal-dl}.

The twelfth (and last) section is devoted to applications and extensions.
As an application, we obtain an existence and uniqueness result for
the Dirichlet problem for the generalized Navier-Stokes system \eqref{NS-Dirichlet} in a smooth
domain $\Omega $ with straight cylindrical ends assuming
$k\in \ZZ _+$, $\operatorname{dim}\, M =n < 2(k+3)$, that the given data is in
$H^k(\Omega ;TM)\oplus H^{k+3/2}(\partial \Omega ; TM)$
and is sufficiently small in a sense that is described in detail in Theorem \ref{thm.main.NS}.
The proof is based on the well-posedness of the Dirichlet problem
for the non-homogeneous generalized Stokes system, Theorem \ref{thm.main.WP-non-1}, and a
fixed point theorem.
In the last subsection of this section (and of this paper), we state the well-posedness of
the Dirichlet problem for the generalized Stokes system on a closed manifold \cite{KNW-2025}.
We also explain what modifications need to be brought to the statements of our results
for manifolds with straight cylindrical ends so that they are valid also in the case of
closed manifolds. These results were proved in \cite{KNW-2025} (for closed manifolds), so
in this paper we content ourselves with reminding a few of these results.

\subsection{Overview of the existing literature}
The layer potentials method is a classical method for solving elliptic boundary value problems
\cite{Chandler-Wilde-1, Costabel88, Massimo-book, D-M, Galdi, H-W, JerisonKenig95, J-K,
Khavinson-Putinar, KMNW-2025, Ladyzhenskaya, Lean, D-I-M-GHA, SchwabBook, 24}.
It plays a crucial role in the analysis of elliptic boundary value
problems in various function spaces, their main advantage consisting in the fact that they
provide explicit formulas for the the solutions and that they provide a convenient way of extending
the results to $L^{p}$-type Sobolev spaces. There are also results that can be proved
at this time only using layer potential methods. An example is provided by our well-posedness
theorem, Theorem \ref{thm.main.WP}. More on the history of the method of layer potentials
can be found, for example, in \cite{Agr, Fa-Ke-Ve, Galdi, H-W, J-K-81, Khavinson-Putinar,
Ladyzhenskaya, Lean, Russo, Varnhorn}. See also
\cite{epsteinBook, GNS08, J-K-82}.

Let us now review some of the more specialized results in the subject,
most of them from research papers devoted to the analysis of the Stokes
and Navier-Stokes equations on domains, in Euclidean spaces, or on Riemannian manifolds.
Fabes, Kenig and Verchota \cite{Fa-Ke-Ve} initiated the study of boundary value problems for
the Stokes system in Lipschitz domains of $\RR^n$ with data in $L^2$-based Sobolev spaces
and obtained various mapping properties for the corresponding layer potentials. They also
obtained well-posedness results for the Dirichlet and Neumann problems in $L^p$-based spaces,
with $p$ in a neighborhood of $2$, by combining Rellich type formulas with layer potential
approaches. Mitrea and Wright \cite{M-W} used layer potential methods and extended the
results in \cite{Fa-Ke-Ve} to a variety of $L^p$, Sobolev, Bessel potential, and Besov
spaces. They also obtained well-posedness results for the main boundary value problems
for the constant coefficient Stokes system in arbitrary Lipschitz domains in Euclidean
spaces, together with the determination of the optimal range for $p$.
The authors of \cite{K-L-M-W} obtained mapping properties for the constant-coefficient
Stokes and Brinkman layer potential operators in standard and weighted Sobolev spaces
on $\RR^3$. Shen \cite{Shen} used a layer potential approach and obtained resolvent
estimates in $L^p$ spaces for the Stokes operator in Lipschitz domains in the Euclidean
spaces. Layer potential theoretical methods have been also combined with fixed point
theorems in the analysis of boundary value problems for (linear and nonlinear) elliptic
systems with nonlinear boundary conditions. For instance, the authors in \cite{K-L-W}
used a layer potential approach and the Leray-Schauder fixed point theorem to show the
existence result in $L^p$, Sobolev, and Besov spaces for the constant-coefficient Stokes
and Brinkman systems with nonlinear Neumann-transmission conditions. Girault and
Sequeira \cite{Gi-Se} used a variational approach to prove the well-posedness of
the exterior Dirichlet problem for the Stokes system on exterior Lipschitz domains
in $\RR^n$, $n=2,3$. Galdi \cite[Chapters 5, 11, 12]{Galdi} used variational
approaches in the analysis of steady Stokes and Navier-Stokes flows in exterior domains.
Medkov\'{a} \cite{medkova26} used an integral equation method to analyze the Dirichlet
problem for the Brinkman and Darcy-Forchheimer-Brinkman systems in Sobolev spaces on
unbounded domains with compact Lipschitz boundary in $\RR^n$. Exterior Dirichlet
and transmission problems for anisotropic {Stokes} and {Navier}-{Stokes} systems with
$L_\infty$ variable  coefficients have been studied in \cite{KMW-DCDS2021} and
\cite[Chapters 5-7]{KMNW-2025} using a variational approach and in \cite{medkova09}.
Other boundary problems for the Stokes system in Sobolev spaces on Lipschitz domains via
layer potential theoretical methods have also been studied in \cite{Medkova13, Medkova15b,
Medkova15a, Russo, Sa-Se, Varnhorn-2004}.

Chandler-Wilde et al. \cite{Chandler-Wilde-Perfekt, Chandler-Wilde-1} studied boundary
integral equations on locally-dilation-invariant Lipschitz domains and fractal screens.
Cakoni, Hsio, and Wendland have studied a mixed boundary value problem for the
biharmonic equation using boundary integral equations.
Haller-Dintelmann, R. and Jonsson, A. and Knees, D. and Rehberg
have studied the regularity of a mixed, second order
boundary value problem in divergence form \cite{H-J-K-R}.
Singularly perturbed boundary value problems based on a functional analytic approach
have been studied by M. Dalla Riva, M. Lanza de Cristoforis, and P. Musolino
\cite{Massimo-book}. This approach, which also uses a layer potential analysis, is very
useful in the study of various linear nonlinear elliptic problems. A nonvariational form
of the acoustic single layer potential has been obtained by Lanza de Cristoforis
\cite{Massimo-CPAA-26}. A nonvariational form of the Neumann problem for H\"{o}lder
continuous harmonic functions, based also on a layer potential analysis, has been
obtained in \cite{Massimo-JDE-25}. Ando, Kang, Miyanishi,
and Putinar have obtained Carleman factorizations of layer potentials \cite{AKMP}. The
recent book \cite{D-I-M-GHA} of D. Mitrea, I. Mitrea and M. Mitrea is a rigorous
interplay between Harmonic Analysis, Geometric Measure Theory, Function Space Theory,
and Partial Differential Equations, and provide a powerful tool in the analysis of
boundary problems for complex coefficient elliptic systems in various geometric settings,
including the class of Lipschitz domains. The theory of Fredholm and layer potential
operators in Euclidean spaces and in Riemannian manifolds plays an important role in
their book \cite{D-I-M-GHA} (and also in our book \cite{KMNW-2025}).
See also \cite{CaKr13, CYKL18, CDDM}.

Layer potentials play an important role in applications, for instance in
the study of numerical methods. In this direction, in an important paper,
Costabel \cite{Costabel88} has establishes many
useful properties of layer potentials using elementary methods.
Buffa, Hiptmair, Petersdorff, and Schwab \cite{Hiptmair} studied numerical
methods for Maxwell's transmission
problems in Lipschitz domains by using boundary element methods. Labarca-Figueroa and
Hiptmair \cite{Hiptmair25} used coupled boundary and volume integral equations in the
study of electromagnetic scattering. Steinbach and Wendland \cite{Ste-We} have
obtained resulsts on a Neumann series method for solving boundary integral equations.
Costabel, Stephan, and Wendland have studied boundary integral equations of the first kind
for the bi-{Laplacian} in a polygonal plane domain \cite{CSW}.
See also \cite{ABMH20, CoDa96, EGK, MunnierBi, MunnierCusp, Pacheco-Steinbach}.

The layer potential methods are also very useful in the analysis of elliptic boundary
value problems on manifolds. Next we provide a brief overview of some of the important
contributions related especially to the Stokes and Navier-Stokes equations in the setting
of compact or non-compact Riemannian manifolds.
Mitrea and Taylor \cite{M-T} and Dindo\u{s} and Mitrea \cite{D-M} obtained the well-posedness
of the Dirichlet problem for the Stokes and Navier-Stokes systems with smooth coefficients
on Lipschitz domains in closed Riemannian manifolds and in $L^p$, Sobolev and Besov spaces,
by using the mapping properties of Stokes layer potentials,
(see also \cite{KNW-2025} for similar results for a modified Stokes operator).
In \cite{K-M-W}, the authors studied transmission problems for the smooth coefficients
Navier-Stokes type equations in Lipschitz domains on closed Riemannian manifolds by using
a layer potential approach combined with a fixed point theorem.
Mitrea and Monniaux \cite{M-Mon1} studied the analyticity of the semigroup generated by the
Stokes operator equipped with Neumann-type boundary conditions on $L^p$-spaces on Lipschitz
domains of closed Riemannian manifolds.
In \cite{K-W}, the authors proved the well-posedness of Dirichlet problems for the Stokes and
Navier-Stokes systems with $L^{\infty }$-variable coefficients in $L^2$-based Sobolev spaces
in Lipschitz domains on closed Riemannian manifolds. They used a variational approach
to define the Newtonian and layer potentials for the non-smooth coefficients Stokes
system on Lipschitz surfaces in closed Riemannian manifolds. Benavides, Nochetto, and
Shakipov \cite{BNS} used a variational approach and studied the well-posedness and
regularity in $L^p$-based Sobolev spaces, $p\in (1,\infty )$, for the weak formulations
of the stationary tangent Stokes and tangent Navier-Stokes systems on a
compact and connected $d$-dimensional manifold without boundary of class $C^m$, $m\geq 2$,
embedded in $\RR^{d+1}$, in terms of the regularity of the source terms and the manifold.

The analysis on non-compact manifolds poses more challenges than that the one on compact manifolds.
A significant issue in this sense is the construction of the classes of pseudodifferential
operators. In this regard, we refer to the contributions of Ammann, Lauter and Nistor
\cite{aln1, ALN-2007},
Mazzeo and Melrose \cite{MazzeoMelroseAsian}, Melrose \cite{MelroseActa}, Melrose and Mendoza
\cite{Melrose-Mendoza}, Nistor \cite{Nistor-05}, Nistor, Weinstein, and Xu \cite{NWX}
Schrohe \cite{Schrohe, SchroheFrechet}, Schulze \cite{SchulzeBook91, Schulze},
Schulze and Wong \cite{SchulzeWongBCalc}, Shubin \cite{Shubin}, to mention just a few
among very many others. For instance, some results relevant for boundary value problems
were obtained by Grieser \cite{Fr-Grrieser-Sc, GSV26}, who used the Calderon projection in the
framework of the cusp-edge calculus \cite{MazzeoMelroseAsian} for related boundary value
problems for (classically) elliptic operators. The cusp-edge calculus is a completion of a
groupoid calculus \cite{NWX} corresponding to a structural Lie
algebra of vector fields that is, in fact, more general than the one appearing in our
setting of manifolds with straight cylindrical ends.

Another important issue in the analysis on non-compact manifolds
is finding the ``right definition'' of the relevant function spaces. In this paper,
as in \cite{KohrNistor1}, we have defined our function spaces by using connections. See
also \cite{AmannSobolev, AmannFunctSp, Schneider}.
Amann \cite{AmannFunctSp} and Ammann, Gro{\ss}e, Nistor
\cite{AGN-1, AGN-2} studied function spaces on manifolds with bounded geometry
and obtained well-posedness results for the Laplace operator. Using the
uniform Shapiro-Lopatinski conditions, Gro{\ss}e and Nistor \cite{GN17} obtained
regularity results for elliptic boundary value problems on manifolds
with bounded geometry. Lie manifolds are special manifolds with bounded geometry
\cite{sobolev, aln1}, which allow the implementation of the layer potential methods.

The layer potential theory for the Laplace operator on manifolds with cylindrical ends was
initiated by Lewis and Parenti \cite{LewisParenti} and further extended by Mitrea and Nistor
\cite{Mitrea-Nistor}. Manifolds with cylindrical ends are very useful
in the study of boundary value problems on manifolds with conical points. See, for example,
the papers \cite{sobolev, BMNZ, CSS07-CPDE, daugeBook, GrieserBCalc, Kondratiev, Ko-Ma-Ro,
Kral-Wendland, LauterMoroianu1, Ma-Ro, Ma-Ni, MelroseAPS, SchulzeWongBCalc},
among many others devoted to the subject. The reason is provided by the fact that a domain with
a conical point is mapped
via the Kondratiev transformation \cite{Kondratiev} to a manifold with cylindrical ends. This
fact was exploited by Qiao and Nistor \cite{Qiao-Nistor} to obtain a solvability result for the
Dirichlet problem for the Laplace operator in an open straight cone in $\RR^{n}$, $n\ge 3$, in
the setting of Kondratiev-Babu\v{s}ka weighted (Sobolev) spaces.
Roidos, Schrohe and Seiler \cite{Ro-Sch-Se} studied the $H_\infty $-calculus in $L^p$-based
Sobolev spaces for parameter-elliptic boundary value problems on manifolds with conical
singularities at the boundary. As an application they treated the porous medium equation on
manifolds with conical singularities (see also \cite{Ro-Sch-16} and the references
therein). Mitrea and Nistor \cite{Mitrea-Nistor} and Kohr, Nistor and Wendland
\cite{KMNW-2025, KNW-22, KohrNistor-Stokes} developed an essentially translation
invariant pseudodifferential calculus on manifolds with cylindrical ends, which is
very useful to provide the invertibility and structure of the Stokes operator and,
as a consequence, the construction and the invertibility the Stokes layer potential
operators on manifolds with cylindrical ends. Examples of domains with conical points
are the polygonal domains and some polyhedral domains. Manifolds with cylindrical ends
are particular cases of Lie manifolds \cite{aln1, CarvalhoNistorQiao} and of manifolds with
bounded geometry \cite{AGN-1, AGN2, AGN-2} (see also the references therein).
The papers by
Carvalho and Qiao \cite{CarvalhoQiao}, Qiao and Li \cite{QiaoLi}, Qiao \cite{QiaoEdge},
Putinar and Perfect \cite{PutinarPerfekt1, PutinarPerfekt2} contain
results on layer potential on polyhedral domains. See also \cite{Ammari19b, Ammari19}.

In \cite{Grosse-Kohr-Nistor-23}, Gro{\ss}e, Kohr, and Nistor have obtained the $L^2$-unique
continuation property for the deformation operator $\Def$ on manifolds with bounded geometry,
a result that was used for the analysis of the Stokes various classes of manifolds,
including manifolds with cylindrical ends, in \cite{KMNW-2025, KNW-22, KohrNistor-Stokes}.
Manifolds with cylindrical ends appear in may applications that involve
``wave-guides'' or ``layers'' (see e.g. \cite{bbendhia21, BCN18, bbendhia23, bbendhia94,
BNSV22, BNSV23, BBKO, bourgeois23, epsteinP1, epsteinP2, EFK, KMNW-2025, KohrNistor-Stokes,
KrejcirikZuazua, Wendland24}). Let us finally add that the main operator in the structure
of the Stokes and Navier-Stokes equations on Riemannian manifolds is the deformation Laplacian
$\bsL := 2\Def^*\Def$. In their groundbreaking paper \cite{Ebin-Marsd}, Ebin and Marsden described
the deformation Laplacian as the right Laplace type operator to describe the Navier-Stokes
equation on compact Riemannian manifolds. Pierfelice \cite{Pierfelice} also used the
deformation-Laplacian for the definition and the analysis of the Navier-Stokes equation on
a non-compact Riemannian manifold with negative Ricci curvature.
This is the choice that we also adopt in the description of our Stokes operator
$\bsXi$ of Equation \eqref{eq.def.bsXi}. Our generalized Stokes operator is
closely related to the Brinkman operator. The Brinkman flow was studied
by Karageorghis, Lesnic, and Marin in \cite{KLM21b, KLM21}, and B\u{a}cu\c{t}\u{a}, Hassell, Hsiao, and
Sayas in \cite{B-H}. A class of solutions to the Navier-Stokes equations for flows in plane-parallel
channels and related singular perturbation problems were studied by Mazzucato and Taylor
\cite{Mazzucato-Taylor}.

\subsection{Acknowledgements}
We thank Massimo Lanza de Cristoforis, Dorina, Irina, and Marius Mitrea, Sergey
Mikhailov, and Mihai Putinar for useful discussions. We also thank Daniel Grieser
and Karl-Mikael Perfekt for useful suggestions.

\section{\cn Green formulas for the Stokes operator}
\label{sec.Green}

In this section we introduce the basic differential operators used for the study of our
(generalized) Stokes operators and establish Green formulas and energy estimates for
the Stokes operator. More precisely, in the first part of this section, Subsection
\ref{ssec.do}, we introduce and study the deformation operator $\Def$, Equation
\eqref{eq.def.Def}. The deformation operator appears in the definition of the
the Stokes operator $\bsXi = \bsXi_{V, V_{0}}$ introduced in Equation \eqref{eq.def.bsXi}.
We also introduce a few other basic differential operators needed for the study of
the Stokes operator. In the second subsection, we recall an abstract integration by
parts formulas, Proposition \ref{prop.bdry.op}. The last subsection,
Subsection \ref{ssec.Green1} establishes the needed Green function formulas for the Stokes
operator, as well as some related energy estimates and representation formulas.

The material contained in this section is thus mostly background material, thus we also
introduce here other needed concepts or notation, mostly related to the differential operators
appearing in the study of the Stokes operator. The results of this section, while mostly
known, are sometimes hard to find in the literature. We thus provide here a complete and
largely self-contained presentation of these results (a different approach is contained in
Section A.3.3 of our book \cite{KMNW-2025}).

In this section, \emph{we assume that $(M, g)$ is a Riemannian manifold, occasionally
with some additional properties. More precisely, in the third subsection of
this section, we assume that $M$ has bounded geometry.} Recall that a manifold with
bounded geometry is one that has positive injectivity radius and all the covariant derivatives
of its curvature are bounded. (See \cite{Grosse-Kohr-Nistor-23} for references and for
more on manifolds with bounded geometry. We note that the results of this section will
be used \emph{only when $M$ has straight cylindrical ends}, and that the manifolds with
straight cylindrical ends have bounded geometry.)

\subsection{Differential operators}
\label{ssec.do}
We now introduce the basic differential operators used for the study of our
(generalized) Stokes operators. If $E \to M$ is a (smooth) vector bundle, then
$\CI(M; E)$ denotes the set of smooth sections of $E$, whereas $\CIc(M; E) \subset
\CI(M; E)$ denotes the subset of compactly supported (smooth) sections of $E$.
When necessary, the real vector bundles and their sections will be complexified.
All our vector bundles will be \emph{smooth.} If $E = \CC$, it will
be omitted from the notation, thus $\CI(M)$ denotes the set of smooth functions on $M$.

One of the most basic differential operators is the \emph{derivation}
\begin{equation*}
  X : \CI(M) \to \CI(M)
\end{equation*}
defined by a smooth vector field $X$. (Recall that a smooth vector field is a section
of the \emph{tangent} bundle $TM \to M$, whose dual is the \emph{cotangent} bundle to
$M$ and is denoted $T^{*}M$.) Another very basic differential operators is then the
\emph{Levi-Civita connection}
\begin{equation*}
  \nabla^{LC} : \CI(M; TM) \to \CI(M; T^{*}M \otimes TM)\,.
\end{equation*}
(Recall that the Levi-Civita connection $\nabla^{LC}$ is the unique torsion-free,
metric preserving connection on $TM$.) The Levi-Civita
connection on $TM$ then induces connections on the tensor bundles $T^{\otimes k} M
\otimes T^{*\otimes l}M$, also be denoted $\nabla^{LC}$. If $X$ is a smooth vector
field on $M$, then contraction with $X$ (on the first component tensor)
defines a map $\iota_{X} : \CI(M; T^{*}M \otimes TM) \to \CI(M; TM)$ by
$\iota(X)(\omega \otimes Y) := \omega(X) Y$, where $\omega \in
\CI(M; T^{*}M)$ (i.e., it is a 1-form) and $Y \in \CI(M; TM)$ (i.e., $Y$ is a smooth
vector field). This yields the differential operator
\begin{equation*}
  \nabla_{X}^{LC} \ede \iota_{X} \circ \nabla^{LC}  : \CI(M; TM) \to \CI(M; TM)
\end{equation*}
that will often be used below. We obtain similar maps
$\nabla_{X}^{LC} \ede \iota_{X} \circ \nabla^{LC}  : \CI(M; T^{*\otimes j} M
\otimes T^{\otimes k} M) \to \CI(M; T^{*\otimes j} M
\otimes T^{\otimes k} M)$ between sections of tensor bundles.

We shall often use the \emph{musical isomorphism} $\sharp : T^{*}M \to TM$
defined by the metric $g$, explicitly, if $X$ and $Y$ are vector fields
on $M$, then
\begin{equation*}
  \sharp : T^{*}M \to TM \ \mbox{ satisfies }\
  \langle X^{\sharp}, Y \rangle \ede X^{\sharp} (Y) \ede g(X, Y)\,.
\end{equation*}
The inverse of $\sharp$ is denoted by the same
symbol: $\sharp : TM \to T^{*}M$. One should be careful not to confuse
$\nabla^{LC}X \in \CI(M; T^{*}M \otimes TM)$,
$X \in \CI(M; TM)$ (a smooth vector field),
with the \emph{gradient} $\nabla f := (df)^{\sharp} \in \CI(M; TM)$,
$f \in \CI(M)$, a smooth function.

A crucial role in this paper will be played by the
\emph{deformation operator}
\begin{equation*}
  \Def : \CI(M; TM) \to \CI(M; T^{*}M \otimes T^{*}M)\,,
  \quad \Def(X) \ede \frac12 \maL_{X}g\,.
\end{equation*}
Let $X, Y$, and $Z$ be smooth vector
fields on $M$ and let $X \cdot Y = g(X, Y)$ denote the scalar product induced by the
metric $g$ on $M$. The following alternative defining formula of $\Def$ will be more useful:
\begin{equation}\label{eq.def.Def}
    \Def(X)(Y, Z) \seq \langle \Def(X), Y \otimes Z \rangle
    \seq \frac12 \big [ (\nabla_{Y}^{LC} X) \cdot Z
    + (\nabla_{Z}^{LC} X) \cdot Y \big]\,.
\end{equation}

To introduce other needed operators, let $\bsnu$ be a fixed smooth vector field on $M$
(in our applications, $\bsnu$ will extend the \emph{outer unit normal}
vector field to a domain $\Omega$ on which our boundary value problems are
formulated; this explains why, in this paper, we require $\Omega$ to be
\emph{on one side of its boundary}, that is, to satisfy Assumption
\ref{assumpt.1side}). The chosen
vector field $\bsnu$ induces a map
\begin{equation*}
  \bsnu \otimes \sharp : T^{*}M \otimes T^{*}M \to TM\,, \quad
  (\bsnu \otimes \sharp)(\xi \otimes \eta) = \xi(\bsnu) \eta^{\sharp}\,.
\end{equation*}
The map $\sharp \otimes \bsnu : T^{*}M \otimes T^{*}M \to TM$ is defined analogously.
We then let
\begin{equation}
  \label{eq.def.Dnu}
  \begin{gathered}
    \Dnu  : \CI(M; TM) \to \CI(M; TM)\\
    \Dnu  X \ede \frac12 (\bsnu \otimes \sharp +
    \sharp \otimes \bsnu) \, \Def(X) \seq \big\langle\, \Def(X)\, ,
    \, \bsnu \otimes 1\, \big\rangle^{\sharp}     \,.
  \end{gathered}
\end{equation}
The operator $\Dnu $ and its (formal) adjoint $\Dnu ^{*}$ will play an important role
in what follows. (In the last formula, $\langle \ \cdot \ ,\, \bsnu \otimes 1\,\rangle
: T^{*}M \otimes T^{*}M \to \CC \otimes T^{*}M =T^{*}M$ is the contraction
with $\bsnu$ on the first variable.)
This allows us to introduce the operator
\begin{equation}\label{def.conormal.Tderiv}
  \begin{gathered}
    \bop : \CI(M; TM \oplus \CC) \to \CI(M; TM)\\
    \bop U \ede -2 \Dnu(\bsu) + p \bsnu \,,
  \end{gathered}
\end{equation}
where $U := (\bsu\ \  p)^{\top }$ is a smooth, compactly supported section of
$TM \oplus \CC$ (see also formula \eqref{eq.def.Dnu}). The formula
$\Dnu X := (\Def(X) \bsnu \otimes 1)^{\sharp}$ is also used \cite{D-M, M-T, 24}.
We shall also write
\begin{equation}\label{def.conormal.Tderiv2}
  \tbop  U \ede  \cvector{-2 \Dnu(\bsu) + p \bsnu}{0} \seq
  \left(
    \begin{array}{cc}
      -2 \Dnu & \bsnu \\
      0 & 0
    \end{array}
  \right)
  U\,.
\end{equation}

Let $V , V_{0} : M \to [0, \infty)$ be measurable functions (usually smooth).
Recall from the introduction
that the \emph{deformation Laplacian} is the second order differential operator
$\bsL \ede 2\Defstar\Def$. The operator
\begin{equation*}
  \bsL_{V} \ede 2\Defstar \Def + V
\end{equation*}
will be called the \emph{perturbed deformation Laplacian.} Also, recall that the
\emph{generalized Stokes operator} (Equation \eqref{eq.def.bsXi}) is the operator
\begin{equation*} 
  \bsXi \ede \bsXi_{V, V_{0}}  \ede \left(\begin{array}{ccc}  \bsL_{V} & \nabla \\
  \nabla^* &  -V_{0}
  \end{array}
  \right)  \in \End(\CIc(M; TM \oplus \CC))\,.
\end{equation*}
(Recall that for a module $E$ over some ring $R$, $\End(E)$ denotes the set
of $R$-linear maps $E \to E$, that is, the set of \emph{endomorphisms} of $E$.)

We now study some of the properties of these operators.
Then a direct calculation gives right away the following result.

\begin{lemma}\label{lemma.bop.star}
  We have
  \begin{equation*}
    \bopstar (\bsu) \seq
    \left (
      \begin{array}{c}
      - 2\Dnu^{*}\bsu \\
      \bsnu \cdot \bsu
      \end{array}
    \right )
    \seq \left ( \begin{array}{c}
      - 2\Dnu^{*} \\
      \bsnu^{\sharp}
    \end{array}
    \right )\, \bsu \seq
    \left(
    \begin{array}{cc}
      -2 \Dnu^{*} & 0 \\
      \bsnu^{\sharp} & 0
    \end{array}
  \right)
  U  \,.
  \end{equation*}
\end{lemma}

In this paper, $\imath$ will denote \emph{the imaginary unit:} $\imath^{2} = -1$.
The \emph{principal symbols} of the operators of interest can be determined from the
functorial properties of the principal symbol using also the following formula
(stationary phase approximation \cite{Hormander1, H-W, Taylor1}).

\begin{lemma}\label{lemma.stat.phase}
  Let $\phi : M \to \RR$ be a smooth function and $P$ be a first order
  differential operator $\CIc(M; E) \to \CIc(M; F)$, where $E, F \to M$ are
  smooth vector bundles. Then
  \begin{equation*} 
    \sigma_{1}(P; d\phi(x)) u(x) \seq \lim_{t \to \infty}
    t^{-1} e^{-\imath t \phi(x)}
    \big[ P(e^{\imath t \phi}u) \big](x)\ \mbox{ if }\ d\phi(x) \neq 0\,.
  \end{equation*}
\end{lemma}

The principal symbol will be discussed in general in Section \ref{sec.nllRn},
see Definition \ref{def.princ.symb}. For more details on the principal symbol,
see the above mentioned textbooks \cite{Hormander1, H-W, Taylor1}
or \cite{KMNW-2025} (or any basic monograph
on pseudodifferential operators).

We shall use repeatedly in what follows the following notation.

\begin{notation} \label{not.def.endomorphism}
  Let $\mathfrak W$ be a vector space and let $v \in \mathfrak W$ and $\xi \in
  \mathfrak W^{*}$. We shall then regard $v \otimes \xi \in \mathfrak W \otimes
  \mathfrak W^{*}$ as an endomorphism $v \otimes \xi \in \End(\mathfrak W) \simeq
  \mathfrak W \otimes \mathfrak W^{*}$ via the formula
  $(v \otimes \xi)x := \xi(x) v$, for any $x \in \mathfrak W$.
  In particular, if $\mathfrak W$ is a real vector space with an inner product
  denoted $\cdot$ and $\sharp : \mathfrak W \to \mathfrak W^{*}$ is the musical
  isomorphism (i.e., $v^{\sharp}(w) := v \cdot w$), then,
  for all $v, w, x \in \mathfrak W$, we have
  \begin{equation*}\label{eq.def.endomorphism}
    \begin{gathered}
      (v \otimes w^{\sharp}) x \ede  (x \cdot w) v \ \  \mbox{ and}\\
      (v \otimes w^{\sharp})^{*} \seq  w \otimes v^{\sharp} \,,
    \end{gathered}
  \end{equation*}
  where, for the last formula, we have used the fact that $\mathfrak W$ is
  a real vector space.
\end{notation}

For any vector space or vector bundle $E$, we let $S^{2}E \subset E\otimes E$
be the symmetric part of the tensor product $E \otimes E$.
The following formulas are well known,
see \cite{D-M, M-T, Taylor1}. We state them here for further use and
we prove them for completeness.

\begin{proposition}\label{prop.prop.Def}
  Let $X$, $Y$, and $Z$ be three smooth vector fields on $M = M' \times \RR$. Then
  \begin{enumerate}[\rm (i)]
    \item $\sigma_{1}(\Def; \xi)X \seq \frac{\imath}2 \big[ \xi \otimes X^{\sharp}
      + X^{\sharp} \otimes \xi] \in S^{2}T^{*}M
      \subset T^{*\otimes 2}M$ and
    \item $\sigma_{1}(\Defstar; \xi) \big( Y^{\sharp} \otimes Z^{\sharp} \big)
      \seq -\frac{\imath}2\big [ \xi(Y)Z + \xi(Z)Y \big]\,.$
  \end{enumerate}
\end{proposition}

Note that, in the last equation, $\sigma_{1}(\Defstar; \xi)$ was tacitly extended
to act on all tensors $T^*M \otimes T^*M$ (not just on $S^{2}T^{*}M
\subset T^{*\otimes 2}M$). Nevertheless, this extension is canonical, in the
sense that it vanishes on anti-symmetric tensors:
\begin{equation*}
  \sigma_{1}(\Defstar; \xi) \big( Y^{\sharp} \otimes Z^{\sharp}
  - Z^{\sharp} \otimes Y^{\sharp}\big)
  \seq -\frac{\imath}2\big [ \xi(Y)Z + \xi(Z)Y
  - \xi(Z)Y - \xi(Y)Z \big] \seq 0\,.
\end{equation*}

\begin{proof}
  Let $x \in M$ and $\phi : M \to \RR$ be a smooth function such that $d\phi(x) \neq 0$.
  Let us replace $X$ with $e^{\imath t \phi}X$ in
  Equation \eqref{eq.def.Def}. Lemma \ref{lemma.stat.phase} then gives:
  \begin{align*}
    \langle \sigma_{1}(\Def; d\phi)X, Y \otimes Z\rangle & \seq
    \lim_{t \to \infty} \frac{e^{-\imath t \phi}}{2t}
     \big [ Z \cdot \nabla_{Y}^{LC} ( e^{\imath t \phi} X)
    + Y \cdot \nabla_{Z}^{LC} (e^{\imath t \phi} X) \big]\\
    & \seq \frac{\imath}2 \big [ d\phi(Y)X \cdot Z + d\phi(Z)X \cdot Y  \big]\\
    & \seq \frac{\imath}2 \langle d\phi \otimes X^{\sharp} + X^{\sharp} \otimes d\phi ,
    Y \otimes Z \rangle \,.
  \end{align*}
  This proves (i). From this, using that the principal symbol is stable for
  adjoints (i.e., $\sigma_{m}(P^{*}; \xi) = \sigma_{m}(P; \xi)^{*}$), we obtain
  \begin{align*}
    (\sigma_{1}(\operatorname{Def}^{*}; \xi) \big( Y^{\sharp} \otimes Z^{\sharp}\big), X)
    & \seq (Y^{\sharp} \otimes Z^{\sharp}, \sigma_{1}(\Def; \xi) X)\\
    & \seq \big(Y^{\sharp} \otimes Z^{\sharp}, \frac{\imath}2 \big[ \xi \otimes X^{\sharp}
    + X^{\sharp} \otimes \xi] \big)\\
    & \seq -\frac{\imath}2 \big[ \xi(Y) X \cdot Z + \xi(Z) X \cdot Y \big] \\
    & \seq -\frac{\imath}2 \big( \xi(Y)Z + \xi(Z) Y,  X \big)\,.
  \end{align*}
  This proves also (ii) and completes the proof.
\end{proof}

For any first order differential operator $P : \CI(M; E) \to \CI(M; F)$, we let
\begin{equation}\label{eq.def.panu}
  \pa_{\bsnu}^{P} \ede -\imath \sigma_{1}(P; \bsnu^{\sharp})\,,
\end{equation}
where $\bsnu$ is the chosen vector field (in this section, this vector field is
only required to be smooth, but, beginning with the next subsection, we will require
$\bsnu$ to be outer unit normal at the boundary of our given domain $\Omega$).
This definition is motivated by the abstract integration formula of the
following subsection.

The next formulas are an immediate consequence of the formulas of
Propositions \ref{prop.bdry.op} and \ref{prop.prop.Def}. (See Section A.3.4
from \cite{KMNW-2025} for more details. Also, recall the formula
$v \otimes w^{\sharp}x \ede  (w\cdot x) v$ from Notation \eqref{not.def.endomorphism}.)

\begin{corollary} \label{cor.formulas0}
  The deformation operator $\Def$ of Equation \eqref{eq.def.Def} satisfies
  the following formulas:
  \begin{enumerate}[\rm (i)]
    \item $\sigma_{1}(\Def; \xi) =
    \frac{\imath}2\big[ \xi \otimes \sharp
    + \sharp \otimes \xi \big] \in \Hom(TM; T^{*\otimes 2}M)$;
    \item $\pa_{\bsnu}^{\Def} := - \imath \sigma_{1}(\Def; \bsnu^{\sharp})
    =  \frac12 \big[ \bsnu^{\sharp}
    \otimes \sharp + \sharp \otimes \bsnu^{\sharp} \big] \in \Hom(TM; T^{*\otimes 2}M)$;
    \item $\sigma_{1}(\Defstar; \xi) =
    - \frac{\imath}2\big[ \xi^{\sharp} \otimes \sharp
    + \sharp \otimes \xi^{\sharp} \big] \in \Hom(T^{*\otimes 2}M; TM)$;
    \item $\pa_{\bsnu}^{\Defstar} := - \imath \sigma_{1}(\Def^{*}; \bsnu^{\sharp})
    = - \frac12 \big[ \bsnu
    \otimes \sharp + \sharp \otimes \bsnu \big] \in \Hom(T^{*\otimes 2}M; TM)$;
  \end{enumerate}
\end{corollary}

\begin{proof}
  The first point is nothing but a reformulation of the first point of
  the last proposition. The second point follows right away from the first point.
  The third point point is obtained by applying again Proposition \ref{prop.bdry.op}.
  Indeed, the point (ii) of that proposition,
  namely, $\sigma_{1}(\Defstar; \xi) \big( Y^{\sharp} \otimes Z^{\sharp} \big)
  \seq -\frac{\imath}2\big [ \xi(Y)Z + \xi(Z)Y \big]$ gives right away
  \begin{equation*}
    - \frac{\imath}2\big[ \xi^{\sharp} \otimes \sharp
    + \sharp \otimes \xi^{\sharp} \big]\big( Y^{\sharp} \otimes Z^{\sharp} \big)
     \ede - \frac{\imath}2 \xi(Y) Z + \xi(Z) Y \\
     \seq \sigma_{1}(\operatorname{Def}^{*}; \xi)\big( Y^{\sharp} \otimes Z^{\sharp} \big)\,.
  \end{equation*}
  The last point follows from (iii) just proved using also
  that $\sharp$ is involutive, that is, $\sharp^2 = id$.
\end{proof}

We also obtain the following slightly more difficult formulas.

\begin{corollary} \label{cor.formulas}
  The deformation operator $\Def$ and the differential operator $\Dnu$ of Equation
  \eqref{eq.def.Dnu} satisfy the following formulas:
  \begin{enumerate}[\rm (i)]
    \item $\pa_{\bsnu}^{\Defstar}\Def = -\Dnu $;
    \item $\sigma_{1}(\Dnu ; \xi) = \frac{\imath}2 \big [ \xi(\bsnu)
     + \xi^{\sharp} \otimes \bsnu^{\sharp} \big]$;
    \item $\sigma_{1}(\Dnu ^{*}; \xi) = - \frac{\imath}2
    \big [ \xi(\bsnu) + \bsnu \otimes \xi \big]$; and
    \item $\sigma_{2}(\Defstar\Def; \xi) = \frac12(| \xi|^{2}
    + \xi^{\sharp} \otimes \xi)$.
  \end{enumerate}
\end{corollary}

\begin{proof}
  The point (i) follows directly from $\Dnu \ede \frac12 (\bsnu \otimes \sharp +
    \sharp \otimes \bsnu) \, \Def$, see Equation \eqref{eq.def.Dnu}, and of the formula
  $\pa_{\bsnu}^{\Defstar} = - \frac12 \big[ \bsnu \otimes \sharp
  + \sharp \otimes \bsnu \big] \in \Hom(T^{*\otimes 2}M; TM)$ of
  Corollary \ref{cor.formulas0}(iv).

  To prove (ii), we use again the definition of $\Dnu
  \ede \frac12 (\bsnu \otimes \sharp + \sharp \otimes \bsnu) \, \Def$ and Corollary
  \ref{cor.formulas0}(i) to obtain
  \begin{multline*}
    \sigma_{1}(\Dnu ; \xi) X  \seq \frac12 \big[ \bsnu
    \otimes \sharp + \sharp \otimes \bsnu \big]\, \sigma_{1}(\Def; \xi)X
    \seq \frac12 \big[ \bsnu
    \otimes \sharp + \sharp \otimes \bsnu \big]\, \frac{\imath}2\big[ \xi \otimes \sharp
    + \sharp \otimes \xi \big]X\\
    \seq \frac{\imath} 4 \big[ \bsnu
    \otimes \sharp + \sharp \otimes \bsnu \big]\, ( \xi \otimes X^{\sharp}
    + X^{\sharp} \otimes \xi)\\
    \seq \frac{\imath} 4 \big( \xi(\bsnu) X + (\bsnu \cdot X) \xi^{\sharp}
    + (\bsnu \cdot X) \xi^{\sharp} + \xi(\bsnu) X \big)\\
    \seq \frac{\imath} 2 \big( \xi(\bsnu) X
    + (\bsnu \cdot X) \xi^{\sharp} \big) \seq \frac{\imath} 2 \big( \xi(\bsnu)
    + \xi^{\sharp} \otimes \bsnu^{\sharp}\big) X\,.
  \end{multline*}
  The transformation properties of the principal symbol with
  respect to adjoints and the formula $(v \otimes w)^{*} = w^{\sharp} \otimes v^{\sharp}$
  give right away (iii).
  Finally, taking into account that $\sharp^{2} = id$ and that
  $ \xi^{\sharp}( \xi) = | \xi|^{2}$ and using the multiplicativity
  of the principal symbol, we obtain
  \begin{align*}
    \sigma_{2}(\Defstar\Def; \xi)
    &\seq \sigma_{1}(\Defstar; \xi) \sigma_{1}(\Def; \xi)\\
    &\seq \frac14( \xi^{\sharp} \otimes \sharp + \sharp \otimes \xi^{\sharp} )
    ( \xi \otimes \sharp + \sharp \otimes \xi^{\sharp} )
    \seq \frac12(| \xi|^{2} + \xi^{\sharp} \otimes \xi)\,,
  \end{align*}
  as claimed in (iv).
\end{proof}

\subsection{An abstract integration by parts formula}
\label{ssec.abs.int}
{\cn We continue to assume that $(M, g)$ is a Riemannian manifold.}
We now recall a general (abstract) integration by parts formula that
will often be used in conjunction with the formulas of the previous subsection.

Our integration by parts formulas will be used either on the manifold $M$
or on some open subset $\Omega \subset M$ with smooth boundary $\Gamma := \pa \Omega$.

\begin{assumption}\label{assumpt.1side}
  Let $\Omega \subset M$ be an open subset. Recall that $\overline{\Omega}$
  denotes the \emph{closure} of $\Omega$ and that $\Gamma$ denotes the boundary
  $\pa \Omega := \overline{\Omega} \smallsetminus \Omega$ of $\Omega$.
  \cn We require in this paper that $\Omega$ \emph{be on one side of its boundary.}
  This means that we require one of the following two equivalent conditions:
  \begin{enumerate}[\rm (i)]
    \item If $\Omega_{-} := M \smallsetminus \overline{\Omega}$, then
    $\Gamma := \pa \Omega = \pa \Omega_{-}$.
    \item $\Omega$ is the interior of $\overline{\Omega}$.
  \end{enumerate}
\end{assumption}

This assumption is necessary in order to define a smooth vector field
$\bsnu$ that, at $\Gamma$, coincides with the \emph{unit outer normal vector}
to $\Omega$.
In the case of integration over $M$, integration by parts will be performed for
\emph{compactly supported} functions (or sections), so there will be no boundary
terms. However, when using integration by parts on the domain $\Omega$, we will
obtain \emph{boundary terms.} It is the purpose of this subsection to give a general
formula for these boundary terms. There is no loss of generality to assume that
$\Omega$ is connected

We endow $\pa \Omega$ with the measure $dS_{g}$ induced by the metric on $M$.
This allows us to introduce the inner product on the boundary $\Gamma$ of $\Omega$.
Also, we let $u \cdot v$ denote the (pointwise) scalar product of two sections
of our Hermitian vector bundle $E \to M$, as before. We then let
\begin{equation}\label{eq.def.inner.prod}
  \begin{gathered}
    (u, v) \seq (u, v)_{\Omega} \ede \int_{\Omega} u(x) \cdot v(x)\dvol_{g}(x)
    \quad \mbox{ and}\\
    %
    (u, v)'
    \ede \int_{\Gamma = \pa \Omega} u(x) \cdot v(x) \,
    dS_{g}(x)
  \end{gathered}
\end{equation}
denote the $L^{2}$-inner product of sections of $E$ over $\Omega$ and,
respectively, over $\Gamma := \pa \Omega$.

We will need some additional formulas continuing the discussion of the previous
subsection. We include them in a remark.

\begin{remark}\label{rem.basic.ip}
  Recall that $\sharp : TM \to T^{*}M$ is the vector bundle isomorphism defined by the
  metric $g$ on $TM$. Also, if $E, F \to M$ are two Hermitian vector bundles, then recall
  that $\sigma_{1}(P; \xi) \in \Hom(E; F)$ is the value at $\xi \in T^{*}M$ of
  the principal symbol of an order one differential operator $P : \CI(M; E)
  \to \CI(M; F)$ and
  $\pa_{\bsnu}^{P} :=   -\imath \sigma_{1}(P; \bsnu^{\sharp})$
  (see Equation \eqref{eq.def.panu}). Let also $\dvol_{g}$ be the
  volume element associated to the metric $g$ on $M$.
  We then have the following simple results.
  \begin{enumerate}[\rm (i)]
    \item or any vector field $X$, we have \cite[page 49]{petersen:98}
    \begin{equation}
      \maL_{X}\dvol_{g} \seq \div(X) \dvol_{g}\,,
    \end{equation}
    where $\maL_{X}$ is the Lie derivative in the direction of $X$.
    \item If $P$ is a first order differential operator,
    then $\sigma_{1}(P^{*}) = \sigma_{1}(P)^{*}$, and, hence
    $\pa_{\bsnu}^{P^*} = - (\pa_{\bsnu}^{P})^{*}$.

    \item If $X$ is a smooth vector field on $M$ and $P = X$, then
    $\sigma_{1}(X; \xi) = \imath \xi(X)$, and hence
    \begin{equation*}
      \pa_{\bsnu}^{X} \seq -\imath \sigma_{1}(X; \bsnu^{\sharp})
      \seq - \imath \big ( \imath \bsnu^{\sharp}(X) \big) \seq
      \bsnu \cdot X \,,
    \end{equation*}
    which, as we will see in the next proposition,
    is consistent with the well-known \emph{divergence formula:}
      \begin{equation*} 
      (X(f), h)_{\Omega} \seq -(f,  X(h) + \div(X) h)_{\Omega}
      + ((X \cdot \bsnu)f, h)'\,,
    \end{equation*}
    \cite{H-W, petersen:98, Wl-Ro-La}.
    (See \cite{KMNW-2025} for more details, where the above formula was used as the first
    step in the proof of Proposition \ref{prop.bdry.op}.)

    \item If $P := \nabla := \sharp d : \CI(M) \to \CI(M; TM)$ (the \emph{gradient}),
    then $\sigma_{1}(\nabla; \xi) = \imath \xi^{\sharp}$
    and a particular case of Proposition \ref{prop.bdry.op} next is the following integration
    by parts formula:
    \begin{equation*}  
      (\nabla f, X)_{\Omega} \seq
      (f, \nabla^{*}X)_{\Omega} + (f \bsnu, X )_{\pa \Omega}
      \seq (f, \nabla^{*}X)_{\Omega} + (f , \bsnu \cdot X)_{\pa \Omega}\,.
    \end{equation*}
    Recalling that $\nabla_{X} := \iota_{X} \circ \nabla$, we obtain that
    $\sigma_{1}(\nabla; \xi) = \iota_{X} \circ \imath \xi^{\sharp} = \imath \xi(X)$.
  \end{enumerate}
\end{remark}

We are ready to formulat now the needed abstract integration by parts formula,
which is Proposition 9.1 from Chapter 2 of \cite{Taylor1}, see also Proposition
A.3.14 from \cite{KMNW-2025}.

\begin{proposition} \label{lemma.prop.panu}
  \label{prop.bdry.op}
  Let $\Omega \subset M$ be an open subset with smooth boundary $\Gamma := \pa \Omega$ and
  let $P : \CI(M; E) \to \CI(M; F)$ be a first order differential operator. Let
  \begin{equation*}
    \pa_{\bsnu}^{P} \seq - \imath \sigma_{1}(P; \bsnu^{\sharp}) \in \Hom(E; F)
  \end{equation*}
  and let $P^{*} : \CI(M; F) \to \CI(M; E)$ be the \emph{formal adjoint} of $P$.
  Then, for all $u \in \CIc(M; E)$ and $v \in \CIc(M; F)$, we have
  \begin{equation*} 
    (Pu, v)_{\Omega} \seq (u, P^{*}v)_{\Omega}
    + (\pa_{\bsnu}^{P} u, v)'
  \end{equation*}
  and $\pa_{\bsnu}^{P} : \CI(M; E) \to \CI(\Gamma; F)$ is the unique operator with this property.
  \lpar Recall that $(\ , \ )'$ is the inner product on the boundary $\pa \Omega$.\rpar
\end{proposition}

For completeness, we sketch now a proof. See \cite{KMNW-2025} for more details.

\begin{proof}
  First, if $P = a$, an order zero multiplication operator, then the
  formula is obviously true.
  Second, if $P = X$, a smooth vector field, the claimed formula is nothing
  but the divergence formula (Remark \eqref{rem.basic.ip}(ii)), which is
  well known. One checks
  easily that if the formula of the proposition is true for $P$ and $a$
  and $b$ are suitable endomorphisms, then the formula is true also for
  $aPb$. Since the formula is linear and all first order differential
  operators can be written as linear combinations of differential operators
  of the form $aXb$ and $a$, with $X$ a smooth vector field, the result follows.
\end{proof}

\subsection{Green formulas and energy estimates for the Stokes operator}
\label{ssec.Green1}
In this subsection, $(M, g)$ is a Riemannian manifold with bounded geometry, that
is, we assume that it has positive injectivity radius and all covariant derivatives
of the curvature are bounded. (See \cite{Grosse-Kohr-Nistor-23} for more on
manifolds with bounded geomery, including references.) Also, $E \to M$
is a Hermitian vector bundle with connection, also assume to have bounded geometry,
meaning that all the covariant derivatives
of its curvature are bounded. We will let $\Omega \subset M$ be an open
subset with boundary $\Gamma :=\pa \Omega$, assumed to be a submanifold with
bounded geometry of $M$ (see \cite{AGN1, KMNW-2025} definitions).
In this paper, we are interested in the case when $M$ is a manifold with straight
cylindrical ends (Definition \ref{def.cyl.end}), $E \to M$ a compatible
Hermitian vector bundle with metric preserving connection (Definition \ref{def.comp.vb}),
and $\Omega \subset M$ also compatible with the straight cylindrical ends structure
of $M$ (Equation \eqref{eq.def.Omega}). Then $M$, $E$, and $\Omega$ fit into the
framework of this subsection (that is $M$, $E$, and $\Gamma$ have bounded
geometry), so the results proved next apply to the setting we are interested in.

Our Green-type formulas are formulated in terms of
the {\it scalar product} on $E$, denoted $\xi \cdot \eta$, and on the innner-product on
$L^{2}(\Omega; E)$, which, we recall, is given by
\begin{equation*}
  (f,g)_{\Omega} \ede \int_{\Omega} f(x) \cdot g(x)\, \dvol_{g}(x)\,.
\end{equation*}
(In particular, we have $(f,g)_{\Omega} \ede \int_{\Omega} f(x) \overline{g(x)}\,
\dvol_{g}(x)$ if $f$ and $g$ are functions.) Our convention is that our scalar products
are conjugate linear in the {\it second variable} and linear in the first.
Let also $\bsu, \bsw \in H^{1}(\Omega; E)$. For the simplicity of
the notations, we shall write
\begin{equation}\label{eq.bdry.conv}
  (\bsu, \bsw)' \ede (\bsu_{+}, \bsw_{+})'\,.
\end{equation}
To avoid confusion (for instance, when $u$ is defined on both sides of
$\Gamma := \pa \Omega$), we shall sometimes write $(\bsu_{+}, \bsw_{+})'$ instead
of $(\bsu, \bsw)'$.
Recall that $\textbf{1}_{A}$ denotes the characteristic function of the set $A$
(that is, $\textbf{1}_{A}(x) = 1$ if $x \in A$ and $\textbf{1}_{A}(x) = 0$ if
$x \notin A$). We shall also need the following notation, which will be
{\it fixed from now on}.

\begin{notation}\label{not.def.UWB}
  We will use the following notation throughout this section:
  \begin{enumerate}[\rm (i)]
    \item $U \ede \cvector{\bsu}{p} \seq (\bsu \ \ p)^{\top}$ and
    $W \ede \cvector{\bsw}{q} \seq (\bsw \ \ q)^{\top}$ will denote
    two generic sections of $H^{2}(\Omega; TM) \oplus H^{1}(\Omega)$;
    thus $\bsu, \bsw \in H^{2}(\Omega; TM)$ are the {\it vector components}
    and $p, q \in H^{1}(\Omega)$ are the { \it scalar components} of $U$ and,
    respectively, $W$.
    \item The sesquilinear form
    \begin{equation*}
      B_{\Omega}(U, W) \ede 2(\Def \bsu, \Def \bsw)_{\Omega}
      + (\nabla^{*} \bsu, q)_{\Omega}
      + (p, \nabla^{*} \bsw)_{\Omega} + (V\bsu, \bsw)_{\Omega} - (V_{0}p, q)_{\Omega}
    \end{equation*}
    and we sometimes let $\mathfrak v \ede (V\bsu, \bsw)_{\Omega}
      - (V_{0}p, q)_{\Omega}$;
    \item If $w$ is the section over $\Omega$ of some vector bundle over $M$, then
    $\textbf{1}_{\Omega}w$ will denote its {\it extension with zero outside
    $\Omega$.}
  \end{enumerate}
\end{notation}

Recall that, for all $h \in H^{1}(\Omega; E)$, $h_{+}$ denotes its
trace (or restriction) at the boundary $\pa \Omega$. Also, recall the definition of
distributions of the form $h \otimes \delta_{\Gamma}$,
\begin{equation}  \label{eq.def.hdelta2}
  \langle h \otimes \delta_{\Gamma}, \phi \rangle \ede \int_{\Gamma}
  h(x) \cdot \phi(x) \, dS_{\Gamma}(x)\,.
\end{equation}
We then have the following two Green-type formulas formula that relate our generalise
Stokes operator $\bsXi := \bsXi_{V, V_{0}}$ with the form $B_{\Omega}$

\begin{proposition} \label{prop.Green}
  Let us assume that $M$ has bounded geometry and that $\pa \Omega \subset M$ is a
  submanifold with bounded geometry, as before. We use the
  Notation \ref{not.def.UWB}, in particular, $U := (\bsu \ \ p)^{\top}$ and
  $W := (\bsw \ \ q)^{\top}$ are in $H^{2}(\Omega; TM) \oplus H^{1}(\Omega)$.
  Then we have the following relations:
  \begin{enumerate}[\rm (i)]
    \item $ \big ( \bsXi U, W \big)_{\Omega}
    \seq B_{\Omega} (U, W)
    + (\bop U, \bsw)'
    \seq B_{\Omega} (U, W) + \big(\tbop  U, W\big)' \,.$

    \item $ \big ( \bsXi U, W \big )_{\Omega} - \big (U, \bsXi W \big )_{\Omega}
    \seq (\bop  U, \bsw)' - (\bsu , \bop  W)'  \seq
    (\tbop  U, W)' - (U,\tbop  W)'\,.$

    \item We let $\textbf{1}_{\Omega} U$ be the section of $TM \oplus \CC \to M$
    that is equal to $U$ on $\Omega$ and equal to $0$ outside $\Omega$. Then
    we have the following representation formula
    \begin{equation*}
      \bsXi \big( \textbf{1}_{\Omega} U \big)
      \seq \textbf{1}_{\Omega} \big(  \bsXi U  \big)
      - (\tbop  U)_{+}  \otimes \delta_{\Gamma}
      + \tbopstar (U_{+} \otimes \delta_{\Gamma})\,.
    \end{equation*}
  \end{enumerate}
\end{proposition}

If $\Omega = M$, the results are still true, but we must remove the boundary
terms, that is the ones involving $(\ , \ )'$. More precisely, this proposition
reduces to
\begin{equation}  \label{eq.prop.Green.M}
  \big ( \bsXi U, W \big)_{M} \seq B_{\Omega} (U, W) \seq \big ( \bsXi U, W \big)_{M}\,.
\end{equation}

\begin{proof}
  Let us assume first that $U$ and $W$ are smooth and compactly supported.
  Let $\mathfrak v := (\bsu, V \bsw)_{\Omega} - (p, V_{0}q)_{\Omega}
  = (V \bsu, \bsw)_{\Omega} - (V_{0} p, q)_{\Omega}$ and $B_{\Omega}$,
  be as in Notation \ref{not.def.UWB}. Then,
  using the relation $\pa_{\bsnu}^{\Defstar}\Def = -\Dnu $
  of Corollary \ref{cor.formulas}(i), we obtain the following relations:
  \begin{multline*}
    \big (  \bsXi  U , W \big )_{\Omega}
     \seq (2\Defstar \Def \bsu, \bsw)_{\Omega} + (\nabla^{*} \bsu, q)_{\Omega}
    + (\nabla p, \bsw)_{\Omega} + \mathfrak v \\
    \seq 2(\Def \bsu, \Def \bsw)_{\Omega} - 2(\Dnu \bsu,  w)'
    + (\nabla^{*} \bsu, q)_{\Omega}
    + ( p, \nabla^{*} \bsw)_{\Omega}
    + (p \bsnu, \bsw)' + \mathfrak v \\
     \seq B_{\Omega} (U, W)
    - 2(\Dnu \bsu,  w)' + ( p \bsnu, \bsw)'
    \seq B_{\Omega} (U, W) + (\bop U, \bsw)'\\
    \seq B_{\Omega} (U, W) + (\tbop  U, W)'
    \,.
  \end{multline*}
  This proves (i) for $U$ and $W$ smooth, compactly supported. To complete the proof of
  (i), it suffices to notice that all the terms in (i) are continuous on
  $H^{2}(M; TM) \oplus H^{1}(M)$ and that $\CIc(M; TM \oplus \CC)$ is dense in the latter,
  because $M$ has bounded geometry \cite{GrosseSchneider, KMNW-2025}.

  The second point is an immediate consequence of the first one using the fact
  that $B_{\Omega}$ is sesquilinear. Indeed,
  \begin{align*}
    \big ( \bsXi U, W \big )_{\Omega} - \big (U, \bsXi W \big )_{\Omega}
    & \seq \big ( \bsXi U, W \big )_{\Omega} -
    \overline{\big (\bsXi W , U \big )}_{\Omega}\\
    & \seq B_{\Omega}(U, W) + (\bop  U, \bsw)' - \overline{B_{\Omega}(W, U)}
    - \overline{(\bop  W, \bsu)}\,{}_{\pa \Omega}\\
    & \seq (\bop  U, \bsw)' - (\bsu , \bop  W)'\,.
  \end{align*}
  The function $\textbf{1}_{\Omega} U$ is potentially discontinuous on $M$, and hence
  $\bsXi \big( \textbf{1}_{\Omega} U\big)$ is defined in distribution sense on $M$. Its
  definition and the second point (already proved) finally give
  \begin{align*}
    \langle \bsXi \big( \textbf{1}_{\Omega} U \big), W \rangle
    & \ede \big \langle \textbf{1}_{\Omega} U,  \bsXi W \big \rangle
    \ede \big ( \textbf{1}_{\Omega} U,  \bsXi W \big)_{M}
    \seq \big (U,  \bsXi W \big)_{\Omega} \\
    & \seq \big (  \bsXi  U, W \big )_{\Omega}
    - ((\tbop  U), W)' + (U, \tbop  W)'\\
    & \seq \big (\textbf{1}_{\Omega}  \bsXi  U, W \big )
    - ((\tbop  U)_{+}, W)' + (U_{+}, \tbop  W)'\\
    & \seq \big (\textbf{1}_{\Omega} \bsXi  U, W \big )
    - \langle (\tbop  U)_{+} \otimes \delta_{\Gamma}, W \rangle +
    \langle U_{+} \otimes \delta_{\Gamma}, \tbop  W \rangle \\
    & \seq \big \langle \textbf{1}_{\Omega} \bsXi  U, W \big \rangle
    - \langle (\tbop  U)_{+} \otimes \delta_{\Gamma}, W \rangle +
    \langle \tbopstar (U_{+} \otimes \delta_{\Gamma}),  W \rangle \,.
  \end{align*}
  Because $W$ was arbitrary, this proves the relation (iii).
\end{proof}

We will need the following definition from \cite{Grosse-Kohr-Nistor-23}.

\begin{definition} \label{def-L2-cont}
  Let $M$ be a manifold and $E$ and $F$ be two vector bundles over $M$. If $M$ is connected,
  we say that a differential operator $T : L^2(M; E) \to \CIc(M; F)'$ satisfies
  the \textit{$L^2$-unique continuation property} if, given $u \in L^2(M; E)$
  that vanishes in a \emph{non-empty} open subset of $M$ and satisfies $T u \seq 0$,
  then $u \seq 0$ \emph{everywhere} on $M$. For general $M$, we say
  that $T$ satisfies the \textit{$L^2$-unique continuation property} if
  it satisfies this property on any connected component of $M$.
\end{definition}

Clearly, an operator that is injective satisfies the $L^2$-unique continuation property.
The concept of $L^{2}$-unique continuation property just recalled
allows us to obtain the following corollary. Recall that a \emph{Killing vector field}
$X$ is a vector field that preserves the metric, equivalently, $\Def X = 0$.
We shall need the following strong form of positivity for the matrix valued potential
$V_{0}$ (in our previous works, we assumed that $V_{0}$ is a scalar multiple of the
identity matrix, but here we allow it to be more general).

\begin{definition}\label{def.succ}
  Let $A$ be a topological space and $V : A \to M_{n}(\CC)$. We shall write $V \succ 0$
  on $A$ if,
  \begin{enumerate}[\rm (i)]
    \item $V \ge 0$ on $A$ and,
    \item for every connected component $A_{0}$ of $A$, there is a point
    $a_{0} \in A_{0}$ such that $V(a_{0}) > 0$ (that is, in addition to $V(a_{0})\ge 0$, we also
    have that $V(a_{0})$ is invertible).
  \end{enumerate}
  The same definition applies to sections of endomorphism bundles.
\end{definition}

If $V$ is scalar valued or a scalar multiple of the identity $1 \in M_{n}(\CC)$
(as in our previous papers on the subject), the condition
$V \succ 0$ on $A$ is equivalent to $V \succ 0$ on every connected component of $A$.
We have the following corollary.

\begin{corollary}\label{cor.e.est}
  Let $V, V_{0} \ge 0$ and
  $U = \cvector{\bsu}{p} \in H^{2}(\Omega; TM) \oplus H^{1}(\Omega)$
  satisfy $\bsXi U = 0$ in $\Omega$ and $(\bop  U, \bsu)' = 0$.
  \begin{enumerate}[\rm (i)]
    \item We have $\Def \bsu \seq 0$, $V \bsu \seq 0$, $\nabla^{*} \bsu \seq 0$,
    $V_{0}p \seq 0$, and $\nabla p \seq 0$
    in $\Omega$.

    \item
    If, furthermore, $V_{0} \succ 0$ in $\Omega$, then
    $p = 0$ on $\Omega$.

    \item Similarly, if one of the following three conditions is satisfied:
    \begin{enumerate}[\rm (a)]
      \item No connected component of $\Omega$ has non-zero Killing vector fields;
      \item $V \succ 0$ on $\Omega$; or
      \item $\pa \Omega \neq \emptyset$ and $\bsu = 0$ on $\pa \Omega$;
    \end{enumerate}
    then $\bsu = 0$ in $\Omega$.
  \end{enumerate}
\end{corollary}

In particular, this corollary gives that
$\bsu = 0$ on $\supp(V) \cap \Omega$ and $p = 0$ on
$\supp(V_{0}) \cap \Omega$. Recall that by the statement ``$\phi \not \equiv 0$
on $A$,'' we mean that
there exists $a$ in the domain of $\phi$ such that $\phi(a) \not = 0$. To
negate this statement, we shall write ``$\phi = 0$ in $A$.''

\begin{proof}
  The real part $\operatorname{Re}(\mathfrak w)$
  of $\mathfrak w := (p, \nabla^{*}\bsu)_{\Omega}- (\nabla^{*}\bsu, p)_{\Omega}$
  vanishes. Let
  \begin{equation*}
    W \ede \cvector{\bsw}{q} \seq \cvector{\bsu}{-p} \, =: \, U'
  \end{equation*}
  in the formula $\big ( \bsXi U, W \big )_{\Omega} \seq B_{\Omega} (U, W)
  + (\bop  U, \bsw)'$ of Proposition \ref{prop.Green}. Together with
  the definition of $B_{\Omega}$ in Notation \ref{not.def.UWB} and with
  $\operatorname{Re}(\mathfrak w) =0$,
  this gives
  \begin{align*}
    0 & \seq \operatorname{Re} \big[\big ( \bsXi U, U' \big )_{\Omega}
    - (\bop U, \bsu)'\big]
    \seq \operatorname{Re} \big[  B_{\Omega} (U, U')\big]\\
    &
    \seq \operatorname{Re} \big[ 2(\Def \bsu, \Def \bsu)_{\Omega}
    - (\nabla^{*}\bsu, p)_{\Omega}
    + (p, \nabla^{*}\bsu)_{\Omega} + (V \bsu, \bsu)_{\Omega}
    + (V_{0}p, p)_{\Omega} \big]
    \\
    &
    \seq 2(\Def \bsu, \Def \bsu)_{\Omega} + (V \bsu, \bsu)_{\Omega}
    + (V_{0}p, p)_{\Omega}\,.
  \end{align*}
  Because $V, V_{0} \ge 0$, all three terms in the last sum are non-negative,
  so each of them equals zero. Therefore $\Def \bsu = 0$, $V \bsu = 0$,
  and $V_{0}p = 0$ in $\Omega$. We also have
  \begin{equation*}
    0 \seq \bsXi U \seq \cvector{2\Defstar \Def \bsu + V \bsu + \nabla p}
    {\nabla^{*} \bsu - V_{0}p}
    \seq \cvector{\nabla p}{\nabla^{*} \bsu}\,,
  \end{equation*}
  and hence we obtain (i). The condition $\nabla p = 0$ just proved implies that $p$
  is locally constant. Since, moreover, $V_{0}p = 0$,
  this constant is zero on the connected components of $\Omega$ on
  which $V_{0} \not \equiv 0$, and this proves (ii). (Notice that this is exactly
  the $L^{2}$-unique continuation property of $\nabla$, see
  Definition \ref{def-L2-cont}.) Similarly, (iii)
  follows from the fact that $\Def$ satisfies the $L^{2}$-unique
  continuation property (see \cite{Grosse-Kohr-Nistor-23}).
\end{proof}

The above result holds true also if $\Omega = M$ (with the same proof).
More precisely, we have the following consequences.

\begin{corollary}\label{cor.e.est.new}
  Let $V, V_{0} \ge 0$ and $U = \cvector{\bsu}{p} \in H^{2}(M; TM) \oplus H^{1}(M)$
  satisfy $\bsXi U = 0$ in $M$, which, we recall, is assumed to be connected.
  Then the following results hold:
  \begin{enumerate}[\rm (i)]
    \item $\Def \bsu \seq 0$, $V \bsu \seq 0$, $\nabla^{*} \bsu \seq 0$,
    $V_{0}p \seq 0$, and $\nabla p \seq 0$ in $M$.

    \item
    If, furthermore, $V_{0} \succ 0$ in $M$, then
    $p = 0$ on $M$.

    \item Similarly, if one of the following two conditions is satisfied:
    \begin{enumerate}[\rm (a)]
      \item $M$ has no non-zero Killing vector fields or
      \item $V \succ 0$ on $M$
    \end{enumerate}
    then $\bsu = 0$ in $M$.
  \end{enumerate}
\end{corollary}

We make the following simple observation.

\begin{remark}
  Let $\varepsilon$ be the diagonal matrix $1 \oplus (-1)$.
  The last corollary and the proof of Corollary \ref{cor.e.est} give
  that the operator
  \begin{equation*} 
    \varepsilon \bsXi \ede \varepsilon \bsXi _{V,V_{0}} \ede
    \left(\begin{array}{ccc}  1 & 0 \\
    0 &  -1
    \end{array}\right) \left(\begin{array}{ccc}  \bsL + V & \nabla \\
    \nabla^* &  -V_{0}
    \end{array}\right)
    \seq
    \left(\begin{array}{ccc}  \bsL + V & \nabla \\
    - \nabla^* &  V_{0}
    \end{array}\right)
  \end{equation*}
  has positive real part, in the sense that $\operatorname{Re}
  (\varepsilon \bsXi U, U)_{M} \ge 0$
  for all $U$ smooth with compact support.
\end{remark}


\section{\cn General background material and manifolds with cylindrical ends}
\label{sec.background}

This section contains some preliminary material on pseudodifferential operators, normal
lateral limits, Fourier transforms and their relation to normal lateral limits, and
manifolds with straight cylindrical ends (definition and vector bundles and  Sobolev spaces
on these manifolds). Several related results are contained in the book \cite{KMNW-2025},
to which we refer for further details as well as for the concepts not defined here.
See also \cite{Hormander1, Hormander3, Taylor1} for related background material.
Except the last subsection, where we introduce manifolds with straight cylindrical
ends, in the other three subsections we work on $\RR^{n}$.

\subsection{Basic function spaces and some Fourier transforms}
\label{ssec.Fourier}
We now recall some basic function spaes, the Fourier transform, some other basic
concepts, and compute some Fourier transforms that will be needed in what follows,
especially for the limit and jump relations.

Let $p \in [1, \infty]$ and $(\mfkM , \mu)$ be a measure space.
If $u : \mfkM \to \CC$ is a measurable function, then its \emph{$L^p$-norm}
is defined by
\begin{equation}\label{eq.def.normp}
  \|u\|_{L^p(\mfkM, \mu)} \ede \,
  \begin{cases}
    \ \left(\int_{\mfkM} |u(x)|^p \, d\mu(x)\right)^{1/p}\,, &
    \quad \mbox{if } \ p < \infty \\
    \ \, \esssup_{x \in X}\, |u(x)|  \,, &
    \quad \mbox{if } \  p = \infty\,.
  \end{cases}
\end{equation}
Of course, it is possible that $\|u\|_{L^p(\mfkM, \mu)} = \infty$.
We identify functions that coincide except on a zero measure set, to
obtain the Lebesgue spaces
\begin{equation*} 
  L^p(\mfkM; \mu) \ede \{ u : \mfkM \to \CC \mid \ u \mbox{ measurable and }
  \|u\|_{L^p(\mfkM)}
  < + \infty \}/\ker (\|\cdot \|_{L^p(\mfkM)})\,.
\end{equation*}
When the measure $\mu$ is clear from the context, we shall write $L^p(\mfkM)$ instead
of $L^p(\mfkM; \mu)$. For instance, unless otherwise explicitly stated, on $\RR^{n}$
we shall consider the Lebesque measure. Thus, $L^{p}(\RR^{n})$ denotes the usual
$L^{p}$-spaces defined using the Lebesque measure.

Our convention is that $\NN := \{1, 2, \ldots \}$ and we let
$\ZZ_{+} := \NN \cup \{0\}$. The elements $\alpha \in \ZZ_{+}^{n}$
will be called \emph{multi-indices}. If $x = (x_{1}, x_{2}, \ldots, x_{n}) \in
\RR^{n}$ and $\alpha \in \ZZ_{+}^{n}$ is a multi-index, we let
$x^{\alpha} := x_{1}^{\alpha_{1}} x_{2}^{\alpha_{2}} \ldots x_{n}^{\alpha_{n}}$.
We also let $\pa_{x_{j}}\pa_{j}:= \frac{\partial}{\partial x_{j}}$ be the $j$th
partial derivative and $\pa^{\alpha} := \pa_{1}^{\alpha_{1}} \pa_{2}^{\alpha_{2}}
\ldots \pa_{n}^{\alpha_{n}}$. Recall then that
\begin{equation}\label{eq.def.maS}
  \maS(\RR^{n}) \ede
  \{ u : \RR^{n} \to \CC \mid
  x^{\alpha} \pa^{\beta} u \in L^{2}(\RR^{n})\,,
  \ \forall\, \alpha, \beta \in \ZZ_{+}^{n} \}
\end{equation}
denotes the space of \emph{Schwartz functions} on $\RR^{n}$ (i.e. smooth rapidly
decaying functions at infinity). By $\maS'(\RR^{n})$ we denote the dual of
$\maS(\RR^{n})$, called the {\it space of tempered distributions} on $\RR^n$.

We let $\maC^{k}(\RR^{n})$ denote the space of $k$-times differentiable functions
on $\RR^{n}$ and $\maC(\RR^{n}) := \maC^{0}(\RR^{n})$ (simply, the space of
continuous functions on $\RR^{n}$).
If $\alpha \in \ZZ_{+}^{n}$, then $|\alpha| := \alpha_{1} + \alpha_{2} +
\ldots + \alpha_{n}$.
We let $\maC_{0}^{k}(\RR^{n})$ denote the set of functions $f : \RR^{n} \to \CC$
such that $\lim_{|x| \to \infty}|\pa^{\alpha} f(x)| = 0$, for all $|\alpha| \le k$.
As usual, $\CIc(\RR^{n})$ denotes the set of smooth functions with compact support in
$\RR^{n}$ (this definition extends to manifolds right away).

We let $\imath ^2=-1$ and $x \cdot \xi := \sum_{j=1}^{n} x_{j} \xi_{j}$ denote the
inner product of two vectors $x, \xi \in \RR^{n}$. The \emph{Fourier transform} of
tempered distributions will be denoted $\maF : \maS'(\RR^{n}) \to \maS'(\RR^{n})$.
More precisely, our convention is that, if
$f \in L^{1}(\RR^{n})$, then $\maF$ and its inverse $\maF^{-1}$ and are given by
\begin{equation}\label{eq.def.inv.F}
  \begin{gathered}
    \hat f(x) \ede \maF f(x) \ede \int_{\RR^{n}}
    e^{-\imath x \cdot \tau} f(\tau)
    \, d\tau\qquad \mbox{and}\\
    \maF^{-1} f(x) \seq \frac1{(2\pi)^{n}} \int_{\RR^{n}}
    e^{\imath x \cdot \tau} f(\tau)
    \, d\tau\,.
  \end{gathered}
\end{equation}
It follows that $\maF f, \maF^{-1}f \in \maC_{0}(\RR^{n}) := \maC_{0}^{0}(\RR^{n})$,
because $f$ was assumed to be integrable.

Let us recall a few basic calculations with the Fourier transform for
later use.

\begin{remark}\label{rem.eq.Fderiv}
First, for $f \in \maS(\RR^{n})$, we have
\begin{equation*}
    \maF(f')(x) \seq \int_{\RR} e^{-\imath x \tau} f'(\tau)\, d\tau
    \seq -
    \int_{\RR} \big(e^{-\imath x \tau} \big)' f(\tau)\, d\tau
    \seq \imath x \hat f(x)\,.
\end{equation*}
An equivalent way of writing this relation is $\maF \pa_{x} f = \imath x \maF f$,
which remains valid for $f \in \maS'(\RR)$ by continuity using
the density of $\maS(\RR)$ in $\maS'(\RR)$ in the weak topology. Similarly, if
$f \in \maS'(\RR^{n})$,
$\pa_{j} := \frac{\pa}{\pa x_{j}},$ and $x_{j}$ denotes also the operator of
multiplication by $x_{j}$, then we have the following equalities of
tempered distributions:
    $\maF( \pa_{j} f) \seq \imath x_{j} \maF(f)$,\
    $\maF( \tau_{j} f) \seq \imath \pa_{j} \maF(f)$,\
    $\maF^{-1}( \pa_{j} f) \seq -\imath x_{j} \maF^{-1}(f)$,\ and\
    $\maF^{-1}( \tau_{j} f) \seq -\imath \pa_{j} \maF^{-1}(f).$
\end{remark}

Recall that $\pv \, x^{-1} =: \mfkV$ denotes the distribution
\begin{equation}\label{eq.def.pvinvx}
    \langle \pv \, x^{-1}, \phi \rangle \ede
    \lim_{\epsilon \searrow 0} \ \int_{|x| > \epsilon} \frac{\phi(x)}{x}\, dx
    \, =: \, \langle \mfkV, \phi \rangle \,, \ \ \phi \in \maS(\RR)\,.
\end{equation}
(sometimes called \emph{Hadamard's ``partie finie'' integral}). (We
use $\mfkV := \pv \, x^{-1}$ in order to streamline the notation.)
We shall need the explicit form of the
Fourier transform $\maF \mfkV := \maF(\pv \, x^{-1})$.
This is well-known and is expressed in terms of the $\sgn$ function:
\begin{equation}\label{eq.def.sgn}
  \sgn(x) \ede \frac{x}{|x|} \seq
    \begin{cases}
        \ \ 1 & \ \mbox{ if } x > 0 \\
        \ \ 0 & \ \mbox{ if } x = 0 \\
        \ -1 & \ \mbox{ if } x < 0 \,,
    \end{cases}
\end{equation}
(the value $\sgn(0)$ can, in fact, be chosen arbitrarily). As usual, we let
\begin{equation*}
  u(0+) \ede \lim_{\epsilon \searrow 0} u(\epsilon) \quad
  \mbox{and} \quad u(0-) \ede \lim_{\epsilon \searrow 0} u(-\epsilon)\,,
\end{equation*}
when these limits exist. For example, $\sgn(0+) - \sgn(0-) = 2$.

For any function $f : \RR \to \CC$, we let
\begin{equation}\label{eq.def.tilde}
  \tilde f(x) \ede f(-x)\,,
\end{equation}
a definition that extends, by duality, to distributions. As for functions, a distribution
$u \in \CIc(\RR)'$ is called \emph{even} (respectively, \emph{odd}) if
$\tilde u = u$ (respectively, $\tilde u = -u$). It is well-known (and
very easy to check) that $\maF(\tilde u) = \widetilde{\maF(u)}$ for all $u \in
\maS'(\RR)$. This shows that the Fourier transform of an even (respectively, odd)
tempered distribution is even (respectively, odd).
This observation will be needed for the proof of the following result
(which is undoubtedly well-known, see, for instance, Proposition 8.2 in
Chapter 3 in \cite{Taylor1}; it was stated without proof in \cite{KMNW-2025}).
For $x \in \RR^{n}$, we let, as usual
\begin{equation}\label{eq.def.jap}
  \jap{x} \ede \sqrt{1 + |x|^{2}} \seq \big ( 1 + x_{1}^{2} + \ldots
  + x_{n}^{2} \big)^{1/2}\,.
\end{equation}

\begin{lemma} \label{lemma.Fpvx-1}
    Let $\mfkV := \pv \, x^{-1}$ be as in \eqref{eq.def.pvinvx}. Then
    \begin{enumerate}[\rm (i)]
      \item $\mfkV \in \maS'(\RR)$ and $\maF^{-1} \mfkV = \frac{\imath}2 \sgn$.

      \item Let $\chi_{0} \in \CIc(\RR)$ be an even function with $\chi_{0}(x) = 1$
      for $x$ close to $0$, let
      \begin{equation*}
        Z \ede (1-\chi_{0})\mfkV \in \maS'(\RR)\,,\ \mbox{ and let }\
        W \ede \maF^{-1} Z \ede \maF^{-1} \big( (1-\chi_{0}) \mfkV )\,.
      \end{equation*}
      Then $W$ extends to a smooth function on each of the intervals $[0,\infty)$
      and $(-\infty, 0]$ and $W(0+) = \frac{\imath}2 = -W(0-)$.

      \item Also,
      for all $j, k \in \ZZ_{+}$, there exist $C_{j,k} > 0$ such that,
      for all $x \neq 0$, we have
      \begin{equation*}
        |\pa_{x}^{j}W(x)| \le C_{j, k} \jap{x}^{-k}\,.
      \end{equation*}
      The same result holds for the restriction of $W$ to $(-\infty, 0)$.
  \end{enumerate}
\end{lemma}

Compare with Proposition 8.2 in Chapter 3 in \cite{Taylor1}.

\begin{proof}
  Let $\delta_{0}$ be the Dirac distribution at $0 \in \RR$.
  The derivative $\sgn'$ of the $\sgn$ function of Equation \eqref{eq.def.sgn}
  is given (in distribution sense) by
  \begin{equation*}
    \sgn' \seq [\sgn(0+)- \sgn(0-)] \delta_{0} \seq 2 \delta_{0}\,.
  \end{equation*}
  Remark \eqref{rem.eq.Fderiv} then gives
  \begin{equation*}
    \imath x \maF(\sgn)(x) \seq \maF(\sgn')(x) \seq 2 \maF(\delta_{0})(x)\seq 2\,.
  \end{equation*}
  It can be checked directly that $x \mfkV := x \big (\operatorname{pv}\, x^{-1}\big) = 1$.
  Therefore
  \begin{equation*}
    x \, \Big( \frac{\imath}{2}\maF(\sgn)(x) - \mfkV \Big )
    \seq 0 \in \maS'(\RR)\,.
  \end{equation*}
  Thus the difference $\frac{\imath}{2}\maF(\sgn)(x) - \mfkV$ is a
  multiple of $\delta_{0}$. Since both $\maF(\sgn)(x)$ and $\mfkV$
  are \emph{odd} distributions and $\delta_{0}$ is even, it follows that
  $\frac{\imath}{2}\maF(\sgn)(x) - \mfkV = 0$. Applying $\maF^{-1}$ to this
  equality yields the point (i).

  To prove (ii), we first notice that the distribution $\chi_{0}\mfkV$
  is compactly supported, and hence the distribution
  $u := \maF^{-1} (\chi_{0}\mfkV)$ is given by a smooth function \cite{Hormander1}.
  (Although we will not use this, note that $u$ is the convolution of $\maF^{-1}\mfkV$
  with a function in $\maS$ with integral one.) Since,
  \begin{equation*} 
    W \ede \maF^{-1} \big( (1-\chi_{0}) \mfkV\big ) \seq \maF^{-1} \mfkV -
    \maF^{-1} (\chi_{0}\mfkV) \seq \frac{\imath}2 \sgn -
    u\,,
  \end{equation*}
  we obtain that $W$ is given by a \emph{locally} $L^{1}$ function
  that is smooth outside 0 and extends to a smooth function
  on $[0,\infty)$. Moreover, $W$ is an odd function and has jump
  $W(0+) - W(0-) = \frac{\imath}2 (\sgn(0+) - \sgn(0-)) = \imath$ at $0$.
  Therefore $W$ also extends to a smooth function on $(-\infty, 0]$ and
  $W(0+) = - W(0-) = \frac{\imath}2$.

  It remains to prove (iii), that is, that
  $\pa_{x}^{j} W$ has rapid decay at $\infty$. Indeed, for all $k \ge 1$,
  $\big[ (1-\chi_{0}) \mfkV\big ]^{(k)} \in L^{1}(\RR)$, and hence
  we have
  \begin{equation}\label{eq.aux.W}
    (-\imath x)^{k} W \seq (-\imath x)^{k} \maF^{-1} \big[ (1-\chi_{0}) \mfkV\big ]
    \seq \maF^{-1} \Big[ \big[ (1-\chi_{0}) \mfkV\big ]^{(k)} \Big]\in \maC_{0}(\RR)\,.
  \end{equation}
  The last displayed equation gives that, for all $k \ge 1$ (and, hence,
  for all $k$), there exists $C_{k} > 0$ such
  that $|W(x)| \le C_{k} |x|^{-k}$, for all $|x| \ge 1$. This proves the result
  for $j = 0$, and hence $W$ has rapid decay at $\pm \infty$.

  Recall that $u := \maF^{-1}(\chi_{0}\mfkV) := \maF^{-1}(\chi_{0}[\pv\, x^{-1}])$.
  To obtain the result
  for the other values of $j$ (and, hence, prove (iii)), let us use Remark
  \eqref{rem.eq.Fderiv} to obtain that
  \begin{equation*}
      \pa_{x} u \ede \pa_{x} \maF^{-1} (\chi_{0}\mfkV)
      \seq \imath \maF^{-1} (x \chi_{0}\mfkV)
      \seq \imath \maF^{-1} (\chi_{0}) \in \maS(\RR) \,.
  \end{equation*}
  By integrating to obtain $u$ from $\pa_{x}u$ (up to a constant),
  this gives that there exists a constant $L \in \CC$
  such that $u\vert_{[0,\infty)} -L \in \maS([0,\infty))$.
  Since $W(x) = \frac{\imath}2 \sgn(x) - u(x)$ has zero limits at infinity
  (see also Equation \eqref{eq.aux.W}), we obtain
  that  $W\vert_{[0,\infty)} \in \maS([0,\infty))$,
  as claimed. (We also obtain $L = \frac{\imath}2$, but that is not
  needed for the proof.)
\end{proof}

As a consequence, we obtain the following lemma on jump-values of certain Fourier
transforms. It contains one of the most important calculations
for the proof of jump relations (see Lemma A.5.5 from \cite{KMNW-2025}).

\begin{lemma} \label{lemma.first.jump}
  Let $L \in \RR$ and $u : \RR \to \RR$ be a continuous function such that
  there exist $C \ge 0$ and $\epsilon > 0$ such that, for all $x \in \RR$, we have
  \begin{equation}\label{eq.cond.lfj}
    \jap{x}^{\epsilon}|x u(x) - L| \le C\,.
  \end{equation}
  Let $\tilde u (x) = u(-x)$, as before. Then $u + \tilde u$ is integrable,
  $u$ is a tempered distribution (i.e., $u \in \maS'(\RR)$), and its
  inverse Fourier transform $\maF^{-1} u$ is a function that is
  continuous everywhere, except maybe at $0 \in \RR$,
  with one-sided limits in $0$ given by
  \begin{equation*}
    \begin{gathered}
      \maF^{-1}u(0+) + \maF^{-1}u(0-) \seq
      \frac1{2 \pi} \, \int_{\RR} \big[ u(x) + u(-x) \big ]\, dx\,,
      \\
      \maF^{-1}u(0+) - \maF^{-1}u(0-) \seq \imath L \,,  \quad \mbox{ and}\\
      \lim_{x \to \infty} xu(x) \seq \lim_{x \to -\infty} xu(x) \seq L\,.
    \end{gathered}
  \end{equation*}
\end{lemma}

\begin{proof}
  As in Lemma \ref{lemma.Fpvx-1}, let $\chi_{0} \in \CIc(\RR)$ be an even
  function with $\chi_{0} = 1$ in a neighborhood of 0 and let
  $\mfkV := \pv \, x^{-1} \in \maS'(\RR)$. Using these distributions, we now decompose
  our given function $u$ as
  \begin{equation}\label{eq.one.one}
    u \seq L\mfkV + w_{1} + w_{2} \in \maS'(\RR)\,,
  \end{equation}
  where $w_{1} := (u-L\mfkV)\chi_{0}$, $w_{2} := (u-L\mfkV)(1-\chi_{0})$, and $L \in \RR$
  is as in the statement.

  The distribution $w_{1} := (u-L\mfkV)\chi_{0}$ is compactly supported, and hence $\maF^{-1}w_{1}$
  is a smooth function on $\RR$ (an easy fact, see, for instance, \cite{Hormander1}).
  In turn, the function $w_{2} := (u-L\mfkV)(1-\chi_{0})$ is a continuous
  function on $\RR$ that vanishes in a neighborhood of $0$. Our assumptions
  then imply that there exists $C' \ge 0$ such that, for all $x \in \RR$, we have
  $|w_{2}(x)| \le C' \jap{x}^{-1-\epsilon}$.
  Consequently,
  $w_{2} \in L^{1}(\RR)$ and thus $\maF^{-1}w_{2}$ is continuous.
  Combining with Equation \eqref{eq.one.one}, we obtain
  \begin{equation}
    \maF^{-1}u - L \maF^{-1}\mfkV \seq \maF^{-1}w_{1}
    + \maF^{-1}w_{2} \in \maC(\RR)\,.
  \end{equation}
  Therefore $\maF^{-1}u$ and $L \maF^{-1}\mfkV$ have the same jump at 0.
  The relation $\maF^{-1}\mfkV := \maF^{-1}(\pv\, x^{-1} = \frac{\imath}2$ of Lemma
  \ref{lemma.Fpvx-1}(i) then gives
  \begin{equation}
    \maF^{-1}u(0+) - \maF^{-1}u(0-)
    \seq L \big[\maF^{-1}\mfkV(0+) - \maF^{-1}\mfkV(0-)\big]
    \seq \imath L\,.
  \end{equation}
  This proves the second displayed relation.

  To prove the first relation, let $w := u + \tilde u$, where
  $\tilde u(x) = u(-x)$, as before. Then the assumption gives that there
  exists $C'' > 0$ such that, for all $x \in \RR$, we have
  $|w(x)| \le C'' \jap{x}^{-1-\epsilon}$. Consequently, $w$ is integrable and
  hence $\maF^{-1} w$ is continuous on $\RR$. Combining this fact with
  the relation $\maF(\tilde u) = \widetilde{\maF(u)}$, we obtain
  \begin{multline*}
    \maF^{-1}u(0+) + \maF^{-1}u(0-) \seq \maF^{-1}u(0+) + \widetilde{\maF^{-1}u}(0+) \seq
    \maF^{-1}u(0+) + \maF^{-1}\tilde u(0+)\\ \seq  \maF^{-1} w(0+)
    \seq \maF^{-1} w(0) \seq \frac1{2 \pi} \int_{\RR} w(x)dx
    \seq \frac1{2 \pi} \int_{\RR} \big[ u(x) + u(-x) \big] dx\,.
  \end{multline*}
  The last statement of this lemma (the limit) follows the assumption that
  $\jap{x}^{\epsilon}|x u(x) - L| \le C$,
  which gives $\lim_{|x| \to \infty} (x u(x) - L) = 0$.
\end{proof}

We will use the last lemma in the following form.

\begin{corollary}
  \label{cor.first.jump}
  Let $u : \RR \to \RR$ be a function satisfying the assumptions of Lemma
  \ref{lemma.first.jump} (that is, satisfying the condition
  $\jap{x}^{\epsilon}|x u(x) - L| \le C$ of Equation \eqref{eq.cond.lfj}).
  Then
  \begin{enumerate}[\rm (i)]
    \item $u$ defines a tempered distribution on $\RR$;
    \item $\maF^{-1}u$ is given by a locally integrable function that is
    continuous on $\RR \smallsetminus \{0\}$;
    \item the following limit exists:
    \begin{equation*}
      \JC(u) \ede \lim_{|x| \to \infty} x u(x)\,;
    \end{equation*}
    \item $u + \tilde u$ is integrable on $\RR$, and
    \begin{equation}\label{eq.jump.explicit}
      \maF^{-1}u(0\pm) \seq \pm \frac{\imath \JC(u)}2
      + \frac1{2 \pi}\, \int_{\RR} \frac{u(x) + u(-x)}{2} \, dx\,.
    \end{equation}
  \end{enumerate}
\end{corollary}

\subsection{Pseudodifferential operators on $\RR^{n}$}
\label{ssec.pdos}

Most of the operators that we will work with (if not all of them!) will be
\emph{pseudodifferential operators.} We thus recall now some basic definitions
that are needed in the following. First, we recall that for $x \in \RR^{n}$, we
use the notation
  $\jap{x} \ede \sqrt{1 + |x|^{2}} \seq \big ( 1 + x_{1}^{2} + \ldots
  + x_{n}^{2} \big)^{1/2},$ see Equation \eqref{eq.def.jap}.

Let $\ZZ_{+}^{n}:= \left\{\alpha =(\alpha _1,\ldots,\alpha _n)\in \ZZ^n:
\alpha _i\geq 0\,, \ 1\leq i\leq n\right\}$, and, for $\alpha \in \ZZ_{+}^n$,
\begin{equation*}
  \pa_{x}^{\alpha} \ede \pa^{\alpha} \ede \frac{\pa ^{\alpha _1}}{\pa x_1^{\alpha _1}}\cdots
  \frac{\pa ^{\alpha _n}}{\pa x_n^{\alpha _n}}\,,
\end{equation*}
as before. We now recall the classical, well-known definition of Kohn-Nirenberg symbols
(see also \cite{Hormander3})

\begin{definition}\label{def.pseudos}
  We let $\Smbmn$ denote the set of functions $a : \RR^{n} \times \RR^{n} \to \CC$
  such that, for every pair of multi-indices $\alpha, \beta \in \ZZ_{+}^{n}$, there exists
  $C_{\alpha\beta} \ge 0$ such that
  \begin{equation*} 
    |\pa_{x}^{\alpha} \pa_{\xi}^{\beta} a(x, \xi)|
    \le C_{\alpha \beta}\jap{\xi}^{m-|\beta|}\,.
  \end{equation*}
  \begin{enumerate}[\rm (i)]
    \item The function $a$ will be called a \emph{symbol.}
    \item For $a \in \Smbmn$, we define $a(x, D): \CIc(\RR^{n}) \to \CI(\RR^{n})$ by the formula
    \begin{equation*}  
      a(x, D) u (x) \ede \frac1{(2 \pi)^{n}}
      \,\int_{\RR^{n}} e^{\imath x \cdot \xi} a(x, \xi) \hat u( \xi) \, d \xi\,.
    \end{equation*}
    The operator $a(x, D)$ is called the \emph{pseudodifferential operator with symbol} $a$.
  \end{enumerate}
\end{definition}

For $a \in \Smbmn$, we shall denote by $a(x, \xi)$ its values, where $x, \xi \in \RR^{n}$.
The elements of $M_{k}(\Smbmn)$ are also called symbols.
See also \cite{H-W, KMNW-2025, Taylor1}.
In order to clarify some later conventions, let us mention here that the second variable
in $a(x, \xi)$,
namely $\xi$, is really in the dual of $\RR^{n}$, which is nevertheless identified with
$\RR^{n}$ using
the standard inner product (this will be relevant for our study of the jump relations,
where the last basis vector in the $\xi$ variable will be denoted $e_{n}^{\sharp}$ instead
of simply $e_{n} = (0, \ldots, 0, 1)$).

The operator $a(x, D)$ is a particular case of a
\emph{pseudodifferential operator} on $\RR^{n}$ and, locally, every differential operator
on $\RR^{n}$ is of this form \cite{Hormander3}. For our purposes, we will not need to use
pseudodifferential operators more general than $a(x, D)$ on $\RR^{n}$.
We shall denote
by $\Psi^{m}(\RR^{n})$ the set of all pseudodifferential operators on $\RR^{n}$,
it is the
set of operators $T : \CIc(\RR^{n}) \to \CI(\RR^{n})$ such that, for all
$\phi \in \CIc(\RR^{n})$,
the operator $\phi T \phi$ is of the form $a(x, D)$ for a suitable symbol $a \in \Smbmn$
depending on $T$ and $\phi$. This definition extends right away to
operators acting between sections of vector bundles on manifolds.
Thus, if $M$ is a smooth manifold
and $E, F \to M$ are smooth vector bundles, then
\begin{multline}\label{def.psdos.M}
  \Psi^{m}(M; E, F) \ede \{ T : \CIc(M; E) \to \CI(M; F) \mid \forall \phi \in \CIc(M)
  \mbox{ supported} \\
  \mbox{in a coordinate neighborhood that trivializes both } E \mbox{ and } F\,,
  \phi T \phi \seq a(x, D)\}\,,
\end{multline}
where, this time, $a \in M_{k,l}(\Smbmn)$ depending on $T$ and $\phi$, $n$ is the dimension
of $M$, $E$ has rank $l$, and $F$ has rank $k$.

If $b \in S^{m}(\RR^{n-1} \times \RR^{n-1})$, we let $b(x', D')$ denote
its associated operator. In general, symbols with a prime (i.e. ${}'$) will refer to objects
on $\RR^{n-1}$. The most important example is that, if $x \in \RR^{n}$,
then $x' \in \RR^{n-1}$ and $x_{n} \in \RR$ are its projections, so that
$x = (x', x_{n})$.

Recall that \emph{the distribution kernel} $k_{a(x, D)}$ of $a(x, D)$ is the
distribution on $\RR^{n} \times \RR^{n}$ such that, for all $\phi, \psi \in \CIc(\RR^{n})$,
we have
\begin{equation}\label{eq.def.dkaD}
  \langle a(x, D)\phi, \psi \rangle \seq \langle k_{a(x, D)}, \psi \boxtimes \phi \rangle\,,
\end{equation}
where $\psi \boxtimes \phi \in \CIc(\RR^{n} \times \RR^{n})$ is given by
$\psi \boxtimes \phi(x, y) = \psi(x) \phi(y)$ and the pairings $\langle \ , \ \rangle$
are those of distributions and test functions (for the first pairing, on $\RR^{n}$, and,
for the second pairing, on $\RR^{n} \times \RR^{n}$). For further use, let us record the
following well-known lemma \cite{Hormander3, H-W, Taylor1}.

\begin{lemma}\label{lemma.d.k}
  The distribution kernel $k_{a(x, D)}(x, y)$ of $a(x, D)$
  is smooth for $x \neq y$  and there it is given by
    $k_{a(x, D)}(x, y) \seq \maF_{\xi}^{-1}a(x, x-y)\,.$
\end{lemma}

A smooth function $b : \RR^{n} \times \RR^{n} \to \CC$ will be called \emph{essentially
homogeneous} of order $m$ if
\begin{equation}\label{eq.def.ess.hom}
  b(x, \lambda \xi) \seq \lambda^{m} b(x, \lambda \xi)\,, \qquad
  \mbox{for all }
  | \xi| \ge 1 \ \mbox{ and } \ \lambda | \xi| \ge 1\,.
\end{equation}
A symbol $a \in \Smbmn$ will be called \emph{classical} if there exist
essentially homogeneous symbols $a_{j} \in S^{m-j}(\RR^{n} \times \RR^{n})$
of order $m - j$
such that, for every $N \in \NN$, $a - \sum_{j=0}^{N} a_{j} \in
S^{m-N-1}(\RR^{n} \times \RR^{n})$. The set of classical symbols of order $m$
will be denoted $S^{m}_{\cl}(\RR^{n} \times \RR^{n})$.
We now recall the definition of the principal symbol, which will
play an important role in what follows.

\begin{definition}\label{def.princ.symb}
  If $a \in S^{m}_{\cl}(\RR^{n} \times \RR^{n})$ and $\xi \in \RR^{n}$,
  $\xi \neq 0$, we let
  \begin{equation*} 
    \sigma_{m}(a; x, \xi) \ede \lim_{\lambda \to \infty}
    \lambda^{-m} a(x, \lambda \xi)
  \end{equation*}
  denote its \emph{principal part} (or \emph{symbol}).
  We shall say that $\sigma_{m}(a(x, D)) := \sigma_{m}(a)$ is \emph{the principal
  symbol} of the pseudodifferential operator $a(x, D)$.
\end{definition}

Therefore, $\sigma_{m}(a) : \RR^{n} \times (\RR^{n} \smallsetminus \{0\}) \to \CC$
satisfies $\sigma_{m}(a; x, \lambda \xi) = \lambda^{m} \sigma_{m}(a; x, \xi)$, for
all $\lambda > 0$. For example,
\begin{equation}\label{eq.ps.dj}
  \sigma_{1}(\pa_{j}; x, \xi) \seq  \imath \xi_{j}\,.
\end{equation}
Given $a \in S_{\cl}^{m}(\RR^{n} \times \RR^{n})$, we
can always find an essentially homogeneous function $b_{m}$ of order $m$,
$b_{m} \in S_{\cl}^{m}(\RR^{n} \times \RR^{n})$,
such that $\sigma_{m}(a; x, \xi) = b_{m}(x, \xi)$, for all $| \xi| \ge 1$
\cite{Hormander3, H-W, Taylor2}.
We shall say that $b_{m}$ is an \emph{essential homogeneous representative
of} $\sigma_{m}(a)$. Also, note that $a \in S^{m}_{\cl}(\RR^{n} \times \RR^{n})$
will have a different principal symbol if regarded as an element of
$S^{m+1}_{\cl}(\RR^{n} \times \RR^{n})$, in fact, then, $\sigma_{m+1}(a) = 0$.

\subsection{Traces and normal lateral limits on $\RR^{n}$}
\label{ssec.problem}

We shall need some general results on the \emph{normal lateral limits} $[a(x, D)u]_{\pm}$
at $\Gamma := \{x_{n} = 0\}$ of the value $a(x, D)u$ of a pseudodifferential operator
$a(x, D)$ acting on suitable distribution $u = h \otimes \delta_{\Gamma}$ supported on
$\Gamma$. This problem will be formulated precisely next and will be studiend in detail
on Euclidean space in Section \ref{sec.nllRn} and on manifolds with cylindrical ends
in Section \ref{sec.NLimits} following \cite{KMNW-2025, KNW-22, KohrNistor-Stokes}.

It will be convenient to use the following notation. Let $\Gamma_{\epsilon} \subset \RR^{n}$
be the hyperplane
\begin{equation}\label{eq.def.Gamma}
  \Gamma_{\epsilon} \ede \{x = (x', x_{n})
  \in \RR^{n-1} \times \RR \mid x_{n} = \epsilon \}\,, \qquad \Gamma := \Gamma_{0}
\end{equation}
with the induced Euclidean measure. We parameterize $\Gamma_{\epsilon}$ via the diffeomorphism
$\Gamma \ni (x', \epsilon) \mapsto x' \in \RR^{n-1}$.
Recall that by $\CIc(\RR^{n})$ we denote the set
of smooth functions with compact support in $\RR^{n}$. We let $e_{n} = (0, 0, \ldots, 0, 1)
\in \RR^{n}$. Then $\Gamma _\epsilon \ede \Gamma + \epsilon e_{n} \ede \{x = (x', x_{n})
\in \RR^n: x_{n} = \epsilon\} \subset \RR^{n}$. Then $\tau_{-\epsilon} : H^{s}(\Gamma_{\epsilon})
\to H^{s}(\Gamma)$ will be the natural isometries induced by translations
$\tau_{\epsilon}(x) := x - \epsilon e_{n}$, $\tau_{\epsilon} : \Gamma_{\epsilon}
\to \Gamma_{0} = \Gamma$. (Note that, in this setting, the role of $\Omega$ is played
by $\Omega = \RR_{+}^{n} := \{ x_{n} > 0\}$, so that $\Gamma := \pa \Omega$,
as usual.)

The following lemma is often used in this paper (see also Lemma A.5.1 in \cite{KMNW-2025}).

\begin{lemma}\label{lemma.def.hdelta}
  Let $\Gamma := \{x_{n} = 0\} \simeq \RR^{n-1}$ be as in \eqref{eq.def.Gamma},
  $s \in \RR$, and
  $h \in H^{s}(\Gamma)$, $s \in \RR$.
  \begin{enumerate}[\rm (i)]

  \item
  Let $s \ge 0$ and let $h \otimes \delta_{\Gamma}$ be defined by
  \begin{equation*}
    \langle h \otimes \delta_{\Gamma}, \phi \rangle \ede \int_{\RR^{n-1}}
    h(x') \phi(x', 0) \, dx'\,, \qquad \forall\, \phi \in \CIc(\RR^{n})\,.
  \end{equation*}
  Then $h \otimes \delta_{\Gamma} \in H^{s'}(\RR^{n}; E)$
  for all $s' < -1/2$.

  \item
  Let $\langle \ , \ \rangle_{\Gamma}$ be the pairing of distributions
  and test functions on $\Gamma$. If $s < 0$, the formula
  \begin{equation*}
    \langle h \otimes \delta_{\Gamma}, \phi \rangle \ede \langle
    h,  \phi\vert_{\Gamma} \rangle_{\Gamma}\,, \qquad \forall\,
    \phi \in \CIc(\RR^{n})\,,
  \end{equation*}
  defines a distribution $h \otimes \delta_{\Gamma} \in H^{s - 1/2}(\RR^{n}; E)$.

  \end{enumerate}
\end{lemma}

\begin{proof}
  Let us assume first that $s' < -\frac12$, $u \in H^{-s'}(\RR^{n})$, and $s = 0$
  (so $h \in L^{2}(\Gamma)$). Then the restriction (or trace) $u\vert_{\Gamma} \in
  L^{2}(\Gamma)$ is well defined and depends continuously on $u$, by classical results
  \cite{Taylor1}, because $-s' > 1/2$. Consequently,
  $\int_{\RR^{n-1}} u(x', 0)h(x')\, dx'$
  is also well defined and depends continuously on $u$, hence it defines an element
  in $H^{-s'}(\RR^{n})^{*} \simeq H^{s'}(\RR^{n})$, as claimed. Thus
  $h \otimes \delta_{\Gamma} \in H^{s'}(\RR^{n})$.

  Finally, let us consider $s < 0$. Let $u \in H^{-s+1/2}(\RR^{n})$. Then
  $u \vert_{\Gamma} \in H^{-s}(\Gamma)$, because $-s+1/2 > 1/2$ and the pairing
  $\langle h , u \vert_{\Gamma}\rangle_{\Gamma}$ is well defined and continuous with
  respect to $u$. Therefore $h \otimes \delta_{\Gamma} \in H^{s-1/2}(\RR^{n})$. This
  completes the proof.
\end{proof}

Let $a \in \Smbmn$ and $a(x, D)$ be as in Definition \ref{def.pseudos}(ii), $m \in \RR$.
Then we define the \emph{potential operator}
\begin{equation}\label{eq.def.maSa}
  \maS_{a}h \ede \maS_{a(x, D)}h \ede a(x, D)(h \otimes \delta_{\Gamma})\,.
\end{equation}
Let $s < -1/2$ and $h \in H^{s}(\Gamma)$. Then Lemma \ref{lemma.def.hdelta} and the
mapping properties of the pseudodifferential
operators \cite{Hormander3, KNW-22, Taylor1} show that
\begin{equation*}
  \maS_{a}h \ede \maS_{a(x, D)}h \ede a(x, D)(h \otimes \delta_{\Gamma})
  \in H^{s - m}(\RR^{n})\,.
\end{equation*}
Thus, if $m < -1$, we can choose $s$ close to $-1/2$ such that
$s-m > 1/2$, and hence the trace $(\maS_{a}h)\vert_{\Gamma}$ is well defined.
(In the case $m = -1$, this result does not hold. See Theorem \ref{thm.main.jump0}
for a detailed discussion in this case.)

\begin{definition}\label{def.limit.values0}
If $u : \RR^{n} \to \CC$ is smooth enough, we then let
\begin{equation*} 
  \begin{gathered}
    u_{\epsilon} \ede u \vert_{\Gamma _\epsilon }\,,\\
    u_{\pm} \ede  \lim_{\epsilon \searrow 0} \tau_{\mp \epsilon}
    u_{\pm \epsilon} \ \mbox{ {\bf =\,the normal lateral limits of }} u, \mbox{ and}\\
    \tau_{-\epsilon} \big[ a(x, D)(h \otimes \delta_{\Gamma}) \big]_{\epsilon}
    \, =:\, a_{\epsilon}(x', D') h\,,
  \end{gathered}
\end{equation*}
whenever these definitions make sense.
\end{definition}

It is known that, for a suitable pseudodifferential operator $a$ and $\epsilon \neq 0$,
$a_{\epsilon}(x', D')$ is defined and is again a pseudodifferential operator. We will prove
this in the cases in which we are interested in Section \ref{sec.nllRn} and
we will study the properties of the family $a_{\epsilon}$, $\epsilon \neq 0$.

We let $\RR_{\pm}^{n} := \{x = (x', x_{n})   \in \RR^{n-1} \times \RR \mid  \pm x_{n} > 0 \}$
be the two open half-spaces defined by the hyperplane $\Gamma = \{x_{n} = 0\}$.
Let $s > 1/2$. If $u \in H_{\loc}^{s}(\RR_{\pm}^{n})$, we let
\begin{equation}
  \gamma_{\pm}(u) \ede \mbox{ the trace of } u \mbox{ at }
  \Gamma \ \mbox{, so that }\ \gamma_{\pm}(u) \in
  H^{s-1/2}(\Gamma)\,,
\end{equation}
using, of course, that $\Gamma = \pa \RR_{\pm}^{n}$. If, furthermore,
$u \in H_{\loc}^{s}(\RR^{n})$, we shall write $u\vert_{\Gamma} :=
\gamma_{+}(u) = \gamma_{-}(u)$.
Then, we have the following useful result.

\begin{lemma}\label{lemma.limits=traces}
  Let $u \in H_{\loc}^{s}(\RR_{\pm }^{n})$, $s > 1/2$. Then
  \begin{equation*}
    u_{\pm } \ede \lim_{t \searrow 0} \tau_{\mp t} u_{t} \seq \gamma_{\pm }(u)
    \in H_{\loc}^{s-1/2}(\Gamma)\,.
  \end{equation*}
  In particular, if $a \in \Smbmn$ for some $m < -1$ and $h \in L^{2}(\RR^{n})$, then
  \begin{equation*}
    \big[ a(x, D)(h \otimes \delta_{\Gamma}) \big]_{+}
    \seq \big[ a(x, D)(h \otimes \delta_{\Gamma}) \big]_{-}
    \seq \big[ a(x, D)(h \otimes \delta_{\Gamma}) \big]_{\Gamma}\,.
  \end{equation*}
\end{lemma}

\begin{proof}
  The first relation of the lemma is clearly true if $u \in \CIc(\RR^{n})$.
  The first relation then follows
  from the uniform continuity (in $t$) of the restrictions
  $H_{\loc}^{s}(\RR_{+}^{n}) \to H_{\loc}^{s-1/2}(\Gamma + t e_{n})$
  and the density of $\CIc(\RR^{n})$ in $H_{\loc}^{s}(\RR_{+}^{n})$.
  The result on $\RR_{-}^{n}$ is proved exactly in the same way.

  The result on pseudodifferential operators follows from the mapping properties of
  pseudodifferential operators, from the properties
  of the traces on the spaces $H_{\loc}^{s}$, and from Lemma
  \ref{lemma.def.hdelta}.
\end{proof}

Typically, in this paper, we shall work with functions (or sections of
vector bundles) to which the above lemma applies, and hence we will not need to
distinguish between the limits $u_{\pm}$ and the traces $\gamma_{\pm}(u)$.
Of course, there exist important situations when the limits $u_{\pm}$ exist
but the traces $\gamma_{\pm}(u)$ do not exist.

\subsection{Manifolds with cylindrical ends,
compatible vector bundles and differential operators}
\label{ssec.cyl.ends}
We continue with some standard background material, but
now we turn from the case of $\RR^{n}$ to that of a manifold with cylindrical
ends $M$. We begin by recalling a definition of manifolds with straight
cylindrical ends that is equivalent to the one given in the introduction
(see \cite{KMNW-2025, KNW-22, KohrNistor-Stokes}).

\begin{definition} \label{def.cyl.end}
  A \emph{manifold with straight cylindrical ends} $M$ is a complete non-compact
  Riemannian manifold that is the disjoint union
  \begin{equation*} 
    M \seq M_{0}\cup \left[M'\times (-\infty , R_{M})\right]\,,
  \end{equation*}
  where $(M_{0}, g_{0})$ is a compact Riemannian manifold, $M':=\partial M_{0}
  \neq \emptyset$, $R_{M} \in \RR$, and the \emph{cylindrical end}
  $M'\times (-\infty , R_{M})$ is endowed with a product metric.
  Any Riemannian manifold isometrically diffeomorphic to one of this form will
  also be called a manifold with straight cylindrical ends.
\end{definition}

There is no loss of generality to assume that $M$ is connected, so we will do that
throughout the paper, in order to simplify our statements. It follows from the fact
that $M$ is complete that the isometric embedding $M'\times (-\infty ,0)
\hookrightarrow M$ extends to an isometry $M'\times (-\infty ,0] \hookrightarrow M$
that maps $M' \times \{0\}$ to $M' \simeq \pa M_{0}$ diffeomorphically. There is
no loss of generality to assume that this diffeomorphism is the canonical one identifying
$(x', 0) \in M' \times \{0\}$ to $x' \in M' \simeq \pa M_{0}$. So, again, we will assume
that this is the case throughout the paper. Also, there is no loss of generality to assume
$R_{E} =0$, and, again, we will usually (but not always) do that. For instance, for a
\emph{cylinder}, it is convenient to take $R_{E} = -1$. The case of a cylinder is a very
important one, so we discuss it more formally in the following remark.

\begin{remark}\label{rem.cylinder}
  Recall that a \emph{closed manifold} is a smooth, compact manifold without
  boundary. Let $\mfkS$ be a closed manifold. An important example of a manifold with straight
  cylindrical ends is $M = \mfkS \times \RR$. In this case, $M_{0} = \mfkS \times [-1, 1]$,
  $R_{M} = -1$ and $M' := \pa M_{0} = \mfkS \times \{-1, 1\}$.  Then we have a canonical isometry
  \begin{equation*}
    \mfkS \times \RR \, =: \, M \simeq M_{0} \cup_{\psi} \big[\mfkS \times \{-1, 1\} \times
    (-\infty, -1] \big]\,,
  \end{equation*}
  where, on the right hand side $\psi : \mfkS \times \{-1, 1\} \times \{-1\}
  \to \mfkS \times \{-1, 1\} := \pa M_{0}$ is given by $\psi (x', -1) = x'$,
  for all $x' \in \mfkS \times \{-1, 1\}$. (The symbol $\cup_{\psi}$ means that
  we identify $y$ to $\psi(y)$ in the union.)

  On cylinders, it is important to notice that $M := \mfkS \times \RR$ carries a
  natural action of $\RR$ by translations and we let $\CI(M)^{\RR} =
  \CI(\mfkS \times \RR)^{\RR} \simeq \CI(\mfkS)$
  denote the space of smooth functions on $\mfkS \times \RR$ that are invariant by translations.
  Many objects on manifolds with cylindrical ends are defined by first defining their
  translation invariant versions on cylinders. If $M = \canDec$ is a general manifold
  with cylindrical ends, then $M' \times (-\infty, 0]$ is its \emph{end} and
  $M' \times \RR$ is its associated cylinder with the induced Riemannian metric
  (induced from $\pa M_{0}$, which, in turn, is induced from the metric on $M$).
\end{remark}

In the following, whenever we will write
$M \seq M_{0}\cup \left[M'\times (-\infty ,0)\right]$, we will
understand that $M$ is a manifold with straight cylindrical ends as in Definition
\ref{def.cyl.end}. It is easy (and well-known) to prove (see e.g. \cite{Grosse-Kohr-Nistor-23}
and the references therein)
that all manifolds with straight cylindrical ends have bounded geometry,
so the results of Subsection \ref{ssec.Green1} apply to them.

\begin{definition}\label{def.comp.vb}
  Let $M = M_{0}\cup \left[M'\times (-\infty ,0)\right]$ be a manifold with straight
  cylindrical ends \lpar Definition \ref{def.cyl.end}\rpar\ and let $E \to M$ be a
  Hermitian vector bundle endowed with a metric preserving connection $\nabla^{E}$.
  We shall say that the pair $(E, \nabla^{E})$ is {\it compatible with the straight
  cylindrical ends structure on $M$} (simply, \emph{compatible})
  if there exists $R_E \le 0$ such that $E\vert_{M' \times (-\infty, R_E)}$
  is isometric to the pull-back to $M' \times (-\infty, R_E)$ of a Hermitian
  vector bundle $E' \to M' $ with connection $\nabla^{E'}$ via an isomorphism
  that maps the connection $\nabla^E$ of $E$ to the pull-back connection of
  $\nabla^{E'}$ on $M' \times (-\infty, R_E)$.
\end{definition}

If $E = \CC$ (the one dimensional trivial bundle), then $\nabla^{E} f := df$,
where $d$ is the de Rham differential. Let $\sharp : T^{*}M \to TM$ be the
vector bundle isomorphism defined by the Riemannian metric $g$ of $M$.
Note then that the {\it gradient} $\nabla f := \sharp d f$ has a similar
notation to a connection, however the two should not be confused.

Let $E$ be a vector bundle on $M' \times \RR$ obtained from the pull-back of
vector bundle on $M'$. Thus $E \to M' \times \RR$ is, in, in particular, a
compatible vector bunde on the manifold with straight cylindrical ends $M' \times \RR$.
We let $\CI(M' \times \RR)^{\RR}$ denote the set of translation invariant sections of
$E \to M' \times \RR$. In the rest of this subsection, $M$ will be a
manifold with straight cylindrical ends. (The results are true also for closed manifolds,
see \cite{KNW-2025} and some of them are recalled in Subsection \ref{ssec.ToAdd}.)
We next introduce an important class of smooth sections.

\begin{definition}\label{def.CIinv}
  \index{translation invariant in a neighborhood of infinity}
  Let $E \to M$ be a compatible vector bundle and let $u \in \CI(M; E)$.
  We shall say that $u$ is {\it translation invariant in a neighborhood
  of infinity} if there exists $R_u \le R_E$ such that $u(x', t) = u(x', s)$ for all
  $s, t \le R_u$ and all $x' \in M' := \pa M_{0}$. We let $\CI_{\inv}(M; E)$
  denote the set of sections of $E$ that are translation invariant in a neighborhood
  of infinity and let $\CI(M'  \times \RR; E)^{\RR}$ be the set of
  translation invariant sections on $M' \times \RR$.
  Then we let $\In: \CI_{\inv}(M; E) \to \CI(M'  \times \RR; E)^{\RR}$ be defined by
  \begin{equation*}
    \In(u)(x', t) \seq u(x', t)\,, \quad \mbox{ for }\ t \le R_{u}\,.
  \end{equation*}
\end{definition}

\begin{remark}\label{rem.on.def}
  If $u$ is as in Definition \ref{def.CIinv}, then there exists $u_{0} \in \CI(M'; E)$
  such that $u(x', t) = u_{0}(x')$ for $t \le R$. Thus $\In u(x', t) \seq u_{0}(x')
  \seq u(x', t) \in E$ for $t \le R_u$. Thus
  $u_{0}$ and $\In(u)$ \emph{do not coincide}, but they correspond to each other under the
  natural isomorphism $\CI(M' \times \RR; E)^{\RR} \simeq \CI(M' ; E)$.
  Because of this, we will sometimes regard $\In(u)$ as an element of $\CI(M' ; E)$.
\end{remark}

We let $\CI_{\inv}(M) := \CI_{\inv}(M; \CC)$.
We next introduce differential operators following, for instance \cite{KohrNistor1}.

\begin{definition}\label{def.not.diff.ops}
  Let $E, F \to M$ be compatible vector bundles on $M$ with connections $\nabla^{E}$
  and $\nabla^{F}$. We let $\operatorname{Diff}_{\inv}^m(M; E, F)$ be the set of
  differential operators $P$ acting on sections of $E$ with values sections of $F$
  with the property that there exist $a_{j} \in \CI_{\inv}(M; \Hom(T^{* \otimes j}
  M \otimes E, F))$, $j = 0, \ldots, m$, (see Definition \ref{def.CIinv}) such that
  \begin{equation*}
    Pu  \seq \sum_{j=0}^m a_j (\nabla^{E})^j u\,.
  \end{equation*}
  For such a $P$, we let
  \begin{equation*}
    \In (P) \ede \sum_{j=0}^m \In (a_j) (\nabla^{E})^j
  \end{equation*}
  be the associated differential operator on $M'  \times \RR$ obtained by ``freezing
  the coefficients'' of $P$ near infinity. It will be called the {\it limit at infinity}
  operator associated to $P$.
\end{definition}

Therefore, $\In (P)$ will be a translation invariant differential operator
on $M' \times \RR$. An example is the Levi-Civita connection
\begin{equation*}
  \nabla^{LC} \in \operatorname{Diff}^1_{\inv}(M; TM, T^*M \otimes TM) \,,
\end{equation*}
in which case $\In \nabla^{LC} = \nabla^{LC}$.
The choice of manifolds (i.e., with straight cylindrical ends) and of the associated
functions and differential operators is motivated by practical
questions, since these are the objects that are the most likely to appear in practical
applications. However, our results extend to the more general case of manifolds with
cylindrical ends (not necessarily straigh, but asymptotically straight). One just needs
to slightly enlarge the class of functions, vector bundles, differential operators, and,
especially, pseudodifferential operators. These questions will be considered in another
paper.

\emph{For the rest of this section, $M  = \canDec$ will be a manifold with straight
cylindrical ends and $E \to M$ will be a fixed compatible Hermitian vector bundle.}

Let us use the notation of Definition \ref{def.comp.vb}. Then the compatible vector
bundle $E \to M$ will have a product structure on $M' \times (-\infty, R_E)$. Hence
$\nabla^{E} = \nabla^{E'} + dt \otimes \partial_t$ on $M' \times (-\infty, R_E)$,
where $\nabla^{E'}$ is the connection on $E' \to M'$. Consequently, the
\emph{Bochner Laplacian} of $E$ satisfies
\begin{equation*}
  \Delta_{E} \ede (\nabla^{E})^{*}\nabla^{E}
  \seq (\nabla^{E'})^{*}\nabla^{E'} - \partial _t^2\,.
\end{equation*}

\begin{definition}\label{def.rem.Bochner}
  Let $\nabla^{E}$ denote the compatible connexion on the compatible vector bundle $E \to M$.
  Let $\Delta_{E} \ede (\nabla^{E})^{*}\nabla^{E}$ be the Bochner Laplacian. The Sobolev
  space $H^s(M; E)$ is defined as the domain of $(1 + \Delta_E)^{\frac{s}{2}}$ for $s \ge 0$.
  The norm on this space is
  $\|u\|_{H^{s}(M; E)} := \|(1 + \Delta_E)^{\frac{s}{2}}u\|_{L^2(M; E)}^{1/2}$.
  For $s\leq 0$, the Sobolev space $H^s(M; E)$ is defined as the dual of the Sobolev space
  $H^{-s}(M; E)$.
\end{definition}

See \cite{AmannFunctSp, AGN-1, AGN2, GN17, GrosseSchneider, KohrNistor1, KNW-22, Strichartz}.

\begin{remark}\label{rem.C0dense}
  The space $\CIc(M; E)$
  is dense in $H^{s}(M; E)$ for all $s \in \RR$. This is a result that holds in
  the greater generality of manifolds with bounded geometry, see
  \cite{AmannFunctSp, AmannFunctSp25, sobolev, GrosseSchneider, Schneider}.
\end{remark}

See \cite{KMNW-2025} for a more thorough discussion of the Sobolev spaces on
manifolds with cylindrical ends, including a definition using dyadic partitions
of unity. See \cite{GrosseSchneider, TriebelBG} for a discussion of Sobolev spaces on
manifolds with bounded geometry.

\section{\cn Normal lateral limits of pseudodifferential operators on $\RR^{n}$}
\label{sec.nllRn}

In this section we provide a detailed study on $\RR^{n}$ of the \emph{normal lateral limits}
\begin{equation*}
  [a(x, D)u]_{\pm} \in H^{s}(\Gamma)
\end{equation*}
at $\Gamma := \{x_{n} = 0\}$ for $u = h \otimes \delta_{\Gamma}$.
Here $a(x, D)$ is a pseudodifferential operator (Definition \ref{def.pseudos}), the
normal lateral limits were defined in Definition \ref{def.limit.values0}, and
$u = h \otimes \delta_{\Gamma}$ was defined in Lemma \ref{lemma.def.hdelta}.
The hyperplane $\Gamma := \{x_{n} = 0\}$ separates the two half-spaces
\begin{equation*}
  \Omega_{+} \ede \RR^{n}_{+} \ede \{x \in \RR^{n} \mid x_{n} > 0\}
  \ \mbox{ and } \ \Omega_{-} \ede \RR^{n}_{-} \ede \{x \in \RR^{n} \mid x_{n} < 0\}\,.
\end{equation*}
(We remark in passing, that the open set $\Omega_{+} := \RR^{n}_{+}$ is
on one side of its boundary, as required throughout this paper.)
Recall that if $a(x, D)u$ is smooth enough on either of the half-spaces
$\RR_{\pm}^{n}$, then the corresponding normal lateral limit
$[a(x, D)u]_{\pm}$ at $\Gamma = \pa \Omega_{\pm}$ coincides with the trace
$\gamma_{\pm} a(x, D)u$ of $a(x, D)u$ at the boundary $\Gamma$ of that
half-space (see Lemma \ref{lemma.limits=traces}).

We have included in the previous section and in this section a few very basic
definition related to pseudodifferential operators. If needed, results not recalled or
proved here about pseudodifferential operators can be found in one
of the following books \cite{HAbelsBook, HintzML, Hormander1, RT-book2010, Taylor, Wong},
as well as in may others. For instance, the book \cite{Taylor2} contains
a clear but concise account of pseudodifferential operators and applications to layer potentials.
An even shorter introduction to some basic facts and definitions related to pseudodifferential
operators and geared towards our applications can be found in \cite{KNW-22}.

A complete and general treatment of normal lateral limits distributions of
the form $[a(x, D) h \otimes \delta_{\Gamma}]_{\pm}$ for pseudodifferential operators $a(x, D)$
is contained in the book by Hsiao and Wendland \cite{H-W}. Here we only deal with the results
(and calculations) needed to treat suitable elliptic, second order differential operators.
Our presentation in this section provide a complete, yet concise, introduction to the
subject of limit and jump relations for potential operators on a half-space. These results,
together with those in Section \ref{sec.NLimits} and in \cite{KMNW-2025} provide a
short--but complete--introduction to jump relations on manifolds with cylindrical ends.
(They also shed light on jump relations on general manifolds.)

The main question discussed in this section is the following.

\begin{problem}\label{problem}
  Given $a \in \Smbmn$, to study the existence and properties of the
  \emph{normal lateral limits} (see Definition \eqref{def.limit.values0}):
  \begin{equation*}
    a_{\pm}(x', D') h \ede \big[ a(x, D)(h \otimes \delta_{\Gamma}) \big]_{\pm}\,,
  \end{equation*}
  in particular, to study their
  relation to the traces $\gamma_{\pm}a(x, D)(h \otimes \delta_{\Gamma})$.
\end{problem}

Often in the literature, the \emph{non-tangential limits} at the boundary
are studied instead of the \emph{normal lateral limits} that we study in this paper.
The non-tangential limits are a stronger form of boundary limit and hence, they
exist less often than the normal lateral limits, however, for functions that are smooth enough,
they are the same. (In fact, they are the same for most of the applications studied in
this paper. We stress, however, that proving the existence of non-tangential limits
is more difficult than the proof of the existence of the normal lateral limits. In
particular, the normal lateral limits may exist even if
the non-tangential limits do not exist.)

A word now about the notation. First, we often parametrize $\Gamma := \{x_{n} = 0\}$
with $\RR^{n-1}$. Thus $\Gamma \subset \RR^{n}$, but $\Gamma \simeq \RR^{n-1} \not
\subset \RR^{n}$. This is the reason we need the isometries $\tau_{t}$,
$\tau_{t}(x) := x - t e_{n}$ mapping
$\Gamma_{t} := \RR^{n-1} \times \{t\}$ to $\Gamma$ (see Equation \eqref{eq.def.Gamma}
and the discussion following it). We distinguish $\Gamma$ from
$\RR^{n-1}$ to make it easier to transition in the later sections to an arbitrary
domain (i.e., and open connected subset) with boundary $\Gamma$. However, in this
paper, there will be no situation when confusions can arise if we omit the
isometries (i.e., identifications)
$\tau_{t}$ from the notation, so we shall do that from now on.

The results in this section expand on those in \cite{KMNW-2025}. The reader should
compare to that book. Part of these results were stated also in \cite{KNW-2025}, but
without proofs.

\subsection{General results on normal lateral limits of pseudodifferential operators}
We shall use the notation of Subsection \ref{ssec.pdos}. Also, when a function $f(x, y, \ldots)$
depends on several variables, $x, y, \dots$ and we are computing its Fourier transform
(only) with respect to the variable $x$, we shall write $\maF_{x}(f)(\xi, y, \ldots)$
for the resulting function. The same comment applies to the inverse Fourier transform.

Let $a \in S_{\cl}^{m}(\RR^{n} \times \RR^{n})$ be a classical symbol with principal part
(or symbol) $\sigma_{m}(a)$, regarded as a homogeneous function in the second variable
(see Definition \ref{def.princ.symb} and the discussion following it). We shall say that
$\sigma_{m}(a)$ is \emph{odd} if $\sigma_{m}(a; - \xi) = -\sigma_{m}(a; \xi)$. We then have
the following result (Lemma A.5.6 from \cite{KMNW-2025}), which is a consequence of Lemma
\ref{lemma.first.jump}. Recall that $e_{n} = (0, \ldots, 0, 1) \in \RR^{n}$ and
that $\tilde u (x) := u(-x)$, see Equation \eqref{eq.def.tilde}.

\begin{lemma} \label{lemma.Fsymbols}
  Let $a \in \Smbmn$, $m \in \RR$.
  \begin{enumerate}[\rm (i)]
    \item For any fixed $x \in \RR^{n}$ and $ \xi' \in \RR^{n-1}$, the function
    \begin{equation*}
      \phi_{x, \xi'} : \RR \to \CC\,, \quad \phi_{x, \xi'}( \xi_{n})
      := a(x, \xi', \xi_{n})\,,
    \end{equation*}
    defines a tempered
    distribution on $\RR$ such that $\maF^{-1} \phi_{x, \xi'}
    := \maF_{ \xi_{n}}^{-1} \phi_{x, \xi'}$ coincides
    with a smooth function outside 0.

    \item If $m < -1$, then $\phi_{x, \xi'}$ is integrable on $\RR$, and hence
    $\maF^{-1} \phi_{x, \xi'} \in \maC_{0}(\RR)$.

    \item Assume that $a$ is a classical symbol of order $-1$
    and that $\sigma_{-1}(a)$ is odd. Let $e_{n} := (0, \ldots, 0, 1) \in \RR^{n}$,
    as before, and $\JC(a; x) := \sigma_{-1}(a; x, e_{n})$.
    Then, for all $x \in \RR^{n}$ and $\xi' \in \RR^{n-1}$,
    $\phi_{x, \xi'}$ satisfies
    the assumptions of Lemma \ref{lemma.first.jump} for $\epsilon = 1$ and $L := \maL(a; x)$.
    That is, for all $(x, \xi') \in \RR^{n} \times \RR^{n-1}$ and $\tau \in \RR$, there
    exists $C_{x, \xi'} > 0$ such that
    \begin{equation*}
      |\tau \phi_{x, \xi'}(\tau) - \JC(a; x)| \ede
      |\tau a(x, \xi', \tau) - \sigma_{-1}(a; x, e_{n})|
      \le C_{x, \xi'} /\jap{\tau}\,;
    \end{equation*}
  \end{enumerate}
\end{lemma}

Usually, the ``jump factor'' $\JC$ will be associated rather to a
pseudodifferential operator $a(x, D)$ (instead of its symbol $a$), so we will
write
\begin{equation}\label{eq.full.notation}
    \JC(a(x, D); x) \ede \JC(a; x) \ede \sigma_{-1}(a; x, e_{n})\,.
\end{equation}

Recall that $\maC_{0}^{k}(\RR^{n})$ denotes the set of functions $f : \RR^{n} \to \CC$
such that, for all $|\alpha| \le k$, we have $\lim_{|x| \to \infty}|\pa^{\alpha} f(x)| = 0$.

\begin{proof}
  The property that $\phi_{x, \xi'}$ defines a tempered distribution
  follows from the inequality $\phi_{x, \xi'}(\tau) \le
  C_{x, \xi'}\jap{\tau}^{m}$. When $m < -1$, this implies that
  $\phi_{x, \xi'} \in L^{1}(\RR)$, and hence its inverse Fourier transform
  is in $\maC_{0}(\RR)$. This proves (ii). To complete the proof of (i),
  let $n > m + k +1$. We will use $t$ as dual variable to $ \xi_{n}$ for
  the Fourier transform in $ \xi_{n} \in \RR$.
  We let $t$ denote both the variable of $\RR$ and
  the multiplication operator by $t$ on $\maS(\RR)$ and on $\maS'(\RR)$.
  Then $t^{n} \maF_{ \xi_{n}}^{-1} \phi_{x, \xi'} = \imath^{n}
  \maF_{\xi_{n}}^{-1} (\pa_{ \xi_{n}}^{n}\phi_{x, \xi'}) \in \maC_{0}^{k}(\RR)$
  (all derivative up to order $k$ tend to 0 at infinity)
  where the last relation is due to the fact that
  $|\pa_{ \xi_{n}}^{n}\phi_{x, \xi'}( \xi_{n})| \le C \jap{ \xi_{n}}^{m-n}$
  and $m-n < -k-1$. This shows that $\maF_{ \xi_{n}}^{-1} \phi_{x, \xi'}$ coincides
  with a smooth function on $\RR \smallsetminus \{0\}$.

  Finally, to prove (iii), we will use Corollary \ref{cor.first.jump} for each of the functions
  $\phi_{x, \xi'}$, $(x, \xi')\in \RR^{n} \times \RR^{n-1}$,
  with $\JC_{\pm}(\phi_{x, \xi}) =: \JC_{\pm}(x)$. To that end, we need to
  first check that the hypotheses of Lemma \ref{lemma.first.jump} are satisfied by all
  functions $\phi_{x, \xi'}$. To that end,
  let us write
  \begin{equation*}
    a \seq b_{-1} + r_{-2}\,, \qquad
    b_{-1} \in S^{-1}(\RR^{n} \times \RR^{n})\,, \ \mbox{ and }\
    r_{-2} \in S^{-2}(\RR^{n} \times \RR^{n})\,,
  \end{equation*}
  where $b_{-1}$ is essentially homogeneous of order $-1$. Recall that this means
  that $b_{-1}(x, \lambda \xi) = \lambda^{-1} b_{-1}(x, \xi)$
  whenever $|\xi|, \lambda | \xi| \ge 1$. It follows that $r_{-2}$ is a
  classical symbol. The hypothesis that $\sigma_{-1}(a; x, \xi)$ is odd in $\xi$
  and the fact that $b_{-1}(a; x, \xi) = \sigma_{-1}(a; x, \xi)$ for $|\xi| \ge 1$ imply
  that $b_{-1}(x, \lambda \xi) = \lambda^{-1} b_{-1}(x, \xi)$
  whenever $|\xi|, |\lambda \xi| \ge 1$ (that is, we have extended the previous relation to
  $\lambda$ negative as well).
  We have that all derivatives of $b_{-1}(x, \xi)$
  are bounded, and hence there exists a constant $C>0$ such that, for $|\tau |\geq 1$,
  \begin{equation*}
    |b_{-1}(x, \tau^{-1} \xi', 1) - b_{-1}(x, 0, 1)| \le C |\tau^{-1} \xi'|\,.
  \end{equation*}
  By increasing $C$, we may assume that $|r_{-2}(x, \xi)| \le C |\xi|^{-2}$ for
  all $|\xi| \ge 1$, because $r_{-2}$ is a symbol of order $-2$.
  Let $L = b_{-1}(x,0, 1) = b_{-1}(x,e_{n}) := \sigma_{-1}(a; e_{n})$.
  For $|\tau|\ge 1$ and $(x, \xi') \in \RR^{n} \times \RR^{n-1}$, we let $\xi := (\xi', \tau)$
  with $|\xi| \ge |\tau| \ge 1$ and use the homogeneity relation for $\tau^{-1}\xi$ to obtain
  \begin{multline*}
    |\tau^{2}a(x, \xi', \tau) - L \tau |
    \leq |\tau^{2}b_{-1}(x, \xi', \tau) - L \tau | + |\tau^{2} r_{-2}(x, \xi', \tau)|     \\
    \seq |\tau| \, \big |\,  b_{-1}(x, \tau^{-1} \xi', 1)
    - b_{-1}(x, e_{n})  \big |
    + |\tau^{2} r_{-2}(x, \xi', \tau)|
    \leq C |\xi'| + C \, =: \, C_{x, \xi'}\,.
  \end{multline*}
  Moreover, the function $\phi_{x, \xi'}(\tau):= a(x, \xi', \tau)$ is bounded for
  $|\tau| \le 1$ (because $a$ is a symbol of order $-1$). This completes the proof.
\end{proof}

Recall that if $x = (x_{1}, x_{2}, \ldots, x_{n}) \in \RR^{n}$, we shall write
$x = (x', x_{n})$, where, $x' := (x_{1}, x_{2}, \ldots, x_{n-1})
\in \RR^{n-1}$. Similarly, if $\xi = (\xi_{1}, \xi_{2}, \ldots, \xi_{n})\in \RR^{n}$,
we shall write $\xi = (\xi', \xi_{n})$, where $\xi' \in \RR^{n-1}$.
As detailed in the Corollary \ref{cor.lemma.Fsymbols} next,
the last point of the last lemma in conjunction with Corollary \ref{cor.first.jump}
allows us to introduce the following definition.

\begin{definition}\label{def.aepsilon}
  Let $a \in \Smbmn$ with $m \in \RR$ and $t \in \RR$ and let
  $\phi_{x, \xi'}( \xi_{n}) := a(x, \xi', \xi_{n})$, $\widetilde \phi_{x, \xi'}( \xi_{n})
  := a(x, \xi', -\xi_{n})$, as before. If $m \ge -1$, we assume
  $t \neq 0$. We then define:
  \begin{enumerate}[\rm (i)]
    \item $a_{x_{n}, t}(x', \xi') := \maF_{ \xi_{n}}^{-1} a(x', x_{n}, \xi', t)
    := \big[ \maF^{-1}\phi_{x, \xi'}\big](t) .$

    \item If $a \in S_{\cl}^{-1}(\RR^{n} \times \RR^{n})$
    and $\sigma_{-1}(a; x, \xi)$ is \emph{odd} in $\xi$, we also let
    \begin{enumerate}[(a)]
      \item $a_{x_{n}, 0}(x', \xi') \ede \frac{1}{2\pi}
        \int_{\RR}  \frac{a(x', x_{n}, \xi', \xi_{n})
        + a(x', x_{n}, \xi', - \xi_{n})}{2} \, d \xi_{n} \seq
        \int_{\RR} \frac{\phi_{x, \xi'} + \tilde \phi_{x, \xi'}}{4\pi}\, dt$;
      \item $\JC(a; x) := \sigma_{-1}(a; x, e_{n})$; and
      \item $a_{x_{n}, 0\pm}(x', \xi') \ede \pm \frac{\imath \sigma_{-1}(a; x, e_{n})}2
        + a_{x_{n}, 0}(x', \xi') \ede \pm \frac{\imath}2 \JC(a;x)
        + a_{x_{n}, 0}(x', \xi')$.
    \end{enumerate}
  \end{enumerate}
\end{definition}

The reason for introducing Definition \ref{def.aepsilon} comes from Section
\ref{ssec.problem}. Of course, if $m < -1$ in this definition, we have, for
all $t \in \RR$,
\begin{equation}\label{eq.def.aepsilon.explicit}
  a_{x_{n}, t}(x', \xi') \ede
  \maF_{ \xi_{n}}^{-1} a(x', x_{n}, \xi', t)
  \ede \frac{1}{2\pi}
  \int_{\RR} e^{\imath t \xi_{n}} a(x, \xi', \xi_{n}) \, d \xi_{n}\,.
\end{equation}

Recall that $\tilde \psi(t) := \psi(-t)$.
The last point of the last lemma, in conjunction with Corollary \ref{cor.first.jump}
yields the following very useful consequence.

\begin{corollary}\label{cor.lemma.Fsymbols}
  Let $a$ be a classical symbol of order $-1$, $\phi_{x, \xi'}( \xi_{n}) :=
  a(x, \xi', \xi_{n})$, $\phi_{x, \xi'} : \RR \to \CC$,
  $e_{n} := (0, \ldots, 0, 1) \in \RR^{n}$,
  and $\JC(a; x) := \sigma_{-1}(a; x, e_{n})$, as before.
  We assume that $\sigma_{-1}(a)$ is odd in the sense that, for all $x, \xi \in \RR^{n}$,
  $\xi \neq 0$, we have $\sigma_{-1}(a; x, - \xi) = -\sigma_{-1}(a; x, - \xi)$. Then, for
  all $x \in \RR^{n}$ and $\xi' \in \RR^{n-1}$, we have the following properties:
    \begin{enumerate}[\rm (i)]
      \item $\phi_{x, \xi'}$ defines a tempered distribution on $\RR$.

      \item $\maF^{-1} \phi_{x, \xi'} \in \maS'(\RR)$ is a locally integrable function that is
      continuous on $\RR \smallsetminus \{0\}$.

      \item
        $\JC(a; x) \ede \sigma_{-1}(a; x, e_{n}) \seq \lim_{|\tau| \to \infty}
        \tau a(x, \xi', \tau) \in \CC.$

      \item $\phi_{x, \xi'} + \tilde \phi_{x, \xi'} \in L^{1}(\RR)$.

      \item The one-sided limits $\maF^{-1} \phi_{x, \xi'} (0\pm)$,
      \begin{equation*}
        \begin{cases}
          \ \maF^{-1} \phi_{x, \xi'} (0+)
          \ede \lim_{t \to 0, t > 0}\maF^{-1} \phi_{x, \xi'} (t)\\
          \ \maF^{-1} \phi_{x, \xi'} (0-)
          \ede \lim_{t \to 0, t < 0}\maF^{-1} \phi_{x, \xi'} (t)\,,
        \end{cases}
      \end{equation*}
      of $\maF^{-1} \phi_{x, \xi'} \in \maS'(\RR)$ at $0$ are given by
      \begin{align*} 
        \maF^{-1} \phi_{x, \xi'} (0\pm) & \seq
        \pm \frac{\imath \sigma_{-1}(a; x, e_{n})}2
        + \frac1{2 \pi}\, \int_{\RR} \frac{a(x, \xi', \tau) +
        a(x, \xi', -\tau)}{2} \, d\tau\\
        & \, =: \,
        \pm \frac{\imath \JC(a; x)}2
        + a_{x_{n}, 0}(x', \xi') \, =: \, a_{x_{n}, 0\pm}(x', \xi')\,.\\
      \end{align*}
    \end{enumerate}
\end{corollary}

In our book \cite{KMNW-2025}, the last lemma and the last corollary
were proved for the case $m = -1$
under a slightly more general assumption, as we explain next.

\begin{remark}
Let $m = -1$, $e_{n} = (0, \ldots, 0, 1) \in \RR^{n}$ (as before), and
let us assume that $a$ is \emph{classical} (but not necessarily odd).
Then, for all $x \in \RR^{n}$ and $ \xi' \in \RR^{n-1}$,
the following limits exist
\begin{equation*}
  \JC_{\pm}(a; x) \ede \lim_{\tau \to \pm \infty}
  \tau a(x, \xi', \tau)  \seq  \pm \sigma_{-1}(a; x, \pm e_{n}) \in \CC\,.
\end{equation*}
In our book, we have proved the last lemma under the weaker assumption that $\JC_{+}(x) = \JC_{-}(x)$.
In this paper, this relation is a consequence of the fact that $\sigma_{-1}(a)$ is assumed to be odd.
This simplifies the presentation
and is anyway the only case needed in this paper. From now on, we shall rather use this assumption
(odd principal symbol, instead of the less restrictive assumption $\JC_{+}(x) = \JC_{-}(x)$
used in our book).
\end{remark}

Before turning to the study of Problem \ref{problem}, let us
recall the following basic results.

\begin{theorem}\label{thm.aux.derivUnderInt}
  Let $(\mfkM, \mu)$ be a measure space, $N \in \ZZ_{+} \cup \{\infty\}$, and
  $U \subset \RR^{k}$ be an open subset. Let also
  $f : U \times \mfkM \to \CC$ be a measurable function satisfying,
  for all multi-indices $\alpha \in \ZZ_{+}^{k}$,
  $|\alpha| \le N$, the following three conditions:

  \begin{enumerate}[\rm (i)]
    \item The derivative $\pa_{x}^{\alpha} f(x, \omega)$ is defined
    for all $(x, \omega) \in U \times \mfkM$.

    \item For all $\omega \in \mfkM$, the function
    $x \mapsto \pa_{x}^{\alpha} f(x, \omega)$
    is continuous in $x \in U$.

    \item There is an \emph{integrable} function $G_{\alpha} : \mfkM \to [0, \infty)$
    such that, for all $x \in U$,
    \begin{equation*}
      |\pa_{x}^{\alpha} f(x, \omega)| \le G_{\alpha}(\omega)\,.
    \end{equation*}
  \end{enumerate}
  Then the function $F: \RR^{k} \to \CC$,
    $F(x) := \int_{\mfkM} f(x, \omega)\, d\mu(\omega)$,
  has the property that $\pa_{x}^{\alpha}F$ is defined and continuous for all $|\alpha| \le N$,
  and it is given by
  \begin{equation*}
    \pa_{x}^{\alpha} F(x) = \int_{\mfkM} \pa_{x}^{\alpha} f(x, \omega)\, d\mu(\omega)\,,
  \end{equation*}
  and $|\pa_{x}^{\alpha}F(x)| \le \int_{\mfkM} G_{\alpha}(\omega)\, d\mu(\omega)$.
\end{theorem}

\begin{proof}
  The result follows by applying Lebesgue's dominated convergence theorem.
  See \cite[Appendix E]{Evans}.
\end{proof}

If $N = 0$ (hence $\alpha = 0$), then this theorem simply states
that $F(x)$ is defined for all $x \in \RR^{k}$,
that it is continuous in $x$, and that it is bounded by
$|F(x)| \le \int_{\mfkM} G_{0}(\omega)\, d\mu(\omega).$ An immediate
consequence of this theorem is the following strong operator
convergence result.

To identify the limits of $\big[ a(x, D)(h \otimes \delta_{\Gamma}) \big]_{\epsilon}$
as $\epsilon \to 0$, we shall need the following
well-known, standard result on pseudodifferential operators.

\begin{lemma}\label{lemma.so.conv}
  Let $m \in \RR$, let $b_{\epsilon} \in \Smbmn$,
  $\epsilon \in [0, 1]$, and assume that the set
  $\{b_{\epsilon}\mid \epsilon \in [0, 1]\}$ is bounded in
  $\Smbmn$ (that is, for all $(x, \xi) \in
  \RR^{n} \times \RR^{n}$, we have $|\pa_{x}^{\alpha}
  \pa_{ \xi}^{\beta} b_{\epsilon}(x, \xi)| \le C_{\alpha, \beta}\jap{ \xi}^{m-|\beta|}$,
  with $C_{\alpha, \beta}$ independent of $\epsilon$). Let us also assume
  that, for all $(x, \xi) \in \RR^{n} \times \RR^{n}$ we have
  $\lim_{\epsilon \to 0} b_{\epsilon}(x, \xi) = b_{0}(x, \xi)$. Then, for
  all $h \in H^{s}(\RR^{n})$, we have the following convergence in
  $H^{s-m}(\RR^{n})$
  \begin{equation*}
    \lim_{\epsilon \to 0} b_{\epsilon}(x, D) h \seq
    b_{0}(x, D) h \in H^{s-m}(\RR^{n})\,.
  \end{equation*}
\end{lemma}

\begin{proof}
  The norm of $a(x, D)$ depends on finitely many of the semi-norms
  defining the topology on $\Smbmn$ (see
  \cite{KNW-22} in addition
  to the standard textbooks mentioned above). The result
  follows by combining this property with Theorem \ref{thm.aux.derivUnderInt}.
\end{proof}We now turn to the study of Problem \ref{problem}. To that end, recall that
$\Gamma := \{(x', x_{n}) \in \RR^{n} \mid x_{n} = 0 \}$ and
that $h \otimes \delta_{\Gamma}$ is the distribution
\[\langle h \otimes \delta_{\Gamma}, \phi \rangle := \int_{\Gamma}
h(x') \phi(x', 0) \, dx',\] where $\phi \in \CIc(\RR^{n})\,,$
see Equation \eqref{eq.def.Gamma} and
Lemma \ref{lemma.def.hdelta}. In particular, if $h \in L^{2}(\Gamma)$, then
$h \otimes \delta_{\Gamma} \in H^{-1/2 - \epsilon}(\RR^{n})$, for
all $\epsilon > 0$.
If $u : \RR^{n} \to \CC^{k}$ is continuous enough, recalling the notation
introduced in Definition \eqref{def.limit.values0}, we obtain that
then $u_{\epsilon} : \RR^{n-1} \to \CC^{k}$ is given by
  $u_{\epsilon} (x') \ede u(x', \epsilon)$, $x' \in \RR^{n-1}$.
Note that, we identify $\Gamma _\epsilon := \Gamma + \epsilon e_{n} :=
\{x=(x',x_{n})\in \RR^n : x_{n} = \epsilon\}$ with $\Gamma$ and their associated
function spaces with the translation $\tau_{-\epsilon}$, $\tau_{t}(x) := x + t e_{n}$,
as explained in Subsection \ref{ssec.problem}. Recall the potential operator
\begin{equation*} 
  \maS_{a}(h) \ede a(x, D)(h \otimes \delta_{\Gamma}) \in
  \CI(\RR^{n} \smallsetminus \Gamma)
\end{equation*}
of Equation \eqref{eq.def.maSa}.

Here is our first result on $\big[ a(x, D)(h \otimes \delta_{\Gamma}) \big]_{\epsilon}$,
regarded as a function on $\RR^{n-1} \simeq \{x_{n} = \epsilon\}$, as explained
in Subsection \ref{ssec.problem}, see especially Definition \eqref{def.limit.values0}.

\begin{proposition}\label{prop.side.limits}
  Let $h \in L^{2}(\RR^{n-1})$, let $a \in \Smbmn$, $m \in \RR$,
  and let $a_{x_{n}, t}$, $t \neq 0$, be as in Definition \ref{def.aepsilon}.
  Then, for any $\epsilon \neq 0$, $a(x, D) (h \otimes \delta_{\Gamma})$
  coincides with a smooth function on $\Gamma_{\epsilon} := \Gamma + \epsilon e_{n}$ and
  its restrictions to these sets satisfy
  \begin{equation*}
    \big[\maS_{a}(h)\big]_{\epsilon} \ede
    \big[a(x, D)(h \otimes \delta_{\Gamma})\big]_{\epsilon}
    \seq a_{\epsilon, \epsilon}(x', D') h\,.
  \end{equation*}
\end{proposition}

\begin{proof}
  Let $h \in \CIc(\Gamma)$ and $H:=h \otimes \delta_{\Gamma}$,
  see Lemma \ref{lemma.def.hdelta}. Then $H \in H^{-1}(\RR^{n})$ (because $-1 < -1/2$), so
  $a(x, D)H$ is defined. Since $H$ has singularities only on $\Gamma$,
  $a(x, D)H$ will also have singularities only on $\Gamma$, because pseudodifferential
  operators preserve the singular support. Therefore, $a(x, D)H$ coincides
  with a smooth function outside $\Gamma$. Moreover,
  \begin{equation*}
    \widehat H( \xi) \ede \maF(h \otimes \delta_{\Gamma})( \xi) \seq \maF_{x'}(h)( \xi')
    \seq \hat h( \xi')\,,
  \end{equation*}
  where $\maF_{x'}$ is the Fourier transform on $\Gamma$. For $a$ of order $< -1$,
  we then obtain
  \begin{align*}
    a(x, D)H(x) & \ede \frac{1}{(2\pi)^{n}}
    \int_{\RR^{n}} e^{\imath x \cdot \xi } a(x, \xi) \widehat H( \xi) \,d \xi\\
    & \seq  \frac{1}{(2\pi)^{n}}\int_{\RR^{n-1}} e^{\imath x'\cdot \xi'}
    \left ( \int_{\RR} e^{\imath x_{n}\cdot \xi_{n}}
    a(x, \xi) \hat h( \xi') \,d \xi_{n} \right ) \, d \xi'\\
    & \seq \frac{1}{(2\pi)^{n-1}} \int_{\RR^{n-1}} e^{\imath x'\cdot \xi'}
    a_{x_{n}, x_{n}}(x', \xi') \hat h( \xi') \, d \xi'\\
    & \, =:\, a_{x_{n}, x_{n}}(x', D')h(x')\,,
  \end{align*}
  with all integrals being absolutely convergent, because $\hat h(\xi')$ has
  rapid decay in $\xi'$ and $a$ is of order $-1$, which ensures the integrability in
  the $\xi_{n}$ variable. (This justifies the
  use of Fubini's theorem in the second equality).

  For $a$ of order $m \ge -1$, let $\phi \in \maS(\RR^{n})$, regarded as a symbol
  on $\RR^{n}$ that is constant in $x$. Let $\phi_{\epsilon}( \xi) := \phi(\epsilon \xi)$
  and assume that $\phi(0) = 1$. Then $b_{\epsilon} := a \phi_{\epsilon}$ is a symbol of
  order $-\infty$ and, for all $u \in H^{s}(\RR^{n})$, $b_{\epsilon}(x, D) u
  \to a(x, D)u$ in $H^{s - m}(\RR^{n})$ as $\epsilon \searrow 0$. The desired result
  for $h \in \CIc(\RR^{n-1})$ is obtained by applying the statement
  already proved to $b_{\epsilon}$ and passing to the limit for $\epsilon \searrow 0$
  and using Lemma \ref{lemma.so.conv}. The result for general $h$ follows by
  continuity in $h$ and the density of the space $\CIc(\RR^{n-1})$ in $L^2(\RR^{n-1})$.
\end{proof}

We have the following simple remark that simplifies our notation.

\begin{remark}\label{rem.aee}
  Proposition \ref{prop.side.limits} gives the equality $a_{\epsilon} = a_{\epsilon, \epsilon}$,
  which thus reconciles the definitions of $a_{\epsilon}$ (introduced in
  Definition \ref{def.limit.values0}) and that of $a_{\epsilon, \epsilon}$
  (introduced in Definition \ref{def.aepsilon}).
\end{remark}

We shall need the following standard lemma.

\begin{lemma}\label{lemma.homogeneity}
  Let $a \in \Smbmn$, $m \in \RR$, be essentially
  homogeneous of order $m$ (i.e., $a(x, \lambda \xi) = \lambda^{m}a(x, \xi)$
  for $\lambda |\xi|, |\xi| \ge 1$, see Equation \eqref{eq.def.ess.hom}).
  Let $a_{x_{n},t}$ be as in Definition \ref{def.aepsilon}(i).
  \begin{enumerate}[\rm (i)]
    \item For all $\lambda |\xi'|, |\xi'| \ge 1$ and all
    $t \neq 0$, we have
    \begin{equation*}
      a_{x_{n},t}(x', \lambda \xi') \seq \lambda^{m+1} a_{x_{n},
      \lambda t}(x', \xi')\,.
    \end{equation*}

    \item Let $\varepsilon \in \{-1, 1\}$ and assume that $a(x, -\xi) = \varepsilon a(x, \xi)$
    for $|\xi| \ge 1$, then, for all $\lambda |\xi'|, |\xi'| \ge 1$, we have
    \begin{equation*}
      a_{x_{n},t}(x', - \xi') \seq \varepsilon a_{x_{n}, - t}(x', \xi')\,.
    \end{equation*}

    \item If $m < -1$, the results of the previous two points hold also for $t = 0$.
  \end{enumerate}
\end{lemma}

\begin{proof}
  Assume first that $m < -1$. The substitution $ \xi_{n} = \lambda \tau$ yields
  \begin{multline*}
    a_{x_{n},t}(x', \lambda \xi') \ede
    \maF_{ \xi_{n}}^{-1} a(x', x_{n}, \lambda \xi', t)
    \ede \frac{1}{2\pi}
    \int_{\RR} e^{\imath t \xi_{n}} a(x', x_{n}, \lambda \xi', \xi_{n})
    \, d \xi_{n}\\
    \seq \frac{1}{2\pi}
    \int_{\RR} e^{\imath t \lambda \tau} a(x', x_{n}, \lambda \xi', \lambda \tau)
    \, \lambda d\tau\seq \frac{\lambda^{m+1}}{2\pi}
    \int_{\RR} e^{\imath t \lambda \tau} a(x', x_{n}, \xi', \tau)
    \, d\tau \\
    \seq \lambda^{m+1} a_{x_{n}, \lambda t}(x', \xi')\,.
  \end{multline*}
  If $m \ge -1$, the proof is the same by pairing test functions supported away from 0.
  This proves (i). The proof carries over to (ii) is completely similar.
  If $m < -1$, the proofs of (i) and (ii) obviously work also for $t = 0$,
  which proves (iii).
\end{proof}

\subsection{Normal lateral limits of pseudodifferential operators of orders $<-1$}
Let $\Gamma := \{x_{n} = 0\} \subset \RR^{n}$, as always in this paper.
For further use, let us record the following calculation.

\begin{lemma}\label{lemma.est.japs}
  If $s < -1$, then there exists $C_{s} > 0$ such that
  $\int_{\RR} \jap{ \xi}^{s} \, d \xi_{n} \seq C_{s}\jap{ \xi'}^{s+1}.$
\end{lemma}

\begin{proof}
  The substitution $ \xi_{n} = \jap{ \xi'} t$ gives:
  \begin{equation*}
    \int_{\RR} \jap{ \xi}^{s} \, d \xi_{n} = \int_{\RR}
    \big (\jap{ \xi'}^{2} + \jap{ \xi'}^{2} t^{2} \big )^{s/2} \jap{ \xi'}\, dt
    = \jap{ \xi'}^{s+1}\int_{\RR} (1 + t^{2})^{s/2} \, dt
    =: C_{s}\jap{ \xi'}^{s+1},
  \end{equation*}
where the last integral is convergent because $s < -1$.
\end{proof}

We can now prove the needed results for order $m < -1$ operators.
Recall the symbols $a_{x_{n}, \epsilon}$, $\epsilon \in \RR$, of Definition
\ref{def.aepsilon}.

\begin{proposition} \label{prop.symbols.zero}
  Let $m < -1$ and let $a \in \Smbmn$.
  \begin{enumerate}[\rm (i)]
    \item For any $(x_{n}, t) \in \RR^{2}$, we have that
    the map $\RR^{n-1} \times \RR^{n-1} \ni
    (x', \xi') \to a_{x_{n}, t}(x', \xi') \in \CC$ defines a symbol in
    $S^{m+1}(\RR^{n-1} \times \RR^{n-1})$.

    \item The family $\{a_{x_{n}, t} \, \mid \, (x_{n}, t) \in
    \RR^{2} \}$ is bounded in $S^{m+1}(\RR^{n-1} \times \RR^{n-1})$.

    \item The function $\RR^{n} \times \RR^{n} \ni (x, \xi', t) \to
    a_{x_{n}, t}(x', \xi') \in \CC$ is continuous.

    \item If $h \in H^{s}(\Gamma)$, then
    the function $\RR^{2} \ni (x_{n}, t) \to a_{x_{n},t}(x', D')h
    \in H^{s-m-1}(\Gamma) \simeq H^{s-m-1}(\RR^{n-1})$ is continuous.

    \item If $a \in S_{\cl}^{m}(\RR^{n} \times \RR^{n})$, then,
    for all $x_{n} \in \RR$,
    $a_{x_{n}, 0} \in S_{\cl}^{m+1}(\RR^{n-1} \times \RR^{n-1})$,
    that is, $a_{x_{n}, 0}$ is also a classical symbol.
  \end{enumerate}
\end{proposition}

\begin{proof}
  Let $a \in \Smbmn$.
  Then, by definition, for each $\alpha, \beta \in \ZZ_{+}^{n}$,
  there exists $C_{\alpha, \beta} > 0$ such that, for all $x, \xi \in \RR^{n}$,
  \begin{equation*}
    |\pa_{x}^{\alpha} \pa_{ \xi}^{\beta} a(x, \xi)|
    \le C_{\alpha, \beta} \jap{ \xi}^{m-|\beta|}\,.
  \end{equation*}
  Since $\jap{ \xi} \ge \jap{\xi_{n}}$ and $m -|\beta| \le m < -1$, we have
  $|\pa_{x}^{\alpha} \pa_{ \xi}^{\beta} a(x, \xi)| \le C_{\alpha, \beta}
  \jap{\xi_{n}}^{m-|\beta|}$. The bound on the right hand side
  of the last inequality is integrable and independent
  of $(x, \xi')$. The hypothesis of the theorem of derivation
  under the integral sign, Theorem \ref{thm.aux.derivUnderInt}, are thus satisfied
  and we obtain, for all $x = (x', x_{n})$ that
  \begin{equation*}
    a_{x_{n},t}(x', \xi') \ede
    \maF_{ \xi_{n}}^{-1} a(x', x_{n}, \xi', t)
    \ede \frac{1}{2\pi}
    \int_{\RR} e^{\imath t \xi_{n}} a(x', x_{n}, \xi', \xi_{n}) \, d \xi_{n}\,.
  \end{equation*}
  is well defined (i.e. the integral converges). We also obtain the relation
  \begin{equation*}
    \pa_{x}^{\alpha} \pa_{ \xi'}^{\beta} a_{x_{n}, t} (x', \xi')
    \seq \frac1{2\pi}\, \int_{\RR}
    \pa_{x}^{\alpha} \pa_{ \xi'}^{\beta} e^{\imath t \xi_{n}} a(x, \xi) \, d \xi_{n}
  \end{equation*}
  and the estimate
  \begin{equation}\label{eq.cont.integr}
    2 \pi \, |\pa_{x}^{\alpha} \pa_{ \xi'}^{\beta} a_{t}(x, \xi)|
    \, \leq \, C_{\alpha, \beta}\, \int_{\RR}  \jap{ \xi}^{m-|\beta|}\,
    d \xi_{n} \leq C_{m-|\beta|} C_{\alpha, \beta}\jap{ \xi'}^{m+1-|\beta|}\,,
  \end{equation}
  by Lemma \ref{lemma.est.japs}. This proves (i), (ii), and (iii) at once.

  The point (iv) is a consequence of Lemma \ref{lemma.so.conv} and the points already
  proved. The point (v) follows from Lemma \ref{lemma.homogeneity}, by linearity.
\end{proof}

We are now ready to formulate our main result concerning the limit/jump
values of classical \emph{matrix valued} pseudodifferential operators of order $< -1$.
So far, we have formulated and proved our results for \emph{scalar} symbols.
The proofs extend, however, immediately to the vector valued case.
We thus formulate the following results for sections of a trivial
$k$-dimensional vector bundle.

\begin{theorem}\label{thm.main.jump-2}
  Let $a \in M_{k}(\Smbmn)$ for some $m < -1$ and $k \in \NN$. We use the
  relation $a_{\epsilon} = a_{\epsilon, \epsilon}$, see Remark \ref{rem.aee}.
  \begin{enumerate}[\rm (i)]
    \item The matrix valued function $a_{0}$ is an order $m+1$ symbol given by
    \begin{equation*}
      a_{0}(x', \xi') \seq \frac1{2\pi}
      \int_{\RR} a( x', 0, \xi', \xi_{n}) \, d \xi_{n}\,.
    \end{equation*}
    If $a$ is classical, then $a_{0}$ is also classical.

    \item
    For all $s \in \RR$ and all $h \in H^{s}(\Gamma; \CC^{k}) = H^{s}(\Gamma)^{k}$ and
    all $\epsilon \neq 0$ small, we have $\big[ a(x, D)(h \otimes \delta_{\Gamma})]_{\epsilon}
    = a_{\epsilon}(x', D')h$.

    \item For all $s \in \RR$ and all $h \in H^{s}(\Gamma; \CC^{k})
    = H^{s}(\Gamma)^{k}$, we have
    \begin{multline*}
      [\maS_{a}h]_{\pm} \ede \big[ a(x, D)(h \otimes \delta_{\Gamma})]_{\pm}
      \ede \lim_{\epsilon \to 0\pm} \big[ a(x, D)(h \otimes \delta_{\Gamma})]_{\epsilon}\\
      \seq \lim_{\epsilon \to 0\pm} a_{\epsilon}(x', D')h \seq
      a_{0} (x', D') h \in H^{s - m - 1 }(\Gamma)^{k}\,.
    \end{multline*}
    (The distribution $h \otimes \delta_{\Gamma}$ was introduced in
    Lemma \ref{lemma.def.hdelta} and $u_{\epsilon}$ is the restriction of $u$
    to $\Gamma + \epsilon e_{n} \simeq \RR^{n-1}$.)

    \item If $h \in L^{2}(\Gamma)^{k}$, then
    $\maS_{a}h \ede a(x, D)(h \otimes \delta_{\Gamma})
    \in H^{s'}(\RR^{n})^{k}$ for any $s' \in (1/2, -m - 1/2)$, and hence we have the equality
    of traces (or restrictions)
    \begin{equation*}
      [\maS_{a}h]_{+} \seq [\maS_{a}h]_{-} \seq [\maS_{a}h]\vert_{\Gamma}
      \seq a_{0}(x', D')h \in H^{s'-1/2}(\Gamma)^{k}\,.
    \end{equation*}

    \item Let $k_{a_{0}(x', D')}$ be the distribution kernel of $a_{0}(x', D')$
    and $k_{a(x, D)}$ be the distribution kernel of $a(x, D)$. Then
    \begin{equation*}
      k_{a_{0}(x', D')}(x', y') \seq k_{a(x, D)}(x', 0, y', 0)\,,
      \quad x', y' \in \RR^{n-1}\,, \ x' \neq y'\,.
    \end{equation*}
  \end{enumerate}
\end{theorem}

\begin{proof}
  We can assume that $k = 1$. The general case is completely similar.

  The fact that $a_{0}$ is a symbol of order $m + 1$ in (i) was proved already in Proposition
  \ref{prop.symbols.zero}. (The fact that $a_{0}$ is classical if $a$ is
  will be proved in a slightly greater generality in Theorem \ref{thm.main.jump0}.)
  Finally, the stated equality in (i) follows from the
  fact that $\RR \ni \xi_{n} \to a(x, \xi', \xi_{n}) \in \CC$ is integrable
  on $\RR$ and the definition of $a_{0} = a_{0, 0}$ (see also
  Equation \eqref{eq.def.aepsilon.explicit}).

  The point (ii) is the content of Proposition \ref{prop.side.limits}.
  The first two equalities in (iii) are definitions. The third equality in (iii)
  is the consequence of (ii) just discussed, and, finally, the last
  equality in (iii) follows from Proposition \ref{prop.symbols.zero}(iv) for
  $x_{n} = t = \epsilon$. The results and limits are in the stated space
  because $a_{0}$ is a symbol of order $m +1$, as proved in (i).

  The first two equalities in (iv) follow from Lemma
  \ref{lemma.limits=traces} and the last equality is a consequence of (iii)
  already proved.

  Finally, let us prove (v). Let $\maF_{\xi}$ is the Fourier transform in the second variable
  (in $\xi$). Since $a$ defines a tempered distribution, $\maF_{\xi}^{-1}a$ is also a
  tempered distribution. We next use Lemma \ref{lemma.d.k}
  and the formula for $a_{0} = a_{0, 0}$ in (i) to obtain
  \begin{multline*}
    k_{a_{0}(x', D')}(x', y')  \seq \maF_{\xi'}^{-1}a_{0}(x', x'-y')
     \seq \maF_{\xi'}^{-1}\maF_{\xi_{n}}^{-1} a(x', 0, x'-y', 0)\\
     \seq \maF_{\xi}^{-1} a (x', 0 , x'-y', 0)
    \seq k_{a(x, D)} (x', 0 , y', 0) \,,
  \end{multline*}
  as claimed.
\end{proof}

\begin{remark}\label{rem.u.det}
  It is known that the distribution kernel of a negative order operator
  cannot be supported on the diagonal (because the pseudodifferential
  operators supported on the diagonal are local, and hence they are multiplication
  operators, by Peetre's theorem \cite{H-W, Hormander1, Hormander3}, that is,
  order zero operators).
  Therefore, the restriction of the distribution kernel of a negative
  order operator outside the diagonal completely determines the operator
  (this is relevant in view of Theorems \ref{thm.main.jump-2}(v)
  and \ref{thm.main.jump0}(i) and (ii), for whose proof, a similar argument was used).
  Thus, if $m < -1$, the knowledge of the distribution kernel $k_{a(x, D)}$
  of a pseudodifferential operator $a(x, D)$ of order $m$
  completely determines the distribution kernel $k_{a_{0}(x', D')}$
  and, hence, it completely determines $a_{0}$ and $a_{0}(x', D')$.
\end{remark}

\subsection{Normal lateral limits of pseudodifferential operators of orders $-1$}
We now study normal lateral limits of pseudodifferential
operators of order $m = -1$. For simplicity, we consider only
\emph{classical pseudodifferential operators.} Several points of the next
proposition have already been proved, but the most significant calculation
(the point (iii) below) has not yet been performed. We again formulate
our result for matrices of symbols.

\begin{proposition} \label{prop.symbols.m=-1}
  Let $a \in M_{k}(S^{-1}_{\cl}(\RR^{n} \times \RR^{n})$. We let $a_{x_{n}, t}(x', \xi')$
  be as in Definition \ref{def.aepsilon} (we apply that definition to each
  entry of the matrix). We assume that $\sigma_{-1}(a; x, \xi)$ is \emph{odd} in
  $\xi \in \RR^{n}$.
  \begin{enumerate}[\rm (i)]
    \item For any multi-indices $\alpha \in \ZZ_{+}^{n}$ and
    $\beta \in \ZZ_{+}^{n-1}$,
    there exists $C_{\alpha, \beta} > 0$ such that, for all $(x, \xi', t)
    \in \RR^{n} \times \RR^{n-1} \times \RR$, we have
    \begin{equation}\label{eq.estimates}
      |\pa_{x}^{\alpha}\pa_{ \xi'}^{\beta} a_{x_n,t}(x', \xi')|
      \le C_{\alpha, \beta}\jap{ \xi'}^{-|\beta|}\,.
    \end{equation}

    \item The set $\{ a_{x_{n}, t}, a_{x_{n}, 0\pm} \mid x_{n}, t \in \RR \}$ is
    a bounded subset of $S^{0}(\RR^{n-1} \times \RR^{n-1})$.

    \item The function $\RR^{n} \times \RR^{n-1} \times \RR \ni (x, \xi', t)
    \mapsto a_{x_{n}, t}(x', \xi') \in M_{k}(\CC)$ is continuous except at $t = 0$,
    where it has lateral limits $a_{x_{n}, 0\pm}(x', \xi')$.
  \end{enumerate}
\end{proposition}

\begin{proof}
  We proceed as in the proof of Lemma \ref{lemma.Fsymbols}, that is, we write
  \begin{equation*}
    a \seq b_{-1} + r_{-2}\,, \qquad
    b_{-1} \in S^{-1}(\RR^{n} \times \RR^{n})\ \mbox{ and }\
    r_{-2} \in S^{-2}(\RR^{n} \times \RR^{n})\,,
  \end{equation*}
  where $b_{-1}$ is essentially homogeneous of order $-1$. It follows that $r_{-2}$ is a
  classical symbol and hence it satisfies all the statements of our proposition,
  by Proposition \ref{prop.symbols.zero} (in the same order of points).
  Since all the statements of our proposition
  depend linearly on $a$, it is enough to assume that $a = b_{-1}$, that is,
  that $a$ is \emph{is eventually
  homogeneous of order $-1$.}

  To prove (i), we shall consider separately the cases $t \neq 0$ and $t = 0$.
  \smallskip

  \textbf{Case 1:\ $t\neq 0$; Step 1.1: The main calculations}\
  The idea is to substract from $a$ the ``main singularity''
  using the function  $Z(x) := x^{-1}(1 - \chi_{0}(x))$ of Lemma \ref{lemma.Fpvx-1},
  whose notation we will use in the rest of the proof.
  This will allow us to express the desired $a_{x_{n}, t}(x', \xi')$ in terms of
  the inverse Fourier transform of an integrable function and of $W := \maF^{-1}Z$.
  Recall from Lemma \ref{lemma.Fpvx-1} that
  $\chi_{0} \in \CIc(\RR)$ is an even function that is equal to 1 in a neighborhood
  of $0$. In addition to the assumptions of that lemma, let us assume that
  $\chi_{0}$ has support in $(-1, 1)$, because then $Z$ will also be essentially
  homogeneous of order $-1$. We shall thus study the function
  \begin{equation}\label{eq.def.u}
    u(x, \xi', \tau) \ede  a(x, \xi', \tau) - \JC(a; x) Z(\tau)
    \seq a(x, \xi', \tau) - a(x, e_{n})\tau^{-1}(1 - \chi_{0}(\tau))\,,
  \end{equation}
  which is essentially homogeneous of order $-1$ (because the support of $\chi_{0}$
  is contained in $(-1, 1)$). The function $u$ will play an important role in what
  follows. It will be written, for short, as $u := a - \JC Z$.
  For any multi-index $\alpha \in \ZZ_{+}^{n}$, the
  fact that $a$ and $u$ are essential homogeneous of order $-1$ with
  respect to $\tau$ gives (by dividing by $|\tau| \ge 1$)
  \begin{align*}
    \tau^{2} | \pa_{x}^{\alpha} u (x, \xi', \tau) |
    & \ede |\tau|\, \big |  \pa_{x}^{\alpha} a(x, |\tau|^{-1} \xi', \sgn(\tau))
    - \pa_{x}^{\alpha} a(x, e_{n})\sgn(\tau) \big |\\
    & \seq |\tau|\, \big |  \pa_{x}^{\alpha} a(x, |\tau|^{-1} \xi', \sgn(\tau))
    - \pa_{x}^{\alpha}
    a(x, 0, \sgn(\tau)) \big |
    \le C| \xi'|\,,
  \end{align*}
  where $C = \sup | \nabla_{ \xi }\pa_{x}^{\alpha} a(x, \xi) |$, which is
  finite, because $a$ is an order $-1$ symbol and where we have used the
  hypothesis that $\sigma_{-1}(a; x, \xi)$ is odd in $\xi$. For $|\tau| \ge 1$, this shows that
  \begin{equation} \label{eq.est.u}
    |\pa_{x}^{\alpha} u (x, \xi', \tau)| \ede |\pa_{x}^{\alpha}
    a(x, \xi', \tau) - \pa_{x}^{\alpha} a(x, e_{n})Z(\tau)|
    \le G_{\alpha} \jap{ \xi'} \jap{\tau}^{-2}\,,
  \end{equation}
  for some $G_{\alpha} > 0$ independent of $(x, \xi')$. The left hand side
  $|\pa_{x}^{\alpha} u (x, \xi', \tau)|$ of the last displayed relation is bounded,
  so the inequality \eqref{eq.est.u} actually holds for all $\tau$ (with possibly
  different constants $G_{\alpha}$).
  Consequently, $\jap{ \xi'}^{-1}\pa_{x}^{\alpha} u (x, \xi', \tau)$ is integrable
  with respect to $\tau$ uniformly in $(x, \xi')$. That is, the assumptions of the
  theorem giving the derivability under the integral sign (Theorem
  \ref{thm.aux.derivUnderInt}) are satisfied and
  we obtain that, for all $(x, \xi') \in \RR^{n} \times \RR^{n-1}$ and $t \neq 0$,
  \begin{equation*}
    |\pa_{x}^{\alpha} \maF_{ \xi_{n}}^{-1} u(x, \xi', t)| \le G_{\alpha}' \jap{ \xi'}\,.
  \end{equation*}
  Let $W := \maF^{-1} Z$, again as in Lemma \ref{lemma.Fpvx-1},
  and recall, from that lemma, that $W\in L^{\infty}(\RR)$, more
  precisely, we shall use $|W| \le 1$. By
  applying the inverse Fourier transform, we then obtain that
  \begin{multline}\label{eq.proof.e0}
    |\pa_{x}^{\alpha} a_{x_{n}, t}(x', \xi')| \ede
    \big | \pa_{x}^{\alpha} \maF_{ \xi_{n}}^{-1} a(x, \xi', t) \big |
    \seq \big |\pa_{x}^{\alpha} \maF_{ \xi_{n}}^{-1} u(x, \xi', t) +
    \pa_{x}^{\alpha} a(x, e_{n}) \maF_{ \xi_{n}}^{-1} Z(t) \big |\\
    \le \big |\pa_{x}^{\alpha} \maF_{ \xi_{n}}^{-1} u (x, \xi', t)\big | +
    \big | \pa_{x}^{\alpha} a(x, e_{n}) W(t) \big |
    \le G_{\alpha}'\jap{ \xi'} + \sup | \pa_{x}^{\alpha} a(x, e_{n}) |
    \le C_{\alpha}\jap{ \xi'}\,.
  \end{multline}

  For $|\beta| > 0$, the derivatives $\pa_{\xi_{j}}$ kill the term $a(x, e_{n})Z(\tau)$,
  in $u := a - \JC Z$ (Equation \eqref{eq.def.u})
  which is independent of $\xi$ and hence,
  in analogy with the estimate \eqref{eq.est.u}, we have
  \begin{multline}\label{eq.est.u2}
    |\pa_{ \xi'}^{\beta} \pa_{x}^{\alpha} u (x, \xi', \tau)| \ede
    |\pa_{ \xi'}^{\beta} \pa_{x}^{\alpha}
    a(x, \xi', \tau) - \pa_{ \xi'}^{\beta} \pa_{x}^{\alpha} a(x, e_{n})Z(\tau)| \\
    \seq |\pa_{ \xi'}^{\beta} \pa_{x}^{\alpha} a(x, \xi', \tau)|
    \le G_{\alpha, \beta} \jap{ \xi}^{-1-|\beta|}
    \le G_{\alpha, \beta} \jap{\tau}^{-2} \jap{\xi'}^{1 -|\beta|}\,,
  \end{multline}
  for some $G_{\alpha, \beta} > 0$ independent of $(x, \xi')$,
  where we have used that $\jap{\tau} \le \jap{(\xi', \tau)}$ and
  $\jap{\xi'} \le \jap{(\xi', \tau)}$.

  In view of our results above, all the $ \xi'$-derivatives
  $\pa_{ \xi'}^{\beta} \pa_{x}^{\alpha} u (x, \xi', \tau)$ of the
  function $\pa_{x}^{\alpha} u (x, \xi', \tau)$ are integrable in $\tau$ uniformly
  in $(x, \xi')$. Therefore
  \begin{equation*}
    |\pa_{ \xi'}^{\beta} \pa_{x}^{\alpha} \maF_{ \xi_{n}}^{-1} u(x, \xi', \tau)|
    \le G_{\alpha, \beta}' \jap{\xi'}^{1 -|\beta|}\,.
  \end{equation*}
  We then proceed as in \textbf{Step 1.1}, but using that
  $\pa_{ \xi'}^{\beta} \pa_{x}^{\alpha} \big[a(x, e_{n}) \maF_{ \xi_{n}}^{-1} Z(t)\big] = 0$
  (which we have already noticed above). We obtain
  \begin{multline}\label{eq.mult.aux}
    |\pa_{ \xi'}^{\beta} \pa_{x}^{\alpha} a_{x_{n}, t}(x', \xi')| \ede
    \big|\pa_{ \xi'}^{\beta} \pa_{x}^{\alpha} \maF_{ \xi_{n}}^{-1} a(x, \xi', t) \big|\\
    \seq \big |\pa_{ \xi'}^{\beta} \pa_{x}^{\alpha} \maF_{ \xi_{n}}^{-1} u(x, \xi', t) +
    \pa_{ \xi'}^{\beta} \pa_{x}^{\alpha} a(x, e_{n}) \maF_{ \xi_{n}}^{-1} Z(t) \big |
    \seq \big|\pa_{ \xi'}^{\beta} \pa_{x}^{\alpha} \maF_{ \xi_{n}}^{-1}u(x, \xi',t)\big|\\
    \le G_{\alpha, \beta}' \jap{\xi'}^{1 -|\beta|} \le G_{\alpha, \beta}' \,.
  \end{multline}

  \textbf{Step 1.2:\ Proof of (i) when $t \neq 0$ and $| \xi'|\le 1$.}\
  In this case, the estimates \eqref{eq.estimates} of point (i)
  follow from Equation \eqref{eq.mult.aux} (the last displayed equation)
  and from Equation \eqref{eq.proof.e0}.

  \textbf{Step 1.3:\ Proof of (i) when $t \neq 0$ and $| \xi'|\ge 1$.}\
  Again, we notice that we have proved that $\pa_{ \xi'}^{\beta} \pa_{x}^{\alpha}
  a_{x_{n},t}(x', \xi')$ exists for all $(x, \xi')$. We shall now remove the
  constraint $| \xi'| \le 1$ using the essential homogeneity of order $-1$ of $a$
  and the homogeneity of the Fourier transform (this easy fact is proved in
  most basic textbooks, see, for instance, \cite{Hormander1, Taylor1}).
  More precisely, Lemma \ref{lemma.homogeneity}
  states that $a_{x_{n},t}(x', \lambda \xi') \seq  a_{x_{n}, \lambda t}(x', \xi')$
  (we use the case $m = -1$). By replacing $ \xi'$ with $\lambda^{-1} \xi'$
  and by taking the product of derivatives $\pa_{x}^{\alpha}$, we obtain
  $\pa_{x}^{\alpha} a_{x_{n},t}(x', \xi') \seq
  \pa_{x}^{\alpha} a_{x_{n}, \lambda t}(x', \lambda^{-1} \xi')$.
  This relation, for $\lambda := | \xi'|$, gives
  \begin{equation*}
    |\pa_{x}^{\alpha} a_{x_{n},t}(x', \xi')| \seq
    |\pa_{x}^{\alpha} a_{x_{n}, \lambda t}(x', \lambda^{-1} \xi')|
    \le \sqrt{2} C_{\alpha}\,,
  \end{equation*}
  by Equation \eqref{eq.proof.e0} for $t \neq 0$ replaced with $\lambda^{-1} t \neq 0$
  and $ \xi'$ replaced with $\lambda^{-1} \xi'$, the later having length $=1$
  (so $\jap{\lambda^{-1} \xi'} = \sqrt{2}$).
  Further taking the $\pa_{ \xi'}^{\beta}$ derivative, this further gives
  \begin{multline*}
    |\pa_{ \xi'}^{\beta} \pa_{x}^{\alpha} a_{x_{n},t}(x', \xi')| \seq
    \big |\pa_{ \xi'}^{\beta} \big [ \pa_{x}^{\alpha} a_{x_{n}, \lambda t}
    (x', \lambda^{-1} \xi') \big ] \big |\\
    \seq \lambda^{-|\beta|} \big|\big [ \pa_{ \xi'}^{\beta}  \pa_{x}^{\alpha}
    a_{x_{n}, \lambda t}\big ]
    (x', \lambda^{-1} \xi')  \big |
    \le G_{\alpha, \beta}' \jap{\xi'}^{-|\beta|}\,,
  \end{multline*}
  where the last step is by Equation \eqref{eq.mult.aux}.
  The relation \eqref{eq.estimates} is thus also proved in this case.
  \smallskip

  \textbf{Case 2:\ $t= 0$.}\ This is similar to the case $t \neq 0$.
  In fact, the only place where we have used $t \neq 0$ was when we estimated
  the function $u := a - \JC Z$ of Equation \eqref{eq.def.u} and
  its Fourier transform employing the properties of the Fourier transform function
  $W := \maF^{-1}Z$. For $t = 0$, we replace $u$ with the symmetrization
  $u_{0}(x, \xi', \tau) := \frac12 \big (a(x, \xi', \tau) + a(x, \xi', -\tau)\big)$
  and the function $W$ becomes unnecessary. We include now the calculation,
  but we do not repeat the arguments that are easy and very similar to the ones
  used in the previous case. Also, we do not separate the proof formally in substeps.
  The function $u$ will be replaced with a function $u_{0}$ that
  has similar decay properties and is moreover even in $\tau$. We now proceed
  according to this plan.

  First, for any multi-index $\alpha \in \ZZ_{+}^{n}$, the essential homogeneity
  of $a$ with respect to $\tau \ge 1$ gives successively the following:
  \begin{multline*}
    2 \tau | \pa_{x}^{\alpha} u_{0} (x, \xi', \tau) |
    \ede  \big |  \pa_{x}^{\alpha} a(x, \tau^{-1} \xi', 1)
    + \pa_{x}^{\alpha} a(x, \tau^{-1} \xi', -1)  \big |\\
    \le  \big |  \pa_{x}^{\alpha} a(x, \tau^{-1} \xi', 1)
    - \pa_{x}^{\alpha} a(x, 0, 1) \big | +
     \big |  \pa_{x}^{\alpha} a(x, \tau^{-1} \xi', -1)
    - \pa_{x}^{\alpha} a(x, 0, -1) \big |
    \le C| \xi'|/\tau\,,
  \end{multline*}
  where $C = \sup | \nabla_{ \xi }\pa_{x}^{\alpha} a(x, \xi) |$
  and where we have used the hypothesis $\sigma_{-1}(a; x, \xi)$
  is odd in $\xi$. This shows that
  \begin{equation*} 
    |\pa_{x}^{\alpha} u_{0} (x, \xi', \tau)|
    \le G_{\alpha} \jap{ \xi'} \jap{\tau}^{-2}\,,
  \end{equation*}
  for some $G_{\alpha} > 0$ independent of $(x, \xi')$, that is,
  $\jap{ \xi'}^{-1}\pa_{x}^{\alpha} u_{0} (x, \xi', \tau)$
  is uniformly integrable in $(x, \xi')$, with $| \xi'| \le 1$. We can therefore use
  the derivability under the integral sign (Theorem \ref{thm.aux.derivUnderInt})
  to obtain that
  $|\pa_{x}^{\alpha} \maF_{ \xi_{n}}^{-1} u_{0}( x, \xi', \tau)| \le G_{\alpha}' \jap{\xi'}$.
  By applying the inverse Fourier transform, we then obtain for $| \xi'| \le 1$ that
  \begin{equation*}
    |\pa_{x}^{\alpha} a_{x_{n}, 0}(x', \xi')| \ede
    \big |\pa_{x}^{\alpha} \maF_{ \xi_{n}}^{-1} u_{0}( x, \xi', 0) \big |
    \le G_{\alpha}'\jap{ \xi'} \le G_{\alpha}'
  \end{equation*}
  is uniformly bounded in $(x, \xi')$.
  This proves the estimate \eqref{eq.estimates} for $t = 0$, $|\beta| = 0$,
  and $| \xi'| \le 1$.

  Let us assume now that $t = 0$, $|\beta|>0$, and $| \xi'| \le 1$.
  This follows more or less the same calculation as in the previous case.
  For $|\beta| > 0$, in analogy with the estimate \eqref{eq.est.u}, we have
  \begin{equation} \label{eq.est.u3}
    2|\pa_{ \xi'}^{\beta} \pa_{x}^{\alpha} u_{0}(x, \xi', \tau)| \ede
    |\pa_{ \xi'}^{\beta} \pa_{x}^{\alpha} \big[ a(x, \xi', \tau) +
    a(x, \xi', -\tau) \big] |
    \le G_{\alpha, \beta} \jap{\tau}^{-2} \jap{\xi'}^{-|\beta|}\,.
  \end{equation}
  In Equation \eqref{eq.est.u3}, all the derivatives
  $\pa_{ \xi'}^{\beta} \pa_{x}^{\alpha} u_{0} (x, \xi', \tau)$ of the
  function $\pa_{x}^{\alpha} u_{0} (x, \xi', \tau)$ are uniformly
  integrable in $(x, \xi')$. Therefore
  $|\pa_{ \xi'}^{\beta} \pa_{x}^{\alpha} \maF_{ \xi_{n}}^{-1} u_{0}( x, \xi', \tau)|
  \le G_{\alpha, \beta}' \jap{\xi'}^{-|\beta|}$ and
  \begin{equation*}
    |\pa_{ \xi'}^{\beta} \pa_{x}^{\alpha} a_{x_{n}, 0}(x', \xi')| \ede
    \big|\pa_{ \xi'}^{\beta} \pa_{x}^{\alpha}
    \maF_{ \xi_{n}}^{-1}u_{0}( x, \xi', 0)\big|
    \le G_{\alpha, \beta}' \jap{\xi'}^{-|\beta|} \le G_{\alpha, \beta}'\,.
  \end{equation*}
  This gives the estimate \eqref{eq.estimates} with $C_{\alpha, \beta}
  = G_{\alpha, \beta}'$ when $| \xi'| \le 1$.

  To complete the proof of (i), let now $t = 0$ and $| \xi'|\ge 1$.\
  Again, we notice that we have proved that
  $\pa_{ \xi'}^{\beta} \pa_{x}^{\alpha} a_{x_{n}, 0}(x', \xi')$ exists
  for all $(x, \xi')$ and that it is obtained by commuting differentiation
  with integration. We shall now remove the constraint $| \xi'| \le 1$ using
  the essential homogeneity of order $-1$ of $a$ and of $u_{0}$ and
  the homogeneity of the Fourier transform
  \cite{Hormander1, Taylor1}. We can nolonger use Lemma \ref{lemma.homogeneity},
  but we can proceed as in its proof, since we are dealing with integrable functions.
  Again for $\lambda = |\xi'|\ge 1$ and using in the last step the substitution
  $\tau = \lambda s$, $s \in \RR$, we have
  \begin{multline*}
    \pa_{ \xi'}^{\beta} \pa_{x}^{\alpha} a_{x_{n}, 0}(x', \xi') \seq
    \frac1{4 \pi} \pa_{ \xi'}^{\beta} \pa_{x}^{\alpha} \int_{\RR}
    \big( a(x, \xi', \tau) + a(x, \xi', -\tau))\, d\tau \\
    \seq
    \frac1{4 \pi \lambda } \int_{\RR} \pa_{ \xi'}^{\beta} \pa_{x}^{\alpha}
    \big(  a(x, \lambda^{-1} \xi', \lambda^{-1}\tau)
    + a(x, \lambda^{-1} \xi', -\lambda^{-1} \tau))\,
    d\tau\\
    \seq \frac{\lambda^{-1-|\beta|}}{4 \pi} \int_{\RR}
    \big( (\pa_{ \xi'}^{\beta} \pa_{x}^{\alpha}a)(x, \lambda^{-1} \xi',
    \lambda^{-1}\tau) + (\pa_{ \xi'}^{\beta} \pa_{x}^{\alpha}a)(x, \lambda^{-1} \xi',
    -\lambda^{-1} \tau))\, d\tau\\
    \seq \frac{\lambda^{-|\beta|}}{4 \pi} \int_{\RR}
    \big( (\pa_{ \xi'}^{\beta} \pa_{x}^{\alpha}a)(x, \lambda^{-1} \xi', s)
    + (\pa_{ \xi'}^{\beta} \pa_{x}^{\alpha}a)(x, \lambda^{-1} \xi', -s))\,
    ds\,.
  \end{multline*}
  Equation \eqref{eq.est.u3} then gives (note that, in that equation, the term
  $\jap{\xi'}^{-|\beta|} \le 1$)
  \begin{multline*}
    |\pa_{ \xi'}^{\beta} \pa_{x}^{\alpha} a_{x_{n}, 0}(x', \xi')|
    \leq \frac{\lambda^{-|\beta|}}{4 \pi} \int_{\RR}
    \big| (\pa_{ \xi'}^{\beta} \pa_{x}^{\alpha}a)(x, \lambda^{-1} \xi', s)
    + (\pa_{ \xi'}^{\beta} \pa_{x}^{\alpha}a)(x, \lambda^{-1} \xi', -s)|\,
    ds\\
    \leq G_{\alpha, \beta}\lambda^{-|\beta|} \int_{\RR} \jap{s}^{-2}ds
    \seq C_{\alpha, \beta}  \jap{\xi'}^{-|\beta|}\,.
  \end{multline*}
  This completes the proof of (i).

  The point (ii) now follows easily. The point (i) gives by definition
  that $\{a_{x_{n}, t} \mid x_{n}, t \in \RR\}$ is a bounded subset of $S^{0}(\RR^{n}
  \times \RR^{n})$. Then, we notice that
  $\JC(a; x) := \sigma_{-1}(a; x, e_{n}) \in S_{\cl}^{0}(\RR^{n-1} \times \RR^{n-1})$ and
  it is a bounded family of
  order 0 classical symbols on $\RR^{n-1}$ parametrized by $x_{n} \in \RR$ (all these
  symbols are multiplications operators, since they are independent of $\xi'$). The
  formula $a_{x_{n}, 0\pm }(x, \xi')
  \seq \pm \frac{\imath}2 \JC(a; x) + a_{x_{n}, 0}(x, \xi')$ then gives (ii).

  Finally, (iii) follows from Corollary \ref{cor.lemma.Fsymbols} using the notation
  of Definition \ref{def.aepsilon}.
\end{proof}

To formulate our main result on lateral (or boundary) limits of
pseudodifferential operators of order $= -1$ (Theorem \ref{thm.main.jump0}
next), recall the following notation and results:
\begin{enumerate}[\rm (i)]
  \item $u_{\epsilon} := u\vert{\Gamma_{\epsilon}}$ and $u_{\pm}
  := \lim_{\epsilon \to 0\pm} u_{\epsilon}$
  were introduced in Definition \eqref{def.limit.values0}.

  \item The distribution $\langle h \otimes \delta_{\Gamma}, \phi \rangle
  : = \int_{\Gamma} h \phi dx'$ was introduced in Lemma \ref{lemma.def.hdelta}.

  \item We know that $a_{t} = a_{t, t}$, $t \in \RR$, and $a_{0\pm} := a_{0, 0\pm}$,
  see Definition \eqref{def.aepsilon}, Proposition \ref{prop.side.limits}, and
  the comment following its proof.
\end{enumerate}

We again formulate the following theorem for sections of a $k$-dimensional
trivial vector bundle on $\RR^{n}$.

\begin{theorem}\label{thm.main.jump0}
  We use the notation in Definition \ref{def.aepsilon} and the relation
  $a_{\epsilon} = a_{\epsilon, \epsilon}$, see Remark \ref{rem.aee}. Let
  $a \in S^{-1}_{\cl}(\RR^{n} \times \RR^{n}; M_{k}(\CC))$ and assume that
  $\sigma_{-1}(a)$ is odd \lpar in the sense that $\sigma_{-1}(a; x, -\xi)
  = - \sigma_{-1}(a; x, \xi)$ for all $\xi \in \RR^{n}$\rpar.

  \begin{enumerate}[\rm (i)]
    \item Let $k_{a_{0}(x', D')}$ be the distribution kernel of $a_{0}(x', D')$
    and $k_{a(x, D)}$ be the distribution kernel of $a(x, D)$. Then both
    $k_{a_{0}(x', D')}(x', y')$ and $k_{a(x, D)}(x, y)$ are smooth for
    $x', y' \in \RR^{n-1}$, $x' \neq y'$, and
    \begin{equation*}
      k_{a_{0}(x', D')}(x', y') \seq k_{a(x, D)}(x', 0, y', 0)\,.
    \end{equation*}

    \item $\sigma_{0}(a_{0})$ is also odd, and, for
    all $h \in \CIc(\RR^{n-1})$,
    \begin{equation*}
      a_{0}(x', D')h(x) \seq \pv \,
      \int_{\RR^{n-1}} k_{a(x, D)}(x', y') h(y')\,dy'\,.
    \end{equation*}

    \item The three symbols $a_{0} := a_{0, 0}$ and $a_{0\pm} := a_{0, 0\pm}$ are
    order zero classical, with principal parts \lpar or symbols\rpar
    \begin{equation*}
      \sigma_{0}(a_{0}; x', \xi') \seq \frac1{2\pi}
      \int_{\RR} \frac{\sigma_{-1}(a; x', 0, \xi', \xi_{n})
      + \sigma_{-1}(a; x', 0, \xi', - \xi_{n})}2 \, d \xi_{n}
    \end{equation*}
    and $\sigma_{0}(a_{0\pm}; x', \xi') = \pm \frac{\imath}2 \sigma_{-1}(a; x', e_{n})
    + \sigma_{0}(a_{0}; x', \xi')$

    \item
    For all $s \in \RR$ and all $h \in H^{s}(\Gamma; \CC^{k}) = H^{s}(\Gamma)^{k}$ and
    all $\epsilon \neq 0$ small, we have $\big[ a(x, D)(h \otimes \delta_{\Gamma})]_{\epsilon}
    = a_{\epsilon}(x', D')h$.

    \item For all $s \in \RR$ and for all
    $h \in H^{s}(\RR^{n-1}; \CC^{k}) = H^{s}(\RR^{n-1})^{k}$, we have
    \begin{equation*}
      [\maS_{a}h]_{\pm} \ede
      \lim_{\epsilon \to \pm 0} \big[ a(x, D)(h \otimes \delta_{\Gamma})]_{\epsilon}
      \seq \lim_{\epsilon \to \pm 0} a_{\epsilon}(x', D')h \seq
      a_{0\pm} (x', D') h \in H^{s}(\RR^{n-1})^{k}\,.
    \end{equation*}
  \end{enumerate}
\end{theorem}

\begin{proof} We again assume $k = 1$. Let us prove (i).
  As is well known the distribution kernel $k_{a(x, D)}(x, y)$ of $a(x, D)$ is
  smooth for $x \neq y$ and there it is given by
    $k_{a(x, D)}(x, y) \seq \maF_{\xi}^{-1}a(x, x-y)\,,$
  where $\maF_{\xi}^{-1}$ is the inverse Fourier transform in the second variable
  (i.e., in $\xi \in \RR^{n}$; see Lemma \ref{lemma.d.k}).
  Let $\tilde a(x, \xi) := a(x, \xi', -\xi_{n})$.
  Using the formula defining $a_{0} = a_{0, 0}$ (see Definition
  \ref{def.aepsilon}), we obtain
  \begin{align*}
    k_{a_{0}(x', D')}(x', y') & \seq \maF_{\xi'}^{-1}a_{0}(x', x'-y')\\
    & \seq \frac12 \maF_{\xi'}^{-1}\maF_{\xi_{n}}^{-1} (a + \tilde a)(x', 0, x'-y', 0)\\
    & \seq \maF_{\xi}^{-1} a (x', 0, x'-y', 0) \\
    & \seq k_{a} (x', 0, x'-y', 0) \,,
  \end{align*}
  because $\maF_{\xi}^{-1} \tilde a (x, y) = \maF_{\xi}^{-1} a (x, y', -y_{n})$
  and $y_{n}=0$ in our equation. This proves the first point.

  Property (ii) is a general result on pseudodifferential operators of order $0$ with
  odd symbol \cite{H-W}. Indeed, (ii) was already proved for operators of order $< -1$,
  see Theorem \ref{thm.main.jump-2}. We can thus assume that $a$ is odd and
  essentially homogeneous of order $-1$. Then it follows from Lemma
  \ref{lemma.homogeneity} that $a_{0}$ is also odd and homogeneous of order 0.
  Let $B u(x') :=  \pv \,
  \int_{\RR^{n-1}} k_{a_{0}(x, D)}(x', y')
  u(y')\,dy'$. In view of (i), it is enough to prove that
  $f := a_{0}(x', D') - B = 0$. The distribution kernel of the operator $f$ is
  hence supported on the diagonal. This means that $f$ preserves supports.
  It is then known (Peetre's theorem) that $f:=a_{0}(x', D') - B$ is a multiplication
  operator by some function $f_{0}$, see, for example,
  Theorem 8.1.8 (ii) in \cite{H-W} (on pages 508 and 509). Thus the distribution
  kernel of $f$ is of the form
  \begin{equation*}
    k_{f}(x', y') \seq f_{0}(x')\delta(x'-y') \seq f_{0}(x')\delta(y'-x')
    \seq k_{f}(x', 2x' - y') \,.
  \end{equation*}
  But $k_{a_{0}(x, D)}(x', y') = - k_{a_{0}(x, D)}(x', 2x' -y')$, since $a_{0}$
  is an odd symbol (we have used a similar calculation to the one just performed,
  the formula for the distribution kernel
  in Lemma \ref{lemma.d.k}, and the fact that the Fourier transform of an odd
  function is an odd function). The same property is satisfied by the distribution
  kernel of $B$. Therefore $k_{f}(x', y') = - k_{f}(x', 2x' - y') = - k_{f}(x', y')$,
  and, hence $f = 0$.

  Let us prove the third point. As in the proof of Proposition \ref{prop.symbols.m=-1},
  we may assume
  that $a$ is essentially homogeneous of order $-1$. Then $\sigma_{-1}(a; x, \xi) =
  a(x, \xi)$ for $| \xi| \ge 1$. Lemma \ref{lemma.homogeneity} then gives that
  $a_{x_{n}, 0}$ is also essentially homogeneous
  of order $0$, and hence $\sigma_{0}(a_{x_{n}, 0}; x', \xi') =
  a_{x_{n}, 0}(x', \xi')$ for $| \xi| \ge 1$. We then have for $|\xi'| \ge 1$
  \begin{multline*}
    \sigma_{0}(a_{x_{n}, 0}; x', \xi') \seq  a_{x_{n}, 0}(x', \xi')
    \seq \int_{\RR} \frac{a(x, \xi', \tau) + a(x, \xi', -\tau)}{4 \pi}\, d\tau\\
    \seq \int_{\RR} \frac{\sigma_{-1}(a; x, \xi', \tau) +
    \sigma_{-1}(a; x, \xi', -\tau)}{4 \pi}\, d\tau\,.
  \end{multline*}
  Setting $x_{n} = 0$ and using the definitions of $a_{0, 0\pm} := a_{0, 0} \pm
  \frac{\imath}2 \JC(a)$ (Definition \ref{def.aepsilon}) we obtain (iii).

  The point (iv) is the content of Proposition \ref{prop.side.limits} (see
  also Remark \ref{rem.aee}).

  Recall that, for every inner product vector space $V$, the map $\sharp : V \to V^{*}$
  is defined by $v^{\sharp}(w) := (w, v)$. Finally, to prove (v), let us consider the
  matrix unit $e_{ij} := e_{i} \otimes e_{j}^{\sharp} \in M_{k}(\CC)$. By splitting our
  symbol as $a = \sum a_{ij} e_{ij}$, we can assume that $k = 1$. The first equality of
  (v) is the definition of $\maS_{a}$. The second equality of (v) is the content of
  (iv) just discussed. Then, Proposition \ref{prop.symbols.m=-1}
  shows that the assumptions of Lemma \ref{lemma.so.conv} are satisfied by the
  family $a_{\epsilon, \epsilon}(x', D')$ as $\epsilon \to 0$ with either
  positive or negative values, which gives the last equality in (v) and
  allows us to conclude (iv) and completes the proof of the theorem.
\end{proof}

Recall that we have $a_{\epsilon} = a_{\epsilon, \epsilon}$ and
that we write $a_{0, 0\pm} =: a_{0\pm}$, for simplicity. Most of the relations of
Theorem \ref{thm.main.jump0} have been written
in a compact form. The expanded form of these relations amounts to the following five
relations:
\begin{equation*}
  \begin{gathered}
    \big[ a(x, D)(h \otimes \delta_{\Gamma})]_{+} \ede
    \lim_{\epsilon \searrow 0} \big[ a(x, D)(h \otimes \delta_{\Gamma})]_{\epsilon}
    \seq \lim_{\epsilon \searrow 0} a_{\epsilon}(x', D')h \seq
    a_{0+} (x', D') h \\
    \big[ a(x, D)(h \otimes \delta_{\Gamma})]_{-} \ede
    \lim_{\epsilon \nearrow 0} \big[ a(x, D)(h \otimes \delta_{\Gamma})]_{\epsilon}
    \seq \lim_{\epsilon \nearrow 0} a_{\epsilon}(x', D')h \seq
    a_{0-} (x', D') h \\
    \sigma_{0}(a_{x_{n}, 0}; x', \xi') \seq \frac1{2\pi}
    \int_{\RR} \frac{\sigma_{-1}(a; x', x_{n}, \xi', \xi_{n})
    + \sigma_{-1}(a; x', x_{n}, \xi', - \xi_{n})}2 \, d \xi_{n}\\
    \sigma_{0}(a_{x_{n}, 0+}; x', \xi') \seq \frac{\imath}2
    \sigma_{-1}(a; \lambda e_{n}) + \sigma_{0}(a_{x_{n}, 0}; x', \xi') \qquad \mbox{and} \\
    \sigma_{0}(a_{x_{n}, 0-}; x', \xi') \seq - \frac{\imath}2
    \sigma_{-1}(a; \lambda e_{n}) + \sigma_{0}(a_{x_{n}, 0}; x', \xi')
    \,.
  \end{gathered}
\end{equation*}

\begin{example}
  \label{ex.DL.Laplacian}
  Let us look at the example of the double layer potential associated to
  $\Delta = \sum_{j=1}^{n} \pa_{j}^{2}$ on the half-space
  $\RR_{+}^{n}$ for $n \ge 3$. The outer unit normal vector field is
  $-e_{n}$. Hence the pseudodifferential operator $a(x, D)$ giving the double layer
  potential has distribution kernel $-\pa_{y_{n}}N(x-y)$, where
  $N(x) = \gamma_{n} |x|^{2-n}$, for some constant $\gamma_{n}$,
  whose value is irrelevant for the current considerations. Let $C_{N}$ be the
  operator with kernel $N(x-y)$. Then $a(x, D) = C_{N}(- \pa_{n})^{*} = C_{N}\pa_{n}$.
  The symbol of the operator $C_{N}$ with
  kernel $N(x-y)$ is $-| \xi|^{-2}$ (that is, the symbol of $\Delta^{-1}$, because
  $\sigma_{2}(\Delta) = - | \xi|^{2}$).
  The symbol of $\pa_{n}$ is $\sigma_{1}(\pa_{n}) = \imath \xi_{n}$.
  Therefore, the symbol of $a(x, D)$ is
  \begin{equation*}
    a(x, \xi) \seq  - \imath \xi_{n}| \xi|^{-2}\,,
  \end{equation*}
  which is essentially homogeneous of order $-1$. This gives $a_{x_{n}, 0} = 0$.
  Consequently
  \begin{equation*}
    [a(x, D)h]_{\pm} \seq \pm \frac{\imath}2 \JC(a) h \seq \pm \frac12 \, h\,.
  \end{equation*}
\end{example}

\section{\cn Pseudodifferential operators on manifolds with straight cylindrical ends}
\label{sec.psdos}

\emph{From now on, $M$ will be a manifold with straight cylindrical ends,
$M = \canDec$. Also, in the following, $E \to M $ will be a fixed compatible
Hermitian vector bundle.}
We now define two needed classes of pseudodifferential operators on $M$, the
``$\inv$-calculus'' and the ``$\ess$-calculus.'' Our presentation is complete, but
short (unlike the previous sections). The reason is that there are many works on
related classes of pseudodifferential operators, beginning with the
$b$- and $c$- calculi of Melrose \cite{MelroseActa, MazzeoMelroseAsian}. Related calculi
have also been considered by Schulze and Schrohe \cite{SchroheFrechet,
SchulzeBook91, SchulzeWongBCalc}. See \cite{KNW-22} for more references and for a discussion
of the relation between our calculi and the ones of Mazzeo, Melrose, Schrohe, and Schulze.
See also \cite{GrieserBCalc, LauterSeiler, LeschBCalc, MelroseNistor, SchulzeWongBCalc}.

The results of this section are close to those in \cite{KMNW-2025, KNW-22},
but the presentation is different and, occasionally, it provides more general
results. For instance, in \cite{KMNW-2025}, we work only with classical pseudodifferential
operators, whereas here we allow the Kohn-Nirenberg classes of symbols recalled in Definition
\ref{def.pseudos}.

\subsection{Translation invariance in a neighborhood at infinity: {\rm{inv}}-calculus}
Let $M = M_{0}\cup \left[M'\times (-\infty ,0)\right]$ be a manifold with straight cylindrical
ends, such that $M_{0}$ is a compact manifold and $M':=\partial M_{0}\neq \emptyset$.
Let $\Psi^{m}(M; E, F)$ denote the set of all order $m \in \RR$ pseudodifferential operators
$P : \CIc(M; E) \to \CI(M; F)$. For any continuous, linear operator
$P : \CIc(M; E) \to \CIc(M; F)'$, we let $k_P$ denote its distribution kernel.

For all $-R, s \geq 0$, we let
\begin{equation}\label{eq.def.Phi}
  \Phi _s :M'  \times (-\infty , R]\to M' \times (-\infty ,R-s] \,,
  \quad \Phi _s(x,t)=(x,t-s) \,,
\end{equation}
which is a bijective isometry (given by the translation with $-s$ in the $t$-direction).
If $s < 0$, then $\Phi _s := \Phi _{-s}$.
We fix $R \ll 0$ (i.e., much less than 0) so the maps $\Phi_s$ extend to sections of our
given compatible vector bundles by
\begin{equation*}
  \begin{cases}
    \Phi_{s}(u) \ede \circ \Phi_{-s} & \mbox{ where } \Phi_{-s} \mbox{ is defined and}\\
    \Phi_{s}(u) \seq 0 & \mbox{ where } \Phi_{-s} \mbox{ is \underline{not} defined.}
  \end{cases}
\end{equation*}
We let $\dist_g$ denote the distance function on $M$ defined by the given complete metric
$g$ (which, we recall, was assumed to be a product on the cylindrical end $M' \times (-\infty, 0]$).
Then, for $\varepsilon >0$, we define the {\it the $\varepsilon$--neighborhood of the diagonal}
to be the set
\begin{equation}\label{eq.def.eps_neigh}
  \{(x,y)\in M \times M : \dist_g(x,y) < \varepsilon \}\,.
\end{equation}
Recall the definition of pseudodifferential operators $\Psi^{m}(M; E, F)$,
Equation \eqref{def.psdos.M}.

\begin{definition} \label{translation-invariant}
  A continuous linear operator $P : \CIc(M,E)\to \CI(M,F)'$ is called
  {\it translation invariant in a neighborhood of infinity} if
  the following two conditions are satisfied:
  \begin{enumerate}[\rm (i)]
    \item There exists $\varepsilon_P > 0$ such that its distribution kernel $k_{P}$ is
    supported in the $\varepsilon_P$--neighborhood of the diagonal, see
    Equations \eqref{eq.def.dkaD} and \eqref{eq.def.eps_neigh}.

    \item \label{item.transl-inv} There exists $R_P < R_{E} - \varepsilon_P$ such
    that, for all $u \in \CIc \big(M' \times (-\infty , R_P),E\big)$ and all $s>0$, we have
      $P \Phi _s(u) = \Phi _s (Pu).$
  \end{enumerate}
  We let $\iPS m(M;E,F)$ denote the {\it space of all classical pseudodifferential
  operators $P : \CIc(M,E)\to \CI(M,F)$ of order $m \in \RR \cup \{\pm \infty\}$
  that are translation invariant in a neighborhood of infinity}.
\end{definition}

Thus $\iPS m(M;E,F) \subset \Psi^{m}(M; E, F)$.  We also let
$\iPS m(M;E) := \iPS m(M;E,E)$, $\iPS m(M) := \iPS m(M;\CC)
:= \iPS m(M;\CC, \CC)$,  $\iPS {-\infty}(M;E,F):=\bigcap_{m\in {\ZZ }}\iPS m(M;E,F)$,
and $\iPS {\infty }(M;E,F) :=\bigcup_{m\in {\ZZ }}\iPS m(M;E,F)$.
We let $\iPS m(M;E,F) := \Psi^{m}(M; E, F)$ if $M$ is closed, and similarly
for the other notation.

\begin{remark}
  Let us notice that, if $u$ has support in $M' \times (-\infty, R)$, $R < 0$ small,
  then $\Phi_s(u)$ has support in $M' \times (-\infty , R - s)$
  and $Pu$ has support in $M' \times (-\infty , R + \varepsilon_P)$.
  Consequently, if $u$ has support in $M' \times (-\infty, R)$ and $s < 0$, then both
  $P\Phi_s (u)$ and $\Phi_s(Pu)$ are defined, because $R_P < R_{E} - \varepsilon_P$.
  Consequently, the condition $P \Phi_s(u) = \Phi_s (Pu)$ of Definition
  \ref{translation-invariant} makes sense.
\end{remark}

The calculus $\iPS{m}(M; E, F)$ of Definition \ref{translation-invariant} will be called
the {\it $\inv$-calculus} \cite{KNW-22}. Our $\inv$-calculus is smaller than the
``$b$-calculus'' as follows from the following remark (see also \cite{KNW-22}).

\begin{remark}
\label{rem.inv.subset.b}
Let $\overline{M}$ denote the compactification of $M$ using the function $r := e^{t}$
(the Kondratiev transform), so that $M = \{r > 0\}$ and
$\pa \overline{M} = \{r = 0\}$. Then
\begin{equation*}
  \iPS{m}(M) \subset \Psi_b^{m}(\overline{M})\,,
\end{equation*}
where $\Psi_b^{m}(\overline{M})$ is the ``$b$-calculus'' of Melrose
\cite{MelroseActa, MelroseAPS} and Schulze \cite{SchulzeBook91,Schulze}.
(Schulze calls it the ``cone-calculus,'' see an important paper by Lauter and
Seiler \cite{LauterSeiler} comparing the two approaches.
See also \cite{CSS07, GrieserBCalc, KNW-22, MazzeoEdge, MazzeoMelroseAsian, SchulzeWongBCalc,
Seiler1999} for further results on the $b$-calculus.)
\end{remark}

Recall that $\Hom(E; F) \to M$ denotes the vector bundle of all vector
bundle homomorphisms $E \to F$. Then $\Hom(E; F) \to M$ is also a compatible vector bundle
over $M$ (because $E$ and $F$ are compatible).
The next remark provides some examples of operators in $\iPS{m}(M; E, F)$.

\begin{remark}\label{rem.def.Cinv}
We have
\begin{align*}
  \CI_{\inv}(M; \Hom(E; F)) & \seq \CI(M; \Hom(E; F)) \cap \iPS{m}(M; E, F)
  \quad \mbox{and}\\
  \operatorname{Diff}_{\inv}^m(M; E, F) & \seq \{P
  \mbox{ a differential operator on } M \} \cap \iPS{m}(M; E, F)\,.
  \end{align*}
  see Definitions \ref{def.CIinv} and \ref{def.not.diff.ops}.
\end{remark}

Recall that $k_{P}$ denotes the distribution kernel of a continuous,
linear operator $P : \CIc(M; E) \to \CIc(M; F)'$.

\begin{definition}\label{def.c.supp}
  Let $\Psi_{\comp}^m(M;E, F)$ denote the subspace of pseudodifferential operators
  $P \in \Psi^{m}(M; E, F)$ with {\it compactly supported distribution kernel} $k_{P}$.
\end{definition}

The following result follows from the definition of the $\inv$-calculus.

\begin{proposition} \label{propv1.comp.k}
  Let $M = \canDec$ be a manifold with cylindrical
  ends. Then $\Psi_{\comp}^m(M;E, F) \subset \iPS{m}(M; E, F)$.
\end{proposition}

\subsection{Principal symbol}
The principal symbol plays a crucial
role in the study of the $\inv$-calculus. Let $S^*M$ be the set of unit vectors of $T^*M$ and let
\begin{equation*} 
  \sigma_m : \Psi^{m}(M;E,F) \to \CI(S^*M; \Hom(E, F))
\end{equation*}
be the {\it principal symbol map of classical pseudodifferential operators
of order $\le m$} (see Definition \ref{def.princ.symb}).
We have the usual notion of ellipticity.

\begin{definition}\label{defv1.elliptic}
  An operator $P \in \iPS{m}(M; E, F)$ is called {\it elliptic} if its principal symbol
  $\sigma_m(P) \in \CI(S^*M; \Hom(E, F))$ is invertible
  on $S^*M\subset T^*M$.
\end{definition}

In the case of $\inv$-calculus, we have the additional property that the principal symbol of
a pseudodifferential operator that is translation invariant in a neighborhood of
infinity is also translation invariant at infinity \cite{KNW-22}.

\begin{proposition}\label{prop.def.inv.symb}
  The manifold $S^*M:= \{\xi \in T^*M : \|\xi \| = 1\}$ is also
  a manifold with straight cylindrical ends. The restriction of the
  principal symbol $\sigma_{m} : \Psi^{m}(M; E, F) \to \CI(S^{*}M; \Hom(E; F))$
  to $\iPS{m}(M; E, F)$ maps it to
  $\CI_{\inv}(S^*M; \Hom(E, F))$, and hence we obtain a map
    $\sigma_m : \iPS{m}(M; E, F) \to \CI_{\inv}(S^*M; \Hom(E, F))\,.$
\end{proposition}

\subsection{Algebra properties of the ${\rm{inv}}$-calculus}
We summarize now some of the algebra properties of the $\inv$-calculus $\iPS{m}$.
All of them are direct analogues of the corresponding properties for the usual
pseudodifferential calculus on compact manifolds or for other calculi on manifolds
with cylindrical ends. For simplicity, we shall mention our results only in the case
$E = F$. The general case is obtained easily by replacing $E$ with $E \oplus F$.
See also \cite{Mitrea-Nistor, KMNW-2025, KNW-22, KohrNistor-Stokes}.

\begin{theorem} \label{thm.prop.inv-calc}
  Let $M= M_{0} \cup \big [M'  \times (-\infty, 0) \big ]$ be a manifold with
  straight cylindrical ends and let $E, F, G \to M$ be compatible vector
  bundles on $M$. Then the following properties hold.
  \begin{enumerate}[\rm (i)]
    \item If $P \in \iPS{m}(M; E, F)$, then its adjoint $P^{*} \in \iPS{m}(M; F, E)$, and,
    for $m , m' \in \RR$,
    \begin{equation*}\iPS{m'}(M; F, G)\iPS{m}(M; E, F)
    \subset \iPS{m' + m}(M; E, G)\,.\end{equation*}
    \item The principal symbol $\sigma_m :\iPS{m}(M; E, F)\to \CI_{\inv}(S^*M; \Hom(E; F))$
    of Propositioin \ref{prop.def.inv.symb} is onto and its kernel is $\iPS{m-1}(M; E, F)$.
    It is $*$-multiplicative, in the sense that
    \begin{equation*}
      \sigma_{m' + m}(QP) \seq \sigma_{m'}(Q) \sigma_{m}(P)\
      \mbox{ and }\ \sigma_m(P^*) \seq \sigma_m(P)^*\,,
    \end{equation*}
    if $Q \in \iPS{m'}(M; F, G)$ and $P \in \iPS{m}(M; E, F)$.
    \item
    An operator $P \in \iPS{m}(M; E, F)$ is elliptic if, and only if, there exists
    $Q \in \iPS{-m}(M; F, E)$ such that $QP - 1 \in \iPS{-\infty}(M; E)$ and
    $PQ - 1 \in \iPS{-\infty}(M; F)$.
  \end{enumerate}
\end{theorem}

\subsection{Translation invariant operators on a cylinder and their kernels}
\label{ssec.ti}
We now look at the special case $M = \mfkS \times \RR$, where
$\mfkS$ is a given Riemannian manifold. This special example of a manifold
with cylindrical ends is useful in studying the general case. Of course, the cylinder
$M = \mfkS \times \RR$ is a special case of manifolds with cylindrical ends
(with $M_{0} = \mfkS \times [-1, 1]$ and $M' := \{-1, 1\}$,
see Definition \ref{def.cyl.end} and Remark \ref{rem.cylinder}.)

To simplify the notation, we shall set as in Melrose's work
\begin{equation}\label{eq.def.suspended}
  \iPSsus{m} {\mfkS} {E, F} \ede \iPS{m} (\mfkS \times \RR; E, F)^{\RR}\,.
\end{equation}
We begin with some simple observations on the supports of our operators that will allow
us to define the normal (or limit) operator associated to an operator that is
translation invariant in a neighborhood of infinity. Let $E, F \to \mfkS$ be two
Hermitian vector bundles that we lift to compatible vector bundles on
$\mfkS \times \RR$ denoted with the same symbols. We have the following useful
characterization of the space $\iPS{m} (\mfkS \times \RR; E, F)^{\RR}$ \cite{KNW-22}.

\begin{lemma}\label{lemma.rem.inv=pr}
  The space $\iPSsus{m} {\mfkS} {E, F} := \iPS{m} (\mfkS \times \RR; E, F)^{\RR}$
  consists of the set of order $m$, translation invariant, properly supported, pseudodifferential
  operators on $\mfkS \times \RR$ (invariant with respect to the translations by $t \in \RR$),
  that is
  \begin{equation*}\label{eq.inv=pr}
    \iPSsus{m} {\mfkS} {E, F} \seq \{P\in \Psi^{m}(\mfkS \times \RR; E, F)^{\RR}
    : \ P \mbox{ is properly supported}\,\}.
  \end{equation*}
\end{lemma}

\subsection{Limit operators}
An important concept associated to the operators in the $b$-calculus is that of
``normal operator'' (in Melrose's terminology) or ``conormal principal symbol''
(in Schulze's terminology, who uses the term ``cone-calculus'' instead of
the $b$-calculus). This concept is specific to non-compact manifolds
and is easy to be defined for the ${\rm{inv}}$-calculus. We shall call these operators
``limit at infinity operators,'' (or ``limit operators,''
for short). None of the results of this section have analogues for closed manifolds.

The translation invariant operators
appear in the following lemma that will provide the definition
of limit operators. The proof of this lemma can be found in \cite{KNW-22}
(see also \cite{KMNW-2025}, \cite{KohrNistor-Stokes}).

\begin{lemma}\label{lemma.limit.operator}
  Let $M = M_{0}\cup \left[M'\times (-\infty ,0)\right]$ be a manifold with straight
  cylindrical ends. Let $P\in \iPS {m}(M;E, F)$ and $u \in \CIc(M'\times \RR; E)$.
  Then, there exists $R_{P, u} > 0$ such that $\Phi _{-s}P\Phi_s(u)$ is defined and
  independent of $s > R_{P, u}$. Then, for $s > R_{P, u}$, we let
  \begin{equation*} 
    \In(u) \ede \Phi_{-s}P\Phi_s(u)\,.
  \end{equation*}
  \begin{enumerate}[\rm (i)]
    \item If $a \in \CI_{\inv}(M; \Hom(E; F)) \subset \iPS{m}(M; E, F)$,
    then $\In(a)$ is the same as the one in Definition {\rm \ref{def.CIinv}}.

    \item In general, $\In (P) \in \iPSsus{m}{M'}{E, F}
    := \iPS{m}(M' \times \RR; E, F)^{\RR}$.

    \item If $M = M' \times \RR$ and $P$ is translation invariant
    {\lpar}that is, $P \in \iPS{m}(M' \times \RR; E, F)^{\RR}${\rpar},
    then $\In(P) = P$.
  \end{enumerate}
\end{lemma}

\begin{definition}\label{def.indicial-oper}
  Let $P\in \iPS {m}(M;E,F)$ and
  \begin{equation*}
    \In(P) \in \iPSsus{m}{M'}{E, F} \ede
    \Psi^m(M \times \RR; E, F)^{\RR}
  \end{equation*}
  be as in Lemma \ref{lemma.limit.operator}. Then $\In(P)$ is
    called the {\it limit operator} associated to $P$.
    \index{limit operator (ex. $\In(P)$ or $\In(P)$)}
  The resulting map
  \begin{equation*}
  \In : \iPS {m}(M;E,F)\ni P\longmapsto \In(P)\in \iPSsus{m}{M'  }{E, F}
  \end{equation*}
  will be called the {\it limit operator map.}
\end{definition}

The following proposition provides a good
intuition about the limit operator $\In P$ \cite{KNW-22}.

\begin{proposition}\label{propv1.desc.kInP}
  Let $P \in \iPS{m}(M; E, F)$. Then, for all $(x,t) \neq (y, s)
  \in M' \times (-\infty, 0)$ and all $\lambda \in \RR$ {\rm large enough}
  \begin{equation*}
    k_{\In(P)}(x, t, y, s)  \seq k_{P}(x, t - \lambda, y, s - \lambda)
    \in (F \boxtimes E)_{(x, t, y, s)} \ede F_{x} \otimes E_{y}
  \end{equation*}
\end{proposition}

\subsection{The structure of the $\inv$-calculus}
We shall make repeated us of the following cut-off
function $\eta$ that will be fixed throughout this paper.

\begin{notation}\label{not.rem.eta}
  We fix a smooth ``cut-off'' function $\eta : M'  \times \RR \to [0, 1]$
  \begin{equation*}
    \eta(x, t) \seq
    \begin{cases}
    \ 1 \ \mbox{ for  }\ t \le -2\\
    \ 0 \ \mbox{ for }\ t \ge -1\,.
    \end{cases}
  \end{equation*}
  We extend $\eta$ to a function on $M= \canDec$ by setting it to be equal to $0$
  on $M_{0}$. Then, for $T \in \iPSsus{m}{M' }{E, F} := \iPS {m}
  (M'  \times \RR ; E, F)^{\RR}$, we let
  \begin{equation*} 
    s_{0}(T) \ede \eta T\eta : \CIc(M; E) \to \CIc(M; F)\,.
  \end{equation*}
\end{notation}

The following lemma from \cite{KNW-22}, proves, in particular, that the limit operator
$\In$ of Definition \ref{def.indicial-oper} is surjective.
Recall that $\Psi_{\comp}^m(M;E, F) \subset \Psi^{m}(M; E, F)$
denote the subspace of pseudodifferential operators with {\it compactly supported
distribution kernel}, see Definition \ref{def.c.supp}. Some of these properties follow
from the corresponding properties of the $b$-calculus.

\begin{proposition}\label{propv1.lemma.onto}
  Let $T \in \iPSsus{m}{M' }{E, F}$ and let $s_{0}(T) := \eta T\eta$ be as in
  \ref{not.rem.eta}. Then the following properties hold.
  \begin{enumerate}[\rm (i)]
    \item $s_{0}(T) \in \iPS{m}(M; E, F)$ satisfies $\In (s_{0}(T)) = T.$

    \item The map $\In : \iPS{m}(M; E, F)\to \iPSsus{m}{M' }{E, F}$ is
    multiplicative and surjective and $\ker( \In ) := \In^{-1}(0) = \Psi_{\comp}^m(M;E, F)$.

    \item The map $P \mapsto \big( s_{0}(\In(P)), P - s_{0}(\In(P)) \big)$ yields the
    following direct sum decomposition
    \begin{equation*}
      \iPS{m}(M; E, F) \seq s_{0}(\iPSsus{m} {M' }{E, F}) + \Psi_{\comp}^m(M; E, F)
      \,.
    \end{equation*}
    \item For any $T \in \iPS{m}(M; E, F)$, we have
    \begin{equation*}
      \sigma_{m}(\In(T)) \seq \In(\sigma_m(T)) \in \CI(S^{*}(M' \times \RR) ;
      \Hom(E; F))^{\RR} \,.
    \end{equation*}
  \end{enumerate}
\end{proposition}

We also have the following characterization of Sobolev spaces \cite{KNW-22, KMNW-2025}.

\begin{corollary} \label{cor.char.Sobolev}
  Let $u \in H^{t}(M; E)$, for some $t \in \RR$.
  We have $u \in H^s(M; E)$ if, and only if, $Pu \in L^2(M; E)$ for
  all $P \in \iPS{s}(M; E)$.
\end{corollary}

\subsection{The $\ess$-calculus of essentially translation invariant operators}
We continue to assume that
$M = M_{0}\cup \left[M'\times (-\infty ,0)\right]$ is a manifold with
{\it straight cylindrical ends}. The algebra $\iPS {\infty }(M)$ (the $\inv$-calculus,
Definition
\ref{translation-invariant}) is {\it not stable under inversion} in the sense that
if $T\in \iPS {\infty }(M)$ is an invertible operator on the space $L^2(M)$, then its
inverse does not necessarily belong to $\iPS {\infty }(M)$. In this section
we introduce the space $\ePS{m}(M; E, F)$ of {\it essentially translation
invariant} (classical) pseudodifferential operators $P : \CIc(M; E) \to \CI(M; F)$,
such that $\iPS {\infty} (M) \subset \ePS {\infty} (M)$
and the latter {\it is stable under inversion}.
The operators in $\ePS {\infty} (M)$ are called {\it essentially translation invariant}
\cite{KNW-22} (in \cite{Mitrea-Nistor}, they were called {\it almost invariant}).
The presentation in \cite{KMNW-2025} is very detailed and complete and we refer
the reader to that book for more details.

We shall continue to use the cut-off function $\eta : M\to [0, 1]$ of Notation
\ref{not.rem.eta}. (Recall that $\eta \in \CI(M)$ has support in $M'  \times (-\infty, -1]$
and $\eta = 1$ on $M' \times (-\infty, -2]$.) The function $\eta$ allowed us to associate
to any $T \in \iPSsus{m}{M'}{E} := \iPS{m}(M' \times \RR;E)^\RR$ the operator
\begin{equation*} \label{eq.redef.s0}
  s_{0}(T) \ede \eta T \eta \in \iPS{m}(M; E)\,,
\end{equation*}
which is well-defined (see Notation \ref{not.rem.eta} and Proposition
\ref{propv1.lemma.onto}(i)).

\begin{definition} \label{def.ess.one}
  Let $E \to M$ be a compatible Hermitian vector bundle over the manifold with straight
  cylindrical ends $M$.
  \begin{enumerate}[\rm (i)]
   \item
    Let $\ePSsus{-\infty}{M}{E}$ consist of the linear operators
    $T \in \Psi^{-\infty}(M\times \RR; E)^{\RR}$ with convolution kernel
    $C_T \in \maS(M\times M \times \RR; E \boxtimes E)$.

    \item We define $\Psi_\maS^{-\infty}(M; E)$ to be the set of pseudodifferential
    operators $P \in \Psi^{-\infty}(M; E)$ with distribution kernel
    $k_{P} \in \maS(M\times M; E \boxtimes E)$.

    \item
    Let $\ePSsus{-\infty}{M'}{E} \subset \Psi ^{-\infty }(M' \times \RR;E)^\RR$. We
    then define the space $\ePS{-\infty}(M; E)$
    of {\it essentially translation invariant, order $-\infty$
    pseudodifferential operators} on $M$ by
    \begin{equation*}
      \ePS{-\infty}(M; E) \ede s_{0}(\ePSsus{-\infty} {M' }{E})
      + \Psi_\maS^{-\infty}(M; E)\,.
    \end{equation*}

    \item
    Finally, for $m \in \RR$, we define the space $\ePS{m}(M; E)$
    of {\it essentially translation invariant, order $m$
    pseudodifferential operators} on $M$ by
    \begin{equation*}
      \ePS{m}(M; E) \ede \iPS{m}(M; E) + \ePS{-\infty}(M; E)\,.
    \end{equation*}
  \end{enumerate}
\end{definition}


It then follows easily from the last definition that $\ePSsus{-\infty}{M}{E} =
\ePS{-\infty}(M\times \RR; E)^{\RR}$.

\begin{remark}\label{rem.incl.two}
  The last two definitions provide the following simple inclusions. First,
  \begin{equation*}
    \Psi_{\comp}^{-\infty}(M; E) \subset \Psi_\maS^{-\infty}(M; E)\,,
  \end{equation*}
  because the space of Schwartz sections contains the space of compactly supported
  sections. Let again $s_{0}$ be as in \ref{not.rem.eta}.
  This inclusion implies that
  \begin{multline*}
    \iPS{-\infty}(M; E) \seq s_{0}(\iPSsus{-\infty} {M' }{E})
    + \Psi_{\comp}^{-\infty}(M; E)\\
    \subset s_{0}(\ePSsus{-\infty} {M'}{E})
    + \Psi_\maS^{-\infty}(M; E) \,=: \, \ePS{-\infty}(M; E)\,.
  \end{multline*}
  Finally, for all $m \in \ZZ \cup \{\pm \infty \}$, we also obtain the inclusion
  \begin{equation*}
    \iPS{m}(M; E) \subset \iPS{m}(M; E) + \ePS{-\infty}(M; E)
    \, =: \, \ePS{m}(M; E)\,.
  \end{equation*}
\end{remark}

\subsection{Properties of essentially translation invariant operators}
We continue to assume that $M$ is a manifold with straight
cylindrical ends and that $E \to M$ is a compatible vector bundle.
The next theorem gives the main properties of the calculus
\begin{equation*}
  \ePS{m}(M; E) \ede \iPS{m}(M; E) + \ePS{-\infty}(M; E)
\end{equation*}
introduced in Definition \ref{def.ess.one} (see \cite{KMNW-2025}, \cite{KNW-22} for the proof).

\begin{theorem}\label{thm.prop.ePS}
  Let $m, m', s, s' \in \RR$, let $M = \canDec$ be a manifold with
  straight cylindrical ends, and let $E \to M$ a compatible vector bundle, as before.
  Then we have the following properties.
  \begin{enumerate}[\rm (i)]
    \item $\ePS{m}(M; E)\ePS{m'}(M; E) \subset \ePS{m+m'}(M; E)$
    and $\ePS{m}(M; E)^* \subset \ePS{m}(M; E)$.

    \item If $P \in \ePS{m}(M; E)$, then $P : H^{s}(M; E) \to H^{s-m}(M; E)$
    is well-defined and continuous.

    \item The morphism
    $\sigma_m: \ePS {m}(M;E)/\ePS {m-1}(M;E)\to \CI_{\inv}(S^*M; \End(E))$
    defined by the principal symbol is an isomorphism.

    \item Let $P \in \ePS{m}(M;E)$. Then $P$ is elliptic if, and only if, there exists
    $Q \in \iPS{-m}(M;E)$ such that $PQ - 1, QP - 1 \in \ePS{-\infty}(M;E)$.
  \end{enumerate}
\end{theorem}

\subsection{Translation invariance, the Fourier transform, and indicial operators}
\label{ssec.FT}
The translation invariance of the limit operators allows us to  consider the
one-dimensional Fourier transform to study them. To do that, in this subsection, we
shall restrict ourselves to a cylinder. Let $\mfkS$ be a generic closed manifold
(i.e., a smooth, compact, boundaryless manifold). Thus, in this subsection, we will work on
$\mfkS \times \RR$
and $E, F$ will be vector bundles on $\mfkS \times \RR$
that are lifts of bundles from $\mfkS$. We then consider the {\it one dimensional}
Fourier transform
\begin{equation} \label{Fourier}
  \mathcal F :L^2\big(\mfkS\times \RR; E\big)
  \to L^2\big(\mfkS\times \RR; E\big),\quad
  \mathcal F(f)(y,\tau) \ede  \int _{\RR} e^{- \imath \tau x} f(y,x)dx\,,
\end{equation}
with $\imath^2=-1$ (see also Equation \eqref{eq.def.inv.F}). We no longer assume
that $F = E.$ We next introduce the space $\ePSsus{m}{\mfkS}{E, F}$ in general.

\begin{definition}\label{def.ePSm}
  We let $ \ePSsus{m}{\mfkS}{E, F} := \ePS {m}(\mfkS\times \RR; E, F)^{\RR}$.
\end{definition}

The definition of $s_{0}$ in Notation \ref{not.rem.eta} then
extends to $T \in \ePSsus{m}{M'}{E}$ and
we obtain $s_{0}(\ePSsus{m}{M'}{E}) \subset \Psi^{m}(M; E)$.
We obtain the following remark.

\begin{remark}\label{rem.Fourier}
  Let $m\in \RR$ and $Q \in \ePSsus{m}{\mfkS}{E, F} := \ePS {m}(\mfkS\times \RR; E, F)^{\RR}$.
  Let $u_{0} \in \CI(\mfkS; E)$ and $\tau \in \RR$ and set
  $u(x', t) := e^{\imath t \tau}u_{0}(x')$.
  Recall that $\Phi_{s}$ denotes the translation by $s$, see Equation \eqref{eq.def.Phi}. Then
  \begin{equation*}
      \Phi_{s}(u)(x', t) \ede u(x, t+s) \seq e^{\imath (t+s) \tau}u_{0}(x')
      \seq e^{\imath s \tau}u(x', t)\,.
  \end{equation*}
  Thus, $\Phi_{s} (u) = e^{\imath \tau s}u$. If $Q \in \iPSsus{m}{\mfkS}{E, F}$,
  then $Qu \in \CI(\mfkS\times \RR; F)$ is defined, because $Q$ is properly supported.
  In general, $Qu \in \CI(\mfkS\times \RR; F)$ is defined by continuity. This gives that
  the function $v(x', t) := e^{-\imath t \tau}[Qu](x', t)$ satisfies
  \begin{multline*}
    [\Phi_{s}(v)](x', t) \ede v(x', t+s) \ede
    e^{-\imath (t+s) \tau}[Qu](x', t+s) \seq e^{-\imath (t+s) \tau}\Phi_{s}[Qu](x', t)\\
    \seq e^{-\imath (t+s) \tau}[Q\Phi_{s}(u)](x', t)
    \seq e^{-\imath (t+s) \tau} e^{\imath s \tau}[Q(u)](x', t) \seq v(x', t)\,,
  \end{multline*}
  because $\Phi_{s}$ and $Q$ commute. That is $\Phi_{s}(v) = v$, and hence $v$ depends
  only on $x'$ and can be identified with a function on $M'$ denoted
  $\widehat Q(\tau)u_{0}$.
\end{remark}

This allows us to extend the Fourier transform to operators as follows.

\begin{definition}\label{def.explicit.FT}
  Let $m \in \RR$, let $Q \in \ePSsus{m}{\mfkS}{E, F} := \ePS {m}(\mfkS\times \RR; E, F)^{\RR}$,
  and let $u_{0} \in \CI(\mfkS; E)$. We write
  $u := e^{\imath t \tau} u_{0}$ for the section $u \in \CI(\mfkS \times \RR; E )$
  satisfying $u(x', t) :=  e^{\imath t \tau} u_{0}(x')$. In view of Remark \ref{rem.Fourier},
  we can define $\widehat Q(\tau)u_{0} \in \CI(\mfkS; F)$ by the formula
  \begin{equation*} 
    [\widehat{Q}(\tau)u_{0}](x') \seq e^{-\imath t \tau} [Q(u)](x', t)
    \ede e^{-\imath t \tau}[Q(e^{\imath t \tau}u_{0})](x', t)\,.
  \end{equation*}
\end{definition}

It will be convenient to write the formula in Definition \ref{def.explicit.FT} as
\begin{equation}\label{eq.explicit.FT}
  Q(e^{\imath t \tau}u_{0}) \seq e^{\imath t \tau} [\widehat Q(\tau)u_{0}]\,,
\end{equation}
where $u_{0} \in \CI(\mfkS; E)$, $Q \in \ePS{m}(\mfkS; E, F)$,
and $\widehat Q(\tau) \in \Psi^{m}(\mfkS; E, F)$. This allows us to obtain the
following proposition whose proof can be found in \cite{KMNW-2025} and \cite{KNW-22}.
Recall that, when defined, $k_{P}$ denotes the distribution
kernel of an operator $P$ and that $C_{P}$ denotes its convolution kernel.
In particular, if $Q$ is a multiplication operator only with respect to $x'$, that is,
if $Qu(x', t) = a(x') u(x', t)$, then Definition \ref{def.explicit.FT} gives
\begin{equation} \label{eq.just.new.notation}
    [\widehat{Q}(\tau)u_{0}](x') \seq e^{-\imath t \tau} [Q(u)](x', t)
    \ede e^{-\imath t \tau}[e^{\imath t \tau}a u_{0}](x', t) \seq a(x')u_{0}(x')\,,
\end{equation}
and, hence, $\widehat{Q}(\tau) = a$ (for this particular choice of $Q$ given as a
multiplication operator with $a$, which is a function of $x'$ alone).

\begin{proposition} \label{propv1.mult.indicial}
  Let $Q \in \ePSsus{m}{\mfkS}{E, F} := \ePS{m}(\mfkS \times \RR; E, F)^{\RR}$,
  $P \in \ePS{m'}(\mfkS \times \RR; F, G)$. We then have
  the following properties:
  \begin{enumerate}[\rm (i)]
    \item $\widehat {PQ}(\tau) \seq \widehat {P}(\tau)\widehat {Q}(\tau)$,
    for all $\tau \in \RR$.

    \item The distribution kernel $k_{\widehat{P}(\tau)}
    \in \CI(\mfkS^{2}; G \boxtimes F)$ of $\widehat{P}(\tau)$ and the convolution kernel
    $C_{P} \in \CI(\mfkS^{2} \times \RR; G \boxtimes F)$ of $P$ are related by the formula
    $k_{\widehat{P}(\tau)}(x', y') = \maF C_{P}(x', y', \tau),$ where
    $x', y' \in \mfkS$.

    \item If $u \in \CIc(\mfkS \times \RR; F)$, then $\widehat {Pu}(\tau) \seq
    \widehat {P}(\tau) \hat u(\tau)$.
  \end{enumerate}
\end{proposition}

The operators $\widehat Q(\tau)$ are useful for checking the invertibility of
$Q \in \ePSsus{m}{M}{E, F}$, as follows from the following well-known
result \cite{MazzeoMelroseAsian, MelroseActa, Schulze}.

\begin{theorem}\label{thm.rem.invert}
  Let $\mfkS$ be a smooth, closed manifold and $Q \in \ePSsus{m}{\mfkS}{E, F}$.
  The operator $Q:H^s(\mfkS \times \RR; E)\to H^{s-m}(\mfkS \times \RR; F)$
  is invertible if, and only if, all the operators $\widehat{Q}(\tau):
  H^s(\mfkS; E)\to H^{s-m}(\mfkS; F)$, $\tau \in {\RR}$, are invertible.
\end{theorem}

We let $t$ denote the second coordinate function on the half-infinite cylinder
$M'  \times (-\infty, 0)$ and we extend it to a smooth function (denoted also by $t$) on
$M = \canDec$ such that $t \ge 0$ on $M_{0}$ (recall that $M' := \pa M_{0}$). This allows
us to write $M' \times (-\infty, 0) = \{t\in \RR : t < 0\}$. Let $\Phi _s(x,t)=(x,t-s)$
be as in Equation \eqref{eq.def.Phi}. Then we obtain the following norm convergence
result \cite{KMNW-2025}, \cite{KNW-22}.

\begin{lemma}\label{lemma.def.indicial}
  If $P\in \ePS {m}(M;E, F)$ and $u \in \CIc(M'  \times \RR; E)$, then, for $s$
  large enough, $P\Phi_s(u)$ is defined and we have the following norm
  convergence in $L^2(M'  \times \RR; F)$,
  \begin{equation*} 
    \In(P)(u) \ede \lim_{s \to \infty}
    \Phi _{-s}\big [ P\Phi_s(u)\vert_{\{t < 0\}} \big]\,.
  \end{equation*}
  In addition, the resulting operator satisfies
  $\In (P) \in \ePSsus{m}{M'}{E}$ and $\In \circ s_{0}(P_{0}) = P_{0}$
  for $P_{0} \in \ePSsus{m}{M'}{E}$.
\end{lemma}

Using the Fourier transform for $Q = \In(P)$ in Definition  \ref{def.explicit.FT},
we obtain Melrose's {\it indicial operators.}

\begin{definition}\label{def.indicial.fam}
  \index{indicial operator (ex. $\widehat{P}(\tau)$)}
  Let $P \in \ePS{m}(M; E)$ and $\widehat{P}(\tau) := \widehat{\In(P)}(\tau)$ be as
  in Remark \ref{rem.Fourier}. Then the family
  \begin{equation*}
    \{\widehat{P}(\tau)\}_{\tau \in \RR} \ede \{\widehat{\In(P)}(\tau)\}_{\tau \in \RR}
    \ede \{\mathcal F \In(P) \mathcal F^{-1}(\tau)\}_{\tau \in \RR}\end{equation*}
  is called the {\it indicial family} of $P$. The operator
  $\widehat{P}(\tau)$ is called the {\it indicial operator} of $P$.
\end{definition}

As for the limit operators, the indicial operators are multiplicative
\cite{KMNW-2025, KNW-22, KohrNistor-Stokes}.

\begin{theorem}\label{thm.complet.indicial}
  Let $M= \canDec$ be a manifold with cylindrical ends and
  $E, F, G \to M$ be compatible Hermitian vector bundles. Let $m , m' \in \RR$.
  \begin{enumerate}[\rm (i)]
    \item If $P \in \ePS m (M; E, F)$ and $Q \in \ePS {m'} (M; F, G)$, then
    $\In (PQ) \seq \In (P)\In(Q)$.

    \item For $P$ and $Q$ as in (i), their indicial families {\rm (Definition
    \ref{def.indicial.fam})} satisfy $\widehat{QP}(\tau) \seq \widehat{Q}(\tau)
    \widehat{P}(\tau)$, $\tau \in \RR\,.$

    \item We have $\In \circ s_{0} = id$ on $\ePSsus{m}{M' }{E}$
    and the map $\In : \ePS{m}(M; E, F) \to \ePSsus{m}{M' }{E}$ is surjective
    with kernel $\Psi_{\maS}^{-\infty}(M; E, F) + \Psi_{\comp}^{m}(M; E, F)$.

    \item We have the following exact sequence of algebras
    \begin{equation*}
      0 \longrightarrow \Psi_\maS^{-\infty}(M; E) \longrightarrow \ePS{-\infty}(M; E)
      \stackrel{\In}{-\!\!\!\longrightarrow} \ePSsus{-\infty}{M'}{E}
      \longrightarrow 0 \,.
    \end{equation*}

    \item For any $T \in \iPS{m}(M;E,F)$, we have
      $\sigma_{m}(\In(T)) \seq \In(\sigma_m(T))\,.$
\end{enumerate}
\end{theorem}

The indicial operator of a matrix of operators is defined componentwise.

\begin{remark}\label{rem.indicial2}
  Let us also notice that, given a translation invariant operator
  \begin{equation*}
    T \in \ePSsus{m}{\mfkS}{E, F} \ede \ePS{m}(\mfkS \times \RR; E, F)^{\RR}
    \subset \ePS{m}(\mfkS \times \RR; E, F)\,,
  \end{equation*}
  we can define its indicial operators $\widehat T(\tau)$ in two ways,
  first as an operator in $\ePSsus{m}{\mfkS}{E, F}$ and then as an operator
  in $\ePS{m}(\mfkS \times \RR; E, F)$. The definition in the second case
  gives us ``twice'' each of the operators $\widehat T(\tau)$ defined in
  the first case, so the corresponding families will have the same properties.
  Unless stated otherwise, our definition of $\widehat T(\tau)$ will be given
  by regarding the operator $T$ as an operator in $\ePSsus{m}{\mfkS}{E, F}$
  (i.e., as an invariant operator, which gives us ``fewer'' indicial operators
  $\widehat T(\tau)$).
\end{remark}

\section{\cn Main properties of the $\inv$ and $\ess$ calculi}
\label{sec.NLimits}

In this section, we establish some of the main properties of the operators in
the ``$\inv$'' and ``$\ess$'' calculi of the previous section. We concentrate on the
properties needed in the study of layer potentials. We first adapt the results from
Section \ref{sec.nllRn} on normal lateral limits on half-spaces to smooth, open domains in
manifolds with cylindrical ends and our calculi of the previous section
($\inv$ and $\ess$). Then we recall \ADN-operators and their
Fredholm, (elliptic) regularity, and spectral invariance properties.

{\cn We thus continue to assume that $M = \canDec$ is
a smooth Riemannian manifold with straight cylindrical ends (Definition \ref{def.cyl.end}) and
that $E, F \to M$ are two compatible vector bundle (see Definition \ref{def.comp.vb}).}
Most importantly, we let $\Omega \subset M$ be an open subset with smooth boundary
$\Gamma := \pa \Omega \neq \emptyset$ (so the assumption is that $\Gamma$
is also a smooth manifold) that is \emph{compatible with the structure of cylindrical
ends} of $M$, as in the introduction, Equation \eqref{eq.def.Omega}.
That is:
\begin{equation*} 
  \Omega \cap \big[ M' \times (-\infty, R_{\Omega}] \big] \seq \Omega'
  \times (-\infty, R_{\Omega}]\,,
\end{equation*}
for an open subset $\Omega' \subset M'$.
Note that $\Gamma$ is also a smooth manifold with straight
cylindrical ends if we write
\begin{equation} \label{eq.matching.ce0}
  \Gamma \cap \big[ M' \times (-\infty, R_{\Omega}] \big]
  \seq \Gamma' \times (-\infty, R_{\Omega}]\,,
\end{equation}
where $\Gamma' := \pa \Omega'$. Recall that in this paper we assume that $\Omega$ is
on one side of its boundary, see Assumption \ref{assumpt.1side}
(in particular, $\Gamma$ is also the boundary of $\Omega_{-} := M \smallsetminus
\overline{\Omega}$). It is no loss of generality to assume that
$\Omega$ is connected; by contrast, we cannot assume that $\Omega'$ is connected.
However, $\Omega'$ also must be on one side of its boundary. See Figure 2:

{\begin{figure}[h]
  \label{fig3}
  \centering
  \includegraphics[width=0.4\textwidth]{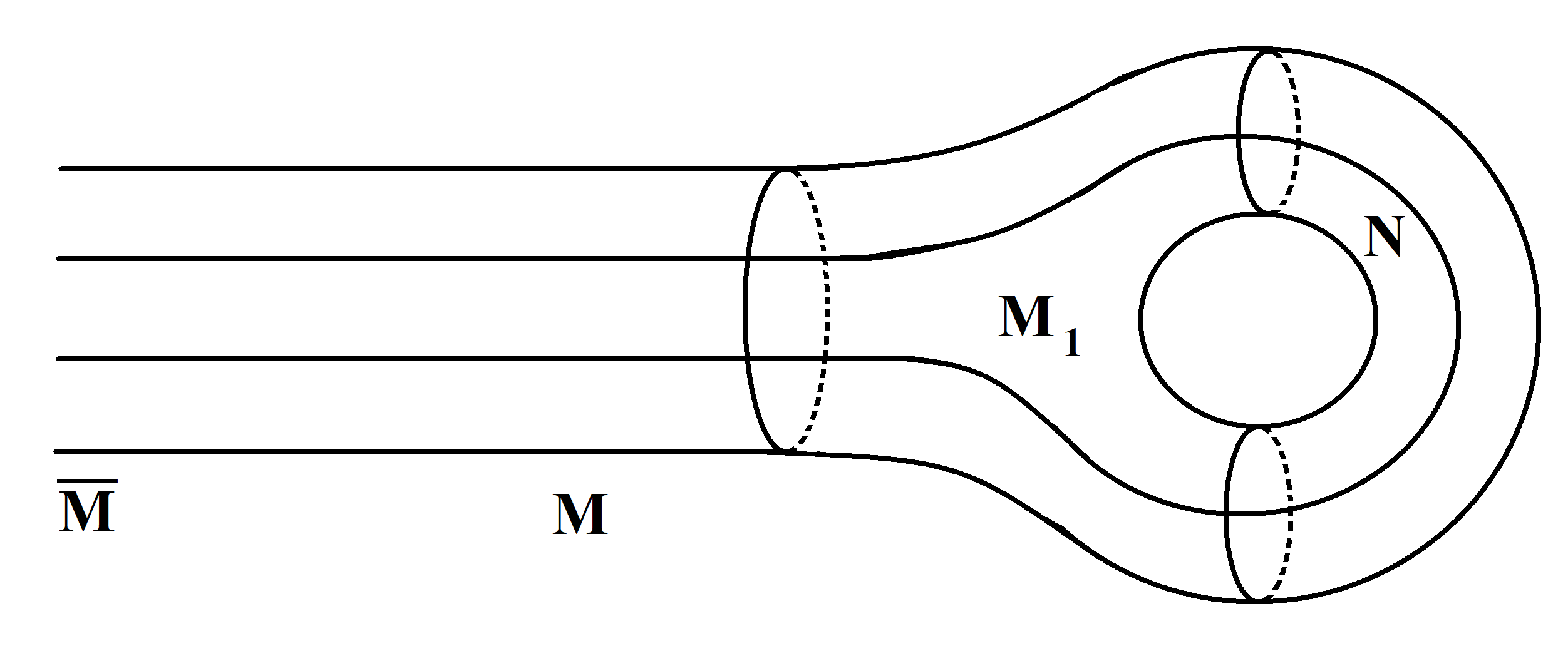}
  \caption{A manifold with boundary and straight cylindrical ends and a
  boundaryless manifold with straight cylindrical ends containing it
  as in Equation \eqref{eq.def.Omega}.}
\end{figure}}

\subsection{Normal tubular neighborhoods}
Recall that the (global) spaces $H^{s}(M)$ are defined using
the metric (see Definition \ref{def.rem.Bochner} and Subsection \ref{ssec.cyl.ends}
in general). Their variants with values in smooth vector bundles are defined similarly.
We shall denote by $\cdot$ the inner product on the compatible vector bundle $E$.
Distributions of the form $h \otimes \delta_{\Gamma}$ (see Equation
\eqref{eq.def.hdelta2}), that is,
\begin{equation*}
  \langle h \otimes \delta_{\Gamma}, \phi \rangle \ede \int_{\Gamma}
  h(x) \cdot \phi(x) \, dS_{\Gamma}(x)\,,
\end{equation*}
will play an important role in what follows. We have the following analog of Lemma
\ref{lemma.def.hdelta}.

\begin{lemma}\label{lemma.def.hdelta2}
  Let $E$ be a compatible, smooth Hermitian vector bundle on $M$ with inner product
  denoted $\cdot$, let $h \otimes \delta_{\Gamma}$ be as in Equation \eqref{eq.def.hdelta2},
  and let $s \in \RR$. Let $s' \in \RR$ be such that
  \begin{equation*}
    \begin{cases}
      \ s' \ede s -1/2 & \mbox{ if } s < 0\\
      \ s' < -1/2 & \mbox{ arbitrary, if }\ s \ge 0\,.
    \end{cases}
  \end{equation*}
  Then, for all $h \in H^{s}(\Gamma; E)$, we have that
  $h \otimes \delta_{\Gamma} \in H^{s'}(M; E)$.
\end{lemma}

\begin{proof}
  The proof is almost identically to that of Lemma \ref{lemma.def.hdelta}, except
  that we use
  that the maps $H^{r}(M; E) \to H^{r-1/2}(\Gamma;E)$, $r > 1/2$ are continuous
  \cite{GrosseSchneider, KMNW-2025}.
\end{proof}

Let $P$ be an order $m$  pseudodifferential operator acting on
the sections of $E$, that is $P \in \Psi^{m}(M; E)$. Recall that we
are assuming that $\Omega$ is on one side of its boundary.
We are interested in
studying
\begin{equation}\label{eq.def.layerPot-0}
  \maS_{P} h \ede P(h \otimes \delta_{\Gamma})\,,
\end{equation}
provided that the latter is defined. For the convenience of the notation, we also let
\begin{equation}\label{eq.def.Omega+}
  \Omega_{+} \ede \Omega \ \mbox{ and }\
  \Omega_{-} \ede M \smallsetminus \overline{\Omega}
  \quad \Rightarrow \quad \Gamma \seq \pa \Omega_{-} \seq \pa \Omega_{+}\,,
\end{equation}
where the last relation is equivalent to saying that $\Omega$ is on one side of its
boundary. We are especially interested in the following two restrictions and their traces
\begin{equation}\label{eq.def.layerPot}
  \maS_{P} h \vert_{\Omega_{\pm}}\quad \mbox{ and }\quad
  [\maS_{P} h]_{\pm} \ede [\maS_{P} h\vert_{\Omega_{\pm}}]\vert_{\pa \Omega_{\pm}}\,.
\end{equation}

\begin{notation}\label{not.bsnu}
We let $\bsnu$ be the \emph{outer} unit normal vector to $\Gamma := \pa \Omega$.
We extend this vector field to a global (smooth) vector field on $M$ (not
necessarily unit everywhere), still denoted $\bsnu$. Also, let $\sharp : TM \to T^{*}M$
be the isomorphism defined by the metric of $M$, as before. We shall write
$\bsnu ^{\sharp}:=\sharp \bsnu$.
\end{notation}

If $v \in TM$ is a tangent vector to $M$ in $x$, we let $\exp(tv)$
denote the image of $tv$ under the exponential map, which is defined for $|t|$
small (depending on $v$).

\begin{definition}\label{def.e.nbhd}
  If $\epsilon > 0$ is such that the \emph{normal exponential map}
  \begin{equation*}
    \exp^{\perp} : \Gamma \times (-\epsilon, \epsilon) \ni (x, t)\, \mapsto\,
    \exp(t \bsnu(x)) \in M
  \end{equation*}
  ($\Gamma =\partial \Omega $) is well defined and is a diffeomorphism onto its image,
  then we shall say that $\Gamma$ has \emph{an $\epsilon$-normal
  tubular neighborhood}.
\end{definition}

It is well-known that if $\Gamma$ is compact or if
$M$ has straight cylindrical ends and $\Omega$ is compatible with this
structure, then $\Gamma$ will have an
$\epsilon$-normal tubular neighborhood, for some $\epsilon > 0$ small enough,
see \cite[Corollary 5.5.3]{petersen:98}. The curves $t \mapsto \exp(t\bsnu(x))$,
$x \in \pa \Omega$, will be called the \emph{normal geodesics} to $\pa \Omega$.
If $u$ is a section of $E$ over $M$, $\Gamma$ has an $\epsilon$-normal
tubular neighborhood, and $t \in (-\epsilon, \epsilon)$, we let
\begin{equation}
  u_{t} \ede u\vert_{\exp^{\perp}(\Gamma \times \{t\})} \in
  \CI(\Gamma \times \{t\}; E) \simeq \CI(\Gamma; E)\,,
\end{equation}
where the last isomorphism is obtained via parallel transport
along the normal geodesics $(-\epsilon, \epsilon) \ni t \to \exp(t\bsnu(x)) \in M$,
$x \in \pa \Omega$. (These isomorphisms are the analogues of the ones induced
by $\tau_{\epsilon}(x) := x - x_{n} e_{n}$, which was introduced after Equation
\eqref{eq.def.Gamma}.)

It will be important for us to study the limits
$u_{\pm} := \lim_{t \to \pm 0} u_{t}$ in some function space on $\Gamma := \pa \Omega$,
for suitable $u$. When they exist, we call these limits, the \emph{normal lateral limits}
of $u$. In case $u$ is smooth enough on $\Omega_{+} := \Omega$ and on
$\Omega_{-} := M \smallsetminus \overline{\Omega}$, then $u_{+}$ is the trace of
$u\vert_{\Omega_{+}} := u\vert_{\Omega}$ at the boundary and, similarly,
$u_{-}$ is the trace of $u\vert_{\Omega_{-}} :=
u\vert_{M \smallsetminus \overline{\Omega}}$ at the boundary, see Lemma
\ref{lemma.limits=traces}.

\subsection{Normal lateral limits for manifolds with cylindrical ends}
We now turn to the study of normal lateral limits of pseudodifferential
operators on general Riemannian manifolds with straight cylindrical ends.
Note that, unlike in the other chapters, the case $M$ compact is not
trivial, but it can be handled exactly as the case of manifolds with cylindrical
ends. (This has been done in \cite{KNW-2025}, so we will not review it in detail
here.)

As usual, the case of pseudodifferential
operators of order $m < -1$ is easier.
We begin with the case of operators with compactly supported distribution
kernels in $M \times M$, which will then be used to deal with the general case.
Let $F \to M$ be a \emph{second} hermitian vector bundle (in addition to $E$).
We have the following simple calculation that will be used repeatedly, so we
formulate it as a lemma.

\begin{lemma}\label{lemma.enough.reg}
  Let $\Omega$, $\Gamma := \pa \Omega$, and $M$ have compatible straight cylindrical ends,
  see Equations \eqref{eq.def.Omega} and \eqref{eq.matching.ce0}, with $\Omega$ on one side
  of its boundary $\Gamma$, as before. Also, we assume that $E, F\to M$ are two compatible hermitian
  vector bundles, again, as before. We also let
  \begin{enumerate}[\rm (i)]
    \item $P \in \ePS{m}(M; E, F)$, $m < -1$;
    \item $s' \in (1/2, -m-1/2) \not = \emptyset$; and
    \item $h \in L^{2}(\Gamma; E)$;
  \end{enumerate}
  then $\maS_{P}h := P(h \otimes \delta_{\Gamma}) \in H^{s'}(M; F)$, and hence,
  \begin{equation*}
    [\maS_{P}h]_{+} \seq [\maS_{P}h]_{-} \seq [\maS_{P}h]\vert_{\Gamma}
    \in H^{s'-1/2}(\Gamma; F)\,.
  \end{equation*}
  In particular, the trace (or restriction) $\maS_{P}h\vert_{\Gamma}$
  of $\maS_{P}h := P(h \otimes \delta_{\Gamma})$ at $\Gamma$ is defined
  and it coincides with the traces $[\maS_{P}h]_{\pm}$ associated to the
  domains $\Omega =: \Omega_{+}$ and $\Omega_{-} := M \smallsetminus \overline{\Omega}$
  with boundary $\Gamma$.
\end{lemma}

\begin{proof}
  The condition $-m -1/2 > 1/2$ shows that the set $(1/2, -m-1/2)$
  is non-empty, so we can choose $s' \in (1/2, -m-1/2)$.
  Because $s'+ m < -1/2$, Lemma \ref{lemma.def.hdelta2}(i)
  shows that $h \otimes \delta_{\Gamma} \in H^{s'+m}(M; E)$, and
  therefore $\maS_{P}h := P(h \otimes \delta_{\Gamma}) \in H^{s'}(M; E)$.
  Since $s' > 1/2$, the trace $\maS_{P}h \in H^{s'-1/2}(\Gamma; E)$ is defined
  and it coincides with the traces from the two open subsets with boundary $\Gamma$,
  by Lemma \ref{lemma.limits=traces}.
\end{proof}

Because the trace of $\maS_{P}h := P(h \otimes \delta_{\Gamma})$
at $\Gamma$ is defined and it coincide with the traces associated to
the domains $\Omega_{\pm}$ with boundary $\Gamma$, we shall concentrate
on the restriction (or trace) $\maS_{P}h\vert_{\Gamma}$
of $\maS_{P}h$ to $\Gamma$. The behavior of
this restriction is the content of the following theorem.

For the rest of this section, we keep the assumptions of Lemma \ref{lemma.enough.reg},
that is, we assume that $\Omega$, $\Gamma:= \pa \Omega$, and
$M$ have compatible straight cylindrical ends (see Equations \eqref{eq.def.Omega} and
\eqref{eq.matching.ce0}) with $\Omega$ on one side of its boundary $\Gamma$
and that let $E, F\to M$ are two compatible hermitian vector bundles.

\begin{theorem}\label{thm.main.jump1b}
  We keep the assumptions of Lemma \ref{lemma.enough.reg}, as explained above.
  Let $m < -1$ and $s' \in (1/2, -m -1/2)$. Then, for $|\epsilon|$ small
  and any $P \in \ePS{m}(M; E, F)$,
  there exists a unique $P_{\epsilon} \in \ePS{m+1}(\Gamma; E, F)$ with
  the following properties:
  \begin{enumerate}[\rm (i)]
    \item For any $h \in L^{2}(\Gamma; E)$,
    $\maS_{P}h := P(h \otimes \delta_{\Gamma}) \in H^{s'}(M; F)$, and
    hence the traces of $\maS_{P}h := P(h \otimes \delta_{\Gamma})$ at the two sides
    of $\Gamma$ are defined and they satisfy
    \begin{equation*}
      [\maS_{P}h]_{+} \seq [\maS_{P}h]_{-} \seq [\maS_{P}h]\vert_{\Gamma}
      \seq P_{0}h \in H^{s'-1/2}(\Gamma; F)\,.
    \end{equation*}

    \item
    For any $x \in \Gamma := \pa \Omega$ and $\xi' \in T_{x}^{*}\Gamma$,
    let $\xi \in T_{x}^{*}M$ be a lift of $\xi'$.
    The principal symbol of $P_{0}$ is then given by
    \begin{equation*}
      \sigma_{m+1}(P_{0}; \xi') \seq \frac1{2\pi}\int_{\RR} \sigma_{m}(P; \xi
      + t \bsnu^{\sharp}_{x}) \, dt\,.
    \end{equation*}

    \item The distribution kernel of the operator $P_{0}$ satisfies
    $k_{P_{0}}(x', y') = k_{P}(x', y')$ for all $x' \neq y'$ in $\Gamma$, and hence
    $(\phi P \psi)_{0} = \phi P_{0} \psi$, for all $\phi, \psi \in C_{0}^{\infty}(M)$.

    \item Let $s \in \RR$ and $h \in H^{s}(\Gamma; E)$,
    then $[\maS_{P}h]_{\epsilon} \ede \big[ P(h \otimes \delta_{\Gamma})]_{\epsilon} \seq
    P_{\epsilon} h$ and
    \begin{equation*}
      [\maS_{P}h]_{\pm}
      \ede \lim_{\epsilon \to \pm 0}
      \big[ P(h \otimes \delta_{\Gamma})]_{\epsilon} \seq
      \lim_{\epsilon \to \pm 0} P_{\epsilon} h
      \seq P_{0} h \in H^{s -m -1}(\Gamma; F)\,,
    \end{equation*}
    where the limit is in the topology of $H^{s -m -1}(\Gamma; F)$.
  \end{enumerate}
\end{theorem}

The operator $P_{0}$ will be called the
\emph{restriction at $\Gamma$ operator} associated to $P$.

\begin{proof}
  Let us notice that the relations $[\maS_{P}h]_{+} \seq [\maS_{P}h]_{-}
  \seq [\maS_{P}h]\vert_{\Gamma} \in H^{s'-1/2}(\Gamma; F)$ have already been
  proved (see Lemma \ref{lemma.enough.reg}). Also, the last equality in (iii)
  is an immediate consequence of the equality of kernels (because $k_{\phi P \psi}
  = \phi k_{P} \psi)$. Thus we do not have to prove these properties anymore,
  so we will ignore these questions from now on.

  Let us assume that the distribution kernel $k_{P}$ of $P$ is
  \emph{compactly supported} (in $M \times M$) and prove our theorem in this case.
  We may also assume that $E$ and $F$ are trivial, one dimensional.
  Since we have assumed that the support $\supp k_{P} \subset M \times M$
  of the distribution kernel of $P$ is compactly supported, its two projections
  $K_{1} := p_{1}\supp k_{P} \subset M$ and $K_{2} := p_{2}\supp k_{P} \subset M$
  are also compact. Hence $K := K_{1} \cup K_{2} \cup \supp h$ is also compact.

  For each $x \in \Gamma$, we choose local coordinates $y$ in a neighborhood $V_{x}$
  of $x$ that straighten out the boundary to the hyperplane by mapping it to $\{x_{n} = 0 \}
  \subset \RR^{n}$. We can choose these coordinates such that they map
  $\exp(t \bsnu)$ to $(y', t) \in \RR^{n-1} \times (-\epsilon, \epsilon)$.
  Let us cover $\Gamma \cap K$ with finitely many such neighborhoods $V_{j} := V_{x_{j}}$,
  which is possible since $K$ is compact. Let us then choose a smooth partition of unity
  $\phi_{0}, \phi_{1}, \ldots, \phi_{N}$ on $M$ subordinated
  to  $M \smallsetminus K, V_{1}, \ldots, V_{N}$.
  We can assume that $\phi_{0}$ vanishes in a neighborhood of
  $\Gamma \cap K$. By refining the covering $\{V_{j}\}$ of $\Gamma$,
  we can assume that the support of each $\phi_{i} P \phi_{j}$, $1 \le i , j \le N$,
  is completely contained in a set of the form $V_{x}$. Then we use Theorem
  \ref{thm.main.jump-2} for each of the operators $\phi_{i} P \phi_{j}$ on the
  coordinate neighborhood $V_{x}$ to obtain the limit operator
  $P_{0ij} \in \Psi^{m+1}(\Gamma)$.
  We define $P_{0} := \sum_{i, j=1}^{N} P_{0ij}$. Then, for each of these operators,
  we have
  \begin{equation}\label{eq.lemma.jump1b.aux}
    \begin{gathered}
      \phi_{i} \big[ \maS_{P}(\phi_{j} h) \big]\vert_{\Gamma}
      \seq P_{0ij} h\,,\\
      \sigma_{m+1}(P_{0ij}; \xi')
      \seq \frac{\phi_{i}}{2\pi} \left( \int_{\RR}
      \sigma_{m}(P; \xi
      + t \bsnu^{\sharp}_{x}) \, dt \, \right) \phi_{j}\,, \quad \mbox{and}\\
      k_{P_{0ij}}(x', y') \seq k_{\phi_{i} P \phi_{j}}(x', y')
      \seq \phi_{i}(x') k_{P}(x', y') \phi_{j}(y')
      \,.
    \end{gathered}
  \end{equation}
  by Theorem \ref{thm.main.jump-2}. Adding up all the corresponding relations
  for $i, j = 1, \ldots, N$, and noticing that $\sum_{i=1}^{N} \phi_{i} = 1$ on
  $\Gamma \cap K$ (recall that $\phi_{0}$ vanishes in a neighborhood of
  $\Gamma \cap K$), we obtain (ii) and
  \begin{equation*}
      k_{P_{0}}(x', y') \ede \sum_{i, j=1}^{N} k_{P_{0ij}}(x', y') \seq
      \sum_{i, j=1}^{N} \phi_{i}(x') k_{P}(x', y') \phi_{j}(y') \seq
      k_{P}(x', y')
  \end{equation*}
  for all $x', y' \in \Gamma$, $x' \neq y'$. We have thus proved also (iii).
  To complete (i), let $h \in L^{2}(\Gamma; E)$.
  Then $\maS_{P}(\phi_{j} h) := P\big[ (\phi_{j} h) \otimes \delta_\Gamma \big ]
  \in H^{s'}(M; E)$, by Lemma \ref{lemma.enough.reg},
  because $s' \in (1/2, -m - 1/2)$. Therefore, $\maS_{P} h\in H^{s'}(M; E)$
  as well, by linearity. This gives
  \begin{equation*}
      \big[ \maS_{P} h \big]\vert_{\Gamma} \seq \sum_{i, j=1}^{N}
      \phi_{i} \big[ \maS_{P}(\phi_{j} h) \big]\vert_{\Gamma} \seq
      \sum_{i, j=1}^{N} P_{0ij} h \, =: \, P_{0}h \,,
  \end{equation*}
  where the second equality is from Equation \eqref{eq.lemma.jump1b.aux}
  (a consequence of Theorem \ref{thm.main.jump-2}). This gives
  the last equality of (i) and hence completes the proof of (i) as well.
  (We have already noticed that the first two equalities in (i) are the
  standard properties of Sobolev spaces discussed in Lemma \ref{lemma.enough.reg}.)

  We have thus proved the first three points of our theorem under the additional hypothesis
  that $k_{P}$ is compactly supported. The general case follows immediately from this
  one by using the results already proved for operators of
  the form $\phi P \psi$, where $\phi, \psi : M \to \CC$ are smooth and compactly
  supported. (Operators of this form will have compactly supported distribution
  kernels.) Indeed, there is no loss of generality to assume that $h$ has compact support.
  Let $\psi \in \CIc(M)$ be equal to 1 on
  the support of $h$. Let $x \in \Gamma$ arbitrary and $U$
  a relatively compact neighborhood of $x$ in $U$. Let $\phi \in \CIc(M)$
  be equal to $1$ on $U$. We first define $P_{0}h \vert_{\Gamma \cap U}
  := (\phi P \psi)_{0} h \vert_{\Gamma \cap U} $. This definition is independent
  of $\phi$ and $\psi$ by (iii) for compactly supported distribution kernels already proved.
  We then have
  \begin{equation*}
    [\maS_{P}h] \vert_{\Gamma \cap U}
    \seq [\maS_{\phi P \psi }h]\vert_{\Gamma \cap U}\\
    \seq (\phi P \psi)_{0}h \vert_{\Gamma \cap U}
    \, =:\, P_{0}h\vert_{\Gamma \cap U}\,.
  \end{equation*}
  Since $x$ was arbitrary, we obtain that $[\maS_{P}h] \vert_{\Gamma}$
  and $P_{0}h$ coincide in the neighborhood of every point, and hence they are equal.

  To prove the last point, let us assume first that $m = -\infty$. Then the result
  follows from the theorem on derivation under the integral sign (Theorem
  \ref{thm.aux.derivUnderInt}). Since the result depends linearly on $P$,
  we may therefore assume then that $P \in \iPS{m}(\Gamma; E, F)$ and hence that $P$
  is properly supported. If this is the case and $h$ has compact support,
  then it follows from the point (ii) already proved that $P_{0} \in
  \Psi^{m+1}(\Gamma; E, F)$.
  The fact that $P_{0} \in \iPS{m+1}(\Gamma; E, F)$ then follows from
  the form of the distribution kernels (our constructions commute with translations).
  The result for a general $h \in H^{s}(\Gamma; E)$ (not necessarily compactly supported)
  is obtained using the
  result already proved, the translation invariance at infinity
  and a dyadic partition of unity on the cylindrical ends of $M$,
  as in \cite{KMNW-2025}.
\end{proof}

We shall need mapping properties of the layer potentials, for which our main reference is
\cite{H-W}, where symbols of rational type are discussed in detail and where further references
are given. Recall \cite{H-W} that a symbol is of \emph{rational type} if in every fiber
it is a quotient of polynomial functions. In particular, a symbol of rational type
is classical.

\begin{theorem}\label{thm.mapping.lp}
  We keep the assumptions of Lemma \ref{lemma.enough.reg}.
  Let $P \in \ePS{m}(M; E, F)$ with symbol of \emph{rational type}
  (a quotient of polynomial functions).
  Then, for any $s \in \RR$, there exists $C_{s} \ge 0$ such that
  \begin{equation*}
    \|\maS_{P} h \vert_{\Omega_\pm}\|_{H^{s - m - \frac12}(\Omega_{\pm}; F)} \ede
    \|P (h \otimes \delta_{\Gamma})
    \vert_{\Omega_\pm}\|_{H^{s - m - \frac12}(\Omega_{\pm}; F)}
    \le C_{s} \|h\|_{H^{s}(\Gamma; E)}\,.
  \end{equation*}
\end{theorem}

\begin{proof}
  If the distribution kernel $k_{P}$ of $P$ has compact
  support, the result follows from Theorem 9.4.7 on page 584 of \cite{H-W}.
  If $P$ is properly supported, the result follows from this case and
  a translation invariant partition of unity on the cylindrical ends of $M$.
  We have $h \otimes \delta_{\Gamma} \in H^{s'}(M; E)$.
  If $P \in \ePS{-\infty}(M; E, F)$, the mapping properties of the operators
  in the essentially translation invariant
  calculus gives that $\maS_{P}(h) := P(h \otimes \delta_{\Gamma})
  \in H^{\infty}(M; F)$, so the result is true in this case as well.
  Since every $P \in \ePS{m}(M; E, F)$ is the sum of an operator in
  $\iPS{m}(M; E, F)$ (which is properly supported) and an operator in
  $\ePS{-\infty}(M; E, F)$, the result follows from the particular cases
  we have already proved.
\end{proof}

By extending the above mapping property to Besov spaces, one can extend our well-posedness
result to $L^{p}$-type Sobolev and Besov spaces. This will be discussed in a future publication.
Recall from Notation \ref{not.bsnu} that $\bsnu$ is a fixed vector field on $M$ that is the outer
unit normal vector to $\Gamma := \pa \Omega$. Also, recall that $\sharp : TM \to T^{*}M$
is the isomorphism defined by the metric. We now turn to the study of the normal lateral limits
of operators of order $m = -1$. We first need to establish some more notation.

\begin{notation}\label{not.ord-1}
  For $P \in \ePS{-1}(M; E, F)$, we let
  $\sigma_{-1}(P) \in \CI(T^*M \smallsetminus \{0\}; \Hom(E; F))$ denote
  its principal symbol. We shall write $\sigma_{-1}(P; \xi) \in
  \Hom(E_{x}, F_{x})$ for its value at $\xi \in T_{x}^*M \smallsetminus \{0\}$
  and we assume that $\sigma_{-1}(P; \xi)$ is odd in $\xi \in T^{*}M$.
  We then let $\JC(P; x) \in \Hom(E_{x}, F_{x})$ be given by
  \begin{equation*}  
    \JC(P; x) \ede \sigma_{-1}(P; -\bsnu_{x}^{\sharp})\,.
  \end{equation*}
  We also let, for all $0\neq \xi' \in T^{*}\Gamma$,
  $\xi \in T_{x}^{*}M$ be such that it projects onto $\xi'$ and is
  orthogonal to $\bsnu_{x}^{\sharp}$ and
  \begin{equation*}  
    b_{0}(\xi') \ede \frac1{4\pi}\, \int_{\RR}\big[
      \sigma_{-1}(P; \xi + \tau \bsnu_{x}^{\sharp}) +
      \sigma_{-1}(P; \xi - \tau \bsnu_{x}^{\sharp}) \big] \, d\tau
      \in \Hom(E_{x}, F_{x})\,.
  \end{equation*}
\end{notation}

The form of the definition of $\JC$ is due to the fact
that $e_{n} = -\bsnu$ in the Euclidean case (see Section \ref{sec.nllRn}) if
$\Omega = \Omega_{+} = \RR_{+}^{n}$. We let $\maS_{P}h := P(h \otimes \delta_{\Gamma})$,
as usual.

\begin{theorem}\label{thm.main.jump1}
  Let $P \in \ePS{-1}(M; E, F)$. We continue to assume the conditions of
  Lemma \ref{lemma.enough.reg}, as before. Let $\JC(P)$ and $b_{0}$
  be as in Notation \ref{not.ord-1} and assume that $\sigma_{-1}(P; \xi)$ is \emph{odd} in
  $\xi \in T^{*}M$. Then, for $|\epsilon|$ small, there exist pseudodifferential
  operators $P_{\epsilon} \in \ePS{0}(\Gamma; E, F)$, such that,
  if we let $P_{0\pm} \ede \pm \frac{\imath}2 \JC(P) + P_{0}$, then
  we have the following properties.
  \begin{enumerate}[\rm (i)]
    \item $\JC_{\pm}(P) \in C_{\inv}^{\infty}(M; \Hom(E; F))$.
    \item $P_{0\pm} \in \ePS{0}(\Gamma; E, F)$.
    \item $\sigma_{0}(P_{0}) = b_{0}$.
    \item $k_{P_{0}}(x', y') = k_{P}(x', y')$ for all $x' \neq y'$ in $\Gamma$.

    \item for all $s \in \RR$ and all $h \in H^{s}(\Gamma; E)$, we have
    $\big[ \maS_{P} h\big]_{\epsilon} := \big[ P(h \otimes \delta_{\Gamma}) \big]_{\epsilon}
    = P_{\epsilon}h$.
    \item $\big[ \maS_{P} h \big]_{\pm}
      \ede \lim_{\epsilon \to \pm 0}
      \big[ P(h \otimes \delta_{\Gamma}) \big]_{\epsilon}
      \seq \lim_{\epsilon \to \pm 0} P_{\epsilon}h
      \seq P_{0\pm} h \in H^{s}(\Gamma; F)\,.$

    \item Let us assume that $P$ is of rational type,
    that $s > 0$, and that $h \in H^{s}(M; E)$. Then $\maS_{P}(h) \in H^{s+1/2}(\Omega_{\pm})$,
    and hence the traces satisfy
    \begin{equation*}
      \gamma_{\pm}(\maS_{P}(h)) \seq \maS_{P}(h)_{\pm}
      \seq P_{0\pm} h \in H^{s}(\Gamma; F)\,.
    \end{equation*}
  \end{enumerate}
\end{theorem}

\begin{proof}
  We have $\JC(P) \in C_{\inv}^{\infty}(M; \Hom(E; F))
  \subset \iPS{0}(M; E, F)$ by the definitions of $\JC_{\pm}(P)$ and the properties
  of the principal symbol. Except for the first and last points, the
  proof is to a large extent, word-for-word the same as the one of Theorem
  \ref{thm.main.jump1b}, whose notations we use here. We sketch the argument
  for the benefit of the reader.

  We agaoin begin with the case when $P$ has compactly supported
  distribution kernel. Then, the first two relations in the
  crucial Equation \eqref{eq.lemma.jump1b.aux} are replaced with
  \begin{equation*}  
    \begin{gathered}
      \phi_{i} \big[ \maS_{P}(\phi_{j} h) \big]_{\pm}
      \seq \big(P_{0ij} \big)_{\pm} h \qquad \mbox{and}\\
      \sigma_{m+1}(P_{0ij}; \xi')
      \seq \frac{\phi_{i}}{4\pi} \left( \int_{\RR}
      \sigma_{m}(P; \xi + t \bsnu^{\sharp}_{x})
      + \sigma_{m}(P; \xi - t \bsnu^{\sharp}_{x}) \, dt \, \right) \phi_{j} \,,
    \end{gathered}
  \end{equation*}
  where $\xi \perp \bsnu_{x}$. (The last relation of that equation does not change.)

  Let us now turn to points (i) and (vii), which are slightly different from the
  proof of Theorem \ref{thm.main.jump1b}.
  First, the only thing that we need to add to complete the proofs of (i) is
  that $\JC_{\pm}(P) = \sum_{ij=1}^{N} \JC_{\pm}(\phi_{i}P \phi_{j})$.
  Finally, the last point follows from Lemma \ref{lemma.limits=traces} and Theorem
  \ref{thm.mapping.lp}.
\end{proof}

If in the previous theorem $P$ is the Laplacian,
$P = \Delta := -d^{*}d$ or the classical Stokes
operator $ \bsXi_{0, 0}$, then $b_{0}= 0$ and hence $P_{0}$
is of order $-1$, but that is not
true in general. For instance, as we will see below,
it is not true if $P =  \bsXi_{V, V_{0}}$ and
$V_{0}$ does not vanish identically on $\Gamma = \pa \Omega$.

\begin{corollary}\label{cor.adjoints}
  Let $P \in \ePS{m}(M; E, F)$ be as in Theorem \ref{thm.main.jump1b}
  (i.e. $m < -1$) or as in Theorem \ref{thm.main.jump1} (i.e. $m = -1$
  and $P$ is classical). Then $(P^{*})_{0} = (P_{0})^{*}$ and, when
  $P$ is classical, $\JC(P^{*}) = \JC(P)^{*}$.
\end{corollary}

\begin{proof}
  Let $x', y' \in \Gamma$, $x' \neq y'$. Theorems \ref{thm.main.jump1b} and
  \ref{thm.main.jump1} give the second and the last of the following equalities:
  \begin{equation*}
    k_{P_{0}^{*}}(x', y') \seq k_{P_{0}}(y', x')^{*} \seq
    k_{P}(y', x')^{*} \seq k_{P^{*}}(x', y') \seq k_{(P^{*})_{0}}(x', y')\,.
  \end{equation*}
  As we have already noticed, the operators $P_{0}^{*}$ and $(P^{*})_{0}$ are
  determined by the values of their distribution kernels outside the diagonal
  (see Remark \ref{rem.u.det} and Theorem \ref{thm.main.jump0}(ii)).
  Since the corresponding distribution kernels of $(P^{*})_{0}$ and $(P_{0})^{*}$ coincide,
  we have $P_{0}^{*} = (P^{*})_{0}$, as claimed. The last relation of the
  Corollary follows from the sequence of equalities
  \begin{equation*}
    \JC(P^{*}) \seq \sigma_{-1}(P^{*}; -\bsnu^{\sharp})
    \seq \sigma_{-1}(P; -\bsnu^{\sharp})^{*} \seq \JC(P)^{*}\,.
  \end{equation*}
  This completes the proof.
\end{proof}

Without ambiguity, we shall therefore write
$P_{0}^{*} = (P^{*})_{0} = (P_{0})^{*}$ from now on.

As noticed in \cite{KNW-2025, KMNW-2025}, the results we have obtained so far
in this section remain valid in the case $\Omega$ bounded.
We conclude this section with a property that is specific to
manifolds with cylindrical ends. Recall the indicial operator
$\In(P)$ and the indicial family $\widehat{P}(\tau)$, obtained by
taking the Fourier transform of $\In(P)$.

\begin{theorem}\label{thm2.indicial.jump1}
  Let $P \in \ePS{m}(M; E, F)$, $m \le -1$. Then
  \begin{equation*}
    [\In (P)]_{0} \seq \In(P_{0}) \quad \mbox{and}
    \quad
    (\widehat{P}(\tau))_{0} \seq \widehat{(P_{0})}(\tau) \,.
  \end{equation*}
  If $m = -1$, we also have $\In[\JC_{\pm}(P)] \seq \JC_{\pm}[\In(P)]$ and,
  consequently, $\widehat{[\JC_{\pm}(P)]}(\tau) \seq \JC_{\pm}[\widehat{P}(\tau)]$
  and $(\widehat{P}(\tau))_{0\pm} \seq \widehat{(P_{0\pm})}(\tau)$.
\end{theorem}

\begin{proof}
  The first relation is follows from the fact that limits at infinity and
  restrictions to $\Gamma$ commute. Both of them commute also with translations,
  and hence with the Fourier transform. This gives the second relation.
  The last relation follows from the fact that $\sigma_{-1}$ and $\In$
  commute, see Proposition \ref{propv1.lemma.onto}(iv).
\end{proof}

\subsection{\ADN\ operators: jumps, ellipticity, and Fredholmness}
The generalized Stokes operator $\bsXi := \bsXi_{V, V_{0}}$ (Equation
\ref{eq.def.bsXi}) is not elliptic in the usual sense, but is elliptic
in \ADN\ sense. As with the usual ellipticity property, the ellipticity
in \ADN\ sense (together with the invertibility of the limit operators) is
enough to establish the Fredholm property of our generalized Stokes operator
$\bsXi := \bsXi_{V, V_{0}}$, but under different Sobolev spaces. This is the
content of Theorem \ref{thm.ADN.Fredholm}. {\cn The definition of
$\ADN$-operators makes sense for general manifolds $M$, so, in the beginning
of this subsection, we allow $M$ to be a \emph{general} smooth manifold.}

\begin{definition}\label{def.ADN}
  Let $M$ be a smooth manifold, let $E_{i} \to M$ smooth vector bundles,
  $1 \le i \le k$, and let $s_i, t_j \in \RR$, $1 \le i, j \le k$. We set
  $\bss := (s_{1}, \ldots, s_{k})$, $\bst := (t_{1}, \ldots, t_{k})$,
  and $E := E_{1} \oplus \ldots \oplus E_{k}$. Then
  \begin{equation*}
    \Psi^{[\bss + \bst ]}(M; E) \ede
    \{ P=[P_{ij}] \mid\, P_{ij} \in \Psi^{s_i + t_j}(M; E_j, E_i),\
    1 \le i, j \le k \}\,.
  \end{equation*}
  An operator $P = [P_{ij}] \in \Psi^{[\bss + \bst ]}(M; TM \oplus \CC)$
  is said to be of \emph{\ADN-order $\le [\bss + \bst ]$}. For $P = [P_{ij}] \in
  \Psi^{[\bss + \bst ]}(M; TM \oplus \CC)$
  let
  \begin{equation*}
    \Symb_{\bss, \bst}(P) \ede [\sigma _{s_i+t_j}(P_{ij})]
  \end{equation*}
  be its \emph{$(\bss,\bst)$--principal symbol,} which is a suitable section of
  the lift of the endomorphism bundle $\End(E)$ to $T^{*}M \smallsetminus 0$.
  The operator $P$ is said to be {\it $(\bss,\bst)$--\ADN\ elliptic} if its
  $(\bss, \bst)$--principal symbol matrix $\Symb_{\bss, \bst}(P)$ is invertible
  outside the
  zero section (i.e., on $T^{*}M \smallsetminus 0$).
\end{definition}

The reason for introducing the spaces $\Psi^{[\bss + \bst ]}(M; E)$
is that our generalized Stokes operator fits into this framework.

\begin{example} \label{ex.bsXi}
  Let $k = 2$, let $\bss = (s_{1}, s_{2}) = \bst = (t_{1}, t_{2}) = (1, 0)$,
  let ${E_{1}} = TM$, and let ${E_{2}} = \CC$. Then
  \begin{equation*}
    \bsXi \ede \bsXi_{V, V_{0}} \in \Psi^{[\bss + \bst ]}(M; TM \oplus \CC) \ede
    \{ P=[P_{ij}] \mid\, P_{ij} \in \Psi^{s_i + t_j}(M; E_j, E_i) \}\,.
  \end{equation*}
  More importantly, $\bsXi$ is not elliptic
  in the usual sense, but we will see in Proposition \ref{prop.ADN.elliptic}
  that $\bsXi$ is $(\bss, \bst)$-\ADN-elliptic.
\end{example}

Recall that $\sharp : T^{*}M \to TM$ denotes the isomorphism defined by the Riemannian
metric $g$ of $M$ (it is sometimes called the ``musical isomorphism''). For instance,
using the notation introduced in Notation \eqref{not.def.endomorphism}, if
$0 \neq \eta \in T_{x}M$, then $\frac1{\|\eta\|^{2}} \eta \otimes \eta^{\sharp}$ is the
projection onto $\eta \neq 0$, regarded as a linear map $T_{x} M \to T_{x}M$.

\begin{proposition} \label{prop.ADN.elliptic}
  Let $(\bss, \bst) = (1,0)$. Then the generalized Stokes operator
  $\bsXi := \bsXi_{V, V_{0}}$ \eqref{eq.def.bsXi} belongs to
  $\Psi^{[\bss + \bst]}(M; TM \oplus \CC)$ and is
  $(\bss,\bst)$--\ADN\ elliptic {\lpar}Definition \ref{def.ADN}{\rpar}. The
  $(\bss,\bst)$--principal symbol of $\bsXi$ is
  \begin{equation*} 
    \Symb_{\bss, \bst}\left(\bsXi \right)( \xi)=
    \left(
      \begin{array}{cc}
        | \xi|^2 + \xi^{\sharp}  \otimes \xi  & \imath \xi^{\sharp}\\
        -\imath \xi & -V_{0}
      \end{array}
    \right) \in \End(TM \oplus \CC)\,,
  \end{equation*}
  which is invertible with inverse
  \begin{equation*} 
    \left(
      \begin{array}{cc}
        \frac{1}{| \xi |^2} - \frac{V_{0}+1}{2V_{0}+1}\frac{1}{| \xi |^4}
     \xi^{\sharp} \otimes \xi
        & \frac{\imath}{(2V_{0}+1)| \xi |^2} \xi^{\sharp}\\
        - \frac{\imath}{(2V_{0}+1)| \xi |^2} \xi & -\frac{2}{2V_{0}+1}
      \end{array}
    \right) \in \End(TM \oplus \CC)\,.
  \end{equation*}
\end{proposition}

\begin{proof}
  The given formula for principal symbol of $\bsXi$ is a consequence of
  \cite[Chapter 5 \S 12, (12.20)]{Taylor1} (but see also Corollary
  \ref{cor.formulas} for the principal symbol of $\Defstar \Def$ and
  Remark \ref{rem.basic.ip}(iii) for the principal symbol of $\nabla$).
  The result then follows by verifying that the
  product of these two matrices is the identity.
\end{proof}

See also \cite[Theorem 5.11 and Remark 5.12]{KohrNistor-Stokes}
and \cite{Co-We}. We notice that
the formula for the inverse of the principal symbol was obtained using Schur's
complement formula.

{\cn We now return to the case when $M$ is a manifold with straight cylindrical ends
(Definition \ref{def.cyl.end}) and the vector bundles are compatible (Definition
\ref{def.comp.vb}).} We begin by adapting Definition \ref{def.ADN} to the
setting cylindrical ends.

\begin{remark}\label{rem.ADN.ess}
  Unless explicitly stated otherwise, we use the notation of the last definition,
  Definition \ref{def.ADN}. Let $M = \canDec$ be a manifold with straight cylindrical
  ends and $E_{i} \to M$ be compatible vector bundles. Recall the definitions of
  $\ePS{m}(M; E)$ and $\ePSsus{m}{M'}{E}$ from Definitions \ref{def.ess.one}
  and \ref{def.ePSm}. We now want to introduce the analogous spaces for \ADN-operators.
  First, we let
  \begin{equation*}
    \ePS{[\bss + \bst]}(M; E) \ede
    \{ P=[P_{ij}] \mid\, P_{ij} \in \ePS{s_i + t_j}(M; E_j, E_i),\
    1 \le i, j \le k \}\,.
  \end{equation*}
  Let also $m = \max\{s_{i} + t_{j}\}$. Then
  $\Psi^{[\bss + \bst ]}(M; E) \subset \Psi^{m}(M; E)$
  and we have
  \begin{equation*}
    \ePS{[\bss + \bst]}(M; E) \seq \ePS{m}(M; E) \cap \Psi^{[\bss + \bst]}(M; E)
    \,.
  \end{equation*}
  Second, by considering $M' \times \RR$ instead of $M$, we can define similarly,
  \begin{multline*}
    \ePSsus{[\bss + \bst]}{M'}{E} \seq \ePSsus{m}{M'}{E} \cap
    \Psi^{[\bss + \bst]}(M' \times \RR; E)\\
    \ede
    \ePS{m}(M' \times \RR; E)^{\RR} \cap \Psi^{[\bss + \bst]}(M' \times \RR; E)
    \,.
  \end{multline*}
  The limit and indicial operators $\In(P)$ and $\widehat{P}(\tau)$
  of $P \in \Psi^{[\bss + \bst ]}(M; E)$ are defined by restriction from
  $\ePS{m}(M; E)$. In particular,
  $\In (P) \in \ePSsus{[\bss + \bst ]}{M'}{TM \oplus \CC}$ (See Lemma \ref{lemma.def.indicial})
  and their indicial operators (of Fourier transforms)
  $\widehat{P}(\tau) \in \Psi^{[\bss + \bst ]}(M'; TM \oplus \CC)$ (see Definition
  \ref{def.indicial.fam}).
\end{remark}

We will use the \ADN-ellipticity of $\bsXi$ to establish (in the next subsection,
under certain positivity assumptions on the potentials $V$ and $V_{0}$) its Fredholm
property. To that end, we shall use the following characterization of the Fredholm
property in the essentially translation invariant calculus. (This result was discussed
and proved in great detail in \cite[Theorem 15.4.19]{KMNW-2025} and \cite[Theorem 7.12]{KNW-22},
so we will not include a proof here. It builds on earlier results of many people,
including Kondratiev, Melrose, and Schulze, see the above quoted works for references.)
Recall the notation of the last remark, Remark \ref{rem.ADN.ess}. Then we have the
following characterization of Fredholm property for a \ADN-elliptic operator on
a manifold with straight cylindrical ends (see \cite{KohrNistor-Stokes} and \cite{KMNW-2025}).

\begin{theorem}\label{thm.ADN.Fredholm}
  Let $M = \canDec$ be a manifold with straight cylindrical ends and $s_i, t_j \in \RR$,
  $1 \le i, j \le k$, as before. Let $P=[P_{ij}] \in \ePS{[\bss + \bst]}(M; E)$ and
  $m \in \RR$. Then the map
  \begin{multline*}
    P : H^{m + [\bst]}(M; E) \ede \oplus_{j=1}^{k} H^{m + t_{1}}(M; E_{j})
    \to H^{m - [\bss]}(M; E)
    \ede \oplus_{i=1}^{k} H^{m - s_{1}}(M; E_{i})
  \end{multline*}
  is Fredholm if, and only if,
  \begin{enumerate}[\rm (i)]
    \item $P$ is \ADN-elliptic and
    \item the limit operator $\In (P) \in \ePSsus{[\bss + \bst]}{M'}{E}$
    is invertible.
  \end{enumerate}
  \lpar If $M$ is closed, we replace $\ePS{[\bss + \bst]}(M; E)$
  with $\Psi^{[\bss + \bst]}(M; E)$ and we
  remove the condition {\rm (ii)}.\rpar
\end{theorem}

The case $M$ closed is, of course, very classical, due to Seeley,
see \cite{KMNW-2025, Hormander3, H-W, Seeley59, Wl-Ro-La}.
The results of the previous subsections of this section on normal lateral
limites then extend immediately
to the components $P_{ij}$ of the operators $P = [P_{ij}]\in \Psi^{[\bss + \bst]}(M; E)$
that satisfy $s_{i} + t_{j} \le -1$.
(See also \cite{CarvalhoNistorQiao, Fr-Grrieser-Sc, Kapanadze-Schulze,
KohrNistor-Stokes, KMNW-2025, Kondratiev, Melrose-Mendoza, SchroheFrechet} and
the references therein.)

\subsection{Further properties: regularity and spectral invariance}
We shall repeatedly make use of the following spectral invariance and
elliptic regularity results from \cite{KMNW-2025, KNW-22}.
We have the following generalization of the classical elliptic regularity
result (see Theorem 15.3.30 of \cite{KMNW-2025} or
\cite[Proposition 6.5]{KNW-22}, see also \cite{MelroseAPS, SchulzeBook91}).
We use the notation of Definition
\ref{def.ADN}. Thus $E_{j}, F_i \to M$, $i, j = 1, \ldots, k$, are compatible vector
bundles, $E = \oplus_{j=1}^k E_{j}$, and $F = \oplus_{i=1}^k F_{i}$, and
$\bss = (s_{1}, \ldots, s_{k})$ and $\bst = (t_{1}, \ldots, t_{k})$, $s_{i}, t_{j} \in \RR$.

\begin{theorem}
\label{thm.regularity}
  Let $P \in \ePS{[\bss + \bst]}(M; E, F)$ be \ADN-elliptic. Let also $m, m' \in \RR$.
  If $u \in H^{m'+[\bst]}(M; E)$ is such that $P u \in H^{m-[\bss]}(M; F)$,
  then $u \in H^{m+[\bst]}(M; E)$.
\end{theorem}

We have the following related spectral invariance property (\cite[Theorem 15.4.11]{KMNW-2025},
see also \cite{BealsSpInv, KNW-22, Mitrea-Nistor}).

\begin{theorem}\label{thm.spectral.inv}
  Let $s_i, t_j \geq 0$, $m \in \RR$, and $T \in \ePS {[\bss + \bst]}(M;E,F)$. Let
  \begin{equation*}
    \maD(T_{\rm{max}}) \ede \{\, \xi \in H^{m-[\bss]}(M;E) :\
    T\xi \in H^{m-[\bss]}(M;F)\, \}
  \end{equation*}
  and assume that $T : \maD(T_{\rm{max}}) \to H^{m-[\bss]}(M;E)$ is \emph{Fredholm} with
  pseudo-inverse denoted $T^{(-1)}$. If one of the $s_i$ or $t_i$ is \lpar strictly\rpar\ positive
  {\rm (i.e., $>0$)}, we also assume that $T$ is $(\bss, \bst)$-\ADN-elliptic. Then its
  Moore-Penrose pseudo-inverse $T^{(-1)} \in \ePS{[\mathbf{-s-t}]}(M; F, E)$.
\end{theorem}

This result was proved in case $T$ is actually invertible in the above mentioned publications,
so let us include a proof for the statement in case $T$ is Fredholm, using the corresponding
statement for ``$T$ invertible''.

\begin{proof}
  We can assume that $E = F$ and $\bss$ and $\bst=0$ are constant, so we work with the
  usual order (this is the only case used in this paper).
  Let us notice that $T$ is elliptic (if it is of order zero, this
  follows from the Fredholm property, Theorem \ref{thm.ADN.Fredholm}). Elliptic regularity,
  Theorem \ref{thm.regularity} gives that its kernel $\maN$ consists of functions in
  $H^{\infty}(M)$ and hence the orthogonal projection onto $p_{\maN} \in \ePS{-\infty}(M; E)$.
  Then $p_{\maN} + T^{*}T$ is invertible and hence $(p_{\maN} + T^{*}T)^{-1}
  \in \ePS{[-2\bss]}(M; E)$. The desired result follows from $T^{(-1)}
  = (p_{\maN} + T^{*}T)^{-1}T^{*} \in \ePS{[-2\bss]}(M; E)$.
\end{proof}

These results are true (and were known for a long time) if $M$ closed,
in which case we replace $\ePS{m}(M; E)$ with $\Psi^{m}(M; E)$.

\section{\cn Green formulas and invertibility of $\bsXi$ on cylinders}
\label{sec.Green2}

We now extend the Green formulas and the results in Section \ref{sec.Green} to the
case of a \emph{cylinder.} We then use these results to prove the invertibility of
our generalized Stokes operator $\bsXi := \bsXi_{V, V_{0}}$ (Equation \eqref{eq.def.bsXi})
on a cylinder.

{\cn We thus assume in this section that $M := \mfkS \times \RR$, where $\mfkS$ is a
closed manifold (i.e., smooth, compact, without boundary).
Moreover, all other objects used in this section will be translation
invariant, thus $\Omega = \Omega' \times \RR$ and $V$ and $V_{0}$ descend to functions
on $\mfkS$, denoted, respectively, $\newtilde V$ and $\newtilde V_{0}$
(see Equation \eqref{eq.just.new.notation} for a justification of this notation).}
Explicitly, we have $V(x', t) = \newtilde V(x')$ and $V_{0}(x', t) = \newtilde V_{0}(x')$
for $x' \in \mfkS$ and $t \in \RR$.
The results of this section will be used for $\mfkS$ possibly disconnected, so we
\emph{will not assume that $\mfkS$ is connected,} although the proofs of the results
in this section can be reduced immediately to the connected case. Note also that we
are continuing to assume that $M$ is a manifold with cylindrical ends, albeit of a
very particular kind. Recall that the scalar product on $L^{2}(M; E)$ is
\begin{equation*}
  (f,g)_{\Omega} \ede \int_{\Omega} f(x) \cdot g(x)\, \dvol_{g}(x)\,.
\end{equation*}
(We have $(f,g)_{\Omega} \ede \int_{\Omega} f(x) \overline{g(x)}\, \dvol_{g}(x)$
if $f$ and $g$ are functions.) Our convention is that
our scalar products are conjugate linear in the {\it second variable} and
linear in the first.

\subsection{Differential operators on cylinders and their indicial operators}
We now make explicit the differential operators studied in the
previous subsection for our manifold $M = \mfkS \times \RR$ (a cylinder). We then use these
explicit formulas to study the associated indicial operators and prove that they
have the $L^{2}$-unique continuation property.

Let $p_{i}$, $i=1, 2$ be the projections of $M = \mfkS \times \RR$ onto $\mfkS$ and,
respectively, $\RR $. Let $p_{1}^{*}(T\mfkS)$ be the {\it pull-back} of the tangent
vector bundle to $\mfkS$ to $M$. It coincides with the vertical tangent bundle to
the fibers of $\mfkS \times \RR \to \RR$. The pull-back $p_{2}^{*}(T\RR)$ is defined
similarly. We will use the following decomposition:
\begin{equation}\label{eq.decomp.TM}
  TM = p_{1}^{*}(T\mfkS) \oplus p_{2}^{*}(T\RR) = p_{1}^{*}(T\mfkS) \times \RR\,.
\end{equation}
Accordingly, any section $X$ of $TM$ will decompose as $X = (X_{1}, f_{X})$,
where $X_{1}$ is a section of $p_{1}^{*}(T\mfkS) \to M:= \mfkS \times \RR$
and $f_{X} \in \CI(M)$. Thus $X_{1}$ is a smooth family of vertical vector fields
on $M$ (i.e., tangent to the submanifolds $\mfkS \times \{t\}$, $t \in \RR$).
Similarly, $f_{X}$ corresponds to $f_{X} \pa_{t}$, where $t$ is the variable on $\RR$.

We shall denote with a `prime' (i.e., with ${}^{\prime}$) the ``vertical'' objects
associated to the fibers $\mfkS \times \{t\}$, $t \in \RR$, of the projection $M \to \RR$.
We shall sometimes write simply $T\mfkS$ instead of its pull-back $p_{1}^{*}T\mfkS$.
We shall do that especially when considering sections of the latter bundle, thus
\begin{equation*}
  \CI(M; T\mfkS) \ede \CI(M; p_{1}^{*}T\mfkS)\,.
\end{equation*}
Let $\nabla f := (df)^{\sharp}$ be the usual gradient.
If $f \in \CI(M)$ and $x = (x', t) \in M = \mfkS \times \RR$, then
$\nabla' f(x', t)$ is the gradient at $x'$ of the restriction
$f\vert_{\mfkS \times \{t\}}$, $t \in \RR$, called the {\it vertical gradient.}
The vertical gradient $\nabla' f$ is thus
a section of $p_{1}^{*} T\mfkS$.
Similarly, $d'$ is the {\it vertical differential} and $\Def'$ is the {\it
vertical deformation tensor}.
In particular, the vertical gradient $\nabla' : \CI(M) \to \CI(M; T\mfkS)$ is
given by $\nabla'(f) = (d'f)^{\sharp}$.

\begin{proposition}\label{prop.grad.prod}
  Let $M = \mfkS \times \RR$ and let $\nabla'$ be the vertical part of the gradient.
  In terms of the decomposition of Equation \eqref{eq.decomp.TM}, we obtain that
  the gradient of $f \in \CI(M)$, $M = \mfkS \times \RR$,
  is $\nabla f = (\nabla'f, \pa_{t} f)$. In matrix notation, we have
  \begin{equation}\label{eq.prod.nabla}
    \nabla \seq \cvector{\nabla'}{\pa_{t}} : \CI(M) \to
    \begin{array}{c}
      \CI(M; T\mfkS)\\
      \oplus\\
      \CI(M)
    \end{array}
  \end{equation}
\end{proposition}

This follows by noticing that $df = d'f + \pa_t f dt$, where $t \in \RR$ is
the canonical variable on the cylinder $\mfkS \times \RR$.

\begin{proposition} \label{prop.Def.prod}
  Let $M = \mfkS \times \RR$ and let $\Def'$ be the vertical deformation operator
  (acting just on the $\mfkS$-directions) and similarly $d'$ is the vertical differentiation.
  Using the decomposition \eqref{eq.decomp.TM}, the deformation operator $\Def$ becomes
  \begin{equation}\label{eq.Def.prod}
    \Def \seq
    \left( \begin{array}{cc}
    \Def' & 0\\
    \frac1{\sqrt{2}} \sharp \pa_{t} & \frac1{\sqrt{2}}d'\\
    0 & \pa_{t}
    \end{array}\right) \ : \
    \begin{array}{c}
    \CI(M; T\mfkS)\\
    \oplus\\
    \CI(M)
    \end{array} \ \rightarrow \ \begin{array}{c}
    \CI(M; T^{*\otimes 2}\mfkS)\\
    \oplus\\
    \CI(M; T^{*}\mfkS)\\
    \oplus\\
    \CI(M) \,.
  \end{array}
\end{equation}
\end{proposition}

\begin{proof}
  Let $X, Y, Z$ be three smooth vector fields on $M = \mfkS \times \RR$.
  The decomposition $TM = p_{1}^{*}(T\mfkS) \oplus p_{2}^{*}(T\RR)
  = p_{1}^{*}(T\mfkS) \times \RR$ of Equation \eqref{eq.decomp.TM}
  gives $X = (X_{1}, f_{X})$,
  where $X_{1}$ is a section of $p_{1}^{*}(T\mfkS) \to M:= \mfkS \times \RR$
  and $f_{X} \in \CI(M)$, as explained above. We decompose similarly
  $Y = (Y_{1}, f_{Y})$ and $Z = (Z_{1}, f_{Z})$. The Levi-Civita
  connection $\nabla_{X}^{LC, M}$ on $M$ is then, by unicity,
  \begin{equation}\label{eq.prod.LC}
    \nabla_{X}^{LC, M} Y \seq (\nabla_{X_{1}}^{LC, \mfkS} Y_{1} + f_{X} \pa_{t}Y_{1},
    X_{1}(f_{Y}) + f_{X} \pa_{t} f_{Y})\,.
  \end{equation}

  The formula \eqref{eq.decomp.TM} gives that the symmetric tensor product
  $S^{2}T^{*}M \subset T^{*\otimes 2}M$ can be written as
  \begin{equation}\label{eq.sym.sum}
    S^{2}T^{*}M \seq S^{2}p_{1}^{*}T^{*}\mfkS \oplus p_{1}^{*}T^{*}\mfkS \oplus \CC\,.
  \end{equation}
  (The factors involving $\sqrt{2}$ appear in order to make the above identification
  isometric.) With these identifications, we have
  \begin{equation}\label{eq.nabla.prod}
    \big(\nabla_{Y}^{LC, M}X\big) \cdot Z \seq
    \big(\nabla_{Y_{1}}^{LC, \mfkS}X_{1}\big) \cdot Z_{1} + f_{Y}(\pa_{t}X_{1})\cdot Z_{1}
      + Y_{1}(f_{X})f_{Z} + f_{Y}f_{Z}\pa_{t}f_{X}\,.
  \end{equation}
  Formula \ref{eq.def.Def} then yields the result.
\end{proof}

Recall that if $P \in \ePSsus{m}{\mfkS}{E, F} := \ePS{m}(\mfkS \times \RR; E, F)^{\RR}$,
then the Fourier transform yields the {\it indicial operators} $\widehat{P}(\tau)
\in \Psi^{m}(\mfkS; E, F)$, $\tau \in \RR$, (see Definition \ref{def.indicial.fam}).
We can now prove the following result.

\begin{proposition}\label{prop.uc.nabla}
  Let $M = \mfkS \times \RR$ and let $\nabla'$ be the vertical gradient.
  In terms of the decomposition of Equation \eqref{eq.decomp.TM}, we obtain that
  the indicial operator $\widehat \nabla (\tau)$ is given by
  \begin{equation*}
    \widehat \nabla (\tau) \seq \cvector{\nabla'}{\imath \tau} : \CI(\mfkS) \to
    \CI(\mfkS; TM) \ \simeq
    \begin{array}{c}
      \CI(\mfkS; T\mfkS)\\
      \oplus\\
      \CI(\mfkS)\,.
    \end{array}
  \end{equation*}
  Let $\Omega' \subset \mfkS$ be an open subset.
  We have that $\widehat \nabla (\tau)$ is injective on $L^{2}(\Omega')$
  for $\tau \neq 0$ and $\widehat \nabla (\tau)$ satisfies the $L^{2}$--unique
  continuation on $L^{2}(\Omega')$
  property for all $\tau$. The kernel of $\widehat \nabla (0)$ on
  $L^{2}(\Omega')$ consists of locally constant functions.
\end{proposition}

\begin{proof}
  The operator $\nabla'$ acts only in the vertical directions, and hence
  $\widehat \nabla'(\tau) = \nabla'$ (it is independent of $\tau$).
  On the other hand, $\widehat{(\pa_{t})}(\tau) = \imath \tau$.
  The formula for $\widehat \nabla (\tau)$ follows then from these formulas
  and from Equation \eqref{eq.prod.nabla}.
  For the rest of the proposition, it suffices to assume that $\mfkS$ is connected.
  Let $f \in \CI(\mfkS)$ be such that $\widehat \nabla (\tau) f = 0$.
  This is equivalent to $\nabla ' f = 0$ and $\imath \tau f = 0$. If
  $\tau \neq 0$, then, of course, $f = 0$, and hence $\widehat \nabla (\tau)$,
  besides being injective, satisfies also the $L^{2}$--unique continuation property on
  $\Omega'$. Let us assume then that $\tau = 0$. Then $f$ is constant
  on $\Omega'$. If, moreover, $f$ vanishes on some non-empty subset of some
  connected component $\Omega_{0}'$ of $\Omega'$, then $f = 0$ in $\Omega_{0}'$.
  (This is the $L^{2}$--unique continuation property of $\nabla'$.)
\end{proof}

This result and its proof extend to the $\Def$ operator.

\begin{proposition}\label{prop.uc.Def}
  Let $M = \mfkS \times \RR$ and let $\Def'$ be the vertical deformation
  operator and $d'$ be the vertical exterior derivative.
  In terms of the decomposition of Equation \eqref{eq.decomp.TM},
  the indicial operator $\widehat \Def (\tau)$ (Definition
  \ref{def.indicial.fam}) of the deformation operator $\Def$ on $M$ is given by
  \begin{equation}\label{eq.Def.prod.tau}
    \widehat{\Def} (\tau)\seq
    \left( \begin{array}{cc}
      \Def' & 0\\
      \frac{\imath}{\sqrt{2}} \sharp \tau & \frac1{\sqrt{2}}d'\\
      0 & \imath \tau
    \end{array}\right) \ : \
    \begin{array}{c}
      \CI(\mfkS; T\mfkS)\\
      \oplus\\
      \CI(\mfkS)
    \end{array} \ \rightarrow \ \begin{array}{c}
      \CI(\mfkS; T^{*\otimes 2}\mfkS)\\
      \oplus\\
      \CI(\mfkS; T^{*}\mfkS)\\
      \oplus\\
      \CI(\mfkS) \,.
    \end{array}
  \end{equation}
  Let $\Omega' \subset \mfkS$ be an open subset. We have that
  the operator $\widehat \Def(\tau)$ is injective on
  $L^{2}(\Omega'; TM) \simeq L^{2}(\Omega'; T\mfkS) \oplus L^{2}(\Omega')$
  for $\tau \neq 0$. For $\tau = 0$ we have
  \begin{equation*}
    \ker \widehat \Def (0) \seq \{ (\bsu, c) \in
    \CI(\Omega'; T\mfkS) \oplus \CI(\Omega') : \Def' \bsu = 0\,,
    \ d' c = 0 \}\,.
  \end{equation*}
  In particular, $\widehat \Def(\tau)$ satisfies the $L^{2}$--unique
  continuation property on $\Omega'$ for all $\tau \in \RR$.
\end{proposition}

\begin{proof}
  The proof of the formula for $\widehat \Def(\tau)$ is similar to the proof
  of the formula for $\widehat \nabla (\tau)$ obtained in the proof of
  Proposition \ref{prop.uc.nabla}.
  Indeed, we notice first that
  the operator $\Def'$ acts only in the vertical directions, and hence
  $\widehat \Def'(\tau) = \Def'$ (it is again independent of $\tau$).
  This formula, together with the formula $\widehat{(\pa_{t})}(\tau) = \imath \tau$
  already discussed (see the proof of Proposition \ref{prop.uc.nabla}) and Equation
  \eqref{eq.Def.prod} then yield the claimed formula for  $\widehat \Def (\tau)$.

  Let $U := \cvector{\bsu}{p} \in \CI(\Omega'; TM)\oplus \CI(\Omega')$ be such
  that $\widehat \Def (\tau) U  = 0$. Our assumption implies, in particular, that
  $\imath \tau p = 0$. Let us assume first that $\tau \neq 0$. Then $p = 0$. The
  second equation then becomes $\sharp \bsu = 0$, and hence $\bsu = 0$ also.
  Thus $\widehat \Def (\tau)$ is injective for $\tau \neq 0$.
  In particular, it satisfies the $L^{2}$-unique continuation property,
  as well.

  It remains to prove that $\widehat \Def (0)$ satisfies the $L^{2}$-unique
  continuation property on $\Omega'$ and to identify its kernel. Again, for the rest of the
  proof, it suffices to assume that $\Omega'$ is connected. Let
  $U = \cvector{\bsu}{p}$, as before. The relation $\widehat \Def (0)U = 0$ is equivalent
  to $\Def' \bsu = 0$ (from the first equation) and $d' p = 0$ (from the second
  equation, the third one being automatically satisfied). This gives the stated
  form for the kernel of $\widehat \Def (0)$. Therefore, if
  $\bsu$ and $p$ vanish on some non-empty open subsets of $\Omega'$,
  then $\bsu = 0$, because $\Def'$ satisfies the $L^{2}$-unique continuation
  property on $\Omega'$, and $p = 0$, because $d'$ satisfies the $L^{2}$-unique
  continuation property on $\Omega'$. Thus $\widehat \Def (0)$ satisfies the
  $L^{2}$-unique continuation property on $\Omega'$ and this completes the proof.
\end{proof}

\subsection{Green formulas and energy estimates for indicial families}
\label{ssec.Green2}
In view of further applications to manifolds with cylindrical ends,
we want to extend the Green and representation formulas and the energy
estimates of Section \ref{sec.Green} (especially, Subsection \ref{ssec.Green1})
to the case of the indicial family $\widehat \bsXi$ of $\bsXi := \bsXi_{V, V_{0}}$
(see Definition \ref{def.indicial.fam}).
{\cn Thus, in this subsection, we continue to consider the cylinder
$M = \mfkS \times \RR$, where $\mfkS$ is a closed manifold.} The results of this section
will be used for $\mfkS$ not necessarily connected, so we \emph{\cn will not
assume that $\mfkS$ is connected,} although the general case can immediately be reduced
to the connected case. Also, we let $\newtilde V (x') = \In V(x', t) = V(x', t) \ge 0$ and
$\newtilde V_{0} (x')= \In V_{0}(x', t)  = V_{0}(x', t) \ge 0$ be smooth, translation invariant
functions (this notation is justified by Equation \eqref{eq.just.new.notation}). Again, this
choice of $V$ and $V_{0}$ satisfies all our previous assumptions.
(Note that in some of our previous papers, the notation $\tilde V$ and $\tilde V_{0}$
was used instead of $\newtilde V$ and $\newtilde V_{0}$. However, in this paper,
the notation $\tilde u$ means $\tilde u(x) := u(-x)$.)

Recall the Fourier transform (in the variable $t \in \RR$)
$\maF : L^{2}(\mfkS \times \RR) \to L^{2}(\mfkS \times \RR)$,
$\hat u = \maF u$, which satisfies
\begin{equation}\label{eq.F.sp}
  2 \pi (u, v)_{M} \seq (\maF u, \maF v)_{M} \seq \int_{\RR}
  (\hat u(\tau), \hat v(\tau))_{\mfkS}\, d\tau \,.
\end{equation}
Note that the first equality in Equation \eqref{eq.F.sp} is a consequence of
the Parceval theorem for the Fourier transform, and we use the simplified
notation $\hat u(\tau)$ instead of $\hat u(x,\tau)$ for $(x,\tau )\in \mfkS\times \RR =M$.
We shall use the notations $\hat u(\tau)$ and $\maF u(\tau)$ interchangeably.

For any $P \in \ePS{m}(M; E)^{\RR}$ we have
\begin{equation}\label{eq.def.hatP}
  \widehat{P u}(\tau) \ede (\maF P u)(\tau) \seq \widehat{P}(\tau) \hat u(\tau)\,.
\end{equation}
This gives
\begin{equation}\label{eq.F.P}
  2 \pi(P u, v)_{M} \seq (\maF P u, \maF v)_{M} \seq
  \int_{\RR} (\widehat{P}(\tau) \hat u(\tau), \hat v(\tau))_{\mfkS} \, d\tau \,.
\end{equation}

For $\phi \in L^{\infty}(\RR)$, we shall need the Fourier multiplier $\MM_{\phi}$
defined by $\maF \MM_{\phi} = \phi \maF$, or, explicitely,
\begin{equation}\label{eq.def.Fm}
  [\maF \MM_{\phi} u](\tau) \ede \phi(\tau) \hat u(\tau)\,.
\end{equation}
All of the above definitions and relations extend to vector bundles that are
lifted from vector bundles on $\mfkS$ (and hence they come with a natural action of
$\RR$ by translations). We then have that the Fourier multiplier operators commute
with the operators in $\ePSsus{m}{M}{E, F} := \ePS{m}(M; E, F)^{\RR}$ (recall that
$M = \mfkS \times \RR$).

Also, in this subsection, we let $\Omega' \subset \mfkS$ be a smooth open subset with
boundary $\Gamma':=\pa \Omega'$ (a smooth manifold). Our results are formulated in
general, but there is no loss of generality to assume that $\Omega'$ is connected
(that is, a smooth domain).
Also, recall that, in this section,
$\newtilde V (x') = V(x', t)$ and $\newtilde V_{0} (x') = V_{0}(x', t)$, as motivated
by Equation \eqref{eq.just.new.notation}. In particular, $V = \newtilde V$ and $V_{0} =
\newtilde V_{0}$ are translation invariant. The main equation of Notation \ref{not.def.UWB}
becomes.
\begin{equation}\label{eq.def.UWBtau}
  \begin{gathered}
    U \ede (\bsu \ \ p)^{\top }\,, \ W \ede (\bsw \ \ q)^{\top }
    \in H^1(\Omega'; TM) \oplus L^{2}(\Omega')\,,\\
    \qquad \mathfrak v \ede ({\newtilde V} \bsu, \bsw)_{\Omega'}
    - ({\newtilde V_{0}}p, q)_{\Omega'}\,, \quad \mbox{and}\\
    B_{\Omega', \tau}(U, W) \ede 2(\widehat \Def(\tau) \bsu,
    \widehat \Def(\tau) \bsw)_{\Omega'}
    + (\widehat \nabla^{*}(\tau) \bsu, q)_{\Omega'}
    + (p, \widehat\nabla^{*}(\tau)\bsw)_{\Omega'} + \mathfrak v \,.
  \end{gathered}
\end{equation}
Also, in this subsection, we let $(\ , \ )' := (\ , \ )_{\pa \Omega'},$
the scalar product on the boundary $\Gamma' = \pa \Omega'$ of $\Omega'$.

Recall that $M = \mfkS \times \RR$ and $\Omega = \Omega' \times \RR$ and all
the operators of the form $\widehat{P}(\tau)$ are obtained from the Fourier transform
with respect to the action of $\RR$ by translation, see, for example,
Equation \eqref{eq.def.hatP}. Proposition \ref{prop.Green} becomes:

\begin{proposition} \label{prop2.Green.tau}
  Let $U^{\top} := (\bsu \ \ p)$ and  $W^{\top}
  := (\bsw \ \ q)$ belong to $H^{2}(\Omega'; TM) \oplus H^{1}(\Omega')$
  and let $(\ , \ )' := (\ , \ )_{\pa \Omega'}$. Then we have the following formulas:
  \begin{enumerate}[\rm (i)]
  \item 
    $(\widehat \bsXi(\tau) U, W \big )_{\Omega'}
    \seq B_{\Omega', \tau} (U, W) + \big(\widehat \bop(\tau) U, \bsw\big)'.$
  \item
    $\big ( \widehat \bsXi(\tau) U, W \big )_{\Omega'}
    - \big (U, \widehat \bsXi(\tau)  W \big )_{\Omega'}
    \seq (\widehat \bop (\tau)  U, \bsw)'
    - \big(\bsu , \widehat \bop (\tau) W\big)'\,.$
  \end{enumerate}
  If $\pa \Omega'$ is empty, then we drop the the inner products $(\ , \ )'$
  on $\pa \Omega'$.
\end{proposition}

\begin{proof}
  The proof is similar to that of Proposition \ref{prop.Green} (sometimes
  it follows from that proposition using the Fourier transform).s
  Let then $U_{0} := (\, \bsu_{0} \ \ p_{0} \, )^{\top},
  W_{0} := (\, \bsw_{0} \ \ q_{0} \, )^{\top} \in
  H^{2}(\Omega' \times \RR; TM) \oplus H^{1}(\Omega' \times \RR)$,
  as in Notation \eqref{not.def.UWB} with $\Omega$ replaced with
  $\Omega' \times \RR$. The first part of Proposition \ref{prop.Green} and
  Equation \eqref{eq.F.P} give then that
  \begin{multline*}
    \frac1{2\pi}\int_{\RR} ( \widehat \bsXi(\tau) \widehat U_{0}(\tau),
    \widehat W_{0}(\tau) \big )_{\Omega'} \, d\tau
    \seq (\bsXi U_{0}, W_{0})_{\Omega} \seq B_{\Omega} (U_{0}, W_{0})
    + (\bop U_{0}, \bsw_{0})'\\
    \seq \frac1{2\pi} \int_{\RR} \big[ B_{\Omega', \tau}
    (\widehat U_{0}(\tau), \widehat W_{0}(\tau))
    + (\widehat \bop(\tau) \widehat U_{0}(\tau), \hat \bsw_{0} (\tau))'\big] \, d\tau \,.
  \end{multline*}
  By replacing $U_{0}$ with its image $\MM_{\phi}U_{0}$ through the
  Fourier multiplier $\MM_{\phi}$, we obtain
  \begin{multline*}
    \int_{\RR} \phi(\tau) (\widehat \bsXi(\tau) \widehat U_{0}(\tau),
    \widehat W_{0}(\tau) \big )_{\Omega'} \, d\tau\\
    \seq \int_{\RR} \phi(\tau) \big[ B_{\Omega', \tau} (\widehat U_{0}(\tau), \widehat W_{0}(\tau))
    + (\widehat \bop(\tau) \widehat U_{0}(\tau), \hat \bsw_{0} (\tau))'\big] \, d\tau\,.
  \end{multline*}
  Since $\phi \in L^{\infty}(\RR)$ is arbitrary, we obtain, for
  every $\tau \in \RR$, that
  \begin{equation*}
    (\widehat \bsXi(\tau) \widehat U_{0}(\tau), \widehat W_{0}(\tau) \big )_{\Omega'}
    \seq B_{\Omega', \tau} (\widehat U_{0}(\tau), \widehat W_{0}(\tau))
    + (\widehat \bop(\tau) \widehat U_{0}(\tau), \hat \bsw_{0} (\tau))'\,.
  \end{equation*}
  The point (i) (for any fixed $\tau$) then follows by choosing $U_{0}$ and $W_{0}$
  such that $\widehat U_{0}(\tau) = U$ and $\widehat W_{0}(\tau) = W$.

  The second point follows immediately from (i) using that $B_{\Omega', \tau}$
  is sesquilinear (following the same method of proof as that for Proposition
  \ref{prop.Green}).
\end{proof}

It follows from the definition of indicial operators (Definition
\ref{def.indicial.fam}) that
\begin{equation}\label{eq.def.Xitau}
  \widehat \bsXi(\tau) \seq
  \left( \begin{array}{cc}
    2\widehat \Defstar (\tau) \widehat \Def (\tau) + \newtilde V
    & \widehat \nabla  (\tau)\\
    \widehat \nabla^{*}  (\tau)  &  - \newtilde V_{0}
  \end{array}\right)\,,
\end{equation}
where, we recall, we are identifying $\newtilde V = \widehat V (\tau)$ and
$\newtilde V_{0} = \widehat V_{0} (\tau)$.
(The formula \eqref{eq.def.Xitau} is valid for any manifold with straight cylindrical ends,
not just for our manifold $M := \mfkS  \times \RR$.) The last proposition gives then the
following lemma analogous to Corollary \ref{cor.e.est}. Recall that, we are
assuming that $V, V_{0} \ge 0$. Also, recall that $\Omega' \subset \mfkS$ is a smooth
domain.

\begin{lemma}\label{lemma2.e.est}
  Let $\tau \in \RR$, let $\newtilde V \in \CI(\mfkS \times \RR; \End(TM))^{\RR} \simeq
  \CI(\mfkS; TM)$ and $\newtilde V_{0} \in \CI(\mfkS \times \RR)^{\RR} \simeq
  \CI(\mfkS)$ be non-negative (i.e., $\newtilde V, \newtilde V_{0} \ge 0$), and let
  \begin{equation*}
    U = \cvector{\bsu}{p} \in H^{2}(\Omega'; TM) \oplus H^{1}(\Omega')
  \end{equation*}
  satisfy $\widehat \bsXi (\tau) U = 0$ in $\Omega'$
  and $(\widehat \bop (\tau) U, \bsu)' = 0$. Then the following
  vanishing relations hold in $\Omega'$:
  \begin{equation*}
    \widehat \Def (\tau) \bsu \seq 0\,,\quad \newtilde V \bsu \seq 0\,,\quad
    \widehat \nabla^{*}(\tau) \bsu \seq 0\,,\quad \widehat \nabla (\tau) p \seq 0
    \,,\quad \mbox{and} \quad \newtilde V_{0}p \seq 0\,.
  \end{equation*}
\end{lemma}

\begin{proof}
    Let us take $W = U' := \cvector{\bsu}{-p}$ in the formula
    \[\big (\widehat \bsXi (\tau) U, W \big )_{\Omega'} \seq B_{\Omega', \tau}
    (U, W) + (\widehat \bop (\tau) U, \bsw)'\] of Proposition \ref{prop2.Green.tau}.
    Let $\operatorname{Re}(z)$ denote the real part of $z \in \CC$.
    Then $\operatorname{Re}\big[- (\widehat \nabla^{*} (\tau) \bsu, p)_{\Omega'}
    + (p, \widehat \nabla^{*} (\tau) \bsu)_{\Omega'}\big] = 0$.
    Let $\mathfrak v := (\newtilde V \bsu, \bsu)_{\Omega'} + (\newtilde V_{0}p, p)_{\Omega'}$,
    as before. Together with the definition of $B_{\Omega', \tau}$ in Equation
    \eqref{eq.def.UWBtau}, this gives
    \begin{align*}
      0 & \seq \operatorname{Re} \big[\big (\widehat \bsXi (\tau) U,
      U' \big )_{\Omega'}
      - (\widehat \bop (\tau) U, \bsu)'\big]
      \seq \operatorname{Re} \big[  B_{\Omega', \tau} (U, U')\big]\\
      &
      \seq \operatorname{Re} \big[2(\widehat \Def (\tau) \bsu,
      \widehat \Def (\tau) \bsu)_{\Omega'}
      - (\widehat \nabla^{*} (\tau) \bsu, p)_{\Omega'}
      + (p, \widehat \nabla^{*} (\tau) \bsu)_{\Omega'} + \mathfrak v \big]
      \\
      &
      \seq 2(\widehat \Def (\tau) \bsu, \widehat \Def (\tau) \bsu)_{\Omega'}
      + \mathfrak v
      \\
      &
      \seq 2(\widehat \Def (\tau) \bsu, \widehat \Def (\tau) \bsu)_{\Omega'}
      + (\newtilde V \bsu, \bsu)_{\Omega'}
      + (\newtilde V_{0}p, p)_{\Omega'}\,.
    \end{align*}
    All three terms in the last sum are non-negative, so each of them equals zero.
    Therefore $\widehat \Def (\tau) \bsu = 0$,
    $\newtilde V \bsu = 0$, and $\newtilde V_{0}p=0$. We also have
    \begin{equation*}
      0 \seq \widehat \bsXi (\tau) U \seq
      \cvector{2 \widehat \Defstar (\tau) \widehat \Def (\tau) \bsu
      + \newtilde V \bsu + \widehat \nabla (\tau) p}{\widehat \nabla^{*} (\tau) \bsu - \newtilde V_{0}p}
      \seq \cvector{\widehat \nabla (\tau) p}{\widehat \nabla^{*} (\tau) u}\,.
    \end{equation*}
    Hence $\widehat \nabla^{*} (\tau) \bsu =0$ and $\widehat \nabla (\tau) p = 0$, as
    claimed. This completes the proof.
\end{proof}

Recall that $V \succ 0$ on $\Omega$ means that $V \ge 0$ and that
$V$ is (strictly) positive definite at at least one point in every connected
component of $\Omega$ (see Definition \ref{def.succ}).
We now take $\tau = 0$ in the above lemma. The $L^{2}$-unique continuation
properties of $\widehat \Def (0)$ and $\widehat \nabla (0)$ then give the
following result.

\begin{corollary}\label{cor2.e.est3}
  Let $\newtilde V, \newtilde V_{0} \ge 0$
  and let $U = (\bsu \ \ p)^{\top} \in H^{2}(\Omega'; TM)
  \oplus H^{1}(\Omega')$ satisfy $\widehat \bsXi (0) U = 0$ in $\Omega'$
  and $(\widehat \bop (0) U, \bsu)_{ \Gamma'} = 0$
  \lpar as in Lemma \ref{lemma2.e.est}, but for $\tau = 0$\rpar.
  \begin{enumerate}[\rm (i)]
    \item If $\newtilde V_{0} \succ 0$ on $\Omega '$ \lpar Definition \ref{def.succ}\rpar, then
    $p = 0$ on $\Omega '$.

    \item Similarly, if one of the following conditions is satisfied:
    \begin{enumerate}[\rm (a)]
      \item $\bsu = 0$ on $\Gamma '$ or
      \item $\newtilde V \succ 0$ on $\Omega '$,
    \end{enumerate}
    then $\bsu = 0$ in $\Omega '$.
  \end{enumerate}
\end{corollary}

\begin{proof}
  Lemma \ref{lemma2.e.est} gives $\widehat \nabla (0) p = 0$
  in $\Omega'$. The assumption that $\newtilde V_{0} \succ 0$
  in $\Omega$ together with Proposition \ref{prop.uc.nabla}
  (stating the $L^{2}$-unique continuation property of $\widehat \nabla (0)$
  in $\Omega'$) then give $p = 0$.
  Similarly, Lemma \ref{lemma2.e.est} gives $\widehat \Def (0) \bsu = 0$.
  Our assumptions allow us to use Proposition \ref{prop.uc.Def}
  (the $L^{2}$-unique continuation property of $\widehat \Def (0)$)
  to conclude that $\bsu = 0$. This completes the proof.
\end{proof}

As before, we obtain the following corollaries.

\begin{corollary}\label{cor2.e.est3.new}
  Let $\newtilde V, \newtilde V_{0} \ge 0$
  and let $U = (\bsu \ \ p)^{\top} \in H^{2}(\mfkS; TM)
  \oplus H^{1}(\mfkS)$ satisfy $\widehat \bsXi (0) U = 0$ in $\mfkS$.
  \begin{enumerate}[\rm (i)]
    \item If $\newtilde V_{0} \succ 0$ on $\mfkS$, then $p = 0$ on $\mfkS$.

    \item Similarly, if $\newtilde V \succ 0$ on $\mfkS$,
    then $\bsu = 0$ in $\mfkS$.
  \end{enumerate}
\end{corollary}

The case $\tau \neq 0$ is very similar, in fact, even easier than the case
$\tau = 0$, because of the injectivity of the relevant operators for $\tau \neq 0$.

\begin{corollary}\label{cor2.e.est2}
  Let $\newtilde V, \newtilde V_{0} \ge 0$, $\tau \neq 0$, and $U^{\top} = (\bsu \ \ p)
  \in H^{2}(\Omega'; TM)
  \oplus H^{1}(\Omega')$ satisfy $\widehat \bsXi (\tau) U = 0$ in $\Omega'$ and
  $(\widehat \bop (\tau) U, \bsu)'  = 0$
  \lpar the same assumptions as in Lemma \ref{lemma2.e.est}, except that
  $\tau \neq 0$ in this corollary\rpar. Then $U = 0$.
\end{corollary}

\begin{proof}
  Lemma \ref{lemma2.e.est} gives $\widehat \nabla (\tau) p = 0$
  and $\widehat \Def (\tau) \bsu = 0$. Because both operators $\widehat \nabla (\tau )$
  and $\widehat \Def (\tau)$ are injective for $\tau \neq 0$ (by Propositions
  \ref{prop.uc.nabla} and \ref{prop.uc.Def}), we obtain $\bsu=0$ and $p=0$.
\end{proof}

In the last Corollary, if $\Omega' = \mfkS$, then the condition
$(\widehat \bop (\tau) U, \bsu)'  = 0$ is void. It is automatically
satisfied, so, in particular, it is unnecessary. This yields the following
result.

\begin{corollary}\label{cor2.e.est2.new}
  Let $\newtilde V, \newtilde V_{0} \ge 0$, $\tau \neq 0$, and $U^{\top} = (\bsu \ \ p)
  \in H^{2}(\mfkS; TM) \oplus H^{1}(\mfkS)$ satisfy $\widehat \bsXi (\tau) U = 0$.
  Then $U = 0$.
\end{corollary}

\subsection{The invertibility of $\bsXi$ on a cylinder}
Recall from Theorem \ref{thm.ADN.Fredholm} that the Fredholm property
of an operator depends on the invertibility of its limit operator, which
acts on the cylinder associated to $M$. The limit operator of $\bsXi_{V, V_{0}}$
is $\bsXi_{\newtilde V, \newtilde V_{0}}$. It is the purpose of this subsection
to establish the invertibility of $\bsXi_{\newtilde V, \newtilde V_{0}}$ acting
on the associated cylinder under the assumption that $\newtilde V \succ 0$ and
$\newtilde V_{0} \succ 0$ and $\newtilde V_{0} \succ 0$ on $\mfkS$.

{\cn We thus continue to assume, in this subsection, that $M$ is the cylinder
$M = \mfkS \times \RR$ with $\mfkS$ a closed manifold and with a product metric.
We also assume that $V(x', t) = \newtilde V(x')$ and that $V_{0}(x', t) =
\newtilde V_{0}(x')$, and hence $V$ and $V_{0}$ are thus translation invariant.
We also assume that they are both $\succ 0$ on $M$ (see Definition \ref{def.succ}).}
This condition is equivalent to $\newtilde V \succ 0$ and $\newtilde V_{0} \succ 0$
on $\mfkS$.

Let us notice that, since $M := \mfkS \times \RR$ is endowed with the product
metric, this metric is translation invariant, and hence $\bsXi$ is also
translation invariant. We can then define its Fourier transform (or indicial
family) $\widehat \bsXi(\tau)$ as in Section \ref{ssec.FT}, especially
Definition \ref{def.explicit.FT} (and the discussion surrounding it). In particular,
we choose the definition that avoids duplications, see Remark also \ref{rem.indicial2}.

\begin{proposition}\label{prop2.invert3}
  Let us assume that $M = \mfkS \times \RR$ is a cylinder
  and that $\newtilde V \succ 0$ and $\newtilde V_{0}\succ 0$ on $\mfkS$.
  Let $\bss = \bst = (1, 0)$. Then $\widehat \bsXi(0) :=
  \widehat \bsXi_{V,V_{0}}(0)$ is invertible and
  $\widehat \bsXi(0)^{-1} \in \Psi^{[- \bss - \bst]}(\mfkS; TM \oplus \CC)$.
\end{proposition}

\begin{proof}
  The generalized Stokes operator $\widehat \bsXi(0)$ is
  $(\bss, \bst)$--Douglis-Nirenberg elliptic because $\bsXi$ is
  $(\bss, \bst)$--Douglis-Nirenberg elliptic (Proposition \ref{prop.ADN.elliptic})
  and because ellipticity is preserved by taking indicial families.
  Because $\mfkS$ is compact, Theorem \ref{thm.ADN.Fredholm} then
  gives that $\widehat \bsXi(0)$ is Fredholm as an operator
  \begin{equation*}
    \widehat \bsXi(0) : H^{1}(\mfkS; TM) \oplus L^{2}(\mfkS) \to H^{-1}(\mfkS; TM)
    \oplus L^{2}(\mfkS)\,.
  \end{equation*}
  The second Green formula of
  Proposition \ref{prop2.Green.tau} shows that $\widehat \bsXi(0)$ is self-adjoint,
  and hence it has index zero. Therefore, in order to show that $\widehat \bsXi(0)$ induces
  a bijection between the above spaces, it suffices to show that it is injective.
  Let then $U = \cvector{\bsu}{p} = (\, \bsu\ \ p\,)^{\top} \in L^{2}(\mfkS; TM \oplus \CC)$
  be such that $\widehat \bsXi(0) U = 0$. The ellipticity of $\widehat \bsXi(0)$
  guarantees then that $U \in H^{\infty}(\mfkS; TM \oplus \CC)$, by elliptic regularity
  (see Theorem \ref{thm.regularity} or \cite[Theorem 15.3.30]{KMNW-2025} or \cite[Proposition 6.5]{KNW-22}).
  Assumption \ref{assumpt.VV0} allow us to use Corollary \ref{cor2.e.est3.new}
  to conclude that $U = 0$. Therefore $\widehat \bsXi(0)$ is invertible as an operator
  between the indicated spaces. The property that its inverse belongs to the space
  $\Psi^{[- \bss - \bst]}(\mfkS; TM \oplus \CC)$ follows from
  the work of Beals and others (see Theorem 7.4 in \cite{KNW-22} for more details).
\end{proof}

The case $\tau \neq 0$ is completely similar, but only requires
the weaker assumptions $\newtilde V := \In (V) \ge 0$ and
$\newtilde V_{0} := \In(V_{0}) \ge 0$.

\begin{proposition}\label{prop2.invert2}
  Let us assume that $M = \mfkS \times \RR$ is a cylinder, that $\tau \neq 0$,
  and that $\newtilde V_{0}, \newtilde V \ge 0$. Let $\bss = \bst = (1, 0)$. Then
  $\widehat \bsXi(\tau)$ is invertible and $\widehat \bsXi(\tau)^{-1}
  \in \Psi^{[- \bss - \bst]}(\mfkS; TM \oplus \CC)$.
\end{proposition}

\begin{proof}
  The proof is almost the same as that of Proposition \ref{prop2.invert3}
  (but easier). The only difference is that, instead of Corollary \ref{cor2.e.est3.new},
  we use Corollary \ref{cor2.e.est2.new}, which only requires that
  $\newtilde V$ and $\newtilde V_{0}$ be non-negative.
\end{proof}

Standard properties of pseudodifferential operators yield the following result,
which will, however, not be used in this paper. It does specify, however,
in what sense we are talking about the invertibility of our operators
$\widehat \bsXi(\tau)$.

\begin{remark} 
 Let us assume that $M = \mfkS \times \RR$ with  $\mfkS$ compact and
 that $\newtilde V$ and $\newtilde V_{0}$ are non-negative.
 Then, for any $\tau, r \in \RR$, $\tau \neq 0$, we have an isomorphism
  \begin{equation*}
    \widehat \bsXi(\tau) \ede \widehat \bsXi_{V, V_{0}}(\tau) : H^{r+1}(\mfkS; TM)
    \oplus H^{r}(\mfkS) \to H^{r-1}(\mfkS; TM) \oplus H^{r}(\mfkS)\,.
  \end{equation*}
  The result is true also for $\tau = 0$ if we further assume
  that $M$, $V$, and $V_{0}$ satisfy the non-vanishing Assumption
  \ref{assumpt.VV0}. Under the same conditions, we have
  $\widehat \bsXi(\tau)^{-1}
  \in \Psi^{[- \bss - \bst]}(\mfkS; TM \oplus \CC)$.
\end{remark}

The last two propositions then give the invertibility of $\bsXi$ on cylinders
(provided that the positivity at infinity assumption \ref{assumpt.VV0weak} is
satisfied).

\begin{theorem}\label{thm2.invert4}
  Let us assume that $M = \mfkS \times \RR$, that $V$ and $V_{0}$ are translation invariant
  and that $\newtilde V \succ 0$ and $\newtilde V_{0} \succ 0$ on $\mfkS$. Then the operator
  \begin{equation*}
    \bsXi \ede \bsXi_{V, V_{0}} :
    H^{1}(M; TM) \oplus L^{2}(M) \to H^{-1}(M; TM) \oplus L^{2}(M)
  \end{equation*}
  is invertible. If $\bss = \bst = (1, 0)$, its inverse satisfies $\bsXi^{-1}
  \in \ePS{[- \bss - \bst]}(M; TM \oplus \CC)$
  and maps $H^{r+1}(M; TM) \oplus H^{r}(M) \to H^{r-1}(M; TM)
  \oplus H^{r}(M)$ isomorphically for all $r \in \RR$.
\end{theorem}

\begin{proof}
  We know from Propositions \ref{prop2.invert3} and \ref{prop2.invert2} that
  $\widehat \bsXi (\tau)$ is invertible for all $\tau \in \RR$. It is well-known
  from the works of Melrose, Schulze and others (see also Theorem 5.11 in
  \cite{KohrNistor-Stokes}) that then $\bsXi$ is invertible as stated. Once
  we have proved that $\bsXi$ is invertible, the fact that $\bsXi^{-1}
  \in \ePS{[- \bss - \bst]}(M; TM \oplus \CC)$ follows from the work of Beals
  and others (see Theorem 7.4 in \cite{KNW-22} for further details).
  The claimed isomorphisms are obtained from the mapping properties of
  translation invariant operators.
\end{proof}

\section{\cn Invertibility of $\bsXi$, layer potentials, and jump relations on
manifolds with cylindrical ends}
\label{sec.8}

We now use the results of the previous sections and prove the invertibility
of several Stokes-type operators (under suitable positivity conditions on the
potentials $V$ and $V_{0}$) and use these results to define our layer potentials.
We then prove the main jump formulas for these layer potentials.

\emph{We assume from now on that $M = \canDec$ is a manifold with straight
cylindrical ends, so $M'$ is a closed manifold (i.e., compact, boundaryless).}
The case of closed manifolds was treated in \cite{KNW-2025} and will be briefly
reminded in Subsection \ref{ssec.ToAdd}. Also, from now on, we shall assume
that $V$ and $V_{0}$ are translation invariant at infinity and that they are
are smooth. (Some more precise assumption will be formulated later, see Assumptions
\ref{assumpt.VV0weak} and \ref{assumpt.VV0}.)

\subsection{The Fredholm and invertibility properties of $\bsXi$}
Using the results of the previous subsection, we now recall the proof of the
Fredholm property, respectively, of the invertibility property of
$\bsXi := \bsXi_{V, V_{0}}$ under suitable positivity assumptions
on $V$ and $V_{0}$ (Assumptions \ref{assumpt.VV0weak} and,
respectively, \ref{assumpt.VV0}). These results were first
established in \cite{KohrNistor-Stokes}.
The proof relies on the ``Mitrea-Taylor'' trick \cite{M-T} and on the results
and methods from \cite{KNW-22}. To establish the Fredholm property,
we shall need the following assumption
(slightly stronger than Assumption ).

\begin{assumption}\label{assumpt.VV0weak}
  Let $M = \canDec$ be a smooth, connected manifold with straight cylindrical ends with
  $M' := \pa M_{0}$ closed, non-empty and let $V \in \CI_{\inv}(M; TM)$ and
  $V_{0} \in \CI_{\inv}(M)$ be our potentials. We shall say that $V$ and $V_{0}$ satisfy
  the \emph{positivity assumption at infinity on $M$} if
  \begin{equation*}
    \newtilde V \ede \In V \succ 0 \quad \mbox{and} \quad
    \newtilde V_{0} \ede \In V_{0} \succ 0 \quad \mbox{on } M\,.
  \end{equation*}
\end{assumption}

Notice that we are not assuming that, globally, $V$ and $V_{0}$ are
non-negative. This will be done in Assumption \ref{assumpt.VV0}, which is, hence
slightly stronger than Assumption \ref{assumpt.VV0weak}.
We also recall that, if $a$ is a multiplication operator, then we use also
the notation $\newtilde a$ for $\In (a)$ (see Equation
\eqref{eq.just.new.notation}
for a justification of this notation). We also recall (see Definition \ref{def.succ})
that the explicit forms of the last two assumptions amount to the following:
\begin{enumerate}[\rm (i)]
  \item $\newtilde V \ge 0$ and, in every connected component of $M'$, there is
  a point $x$ (and hence a whole open set containing it) where $\newtilde V(x)$
  is invertible (i.e., $V_{0}(x) > 0$ in $M_{n}(\CC)$).
  \item $\newtilde V_{0} \ge 0$ and, in every connected component of $M'$, there
  is a point
  $x$ (and hence a whole open set containing it) where $\newtilde V_{0}(x)$
  is invertible (i.e., $V_{0}(x) \in (0, \infty)$).
\end{enumerate}

Recall that by the statement ``$\phi \not \equiv 0$
on $A$,'' we mean that
there exists $a$ in the domain of $\phi$ such that $\phi(a) \not = 0$. To
negate this statement, we shall write ``$\phi = 0$ in $A$.''
The role of this assumption (Assumption \ref{assumpt.VV0weak}) is to be able to use
the results of the last subsection (which relied on the second part of Corollary
\ref{cor2.e.est3} to obtain the vanishing of $p$ and $\bsu$ in $\Omega'$, if
$\Omega = M$ we used Corollary \ref{cor2.e.est3} instead). The assumption that
$M$ is connected is just to simplify some of our statements. For instance, it
ensures that $V_{0} \not \equiv 0$ on all {\it connected components} of $M$
(since there is only one such connected component!). The general case follows
from the case $M$ connected easily. By contrast, we cannot assume $M'$ to be connected,
as it may be disconnected even if $M$ is connected.

Let us study first the Fredholm property of $\bsXi$ on our manifold $M = \canDec$
with straight cylindrical ends. Recall that $M' := \pa M_{0}$ is always assumed to
be compact. We first notice that the invertibility of the limit operator $\In \bsXi$
follows from the results of the previous section.

\subsection{The Moore-Penrose pseudoinverse of $\bsXi$}
Let $\mathcal N$ be the kernel of $\bsXi : H^{1} (M; TM) \oplus L^{2}(M)
\to H^{-1} (M; TM) \oplus L^{2}(M)$. The classical elliptic regularity of $\bsXi$ implies that
$\mathcal N$ will consist of {\it smooth}
sections. Assume that $\mathcal N \subset L^{2}(M; TM \oplus \CC)$ and that
it is finite dimensional and let $p_{\maN}$ be the $L^{2}$--orthogonal projection onto $\maN$.
Let $\bss =\bst =(1,0)$ in Definition \ref{def.ADN}.
Assume also that there exists a pseudodifferential operator
\begin{equation}\label{eq.def.psdinv}
  \psdinv \in \Psi_{\cl}^{[- \bss - \bst]}(M; TM \oplus \CC)
  \ede \left(
      \begin{array}{cc}
        \Psi^{-2}(M; TM)   & \Psi^{-1}(M; TM, \CC)\\
        \Psi^{-1}(M; \CC, TM) &  \Psi^{0}(M)
      \end{array}
    \right)
\end{equation}
such that
\begin{equation}\label{eq.rel.pnu}
  \begin{gathered}
    \psdinv (1 - p_{\maN}) \seq (1 - p_{\maN}) \psdinv \seq \psdinv \ \mbox{ and }\\
    \bsXi \psdinv \seq \psdinv \bsXi =  1 - p_{\maN}
  \end{gathered}
\end{equation}
(i.e., we assume that $\bsXi$ is invertible on the orthogonal complement of its kernel $\maN$).
The operator $\psdinv$ is uniquely determined by the relations \eqref{eq.rel.pnu}
and  will be called the {\it Moore-Penrose pseudoinverse} of $\bsXi := \bsXi_{V, V_{0}}$.

We also obtain the following result:

\begin{theorem}\label{thm2.main.Fredholm}
  Let $M$ be a manifold with straight cylindrical ends and assume that
  $\newtilde V \succ 0$ and $\newtilde V_{0} \succ 0$
  \lpar Assumption \ref{assumpt.VV0weak}\rpar. Then the following properties hold.
  \begin{enumerate}[\rm (i)]
    \item The limit operator $\In (\bsXi) \in \ePSsus{m}{M'}{TM, \CC}$
    is invertible \lpar see Lemma \ref{lemma.limit.operator}\rpar.

    \item The generalized Stokes operator $\bsXi := \bsXi_{V, V_{0}}$ is
    Fredholm.

    \item Let $\bss = \bst = (1, 0)$. Consequently, the Moore-Penrose pseudoinverse of
    $\bsXi$ satisfies $\psdinv \in \ePS{[- \bss - \bst]}(M; TM \oplus \CC)$.
  \end{enumerate}
\end{theorem}

\begin{proof}
  Assumption \ref{assumpt.VV0weak} implies that $\In (\bsXi) =
  \bsXi_{\newtilde V, \newtilde V_{0}}$ satisfies the
  assumptions of Theorem \ref{thm2.invert4}, which then yields right away the first point.
  The generalized Stokes operator $\bsXi := \bsXi_{V, V_{0}}$ is elliptic by
  Proposition \ref{prop.ADN.elliptic}. Since $\In (\bsXi)$ is invertible,
  Theorem \ref{thm.ADN.Fredholm} implies that $\bsXi$ is Fredholm. The last point follows
  from the Beals-type result proved in \cite[Theorem 7.4]{KNW-22} (this result can also be
  found in \cite{KMNW-2025}).
\end{proof}

We now use this result to prove the invertibility of $\bsXi$ under the following
assumption.

\begin{assumption}\label{assumpt.VV0}
  We say that $V$ and $V_{0}$ satisfy the
  \emph{positivity assumption on $M$} if they satisfy the
  positivity assumption at infinity on $M$ \lpar Assumption
  \ref{assumpt.VV0weak}\rpar\ and $V \ge 0$
  and $V_{0} \ge 0$.
\end{assumption}

It follows from this formulation that Assumption \ref{assumpt.VV0} implies
Assumption \ref{assumpt.VV0weak}. We can now prove our main invertibility result
({Theorem 5.11} from \cite{KohrNistor-Stokes}).

\begin{theorem}\label{thm2.main.invert}
  Let $M$ be a manifold with straight cylindrical ends and assume that
  $V$ and $V_{0}$ satisfy the positivity assumption {\lpar Assumption \ref{assumpt.VV0}\rpar}.
  Then the generalized Stokes operator $\bsXi := \bsXi_{V, V_{0}}$ is
  invertible. Let $\bss = \bst = (1, 0)$.
  Consequently, its inverse satisfies $\bsXi^{-1}
  \in \ePS{[- \bss - \bst]}(M; TM \oplus \CC)$
  and
  \begin{equation*}
    \bsXi := \bsXi_{V, V_{0}} : H^{r+1}(M; TM) \oplus H^{r}(M) \to H^{r-1}(M; TM)
    \oplus H^{r}(M)
  \end{equation*}
  is an isomorphism for all $r \in \RR$.
\end{theorem}

\begin{proof}
  The induced operator
  \begin{equation*}
    \bsXi \ede \bsXi_{V, V_{0}} :
    H^{1}(M; TM) \oplus L^{2}(M) \to H^{-1}(M; TM) \oplus L^{2}(M)
  \end{equation*}
  is Fredholm by Theorem {thm2.main.Fredholm} just proved. It is also self-adjoint,
  and hence it is of index zero (between the indicated spaces). To prove that $\bsXi$ is
  invertible, it hence suffices to prove that it is injective. We proceed as
  in the proof of the last two propositions, using energy estimates.

  Let $U \in H^{1}(M; TM) \oplus L^{2}(M)$ be such that $\bsXi U = 0$
  on $M$. Then $U \in H^{2}(M; TM) \oplus H^{1}(M)$, by elliptic regularity
  (see Theorem \ref{thm.regularity} or \cite[Theorem 15.3.30]{KMNW-2025} or \cite[Proposition 6.5]{KNW-22}).
  Let us notice that the fact that $V$ and $V_{0}$ satisfy the non-vanishing
  assumption on $M$ {\lpar Assumption \ref{assumpt.VV0}\rpar} and that $M$ is connected implies
  that $V_{0} \not \equiv 0$ and $V \succ 0$ on $M$. The assumption $\bsXi U = 0$
  then allows us to use Corollary \ref{cor.e.est.new} to conclude that $U = 0$.
  Thus $\bsXi$ is invertible between the indicated spaces, and hence also as a
  pseudodifferential operator. The fact that its inverse belongs to the set
  $\ePS{[- \bss - \bst]}(M; TM \oplus \CC)$ follows from a ``Beals-type theorem''
  (Theorem 7.4 in \cite{KNW-22}).
  The last stated isomorphism follows from the mapping properties of operators in
  the spaces $\ePS{[\bss + \bst]}(M; TM)$ and in
  $\ePS{[- \bss - \bst]}(M; TM)$, as in the last
  corollaries. This completes the proof.
\end{proof}

The invertibility of $\bsXi := \bsXi_{V, V_{0}}$ in the case $M$ closed is
obtained when both $V$ and $V_{0}$ do not vanish identically on $M$. Moreover,
the condition that $V$ does not vanish identically on $M$ can be replaced with
the condition that $M$ does not have non-vanishing Killing vector fields.
These results were proved in \cite[Theorem 6.1(4)]{KNW-2025}.
In relation to this result, the reader might have noticed that we are not considering
the case when $M = \canDec$ has straight cylindrical ends but does not have non-vanishing
Killing vector fields. In fact, this condition is not very useful for manifolds with
straight cylindrical ends because, even if $M$ does not have non-vanishing
Killing vector fields, its ``end'' $M' \times \RR$ will always have non-vanishing Killing
vector fields, for instance $\pa_{t}$ (the infinitesimal generator of the translations
along $\RR$).
(See Subsection \ref{ssec.ToAdd} for more on the case $M$ closed.)

\subsection{Definition of layer potentials on $M$}
\label{ssec.def.lp}
The Fredholm result of the previous section allow us to define and study several layer
potential type operators. {\it We continue to assume that $M$ has straight cylindrical ends.
We also assume that $V$ and $V_{0}$ satisfy the positivity at infinity
Assumption \ref{assumpt.VV0weak},}
in order to be able to use the Fredholmness results of the previous section. (The case $M$
closed was treated similarly in great detail in \cite{KNW-2025}.) Recall that
the Moore-Penrose pseudoinverse $\psdinv$ of $\bsXi$ (which was proved
to be Fredholm in Theorem \ref{thm2.invert4}) is defined as the inverse of the
operator $\bsXi$ acting from the orthogonal of $\ker \bsXi$ to its image
(see \cite{KNW-2025} for details and references).
As in \cite{KNW-2025}, we also obtain the following result:

\begin{proposition}\label{prop.invert.inverse}
  Let $M$ be a manifold with straight cylindrical ends.
  Let $\bss=\bst=(1,0)$, as in Proposition
  \ref{prop.ADN.elliptic}. We assume that $V$ and $V_{0}$ satisfy
  the non-vanishing Assumption \ref{assumpt.VV0}. Then
  \begin{enumerate}[\rm (i)]
    \item The pseudoinverse $\psdinv$ has the form
    \begin{equation*} 
      \psdinv \, =:\,
      \left(
      \begin{array}{cc}
        \maA  & \maB  \\
        \maC  & \maD
      \end{array}
      \right) \in \ePS{[- \bss - \bst]}(M; TM \oplus \CC)\,,
    \end{equation*}
    where $\maA \in \ePS{-2}(M; TM)$, $\maC = \maB^{*} \in \ePS{-1}(M; TM, \CC)$,
    and $\maD \in \ePS{0}(M)$.

    \item Consequently, we have $\sigma_{-2}(\maA )(x, \xi) \seq \frac{1}{| \xi |^2} -
    \frac{V_{0}+1}{2V_{0}+1}\frac{1}{| \xi |^4}
    \xi^{\sharp} \otimes \xi \,,$
    $\sigma_{-1}(\maB )(x, \xi) \seq \sigma_{-1}(\maC )(x, \xi)^{*} \seq
    \frac{\imath}{(2V_{0}+1)| \xi |^2} \xi^{\sharp} \,,$
    and $\sigma_{0}(\maD )(x, \xi) \seq -\frac{2}{2V_{0}+1}.$
  \end{enumerate}
\end{proposition}

\begin{proof}
  The point (i) makes explicit (using the definition of \ADN-operators) the fact that
  $\bsXi^{-1} \in \ePS{[- \bss - \bst]}(M; TM \oplus \CC)$,
  a result that was already proved in Theorem
  \ref{thm2.main.invert}. Let $\bss=\bst=(1,0)$ be as in the statement.
  The multiplicativity of the principal symbol $\Symb$ gives that
  \begin{equation*}
    \quad \Symb_{\bss, \bst}(\bsXi)\Symb_{\bf -t, -s}(\psdinv)
    \seq \Symb_{\boldsymbol 0, 0}(1) \seq 1\,.
  \end{equation*}
  Therefore, the $({- \bss}, {- \bst})$-principal symbol
  of the inverse $\psdinv$ of $\bsXi$ is the inverse of the
  $(\bss, \bst)$-principal symbol of $\bsXi$,
  which is given by Proposition \ref{prop.ADN.elliptic}. Thus, the
  principal symbols of the operators $\maA$, $\maB$, $\maC$, and $\maD$
  (the entries of $\psdinv$)
  are as stated (as given by Proposition \ref{prop.ADN.elliptic}).
\end{proof}

The existence of the inverse $\psdinv$ of $\bsXi$ allows us to use classical methods to
define the single and double layer potential operators for the Stokes operator (see also
\cite{D-M, Ladyzhenskaya, M-T, 24}). We let $\Omega \subset M$ be a subdomain
with straight cylindrical ends, that is
\begin{equation*}
  \Omega \cap \big[ M'  \times (-\infty, 0) \big]
  \seq \Omega' \times (-\infty, 0)\,,
\end{equation*}
as before (recall that such an open subset is called a \emph{compatible}
open subset of $M$, see Equation \eqref{eq.def.Omega}).
Let $\Gamma := \pa \Omega$, as usual in this paper.
Thus $\Gamma$ is also a manifold with straight cylindrical ends. For the following
definition, recall the Stokes operator $\bsXi = \bsXi_{V, V_{0}}$ \eqref{eq.def.bsXi}.
Also, recall the distribution $\bsh\otimes \delta_{\Gamma}$ of Lemma
\ref{lemma.def.hdelta} and the operator $\bopstar$ of Lemma \ref{lemma.bop.star}.
We now introduce the layer potential operators associated to our generalized
Stokes operator $\bsXi := \bsXi_{V, V_{0}}$.

\begin{definition} \label{def.lp}
  Let $\bsh\in H^{s}(\Gamma; TM)$, $s \in \RR$. The {\it single-layer potential}
  $\maS_{\rm{ST}}(\bsh)$ for $\bsXi$ is then
  \begin{equation*} 
    \maS_{\rm{ST}}(\bsh)
      \ede \psdinv \, \left[\cvector{\bsh}{0}
      \otimes \delta_{\Gamma}\right]\,.
  \end{equation*}
  Similarly, the {\it double-layer potential} $\maD_{\rm{ST}}(\bsh)$ for $\bsXi$ is given by:
  \begin{equation*} 
    \maD_{\rm{ST}}(\bsh)
      \ede \psdinv \, \big[ \bopstar
      \left(
        \bsh\otimes \delta_{\Gamma}
      \right) \big]\,.
  \end{equation*}
\end{definition}

It is useful sometimes to separate the components of $\maS_{\rm{ST}}(\bsh)$ and
$\maD_{\rm{ST}}(\bsh)$ into \emph{velocity} and \emph{pressure} components.

\begin{definition} \label{def.lp2}
  We continue to assume that $\bsh\in H^{s}(\Gamma; TM)$, $s \in \RR$.
  The components of these definitions yield the
  {\it single-layer \underline{velocity} potential} $\maV_{\rm{ST}}(\bsh)$
  and the {\it single-layer \underline{pressure} potential} $\maP_{\rm{ST}}(\bsh)$ for $\bsXi$:
  \begin{equation*} 
    \left(
      \begin{array}{c}
        \maV_{\rm{ST}}(\bsh)\\
        \maP_{\rm{ST}}(\bsh)
      \end{array}
      \right)
      \ede \maS_{\rm{ST}}(\bsh)
      \ede \psdinv \, \left[\cvector{\bsh}{0}
      \otimes \delta_{\Gamma}\right]\,.
  \end{equation*}
  Similarly, we obtain the {\it double-layer
  \underline{velocity} potential} $\mathcal{W}_{\rm{ST}}(\bsh)$ and
  the {\it double-layer \underline{pressure} potential}
  $\mathcal{Q}_{\rm{ST}}(\bsh)$ for $\bsXi$
  \begin{equation*} 
    \left(
      \begin{array}{c}
        \maW_{\rm{ST}}(\bsh)\\
        \maQ_{\rm{ST}}(\bsh)
      \end{array}
      \right)
      \ede \maD_{\rm{ST}}(\bsh)
      \ede \psdinv \, \big[ \bopstar
      \left(
        \bsh\otimes \delta_{\Gamma}
      \right) \big]\,.
  \end{equation*}
\end{definition}

These definitions can be made more explicitly as follows.

\begin{remark}\label{rem.components}
  We have
  \begin{equation*} 
    \maW_{\rm{ST}}(\bsh) \seq
    \left(
      \begin{array}{cc}
        \maA  & \maB
      \end{array}
    \right)
    \left( \begin{array}{c}
        - 2\Dnu^{*} \\
        \bsnu^{\sharp}
      \end{array}
    \right)
    (\bsh\otimes \delta_{\Gamma })
    \seq (- 2 \maA  \Dnu^{*} + \maB \bsnu^{\sharp})\,
    (\bsh\otimes \delta_{\Gamma })
  \end{equation*}
  and
  \begin{equation*} 
    \maQ_{\rm{ST}}(\bsh)
    \seq \left(
      \begin{array}{cc}
        \maC  & \maD
      \end{array}
      \right)
      \left( \begin{array}{c}
        - 2\Dnu^{*} \\
        \bsnu^{\sharp}
      \end{array}
      \right)
      (\bsh\otimes \delta_{\Gamma }) \\
       \ \seq (- 2 \maC  \Dnu^{*} + \maD \bsnu^{\sharp})\,
      (\bsh\otimes \delta_{\Gamma })\,.
\end{equation*}
  Similarly, $\maS_{\rm{ST}}(\bsh) \seq \big(\, \maA (\bsh\otimes \delta_{\Gamma }) \ \ \
  \maC(\bsh\otimes \delta_{\Gamma })\, \big)^{\top }$.
\end{remark}

Theorem \ref{thm.mapping.lp} gives the following result.

\begin{proposition}\label{prop.H2estimates}
  We continue to assume that $M$ is a manifold with straight cylindrical ends and
  that $\Omega \subset M$ is a compatible open subset \lpar Equation \eqref{eq.def.Omega}\rpar.
  Let $\Omega_{+} := \Omega$, $\Omega_{-}:= M \smallsetminus \overline{\Omega}$,
  and $\Gamma := \pa \Omega$ be as before. Let $s\in \RR$ and $\bsh \in H^{s}(\Gamma; TM)$.
  Then
  \begin{equation*}
    \maV_{\rm{ST}}(\bsh)\vert_{\Omega_{\pm}}
    \in H^{s + \frac32}(\Omega_{\pm}; TM) \ \mbox{ and }\
    \maP_{\rm{ST}}(\bsh)\vert_{\Omega_{\pm}} \in H^{s + \frac12}(\Omega_{\pm})\,.
  \end{equation*}
  Similarly, let $s\in \RR$ and $\bsh \in H^{s}(\Gamma; TM)$. Then
  \begin{equation*}
    \maW_{\rm{ST}}(\bsh)\vert_{\Omega_{\pm}} \in H^{s + \frac12}(\Omega_{\pm}; TM)\
    \mbox{ and }\
    \maQ_{\rm{ST}}(\bsh)\vert_{\Omega_{\pm}} \in H^{s - \frac12}(\Omega_{\pm})\,.
  \end{equation*}
\end{proposition}

\begin{proof}
  Indeed, this follows from Theorem \ref{thm.mapping.lp}
  and Remark \ref{rem.components} because
  $\maA$ has order $-2$, $\maC$ and $\bsP := - 2 \maA  \Dnu^{*} + \maB \bsnu^{\sharp}$
  have order $-1$, and $- 2 \maC  \Dnu^{*} + \maD \bsnu^{\sharp}$ has
  order zero. Moreover, all these four operators have rational type symbols.

  We obtain in this way the eight stated relations (including the choices of $+/-$).
  Just to make things clearer,
  let us show in detail how the first relation is proved.
  Theorem \ref{thm.mapping.lp} for $m = -2$, $s \in \RR$, $P = \maA$,
  gives
  \begin{equation*}
    \maV_{\rm{ST}} (\bsh) \vert_{\Omega_{+}} \ede
    \maA ( \bsh \otimes \delta_{\Gamma})
    \vert_{\Omega_{+}} \in H^{s -(-2)-\frac12}(\Omega_{+}) \seq
    H^{s + \frac32}(\Omega_{+})\,.
  \end{equation*}
  This proves the first of the eight mentioned relations. The other seven are proved in
  exactly the same way.
\end{proof}

The following result is a consequence of the definition of the single and double
layer potentials. Recall that $\maN$ is the kernel of $\bsXi$ and $p_{\maN}$ is the
$L^{2}$-orthogonal projection onto $\maN$.

\begin{proposition}\label{prop.compatibility}
  Let us assume that $V$ and $V_{0}$ satisfy the positivity assumption on $M$
  (Assumption \ref{assumpt.VV0}, and hence that $\bsXi$ is invertible).
  Let $\bsh \in H^{s}(\Gamma; TM)$, $s \in \RR$. Then $\bsXi \maS_{\rm{ST}}(\bsh)$
  and $\bsXi \maD_{\rm{ST}}(\bsh)$ vanish on $M \smallsetminus \Gamma$.
\end{proposition}

\begin{proof}
  We have
  \begin{equation*}
    \bsXi \maS_{\rm{ST}}(\bsh) \seq \bsXi \bsXi^{-1}(\bsh \otimes \delta_{\Gamma})
    \seq \bsh \otimes \delta_{\Gamma}\,
  \end{equation*}
  which vanishes outside $\Gamma$, as claimed.
\end{proof}

We notice that in \cite{KNW-2025} we have considered also the case when $\bsXi$ is
only Fredholm, when some additional compatibility relations need to be added to obtain
the vanishing outside $\Gamma$ (respectively, $\Gamma'$) in the above result.
We shall need the following consequences of the representation formula
in Proposition \ref{prop.Green}, the first one of which we will call
\emph{Pompeiu's formula}. Recall that $f_{+}$ denotes the trace of
$f \in H^{1}(\Omega)$ at $\Gamma := \pa \Omega$ (thus the trace from the {\it interior}).

\begin{proposition}\label{prop.rep.formula}
  Let us assume that $V$ and $V_{0}$ satisfy the positivity assumption on $M$
  (Assumption \ref{assumpt.VV0}, and hence that $\bsXi$ is invertible).
  Let $U$ be such that $U\vert_{\Omega} \in H^{2}(\Omega; TM) \oplus H^{1}(\Omega)$.
  We extend $U$ to be 0 outside $\Omega$, and call the extension ${\bf 1}_{\Omega}U$,
  where ${\bf 1}_{\Omega}$ be the characteristic function of $\Omega$ (as before).
  \begin{enumerate}[\rm (i)]
  \item
    ${\bf 1}_{\Omega} U \seq \bsXi^{-1} \left({\bf 1}_{\Omega}
    \big( \bsXi U  \big) \right) -
    \maS_{\rm{ST}} \big[(\bop U)_{+}\big]
    + \maD_{\rm{ST}}  (\bsu_{+})\,.$
  \item Consequently, if we also have $\bsXi U = 0$ in $\Omega $, then
    \begin{equation*}  
    \maD_{\rm{ST}} (\bsu_{+})(x) - \maS_{\rm{ST}} \big[(\bop U)_{+}
    \big] (x) \seq
    \begin{cases}
      \ U(x) & \mbox{ if } x \in \Omega\\
      \ \ \, 0    & \mbox{ if } x \in M \smallsetminus \overline{\Omega} \,.
    \end{cases}
    \end{equation*}
  \end{enumerate}
\end{proposition}

\begin{proof}
  Recall the representation formulas of Proposition \ref{prop.Green}(iii)
  \begin{align*}
    \bsXi \big( \textbf{1}_{\Omega} U \big) &
    \seq \textbf{1}_{\Omega} \big(  \bsXi U  \big)
    - (\tbop U)_{+} \otimes \delta_{\pa \Omega}
    + \tbopstar (U_{+} \otimes \delta_{\pa \Omega})\\
    & \seq \textbf{1}_{\Omega} \big(  \bsXi U  \big)
    - \cvector{(\bop U)_{+}}{0} \otimes \delta_{\pa \Omega}
    + \bopstar (\bsu_{+} \otimes \delta_{\pa \Omega})\,.
  \end{align*}
  By applying $\bsXi^{-1}$ to this formula, we obtain Pompeiu's formula:
  \begin{align}\label{eq.Pompeiu}
    \textbf{1}_{\Omega} U & \seq \bsXi^{-1} \left(\textbf{1}_{\Omega} \big( \bsXi U  \big) -
    \cvector{(\bop U)_{+}}{0} \otimes \delta_{\Gamma}
    + \bopstar \, (\bsu_{+} \otimes \delta_{\Gamma})\right)\\
    & \seq \bsXi^{-1} \left(\textbf{1}_{\Omega}  \big( \bsXi U  \big) \right) -
    \maS_{\rm{ST}} ({(\bop U)_{+}})
    + \maD_{\rm{ST}}  (\bsu_{+})\,. \nonumber
  \end{align}
  The second part follows immediately from Pompeiu's formula \ref{eq.Pompeiu}.
\end{proof}

\subsection{Jump relations}
\label{ssec.jump}
As usual, the limit and jump relations of our layer potentials will play a crucial
role. They are a delicate and difficult topic, especially for the Stokes operator.
We now establish the needed limit and jump relations following Subsection \ref{ssec.do}
and, especially, Section \ref{sec.NLimits} to establish some needed jump relations
for the needed components of our potential operators $\maS_{\rm{ST}}$ and $\maD_{\rm{ST}}$
(see Remark \ref{rem.components}). We continue to
assume that $M$ is with straight cylindrical ends and that $V$ and $V_{0}$ satisfy
the positivity assumption at infinity \lpar Assumption \ref{assumpt.VV0weak}\rpar,
which allows us to conclude that $\bsXi$ is Fredholm and hence that the layer
potentials are defined.
Close versions of the relations established in this section are contained also in
\cite{KMNW-2025} and \cite{KNW-2025}. They were also proved by Mitrea and Taylor
\cite{M-T} and Dindo\u{s} and Mitrea \cite{D-M} with a different approach.

\begin{proposition}\label{prop.local.stokes}
  Let $\bsP := - 2 \maA  \Dnu^{*} + \maB \bsnu^{\sharp}$ and denote
  $\mathfrak f = (V_{0} + 1)/(2V_{0} + 1)$ and $\mathfrak  g =
  1/(2V_{0} + 1)$, as before. Then
  \begin{equation*}
    \sigma_{-1}(\bsP; \xi) \, =: \, a( \xi) \seq
    \frac{\imath}{|\xi|^2} \left( \xi(\bsnu) + \bsnu \otimes \xi
    - \frac{2 \mathfrak f \xi(\bsnu)}{| \xi|^{2}} \xi^{\sharp} \otimes \xi
    + \mathfrak g \xi^{\sharp} \otimes \bsnu^{\sharp} \right)\,.
  \end{equation*}
  Consequently, $\JC = -\imath$ for this operator.
\end{proposition}

\begin{proof}
  The calculations are local, so we can assume that $\Omega = \RR_{+}^{n}$. Then
  $\bsnu = - e_{n} = (0, \ldots, 0, -1)$. We shall write, $\xi = \xi' + t \bsnu$.
Substituting $\xi = \xi' + t e_{n}$, Proposition \ref{prop.invert.inverse} gives
  \begin{equation*}
    \begin{gathered}
    \sigma_{-2}(\maA; \xi) \seq \frac{1}{| \xi |^2} - \frac{\mathfrak f}{| \xi |^4}
  \xi^{\sharp} \otimes \xi\,, \qquad
     \sigma_{-1}(\maB )(x, \xi)
     \seq \frac{\imath \mathfrak g}{| \xi |^2} \xi^{\sharp}\,,\\
    \mbox{and } \quad \sigma_{1}(\Dnu ^{*}; \xi) \seq - \frac{\imath}2
    \big [\xi(\bsnu) + \bsnu \otimes \xi \big] \seq \frac{\imath}2
    \big [t  + e_{n} \otimes \xi \big]
    \,.
    \end{gathered}
  \end{equation*}
  Using that $\bsnu^{\sharp} = -e_{n}^{\sharp}$, we can compute
  \begin{align*}
    \sigma_{-1}(\bsP ; \xi) & \seq -2 \sigma_{-2}(\maA; \xi)\sigma_{1}(\Dnu^{*}; \xi)
    + \sigma_{-1}(\maB; \xi)\bsnu^{\sharp}\\
    & \seq -2\left(\frac1{| \xi|^{2}} - \frac{\mathfrak f}{|\xi|^{4}}
 \xi^{\sharp} \otimes \xi \right)
    \frac{\imath}2 (t  + e_{n} \otimes \xi)
    - \frac{\imath \mathfrak g}{| \xi|^{2}} \xi^{\sharp} \otimes e_{n}^{\sharp}\\
    & \seq -\frac{\imath}{|\xi|^2} \big(   t  + e_{n} \otimes \xi
    - \frac{2 \mathfrak f t }{| \xi|^{2}} \xi^{\sharp} \otimes \xi
    + \mathfrak g \xi^{\sharp} \otimes e_{n}^{\sharp} \big)\,.
  \end{align*}
  To obtain the ``jump'' coefficient $\JC$, we expand the
  formula for the principal symbol from the last equation according
  to the highest powers of $t$, using $ \xi = \xi' + t e_{n}^{\sharp}$,
  with $\xi' = ( \xi_{1}, \ldots, \xi_{n-1}, 0)$, to obtain
  \begin{align*}
    \lim_{t \to \pm \infty} t \sigma_{-1}(\bsP ; \xi)
    &  \seq - \lim_{  t  \to \pm \infty}\frac{\imath t}{|\xi|^2}
    \big(t + e_{n} \otimes \xi
    - \frac{2 \mathfrak f t}{| \xi|^{2}} \xi^{\sharp} \otimes \xi
    + \mathfrak g \xi^{\sharp} \otimes {e_{n}^{\sharp}} \big)\\
    &  \seq -
    \lim_{t \to \pm \infty}\frac{\imath t^{2}}{| \xi|^2}
    \big(1 + e_{n} \otimes e_{n}^{\sharp}
    - \frac{2 \mathfrak f t ^{2}}{| \xi|^{2}} e_{n} \otimes e_{n}^{\sharp}
    + \mathfrak g e_{n} \otimes e_{n}^{\sharp} \big)\\
    &  \seq - \imath \big(1 + e_{n} \otimes e_{n}^{\sharp}
    - 2 \mathfrak f  e_{n} \otimes e_{n}^{\sharp}
    + \mathfrak g e_{n} \otimes e_{n}^{\sharp} \big)\\
    &  \seq - \imath
    \big[1 + (1 - 2 \mathfrak f + \mathfrak g) e_{n} \otimes e_{n}^{\sharp}
    \big] \seq - \imath \,,
  \end{align*}
  because $1 - 2 \mathfrak f + \mathfrak g = 0$.
\end{proof}

The following integrals can be obtained by using the Residue Theorem (the details
can be found in the proof of Lemma 16.3.2 of \cite{KMNW-2025}, see also
\cite{GabrielaBook}).

\begin{lemma}\label{lemma.Mirela} Let $a > 0$. We have
  \begin{equation*}
    \displaystyle\int_{\RR}\frac{x^2 dx}{(a^{2} + x^{2})^{2}}
    \seq \displaystyle\frac{\pi}{2a}\,,\quad
    \displaystyle\int_{\RR}\frac{dx}{(a^{2} + x^{2})^{2}}
    \seq \displaystyle\frac{\pi}{2a^3}\,, \quad
    \mbox{and} \quad \displaystyle\int_{\RR}\frac{dx}{a^{2} + x^{2}}
    \seq \displaystyle\frac{\pi}{a}\,.
  \end{equation*}
\end{lemma}

For the rest of the paper, we shall let
$\bsP := - 2 \maA  \Dnu^{*} + \maB \bsnu^{\sharp}$
be the pseudodifferential operator defining the
vector part $\maW_{\rm{ST}}$ of the double layer potential (Definition
\ref{def.lp}) and let
\begin{equation}\label{eq.def.bsk}
  \bsK \ede \bsP _{0} \ede \big(- 2 \maA  \Dnu^{*} + \maB \bsnu^{\sharp})_{0}
  \in \ePS{0}(\Gamma; TM)
\end{equation}
be the restriction of $\bsP := - 2 \maA  \Dnu^{*} + \maB \bsnu^{\sharp}$
to $\Gamma :=\pa \Omega$ (see Theorem \ref{thm.main.jump1}).

\begin{theorem}\label{thm.jump.K1}
  Let $\bsK := \bsP_{0} \in \ePS{0}(\Gamma; TM)$ be as in Equation
  \eqref{eq.def.bsk}. Then
  \begin{equation*}
    \maW_{\rm{ST}}(\bsh)_{\pm}
    \ede \big[ \bsP  (\bsh \otimes \delta_{\Gamma})\big]_{\pm}
    \seq \left [\pm \frac12  +
    \bsK \right ] \bsh\,,
  \end{equation*}
  where $\sigma_{0}(\bsK; \xi') = \frac{\imath V_{0}}{2(2V_{0} + 1)| \xi'|}
  \big( \bsnu \otimes \xi' - \xi^{'\sharp} \otimes \bsnu^{\sharp} \big )$.
  In particular, the two operators $\pm \frac12  + \bsK$
  are elliptic and have self-adjoint principal symbols.
\end{theorem}

\begin{proof}
  Let $\mathfrak f = (V_{0} + 1)/(2V_{0} + 1)$ and $\mathfrak g = 1/(2V_{0} + 1)$ be as in
  Proposition \ref{prop.local.stokes}. As in the proof of Theorem \ref{thm.main.jump1}, we
  use local coordinates such that $\bsnu = -e_{n}$. We need only to identify
  $\sigma_{0}(\bsK; \xi') = \sigma_{0}(\bsP_{0}; \xi')$ (the rest follows
  from Theorem \ref{thm.main.jump1} and Proposition \ref{prop.local.stokes}).
  To that end, we separate the terms that are \emph{even} in $ t $ in
  $\sigma_{-1}(\bsP ; \xi)$. For instance, recalling that $\xi =\xi '+te_{n}$, the even
  part of $\xi^{\sharp} \otimes \xi$ is $\xi^{'\sharp} \otimes \xi'
  + t^{2} e_{n} \otimes e_{n}^{\sharp}$, whereas its odd part is
  $t(e_{n} \otimes \xi'+ \xi^{'\sharp} \otimes e_{n}^{\sharp})$.
  This gives
  \begin{multline*}
    b(\xi', t) \ede \sigma_{-1}( \bsP; \xi) + \sigma_{-1}(\bsP; \xi', -t) \\
    \seq -\frac{2 \imath}{|\xi|^2} \Big[ e_{n} \otimes \xi' +
    \mathfrak g \xi^{'\sharp} \otimes e_{n}^{\sharp} -
    \frac{2 \mathfrak f t ^{2}}{| \xi|^{2}}(e_{n} \otimes \xi'
    + \xi^{'\sharp} \otimes e_{n}^{\sharp}) \Big]\\
    \seq -2 \imath \Big[ \left( \frac1{| \xi|^2} -
    \frac{2 \mathfrak f t ^{2}}{| \xi|^{4}} \right)
    e_{n} \otimes \xi' + \left( \frac{\mathfrak g}{| \xi|^2} -
    \frac{2 \mathfrak f t ^{2}}{| \xi|^{4}} \right)
 \xi^{'\sharp} \otimes e_{n}^{\sharp} \Big]
    \,.
  \end{multline*}
  We next use Theorem \ref{thm.main.jump1} and the relations
  \begin{equation*}
    \int_{\RR} \frac{1}{| \xi|^2}\, dt \seq \frac{\pi}{| \xi'|}
    \ \mbox{ and }\
    \int_{\RR} \frac{t ^{2}}{| \xi|^4}\, dt  \seq \frac{\pi}{2| \xi'|}
    \,,\quad \xi' \neq 0\,,
  \end{equation*}
  (see Lemma \ref{lemma.Mirela}) to obtain
  \begin{align*}
    \sigma_{0}(\bsK; \xi') &\seq \frac1{4 \pi}
    \int_{\RR} b( \xi',   t )\, d  t \\
    &\seq - \frac{\imath}{2\pi} \int_{\RR}
    \Big[ \left( \frac1{|\xi|^2} - \frac{2 \mathfrak f t ^{2}}{| \xi|^{4}} \right)
    e_{n} \otimes \xi' +
    \left( \frac{\mathfrak  g}{| \xi|^2} - \frac{2 \mathfrak f   t ^{2}}{| \xi|^{4}} \right)
  \xi^{'\sharp} \otimes e_{n}^{\sharp} \Big]\, d  t \\
    &\seq \frac{\imath V_{0}}{2(2V_{0} + 1)| \xi'|}
    \big( \xi^{'\sharp} \otimes e_{n}^{\sharp} - e_{n} \otimes \xi'\big )\\
    & \seq \frac{\imath V_{0}}{2(2V_{0} + 1)| \xi'|}
    \big( \bsnu \otimes \xi' - \xi^{'\sharp} \otimes \bsnu^{\sharp} \big )\,.
  \end{align*}
  This explicit formula gives that $\sigma_{0}(\bsK)^{*} = \sigma_{0}(\bsK)$.
  To complete the proof, we notice that the \emph{non-zero} eigenvalues
  $\lambda = \pm \frac{V_{0}}{2(2V_{0} + 1)}$ of $\sigma_{0}(\bsK; \xi')$ satisfy
  $|\lambda|<1/4$, and hence the operators $\pm \frac 12  + \bsK$ are elliptic.
\end{proof}

\begin{remark}\label{rem.order.K}
  Let us notice that $\bsK = \bsP _{0} \in \Psi^{0}(\Gamma; TM)$ is of order $-1$
  if, and only if, $V_{0} = 0$ on $\Gamma$ (see Equation \eqref{eq.def.bsk}
  and Theorem \ref{thm.jump.K1} above).
\end{remark}

Let $\mathfrak f = (V_{0} + 1)/(2V_{0} + 1)$ and $\mathfrak g = 1/(2V_{0} + 1)$,
as in Proposition \ref{prop.local.stokes}. Also, recall the pseudodifferential
operators $\maA$ and $\maC$ introduced in Proposition \ref{prop.invert.inverse}
as two of the matrix components of $\psdinv$. (These operators were used then
in the definition of the single layer potential $\maS_{\rm{ST}}$.)
We obtain the following result.

\begin{theorem}\label{thm.jump.rel}
  Let $\bsh \in L^{2}(\Gamma; TM)$, $\Gamma := \pa \Omega$,
  $\mathfrak f := \frac{V_{0}+1}{2V_{0}+1}$ and $\mathfrak g = \frac1{2V_{0} + 1}$. Then
  \begin{enumerate}[\rm (i)]
    \item $\maV_{\rm{ST}}(\bsh)_{\pm} \seq \maA_{0}(\bsh)$ and
      $\sigma_{-1}(\maA_{0}; \xi') \seq
      \frac1{4|\xi'|}(2 - \mathfrak f \bsnu \otimes \bsnu^{\sharp}
      - \mathfrak f \eta^{\sharp} \otimes \eta)
      \,,$
    where $\eta := |\xi'|^{-1}\xi'$.  If $V_{0} \ge 0$,
    then $\maA_{0}$ is elliptic with self-adjoint symbol.

    \item $\maP_{\rm{ST}}(\bsh)_{\pm} = \left(\mp \displaystyle \frac {\mathfrak g}2
    \bsnu^{\sharp} + \maC_{0} \right) \bsh$, where $\sigma_{0}(\maC_{0};\xi')
    = -\displaystyle \frac{\imath \mathfrak g }{2|\xi'|} \xi'.$

    \item $[\bop \maS_{\rm{ST}}(\bsh)]_{\pm} \seq \left(\mp \frac12
    + {\bsK}^{*}\right)\bsh\,,$ where $\bsK = \bsP_{0} := (-2 \maA \Dnu^{*} +
    \maB \bsnu^{\sharp})_{0}$ is defined as in Theorem \ref{thm.jump.K1}.
  \end{enumerate}
\end{theorem}

\begin{proof}
  Recall that, for $\xi \in T_{x}^{*}M$, the linear map
  $\xi^{\sharp} \otimes \xi \in \End(T_{x}M)$ is
  defined by $(\xi^{\sharp} \otimes \xi) (v) := \xi(v) \xi^{\sharp}$.
  In particular, $\sigma_{-2}(\maA, \xi) = \frac1{|\xi|^{4}}
  \big(|\xi|^{2} - \mathfrak{f}\xi^{\sharp} \otimes \xi\big)$ (recall that
  $\mathfrak{f}:= \frac{V_{0}+1}{2V_{0}+1}$).
  For $\xi \in T^{*}M$, let us write $\xi = \xi' + t \bsnu^{\sharp}$,
  with $\xi'(\bsnu) = 0$ and we identify $\xi'$ with an element
  of $T^{*}\Gamma$, when $\xi' \in T_{x}^{*}M$ with $x \in \Gamma$. To
  prove the first equality, we use Theorem \ref{thm.main.jump1b}(ii)
  (see also Theorem \ref{thm.main.jump-2}) and then Proposition
  \ref{prop.invert.inverse}
  to obtain the following relation
  \begin{align*}
    \sigma_{-1}(\maA_{0}; \xi') & \seq \frac1{2\pi} \int_{\RR}
    \sigma_{-2}(\maA; \xi) \, dt\\
    & \seq \frac1{2\pi} \int_{\RR} \frac1{|\xi|^{4}}
    \big(|\xi|^{2} - \mathfrak{f} \xi^{\sharp} \otimes \xi\big) \, dt\\
    & \seq \frac1{2\pi} \int_{\RR} \frac1{|\xi|^{4}}
    \big[|\xi|^{2} - \mathfrak{f} (\xi^{\prime \sharp} \otimes \xi'
    + t \bsnu \otimes \xi' +
    t \xi^{\prime \sharp} \otimes \bsnu^{\sharp}
    + t^{2} \bsnu \otimes \bsnu^{\sharp}) \big] \, dt\\
    & \seq \frac1{2\pi} \int_{\RR} \frac1{|\xi|^{4}}
    \big[|\xi|^{2} - \mathfrak{f} (\xi^{\prime \sharp} \otimes \xi' +
    t^{2} \bsnu \otimes \bsnu^{\sharp}) \big] \, dt\\
    & \seq \frac1{2\pi} \int_{\RR} \left(\frac1{|\xi'|^{2} + t^{2}}
    - \frac {\mathfrak{f}}{(|\xi'|^{2} + t^{2})^{2}} \xi^{\prime \sharp} \otimes \xi'
    - \frac {\mathfrak{f}t^{2}}{(|\xi'|^{2} + t^{2})^{2}}
    \bsnu \otimes \bsnu^{\sharp}\right) \, dt\\
    & \seq \frac1{2\pi} \left(\frac{\pi}{|\xi'|}
    - \frac {\mathfrak{f} \pi}{2|\xi'|^{3}} \xi^{\prime \sharp} \otimes \xi'
    - \frac {\mathfrak{f} \pi}{2|\xi'|}
    \bsnu \otimes \bsnu^{\sharp}\right) \\
    & \seq \frac{1}{2|\xi'|}
    - \frac {\mathfrak{f}}{4|\xi'|^{3}} \xi^{\prime \sharp} \otimes \xi'
    - \frac {\mathfrak{f}}{4|\xi'|} \bsnu \otimes \bsnu^{\sharp}\\
    & \seq \frac1{4|\xi'|}\left(2 - \mathfrak f \bsnu \otimes \bsnu^{\sharp}
      - \mathfrak f \eta^{\sharp} \otimes \eta\right)\,.
  \end{align*}
  This proves (i). The operator $\maA_{0}$ is elliptic because
  $\mathfrak f < 1$, and hence $\sigma_{-1}(\maA_{0})$ is invertible.

  For the relation (ii), we use the identity
  $\sigma_{-1}(\maC; \xi) = -\frac{g \imath}{|\xi|^{2}} \xi$ (see
  Proposition \ref{prop.invert.inverse}) and then Theorem \ref{thm.main.jump1}
  (see also Theorem \ref{thm.main.jump0}). We also write $\xi = \xi' + t \bsnu^{\sharp}$
  with $\xi' \perp \bsnu$, as in the proof of the previous point.
  Then we notice that the even part of $\sigma_{-1}(\maC; \xi)$
  (in $\tau$) is $-\frac{g \imath}{|\xi|^{2}} \xi'$. Therefore
  \begin{equation*}
    \sigma_{0}(\maC_{0}; \xi') \seq -\frac1{2\pi}
    \int_{\RR}\frac{g \imath}{|\xi|^{2}} \xi' \, d t
    \seq -\frac{g \imath}{2\pi}
    \left(\int_{\RR}\frac{1}{|\xi|^{2}} \, d t \right) \xi'
    \seq -\frac{g \imath}{2 |\xi'|} \xi'\,.
  \end{equation*}
  The ``jump part'' is also obtained from the principal symbol of
  $\maC$, namely, it is $\mp \frac{\imath}2 \sigma_{-1}(\maC; \bsnu^{\sharp})
  = \mp \frac{g}{2}\bsnu^{\sharp}$.

  Let us now prove the third relation. We have $\bsXi^{*} = \bsXi$, and hence
  $\big(\psdinv \big)^*= \psdinv$. Theorem \ref{thm.main.jump1} and Proposition
  \ref{prop.local.stokes} then give $\JC (\bop \psdinv) =
  \JC (\psdinv\bopstar)^{*} = \overline{(-\imath)} = \imath$.
  Moreover, $(\bop \psdinv)_{0} = (\psdinv \bopstar)_{0}^{*} = \bsK^{*}$.
  This then gives the following relation:
  \begin{multline*}
    [\bop \maS_{\rm{ST}}(\bsh)]_{\pm}
    \seq [\bop \psdinv(\bsh \otimes \delta_{\Gamma})]_{\pm}
    \seq \left( \pm \frac{\imath}2 \JC (\bop \psdinv)
      + (\bop \psdinv)_{0} \right) \bsh \\
     \seq \left( \pm \frac{\imath}2 \JC (\psdinv \bopstar)^{*}
      + (\psdinv \bopstar\big)_{0}^{*} \right) \bsh
    \seq \left(\mp \frac12 + \bsK^{*} \right) \bsh\,.
  \end{multline*}
  The proof is complete.
\end{proof}

For the Stokes operator $\bsXi _{0,0}$, see also\cite{D-M},
\cite[Lemma 3.1]{K-L-W}, \cite[(6.1), (6.2)]{K-W}, \cite[Theorem 3.1]{M-T},
\cite{M-W}, \cite[Lemma 1.3]{24}.
Note that $\bsS :=\maA_{0}\in \ePS{-1}(\Gamma ; TM)$ and
$\frac{1}{2} + \bsK \in \ePS{0}(\Gamma ; TM)$.
Therefore $\Big(\frac{1}{2} + \bsK\Big) \bsS \in  \ePS{-1}(\Gamma ; TM)$
and $\bsS \Big(\frac{1}{2} + \bsK^{*}\Big) \in \ePS{-1}(\Gamma ; TM)$.
In addition, we obtain the following formula.

\begin{proposition}\label{prop.S-K}
  We have the equality
  \begin{equation*}
    \Big(\frac{1}{2} + \bsK\Big) \bsS \seq
    \bsS\Big(\frac{1}{2}+ \bsK^*\Big) \in \ePS{-1}(\Gamma ; TM)\,.
  \end{equation*}
\end{proposition}

\begin{proof}
  Let $\bsh\in H^{1/2}(\Gamma; T^*M)$.
  Theorem \ref{thm.mapping.lp} then gives that
  \begin{equation*}
    \cvector{\bsu}{p} \ede U \ede \maS_{\rm{ST}}(\bsh)\vert_{\Omega}
    \in H^{2}(\Omega; TM) \oplus H^{1}(\Omega)\,.
  \end{equation*}
  Since $\bsXi \maS_{\rm{ST}}(\bsh) = 0$ in $M \smallsetminus \Gamma$
  (Proposition \ref{prop.compatibility}), we can use the Green representation formula
  (Proposition \ref{prop.rep.formula})
  applied to $\cvector{\bsu}{p} := U$, and this gives:
    \begin{equation*}  
    \maD_{\rm{ST}} (\bsu_{+})(x) - \maS_{\rm{ST}}
    \big[(\bop U)_{+} \big](x)\, \seq
    \, U(x) \ \mbox{ if }\  x \in \Omega\,.
  \end{equation*}

  The formula $U := \maS_{\rm{ST}}(\bsh)$ means that $\bsu
  := \maS_{\rm{ST}}(\bsh)\vert_{\Omega}$ and $p:=\maP_{\rm{ST}}(\bsh)\vert_{\Omega}$
  are the restrictions to $\Omega$ of the
  single-layer velocity and pressure potentials. By taking the traces
  on $\Gamma :=\pa \Omega $ from $\Omega $, we obtain
  \begin{equation*}  
    \maD_{\rm{ST}} \big(\big(\maV_{\rm{ST}}(\bsh)_{+}\big)_{+}
    - \maS_{\rm{ST}} \big(\big(\bop \maS_{\rm{ST}} (\bsh)\big)_{+} \big)_{+}
    \seq \maS_{\rm{ST}}(\bsh)_+ \in H^{1}(\Gamma; TM) \oplus L^{2}(\Gamma)\,.
  \end{equation*}
  The vector part of this relation is
  \begin{equation*}  
    \maW_{\rm{ST}} \big(\maV_{\rm{ST}}(\bsh)_{+}\big)_{+}
    - \maV_{\rm{ST}} \big(\big(\bop \maS_{\rm{ST}}
    (\bsh)\big)_{+}\big)_{+} \seq \maV_{\rm{ST}}(\bsh)_{+} \in H^{1}(\Gamma; TM)\,.
  \end{equation*}
  Using the limit relations in Theorems \ref{thm.jump.K1}
  and \ref{thm.jump.rel}, we obtain the formula
  \begin{equation*}  
    \Big(\frac12  + \bsK \Big) \bsS \bsh
    - \bsS \Big(-\frac12  + \bsK^{*} \Big)\bsh \seq \bsS \bsh \,,
  \end{equation*}
  which gives the desired result:
  \begin{equation*}  
    \Big(\frac12  + \bsK \Big) \bsS \bsh
    - \bsS \Big(\frac12  + \bsK^{*} \Big)\bsh \seq 0\,.
  \end{equation*}
  Since $\bsh\in H^{1/2}(\Gamma; T^*M)$ is arbitrary, the previous
  formula leads to the desired property. This completes the proof.
\end{proof}


\section{\cn Calculations on a cylinder}

We shall need definitions and results for the limit operators $\In(\bsXi)$ that are
analogous to those in the previous two sections. To this end, {\cn in this section we
consider the particular case when our manifold $M$ with straight cylindrical
ends is actually a cylinder $M = \mfkS \times \RR$ with a product metric. We do not
assume $\mfkS$ to be connected, though. Then, as in Section \ref{sec.Green2}, we let
$\Omega = \Omega' \times \RR \subset M$ be a product type open subset,
and hence with straight cylindrical ends, where $\Omega ' \subset \mfkS$.
Let $\Gamma ':=\partial \Omega '$.
In this section, in addition to $V$ and $V \ge 0$ smooth, we assume
that they are translation invariant, and hence $V(x', t) = \newtilde V(x')$ and
$V_{0}(x', t) = \newtilde V_{0}(x')$ identify to sections on $\mfkS$.}
Also, {\cn in this section, we use Assumption \ref{assumpt.VV0weak}, that is
that $\newtilde V \succ 0$ and that $\newtilde V_{0} \succ 0$ on $\mfkS$}
(that is, they have points where they are strictly positive
in every connected component of the space of definition, see Definition \ref{def.succ}).
{\cn Also, recall that $\Omega$ is on one side of its
boundary (Assumption \ref{assumpt.1side}), and hence $\Omega'$ is
also on one side of its boundary.}

We follow the usual convention and we let
\begin{equation}\label{eq.def.Omega'}
  \Omega_{-}' \ede \mfkS \smallsetminus \overline{\Omega}' \quad
  \mbox{and} \quad \Omega_{+}' \ede \Omega' \,.
\end{equation}
We also let $\Gamma' := \pa \Omega'$, the boundary of $\Omega'$, a smooth
manifold, because the boundary $\Gamma$ of $\Omega$ is a smooth manifold
(throughout this paper) and
\begin{equation}\label{eq.def.Gamma'}
  \Gamma \seq \Gamma' \times \RR
\end{equation}
in this section. The assumption that $\Omega'$ be on one side of its boundary
is then equivalent to
\begin{equation}\label{eq.1side.Omega'}
  \pa \Omega_{-}' \seq \Gamma' \ede \pa \Omega_{+}' \ede \pa \Omega'
\end{equation}
(see also Assumption \ref{assumpt.1side}). We also note that, for translation invariant
potentials $V$ and $V_{0}$ the assumptions \ref{assumpt.VV0weak} and \ref{assumpt.VV0}
are equivalent.

\subsection{Layer potentials and jump relations for the indicial
operators}

With these hypotheses, we know from Theorem \ref{thm2.invert4}
that $\bsXi$ is invertible on $M = \mfkS \times \RR$ and hence we
can define the layer potential operators. These layer potential operators
are translation invariant, so we can consider their indicial families
(their Fourier transforms, see Definition \ref{def.indicial.fam}).
We have $[\widehat \bsXi(\tau)]^{-1} = \widehat{(\bsXi^{-1})}(\tau)$,
so $[\widehat \bsXi(\tau)]^{-1}$ has a similar structure to
$\bsXi^{-1}$ of Proposition \ref{prop.invert.inverse}. That is,
we have the following remark.

\begin{remark}\label{rem.indicial.inv}
  The inverse $[\widehat \bsXi(\tau)]^{-1}$ has the form
  \begin{equation*} 
    [\widehat \bsXi(\tau)]^{-1} \, =:\,
    \left(
    \begin{array}{cc}
      \widehat \maA(\tau)  & \widehat \maB(\tau)  \\
      \widehat \maC(\tau)  & \widehat \maD(\tau)
    \end{array}
    \right) \in \ePS{[- \bss - \bst]}(M; TM \oplus \CC)\,,
  \end{equation*}
  where $\widehat \maA(\tau) \in \ePS{-2}(M; TM)$,
  $\widehat \maC(\tau) = \widehat \maB^{*}(\tau) \in \ePS{-1}(M; TM, \CC)$,
  and $\widehat \maD(\tau) \in \ePS{0}(M)$.
\end{remark}

\begin{definition} \label{def.lp.tau}
  Assume that $\newtilde V \succ 0$ and that $\newtilde V_{0} \succ 0$ on $\mfkS$
  \lpar Assumption \ref{assumpt.VV0weak}\rpar.
  Let $\bsh\in H^{s}(\Gamma'; TM)$, $s \in \RR$. The {\it indicial layer potentials
  operators} associated to $\widehat \bsXi(\tau)$ are defined by
  \begin{equation*} 
    \maS_{\tau}(\bsh) \ede
    \left(
      \begin{array}{c}
        \maV_{\tau}(\bsh)\\
        \maP_{\tau}(\bsh)
      \end{array}
      \right)
      \ede \widehat \bsXi(\tau)^{-1} \, \left[\cvector{\bsh}{0}
      \otimes \delta_{\Gamma}\right]\,,
  \end{equation*}
  and by
  \begin{equation*} 
    \maD_{\tau}(\bsh) \ede
    \left(
      \begin{array}{c}
        \maW_{\tau}(\bsh)\\
        \maQ_{\tau}(\bsh)
      \end{array}
      \right)
      \ede \widehat \bsXi(\tau)^{-1} \, \big[ \widehat \bopstar(\tau)
      \left(
        \bsh\otimes \delta_{\Gamma}
      \right) \big]\,.
  \end{equation*}
\end{definition}

We have the following analogue of Remark \ref{rem.components}.

\begin{remark}\label{rem.components.tau}
  We have
  \begin{equation*} 
    \maW_{\tau}(\bsh) \seq
    \left(
      \begin{array}{cc}
        \widehat \maA(\tau)  & \widehat \maB (\tau)
      \end{array}
    \right)
    \left( \begin{array}{c}
        - 2\Dnu^{*} \\
        \bsnu^{\sharp}
      \end{array}
    \right)
    (\bsh\otimes \delta_{\Gamma })
    \seq (- 2 \widehat \maA (\tau)  \Dnu^{*} + \widehat \maB (\tau) \bsnu^{\sharp})\,
    (\bsh\otimes \delta_{\Gamma })
  \end{equation*}
  and
  \begin{equation*} 
    \maQ_{\tau}(\bsh)
    \seq \left(
      \begin{array}{cc}
        \widehat \maC (\tau) & \widehat\maD (\tau)
      \end{array}
      \right)
      \left( \begin{array}{c}
        - 2\Dnu^{*} \\
        \bsnu^{\sharp}
      \end{array}
      \right)
      (\bsh\otimes \delta_{\Gamma }) \\
       \ \seq (- 2 \widehat \maC (\tau) \Dnu^{*} + \widehat \maD (\tau) \bsnu^{\sharp})\,
      (\bsh\otimes \delta_{\Gamma })\,.
\end{equation*}
  Similarly, $\maS_{\tau}(\bsh) \seq \big(\, \widehat \maA (\tau)
  (\bsh\otimes \delta_{\Gamma }) \ \ \
  \widehat \maC (\tau) (\bsh\otimes \delta_{\Gamma })\, \big)^{\top }$.
\end{remark}

The following result follows from the definition of the indicial layer potentials
and it has the same proof as Proposition \ref{prop.compatibility}

\begin{proposition}\label{prop.compatibility.tau}
  Let $\tau \in \RR$ and $\bsh' \in L^{2}(\Gamma'; TM)$.
  Then $\widehat \bsXi (\tau) \maS_{\tau}(\bsh')$
  and $\widehat \bsXi (\tau) \maD_{\tau}(\bsh')$ vanish on
  $\mfkS \smallsetminus {\Gamma'}$.
\end{proposition}

The following result result for the layer potential operators associated to
$\widehat \bsXi(\tau)$ and its proof are similar to those of Proposition
\ref{prop.H2estimates}.

\begin{proposition}\label{prop2.H2estimates}
  Let $\Omega_{+}' := \Omega' \subset \mfkS$, $\Omega_{-}' := \mfkS \smallsetminus
  \overline{\Omega}'$, and $\Gamma' := \pa \Omega'$, as before. Let $s\in \RR$ and
  $\bsh' \in H^{s}(\Gamma'; TM)$. Then
  \begin{equation*}
    \maV_{\tau}(\bsh')\vert_{\Omega_{\pm}'}
    \in H^{s + \frac32}(\Omega_{\pm}'; TM) \ \mbox{ and }\
    \maP_{\tau}(\bsh')\vert_{\Omega_{\pm}'} \in H^{s +\frac12}(\Omega_{\pm}')\,.
  \end{equation*}
  Similarly, let $\bsh \in H^{s}(\Gamma'; TM)$. Then
  \begin{equation*}
    \maW_{\tau}(\bsh')\vert_{\Omega_{\pm}'} \in H^{s + \frac12}(\Omega_{\pm}'; TM)\
    \mbox{ and }\
    \maQ_{\tau}(\bsh')\vert_{\Omega_{\pm}'} \in H^{s - \frac12}(\Omega_{\pm}')\,.
  \end{equation*}
\end{proposition}

Recall the \emph{indicial layer potentials} of Definition \ref{def.lp.tau}.
Then we have the following result.

\begin{proposition}\label{prop2.jump.rel}
  Let $\bsh' \in L^{2}(\Gamma'; TM)$, $\Gamma' := \pa \Omega'$,
  $\newtilde{\mathfrak f} := \frac{\newtilde V_{0}+1}{2 \newtilde V_{0}+1}$
  $\newtilde{\mathfrak g} := \frac1{2 \newtilde V_{0} + 1}$. Let $\bsK := \bsP_{0}
  \in \Psi^{0}(\Gamma'; TM)$ be as in Equation \eqref{eq.def.bsk}. Then
  \begin{enumerate}[\rm (i)]
    \item $\maV_{\tau}(\bsh')_{\pm} \seq \widehat \maA_{0}(\tau) \bsh'$,
    and $\widehat \maA_{0}(\tau)$ is elliptic with self-adjoint symbol
    $\sigma_{-1}(\maA_{0}(\tau); \xi') \seq
      \frac1{4|\xi'|}(2 - \newtilde{\mathfrak f} \bsnu \otimes \bsnu^{\sharp}
      - \newtilde{\mathfrak f} \eta^{\sharp} \otimes \eta)
      \,,$
    where $\eta := |\xi'|^{-1}\xi'$.

    \item $\maP_{\tau}(\bsh')_{\pm} = \big(\mp \frac {\newtilde{\mathfrak g}}2
      \bsnu^{\sharp} + \widehat \maC_{0}(\tau) \big) \bsh'$.

    \item $\maW_{\tau}(\bsh')_{\pm} \ede \big[ \widehat \bsP (\tau)
    (\bsh \otimes \delta_{\Gamma})\big]_{\pm} \seq \left [\pm \frac12
    + \widehat \bsK (\tau)\right ]
    \bsh'$ and the two operators $\pm \frac12  + \widehat \bsK(\tau)$ are elliptic
    and have self-adjoint principal symbols.

    \item $[\bop \maS_{\tau}(\bsh')]_{\pm} \seq \left(\mp \frac12
      + \widehat \bsK{}^{*} (\tau)\right)\bsh',$ where $\bsK = \bsP_{0}
      := (-2 \maA \Dnu^{*} + \maB \bsnu^{\sharp})_{0}$, as in Theorem \ref{thm.jump.K1}.
  \end{enumerate}
\end{proposition}

\begin{proof}
  The result follows from the ``jump relations'' established in Theorems \ref{thm.jump.K1}
  and \ref{thm.jump.rel}   and the compatibility of the normal limits with the indicial
  operators, Theorem \ref{thm2.indicial.jump1}.
\end{proof}

\subsection{Invertibility of $\bsS$ on a cylinder}
{\cn We keep the assumptions of the previous subsection. That is,
\begin{itemize}
  \item $M = \mfkS \times \RR$,
  \item $V$ and $V \ge 0$ are smooth, translation invariant,
  and hence $V(x', t) = \newtilde V(x')$ and $V_{0}(x', t) = \newtilde V_{0}(x')$ identify to
  sections on $\mfkS$.
  \item $V$ and $V_{0}$ satisfy the positivity assumption at
  infinity (Assumption \ref{assumpt.VV0weak}, that is,
  $\newtilde V \succ 0$ and $\newtilde V_{0} \succ 0$).
\end{itemize}}
(Recall that $\succ$ was introduced in Definition \ref{def.succ}.) We also keep the notation
and assumptions of equations \eqref{eq.def.Gamma'}, \eqref{eq.def.Omega'}, and
\eqref{eq.1side.Omega'}.

\emph{\cn From now on, all layer potential operators $\bsS$, $\bsK$, ...
\lpar possibly decorated with indices\rpar\ are
associated to the boundaries $\Gamma$ and $\Gamma'$
of the open sets $\Omega =: \Omega_{+}$ and $\Omega' =: \Omega_{+}'$,}
see Definitions \ref{def.lp2}, \ref{def.lp2}, and \ref{def.lp.tau}.

We now prove the invertibility of the single layer potential operator
$\bsS := \maA_{0}$, see Theorem \ref{thm.jump.rel} on our
cylinder $M = \mfkS \times \RR$, if $V$ and $V_{0}$ satisfy the positivity assumption at infinity
(Assumption \ref{assumpt.VV0weak}). In order to do that, we need to first study its indicial
operators $\widehat{\bsS\,} (\tau)$, $\tau \in \RR$. We begin with the case $\tau = 0$, which
requires more assumptions than the case $\tau \neq 0$.

\begin{proposition}\label{prop.Stau=0}
  Let us assume that $M = \mfkS \times \RR$ \lpar is a cylinder\rpar, that $V$ and $V_{0}$
  are smooth, translation invariant at infinity, that $\Omega = \Omega' \times \RR$,
  and that $\Omega' \subset \mfkS$ is on one side of
  its boundary $\Gamma' := \pa \Omega'$ (i.e., $\Omega'$ is the interior of
  $\overline{\Omega'}' \subset \mfkS$). We also assume that
  \begin{enumerate}[\rm (i)]
    \item $\newtilde V \succ 0$ on $\mfkS$
    and
    \item $\newtilde V_{0} \succ 0$ on $\mfkS \smallsetminus \Gamma '$.
  \end{enumerate}
  Let $\widehat{\bsS\,}(0) \in \Psi^{-1}(\Gamma'; TM)$ be
  the indicial operator of the operator $\bsS := \maA_{0} \in \Psi^{-1}(\Gamma; TM)$ of
  Theorem \ref{thm.jump.rel}. Then $\widehat{\bsS\,}(0)$
  is invertible and $\widehat{\bsS\,}(0)^{-1} \in \Psi^{1}(\Gamma'; TM)$.
\end{proposition}

\begin{proof}
  Our assumptions imply that $V$ and $V_{0}$ satisfy the non-vanishing conditions
  of Assumption \ref{assumpt.VV0} on $M = \mfkS \times \RR$, and hence the layer
  potential operator $\bsS := \maA_{0} \in \ePS{-1}(\Gamma; TM)$ is defined
  (see Proposition \ref{prop.invert.inverse}).
  We know from Proposition \ref{prop2.jump.rel} that $\widehat{\bsS\,} (0)$
  is elliptic with self-adjoint principal symbol. Because $\mfkS$ is compact,
  $\widehat{\bsS\,} (0)$ is then Fredholm of index zero, by classical results
  (the ``classical case'' of a closed manifold in Theorem \ref{thm.ADN.Fredholm}).

  Let us show now that $\ker \widehat{\bsS\,}(0) = 0$. Let then $\bsh' \in L^{2}(\Gamma' ; TM)$
  be such that $\widehat{\bsS\,}(0) \bsh' = 0$. Our assumptions imply that
  $\widehat \bsXi(0)$ is invertible (Proposition \ref{prop2.invert3}), and hence the
  {\it indicial single layer potential} $U = (\, \bsu \ \ p\, )^{\top} := \maS_{0}(\bsh')$
  is defined (see Definition \ref{def.lp.tau}).
  Since $\widehat{\bsS\,}(0)$ is elliptic and $\Gamma'$ is smooth and compact, we have that
  $\bsh' \in H^{s}(\Gamma'; TM)$ for all $s \in \RR$, by classical elliptic regularity.
  Proposition \ref{prop2.H2estimates} for $s = 1/2$ then gives that
  the restrictions of $U$ to $\Omega_{+}' := \Omega'$ and to $\Omega_{-}'
  := \mfkS \smallsetminus \overline \Omega{\,}'$ satisfy
  \begin{equation}\label{eq.needed.est}
    U\vert_{\Omega_{\pm}'} \ede (\, \bsu \ \ p\, )^{\top}\vert_{\Omega_{\pm}'}
    \ede \maS_{0}(\bsh') \vert_{\Omega_{\pm}'}
    \in H^{2}(\Omega_{\pm}'; TM) \oplus H^{1}(\Omega_{\pm}')\,.
  \end{equation}
  We also know that $\widehat \bsXi(0) U = 0$ in $\Omega_{\pm}'$ by Proposition
  \ref{prop.compatibility}. Proposition \ref{prop2.jump.rel} gives then that
  $\bsu_{+} = \bsu_{-} = \widehat{\bsS\,}(0) \bsh' = 0$.
  Therefore $\bsu = 0$ in $\Omega_{-}'$ and in $\Omega_{+}'$, by Corollary \ref{cor2.e.est3}.
  The same corollary gives that $p$ is constant on each connected component
  of $\mfkS \smallsetminus \Gamma'$. Because
  $\newtilde V_{0}$ is not identically equal to zero on any connected component of
  $\mfkS \smallsetminus \Gamma'$, we obtain that, in fact, $p = 0$. We have thus obtained
  that $U = 0$ and hence $\widehat{\bsS\,}(0)$ is injective.

  Finally, the fact that $\widehat{\bsS\,}(0)^{-1} \in \Psi^{1}(\Gamma'; TM)$
  follows from Beals' theorem (see, for instance, Theorem 7.4 in \cite{KNW-22}).
\end{proof}

The same result holds for $\tau \neq 0$ under the weaker hypothesis that
$V \succ 0$ and $V_{0} \succ 0$ on $M$
(equivalently, that $\newtilde V \succ 0$ and $\newtilde V_{0} \succ 0$ on $\mfkS$).

\begin{proposition}\label{prop.Sneq0}
  Let us assume that $M = \mfkS \times \RR$ is a cylinder and that $V$ and $V_{0}$ are
  smooth, translation invariant, and $\newtilde V \succ 0$ and $\newtilde V_{0} \succ 0$
  on $\mfkS$. We furthermore assume that $\Omega'$ is the interior of $\overline{\Omega '}$.
  Let $\widehat{\bsS\,}(\tau) \in \Psi^{-1}(\Gamma'; TM)$ be the
  indicial operator of the operator $\bsS := \maA_{0} \in \Psi^{-1}(\Gamma; TM)$
  of Theorem \ref{thm.jump.rel}. If $\tau \neq 0$, then $\widehat{\bsS\,}(\tau)$ is
  invertible and $\widehat{\bsS\,}(\tau)^{-1} \in \Psi^{1}(\Gamma'; TM)$.
\end{proposition}

\begin{proof}
  The proof is exactly the same as that of Proposition \ref{prop.Stau=0},
  except that we are using Corollary \ref{cor2.e.est2} (which only
  requires $\newtilde V$ and $\newtilde V_{0}$ to be non-negative) instead of
  Corollary \ref{cor2.e.est3}.
\end{proof}

We can now prove the invertibility of the single layer potential $\bsS$
on $\Gamma = \Gamma' \times \RR$ in the case of cylinders and under Assumption
\ref{assumpt.VV0weak}.

\begin{theorem}\label{thm.cyl.S}
  Let us assume that $M = \mfkS \times \RR$ is cylinder, that $V$ and $V_{0}$
  are smooth, translation invariant at infinity, and $\newtilde V \succ 0$ on
  $\mfkS$ and $\newtilde V_{0} \succ 0$ on $\mfkS \smallsetminus \Gamma'$.
  We furthermore assume that $\Omega'$ is the interior of $\overline{\Omega '}$.
  Let $\bsS := \maA_{0} \in \ePS{-1}(\Gamma; TM)$ be as in Theorem \ref{thm.jump.rel}.
  Then $\bsS$ is invertible, $\bsS^{-1} \in \ePS{1}(\Gamma; TM)$, and hence
  $\bsS : H^{s}(\Gamma; TM) \to H^{s+1}(\Gamma; TM)$ is an isomorphism for
  all $s \in \RR$.
\end{theorem}

\begin{proof}
  Propositions \ref{prop.Stau=0} and \ref{prop.Sneq0}
  show that $\widehat \bsS(\tau)$ is invertible for all $\tau \in \RR$.
  It is well-known from the work of Melrose, Schulze and others (see also Theorem 5.11 in
  \cite{KohrNistor-Stokes}) that $\bsS$ is then invertible.
\end{proof}

\subsection{Invertibility of $\frac12  + \bsK$ on a cylinder}
In this subsection, {\cn we keep the assumptions of the previous two subsections.
In particular, we assume that $M = \mfkS \times \RR$ is a cylinder, that
$V$ and $V_{0}$ are smooth, translation invariant at infinity, and
$\newtilde V \succ 0$ and $\newtilde V_{0} \succ 0$.}

In the following, we will need the operators
$\bsP := - 2 \maA  \Dnu^{*} + \maB \bsnu^{\sharp}$ and
$\bsK \ede \bsP_{0}$. The first one is needed in the definition of the double layer
potential operator $\maW_{\rm{ST}}$ and the second one was
introduced in Equation \eqref{eq.def.bsk} and studied in Theorem \ref{thm.jump.K1}.
Each of these operators depends on the potentials $V$ and $V_{0}$.
(The correspondence $\bsP  \mapsto \bsP _{0}$ is the basic correspondence studied,
for example, in Theorems \ref{thm.main.jump0} and \ref{thm.main.jump1}.)

The proof of the invertibility of $\frac12  + \bsK$ (under suitable
conditions on $V$ and $V_{0}$) requires the invertibility of its indicial operators.
As for the operator $\bsS$, we first consider the case $\tau = 0$.

\begin{proposition}\label{prop.Ktau=0}
  Let us assume that $M = \mfkS \times \RR$ is a cylinder, that $V$ and $V_{0}$ are smooth,
  translation invariant at infinity, that $\newtilde V \succ 0$ on $\Omega_{-}'$ and
  that $\newtilde V_{0} \succ 0$ on $\Omega_{+}' := \Omega'$. We furthermore assume
  that $\Omega'$ is on one side of its boundary, as usual \lpar i.e., $\Omega'$ is
  the interior of $\overline{\Omega}'$, see Assumption \ref{assumpt.1side}\rpar. Let
  $\widehat \bsK(0) \in
  \Psi^{0}(\Gamma'; TM)$ be the indicial operator of the operator $\bsK \in \ePS{0}(\Gamma; TM)$
  given in Theorem \ref{thm.jump.K1} {\lpar}see also Equation \eqref{eq.matching.ce0}\rpar.
  Then $\frac12  + \widehat \bsK(0)$ is invertible on $L^2(\Gamma'; TM)$ and
  $\big [\frac12  + \widehat \bsK(0) \big]^{-1} \in \Psi^{0}(\Gamma'; TM)$.
\end{proposition}

Recall that by the statement ``$\phi \not \equiv 0$
on $A$,'' we mean that
there exists $a$ in the domain of $\phi$ such that $\phi(a) \not = 0$. To
negate this statement, we shall write ``$\phi = 0$ in $A$.''
Of course, in view of our conventions, the condition $\newtilde V_{0} \not\equiv 0$
on $\Omega'$ implies the condition $V_{0} \not\equiv 0$
on $\Omega$ (since $V$ and $V_{0}$ are translation invariant at infinity).

\begin{proof}
  We know that that the operator $\bsK \ede \bsP_{0}
  \ede \big(- 2 \maA  \Dnu^{*} + \maB \bsnu^{\sharp})_{0}$
  satisfies $\bsK \in \ePS{0}(\Gamma; TM)$, by Theorem \ref{thm.main.jump1}.
  Therefore $\widehat \bsK(0) \in \Psi^{0}(\Gamma'; TM)$, by the definition of the
  indicial family. We also know from Theorem \ref{thm.jump.K1}
  that $\frac12  + \bsK$ is elliptic with self-adjoint principal
  symbol. Consequently, $\frac12  + \widehat \bsK(0)$ is also
  elliptic with self-adjoint principal
  symbol. Because $\Gamma'$ is compact, $\frac12  + \widehat \bsK(0)$ is then
  Fredholm of index zero by classical results \cite{Hormander3, H-W}. Therefore,
  in order to prove that $\frac12 + \widehat \bsK(0)$ is invertible, it suffices
  to prove that its adjoint operator
  $\frac12 + \widehat \bsK(0)^{*}$ is injective.

  Let then $\bsh' \in L^{2}(\Gamma' ; TM)$ be such that
  $\big[\frac12 + \widehat \bsK(0)^{*}\big] \bsh' = 0$. Since
  $\frac12 + \widehat \bsK(0)^{*}$ is elliptic and $\Gamma'$
  is smooth and compact, we have that $\bsh' \in H^{s}(\Gamma'; TM)$
  for all $s \in \RR$, by classical elliptic regularity. Let us consider the
  \emph{\underline{indicial} single layer potential} $U := \maS_{0}\bsh'$
  (see Definition \ref{def.lp.tau} with $\tau = 0$).
  Proposition \ref{prop2.H2estimates} then gives that the restrictions of
  $U$ to $\Omega_{+}' := \Omega'$ and to $\Omega_{-}'
  := \mfkS \smallsetminus \overline{\Omega}'$
  satisfy
  \begin{equation*} 
    U\vert_{\Omega_{\pm}'} \ede (\, \bsu \ \ p\, )^{\top}\vert_{\Omega_{\pm}'}
    \ede \maS_{0}(\bsh')\vert_{\Omega_{\pm}'}
    \in H^{2}(\Omega_{\pm}'; TM) \oplus H^{1}(\Omega_{\pm}')\,.
  \end{equation*}
  We also have that $\widehat \bsXi(0) U = 0$ in $\Omega_{\pm}'$ by Proposition
  \ref{prop.compatibility}. Proposition \ref{prop2.jump.rel} then gives that
  $[\widehat \bop(0) U]_{-} = \big[\frac12 + \widehat \bsK(0)^{*}\big]\bsh' = 0$.
  (Recall that $[\widehat \bop(0) U]_{-}$ is the trace at $\Gamma'$ of
  the restriction of $\widehat \bop(0) U$ to the domain $\Omega_{-}'$.)

  Let next $\Omega_{0}$ be a connected component of $\Omega_{-}'$.
  Our assumptions give that $\newtilde V \not\equiv 0$ on $\Omega_{0}$.
  Hence we obtain that $\bsu = 0$
  on $\Omega_{0}$, by Corollary \ref{cor2.e.est3}. The same corollary gives
  that $p$ is constant in $\Omega_{0}$.
  Let us now show that this constant is actually 0.
  Indeed, the relation
  \begin{equation*}
    0 \seq [\widehat \bop(0) U]_{-} \ede  [-2 \widehat \Dnu (0)\bsu + p \bsnu]_{-}
    \seq p_{-} \bsnu
  \end{equation*}
  on $\Gamma$ implies that $p_{-} = 0$, and hence $p = 0$ on $\Omega_{0}$.
  Since $\Omega_{0}$ was an arbitrary connected component
  of $\Omega_{-}'$, we obtain that $U = 0$ on $\Omega_{-}'$.

  We have already observed that $\widehat \bsXi(0) U = 0$ in $\Omega_{+}'$.
  Proposition \ref{prop2.jump.rel} then gives the ``no-jump relation''
  $\bsu_{+} = \bsu_{-} = 0$ at the boundary $\Gamma' := \partial \Omega'$ of
  $\Omega'$ (interior and exterior traces).
  Corollary \ref{cor2.e.est3} then gives that $\bsu=0$ on $\Omega'$.
  Because $\newtilde V_{0} \not \equiv 0$ on {every connected component of }$\Omega'$,
  the same corollary gives that $p = 0$ on $\Omega'$.  We have thus obtained
  that $U = 0$ on both $\Omega_{\pm}'$, and hence $U = 0$ on $\mfkS$.


  We next use relation (iv) of Proposition \ref{prop2.jump.rel} to conclude
  that
  \begin{align*}
    \bsh' & \seq \left(\frac12
    + \widehat \bsK{}^{*} (\tau)\right)\bsh' - \left(-\frac12
    + \widehat \bsK{}^{*} (\tau)\right)\bsh'\\
    & \seq [\bop \maS_{\tau}(\bsh')]_{-} - [\bop \maS_{\tau}(\bsh')]_{+}\\
    & \seq [\bop U]_{-} - [\bop U]_{+} \seq 0 \,.
  \end{align*}
  Therefore $\widehat{\bsS}(0)$ is injective, and hence it is
  invertible.

  The fact that $\big[\frac12 + \widehat \bsK(0)\big]^{-1} \in \Psi^{0}(\Gamma'; TM)$
  is also a classical result on pseudodifferential operators
  \cite{BealsSpInv} (see also Theorem \ref{thm.spectral.inv};
  a proof can also be found in \cite{KMNW-2025}, Theorem 15.4.11, and in
  \cite{KNW-22}, Theorem 7.4). This completes the proof.
\end{proof}

As for the single layer potential, the case $\tau \neq 0$ requires weaker
assumptions on the potentials.

\begin{proposition}\label{prop.Kneq0}
  Let us assume that $M = \mfkS \times \RR$ is a cylinder $V$ and $V_{0}$ are smooth,
  translation invariant, and that $\newtilde V \succ 0$ and $\newtilde V_{0} \succ 0$
  on $\mfkS$. We furthermore assume that
  $\Omega'$ is the interior of $\overline{\Omega '}$. Let
  $\widehat \bsK(\tau) \in \Psi^{0}(\Gamma'; TM)$ be the indicial operator of the
  operator $\bsK \in \Psi^{0}(\Gamma; TM)$ of Theorem \ref{thm.jump.K1} (see also Equation
  \eqref{eq.matching.ce0}). If $\tau \neq 0$, then $\frac12 + \widehat \bsK(\tau)$ is
  invertible on $L^2(\Gamma'; TM)$ and $\big [\frac12 + \widehat \bsK(\tau) \big]^{-1}
  \in \Psi^{0}(\Gamma'; TM)$, see Theorem \ref{thm.spectral.inv}.
\end{proposition}

\begin{proof}
  The proof follows the same arguments as those in the
  proof of Proposition \ref{prop2.Ktau=0}, except that, instead of
  Corollary \ref{cor2.e.est3}, we use Corollary \ref{cor2.e.est2}.
\end{proof}

We are ready to prove the invertibility of the double layer potential
operator $\frac12 + \bsK$ on the cylinder $\Gamma := \Gamma' \times \RR$.

\begin{theorem}\label{thm.cyl.K}
  Let us assume that $M = \mfkS \times \RR$ is cylinder, that $V$ and $V_{0}$
  are smooth, translation invariant at infinity, and that $\newtilde V \succ 0$ on
  $\Omega_{-}'$ and $\newtilde V_{0} \succ 0$ on $\Omega_{+}':= \Omega'$.
  We furthermore assume that $\Omega'$ is the interior of $\overline{\Omega '}$,
  as usual. Let $\bsK \in \ePS{0}(\Gamma; TM)$ be as in
  Theorem \ref{thm.jump.K1} (see also Equation \eqref{eq.matching.ce0}).
  Then $\frac12 + \bsK $ is
  invertible on $L^2(\Gamma ; TM)$ and $\big [\frac12 + \bsK \big]^{-1}
  \in \Psi^{0}(\Gamma ; TM)$.
\end{theorem}

\begin{proof}
  The proof is the same as that of Theorem \ref{thm.cyl.S},
  but using Propositions \ref{prop.Ktau=0} and \ref{prop.Kneq0}
  instead of Propositions \ref{prop.Stau=0} and \ref{prop.Sneq0}.
\end{proof}


\section{\cn Fredholmness and invertibility of layer potential operators
on manifolds with straight cylindrical ends}
\label{sec.Fredholmness}
{\cn We now come back to the case of a general manifold
with straight cylindrical ends $M = \canDec$ (Definition \ref{def.cyl.end}),
where $M' := \pa M_{0}$ is compact,
as usual. We will also assume that $V$ and $V_{0}$ satisfy the corresponding
non-vanishing assumptions \lpar Assumption \ref{assumpt.VV0}\rpar.
Also, we assume that we are given a \emph{compatible} open subset
$\Omega \subset M \seq  \canDec$,
as in Equation \eqref{eq.def.Omega}.} Recall that this means that $\Omega$
also has straight cylindrical end at infinity, which are assumed to ``match''
that of $M$, in the sense that there exists $R_{\Omega} \le 0$ and an open subset
$\Omega' \subset M$ such that
\begin{equation}\label{eq.matching.ce}
  \Omega \cap \big[ M' \times (-\infty, R_{\Omega})\big]
  \seq \Omega' \times (-\infty, R_{\Omega})\,.
\end{equation}
Then $\Gamma := \pa \Omega$ is also a manifold with straight
cylindrical ends. Let $\Gamma' := \pa \Omega' \subset M'$, we thus have
also the compatibility relation
\begin{equation}\label{eq.matching.ce2}
  \Gamma \cap \big[M' \times (-\infty, R_{\Omega})\big]
  \seq \Gamma' \times (-\infty, R_{\Omega})\,,
\end{equation}
(compare to Equation \eqref{eq.matching.ce0}).
We can assume, in fact, that $R_{\Omega} = 0$ in the last equations.
We set $\Omega_{-} := M \smallsetminus \overline \Omega$,
$\Omega_{-}' := M' \smallsetminus \overline \Omega{\,}'$, $\Omega_{+} := \Omega$,
and $\Omega_{+}' := \Omega'$. In the results of this section, we can split our
generalized Stokes operator as a direct sum according
to the connected components of $\Omega$, so there is no loss of generality
to assume that {\it $\Omega$ is connected}. (We have assumed that $M$ is
connected for the same reason.) The general case follows immediately from this
particular case.

\subsection{Invertibility of $\bsS$}
We first prove the invertibility of the single layer potential operator
$\bsS := \maA_{0}$, see Theorem \ref{thm.jump.rel}. In order to do
that, we need to first study its indicial operators $\widehat{\bsS\,} (\tau)$,
$\tau \in \RR$. We begin with the case $\tau = 0$, which requires more
assumptions than the case $\tau \neq 0$. Recall the definition of $\succ$
(Definition \ref{def.succ}). We notice that Assumption \ref{assumpt.VV0}
implies Assumption \ref{assumpt.VV0weak} on $M' \times \RR$.

\begin{proposition}\label{prop2.Stau=0}
  Let us assume that $M$ is a manifold with straight cylindrical ends and that
  $V$ and $V_{0}$ are smooth, translation invariant at infinity,
  and that $\newtilde V \succ 0$ on $M'$ and $\newtilde V_{0} \succ 0$ on
  $M' \smallsetminus \Gamma'$. Let $\widehat{\bsS\,}(0) \in \Psi^{-1}(\Gamma'; TM)$ be the
  indicial operator of the operator $\bsS := \maA_{0} \in \Psi^{-1}(\Gamma; TM)$ of
  Theorem \ref{thm.jump.rel}. Then $\widehat{\bsS\,}(0)$
  is invertible and $\widehat{\bsS\,}(0)^{-1} \in \Psi^{1}(\Gamma'; TM)$.
\end{proposition}

\begin{proof}
  The hypotesis imply that the operator $\bsXi$ is Fredholm, and hence
  the layer potentials are defined. The result then follows from Proposition
  \ref{prop.Stau=0} applied to $\In(\bsXi)$ acting on $M' \times \RR$
  and the fact that layer potentials commute with limit operators,
  Theorem \ref{thm2.indicial.jump1}.
\end{proof}

The same result holds for $\tau \neq 0$ under the weaker hypothesis that
$\newtilde V \succ 0$ and $\newtilde V_{0} \succ 0$.

\begin{proposition}\label{prop2.Sneq0}
  Let us assume that $M$ is a manifold with straight cylindrical
  ends and that $V$ and $V_{0}$ are smooth, asymptotically translation invariant
  at infinity, and $\newtilde V \succ 0$ and $\newtilde V_{0} \succ 0$ on $M'$.
  Let $\widehat{\bsS\,}(\tau) \in \Psi^{-1}(\Gamma'; TM)$ be the
  indicial operator of the operator $\bsS := \maA_{0} \in \Psi^{-1}(\Gamma; TM)$
  of Theorem \ref{thm.jump.rel}. If $\tau \neq 0$, then $\widehat{\bsS\,}(\tau)$ is
  invertible and $\widehat{\bsS\,}(\tau)^{-1} \in \Psi^{1}(\Gamma'; TM)$.
\end{proposition}

\begin{proof}
  The proof is the same as that of Proposition \ref{prop2.Stau=0},
  but using Proposition \ref{prop.Sneq0} instead of Proposition \ref{prop.Stau=0}.
  (As in that proof, the assumptions $\newtilde V \succ 0$ and $\newtilde V_{0} \succ 0$ on $M'$
  are needed for $\bsXi$ to be Fredholm, so that the layer potentials are defined.)
\end{proof}

The above propositions lead to the following useful corollary.

\begin{corollary}\label{cor2.SFred}
  Let us assume that $M = \canDec$ is a manifold with straight cylindrical ends,
  that $V$ and $V_{0}$ are smooth, translation invariant at infinity,
  and that $\Omega'$ is the interior of $\overline{\Omega}'$.
  We also assume that $\newtilde V \succ 0$ on $M'$
  and $\newtilde V_{0} \succ 0$ on $M' \smallsetminus \Gamma'$.
  Then $\bsS$ is an unbounded Fredholm operator on $L^2(\Gamma ; TM)$.
  Equivalently, it is a \lpar bounded\rpar\ Fredholm operator
  $\bsS : H^{-1/2}(\Gamma; TM) \to H^{1/2}(\Gamma; TM)$.
\end{corollary}

\begin{proof}
  The operator $\In(\bsS)$ is the single layer potential
  operator associated to $\In(\bsXi_{V, V_{0}}) = \bsXi_{\newtilde V, \newtilde V_{0}}$,
  by Theorem \ref{thm2.indicial.jump1}.
  The assumptions of the corollary give that the assumptions
  of Theorem \ref{thm.cyl.S} are therefore satisfied by $\In(\bsS)$, which
  is hence invertible. Because $\bsS$ is elliptic, Theorem \ref{thm.ADN.Fredholm}
  gives that $\bsS$ is Fredholm.
\end{proof}

We can now prove the invertibility of the single layer potential $\bsS$
on $\Gamma$.

\begin{theorem}\label{thm2.S}
  Let us assume that $M, V$, and $V_{0}$ satisfy the non-vanishing assumptions
  \ref{assumpt.VV0} \lpar hence, in particular, $M$ is with straight cylindrical
  ends and $\newtilde V \succ 0$ on $M'$\rpar. Let $\bsS := \maA_{0} \in \ePS{-1}(\Gamma; TM)$
  be as in Theorem \ref{thm.jump.rel}. We furthermore assume the following conditions.
  \begin{enumerate}[\rm (i)]
    \item $\newtilde V_{0} \succ 0$ on $M' \smallsetminus \Gamma'$.

    \item $V_{0} \succ 0$ on $M \smallsetminus \Gamma$.
    \item $\Omega$ is on one side of its boundary $\Gamma$ \lpar i.e., $\Omega$ is
    the interior of $\overline{\Omega}$\rpar.
  \end{enumerate}
  Then $\bsS$ is invertible, $\bsS^{-1} \in \ePS{1}(\Gamma; TM)$, and hence
  $\bsS : H^{s}(\Gamma; TM) \to H^{s+1}(\Gamma; TM)$ is an isomorphism for
  all $s \in \RR$.
\end{theorem}

\begin{proof}
  Because $V$ and $V_{0}$ satisfy the non-vanishing assumptions
  \ref{assumpt.VV0} on $M$, Theorem \ref{thm2.main.invert} implies that the operator
  $\bsXi$ is invertible, and hence the layer potentials are defined. In
  particular, the single layer potential operator
  $\bsS := \maA_{0} \in \ePS{-1}(\Gamma; TM)$ is defined.
  In addition, Theorem \ref{thm.jump.rel} shows that $\bsS$ is elliptic and
  self-adjoint from $H^{-1/2}(\Gamma; TM)$ to $H^{1/2}(\Gamma; TM)$ with respect
  to the $L^{2}$-inner product.
  We know from Corollary \ref{cor2.SFred} just proved that
  $\bsS : H^{-1/2}(\Gamma; TM) \to H^{1/2}(\Gamma; TM)$ is a Fredholm operator.
  The self-adjoint property of $\bsS$ shows that its index is zero.
  Therefore, if we show that this operator is injective, then it will be invertible.

  To this end, let $\bsh \in H^{-1/2}(\Gamma; TM)$ be such that $\bsS \bsh = 0$ and let $U =
  (\, \bsu \ \ p\, )^{\top}:= \maS_{\rm{ST}}(\bsh)$ be the corresponding \emph{single layer
  potential}. The ellipticity of $\bsS \in \ePS{-1}(\Gamma; TM)$, the ellipticity property of
  Theorem \ref{thm.jump.rel}, and elliptic regularity (see Theorem \ref{thm.regularity} or
  \cite[Theorem 15.3.30]{KMNW-2025} or \cite[Proposition 6.5]{KNW-22}) assure that
  $\bsh \in H^{s}(\Gamma; TM)$ for all $s \in \RR$. Then, by
  Proposition \ref{prop.H2estimates},
  the restrictions of $U$ to $\Omega_{+} := \Omega$ and to $\Omega_{-}$
  satisfy the property
  \begin{equation}\label{eq.needed.est-2}
    U\vert_{\Omega_{\pm}} \ede (\, \bsu \ \ p\, )^{\top}\vert_{\Omega_{\pm}}
    \ede \maS_{\rm{ST}}(\bsh) \vert_{\Omega_{\pm}}
    \in H^{2}(\Omega_{\pm}; TM) \oplus H^{1}(\Omega_{\pm})\,.
  \end{equation}
  We also know that $\bsXi U = 0$ in $\Omega_{\pm}$ by Proposition
  \ref{prop.compatibility}. Theorem \ref{thm.jump.rel} gives then that
  $\bsu_{+} = \bsu_{-} = \bsS \bsh = 0$, and by Corollary \ref{cor.e.est},
  $\bsu = 0$ in $\Omega_{-}$ and in $\Omega_{+}$.
  The Stokes equation and again Corollary \ref{cor.e.est} give that $p$ is
  constant on each connected component of $M \smallsetminus \Gamma$, but
  this constant is $0$, because
  $V_{0}$ is not identically equal to zero on any connected component of
  $M \smallsetminus \Gamma$. Thus, $U = 0$. This implies that
  $[\tbop U]_{\pm} = 0$ on $\Gamma$, and Theorem \ref{thm.jump.rel}(iii) gives
  \begin{equation*}
    \bsh \seq \Big(\frac12  + K^{*} \Big )\bsh -
    \Big(-\frac12  + K^{*} \Big )\bsh \seq [\tbop U]_{-} - [\tbop U]_{+} \seq 0\,.
  \end{equation*}
  Consequently, $\bsS$ is injective and, in view of the above arguments, it
  is invertible.
  Finally, the property that $\bsS^{-1} \in \ePS{1}(\Gamma; TM)$ follows from
  Theorem \ref{thm.spectral.inv} (see Theorem 7.4 in \cite{KNW-22} or Theorem
  15.4.11 in \cite{KMNW-2025}).
\end{proof}

\subsection{Invertibility of $\frac12  + \bsK$}
In the following, we will use
the operators $\bsP := - 2 \maA  \Dnu^{*} + \maB \bsnu^{\sharp}$ and
$\bsK \ede \bsP_{0}$. The first one is needed in the definition of the double layer
potential operator $\maW_{\rm{ST}}$ and the second one was
introduced in Equation \eqref{eq.def.bsk} and studied in Theorem \ref{thm.jump.K1}.
Each of these operators depends on the potentials $V$ and $V_{0}$.
(The correspondence $\bsP  \mapsto \bsP _{0}$ is the basic correspondence studied,
for example, in Theorems \ref{thm.main.jump0} and \ref{thm.main.jump1}.)

The proof of the invertibility of $\frac12  + \bsK$ (under suitable
conditions on $V$ and $V_{0}$) requires the invertibility of its indicial operators.
We first consider the case $\tau = 0$. This case needs additional assumptions
and arguments. Recall that we have $V(x', t) = \newtilde V(x', t)$ for $(x', t) \in
M' \times (-\infty, 0)$
and $t \ll 0$ and that $\newtilde V(x', t) = \newtilde V(x')$
is independent of $t$. The function $\newtilde V_{0}$ on $M'$ is defined similarly.
The functions $\newtilde V_{0}$ and $\newtilde V$ correspond to the limit operators of
the operators defined by $V_{0}$ and $V$.

\begin{proposition}\label{prop2.Ktau=0}
  Let us assume that $M, V$, and $V_{0}$ satisfy the non-vanishing assumptions
  {\lpar Assumption \ref{assumpt.VV0}\rpar} and let
  $\widehat \bsK(0) \in \Psi^{0}(\Gamma'; TM)$ be the indicial operator of the
  operator $\bsK \in \ePS{0}(\Gamma; TM)$ given in Theorem \ref{thm.jump.K1}
  {\lpar}see also Equation \eqref{eq.matching.ce2}\rpar. We furthermore assume that
  $\newtilde V_{0} \succ 0$ on $\Omega'$, that $\newtilde V \succ 0$ on $\Omega_{-}'$,
  and that $\Omega'$ on one side of its boundary $\Gamma' := \overline{\Omega '}$.
  Then $\frac12  + \widehat \bsK(0)$ is
  invertible on $L^2(\Gamma'; TM)$ and $\big [\frac12  + \widehat \bsK(0) \big]^{-1}
  \in \Psi^{0}(\Gamma'; TM)$.
\end{proposition}

Of course, in view of our conventions, the condition $\newtilde V_{0} \not\equiv 0$
on $\Omega'$ implies the condition $V_{0} \not\equiv 0$
on $\Omega$ (since $V$ and $V_{0}$ are translation invariant at infinity).

\begin{proof}
  The proof is very similar to that of Proposition \ref{prop2.Stau=0}, but
  using Proposition \ref{prop.Ktau=0} instead of \ref{prop.Stau=0}.
  Indeed, first, the hypotesis imply that the operator $\bsXi$ is Fredholm, and hence
  the layer potentials are defined. The result then follows from Proposition
  \ref{prop.Ktau=0} applied to $\In(\bsXi)$ acting on $M' \times \RR$
  and the fact that layer potentials commute with limit operators,
  Theorem \ref{thm2.indicial.jump1}. The last point follows from Theorem
  \ref{thm.spectral.inv}.
\end{proof}

\begin{proposition}\label{prop2.Kneq0}
  Let us assume that $V, V_{0} \ge 0$ satisfy Assumption \ref{assumpt.VV0}
  \lpar so $\newtilde V \succ 0$ and $\newtilde V_{0} \succ 0$ on $M'$\rpar and let
  $\widehat \bsK(\tau) \in \Psi^{0}(\Gamma'; TM)$ be the indicial operator of the
  operator $\bsK \in \Psi^{0}(\Gamma; TM)$ of Theorem \ref{thm.jump.K1} (see also Equation
  \eqref{eq.matching.ce2}). If $\tau \neq 0$, then $\frac12 + \widehat \bsK(\tau)$ is
  invertible on $L^2(\Gamma'; TM)$ and $\big [\frac12 + \widehat \bsK(\tau) \big]^{-1}
  \in \Psi^{0}(\Gamma'; TM)$.
\end{proposition}

\begin{proof}
  The proof is the same as that of Proposition \ref{prop2.Ktau=0},
  but using Proposition \ref{prop.Kneq0} instead of Proposition \ref{prop.Ktau=0}.
\end{proof}

We above propositions lead to the following useful corollary.

\begin{corollary}\label{cor2.KFred}
  Let us assume that $M, V$, and $V_{0}$ satisfy the non-vanishing assumptions
  \ref{assumpt.VV0} and let $\bsK \in \Psi^{0}(\Gamma; TM)$ be as in
  Theorem \ref{thm.jump.K1} (as before, see also Equation \eqref{eq.matching.ce2}).
  We furthermore assume that $\newtilde V_{0} \succ 0$ on $\Omega'$, that
  $\newtilde V \succ 0$ on $\Omega_{-}'$, and that $\Omega '$ is the interior of
  $\overline{\Omega'}$. Then $\frac12 + \bsK$ is a Fredholm operator on $L^2(\Gamma ; TM)$
  and its pseudo-inverse $(\frac12 + \bsK)^{(-1)} \in \ePS{0}(\Gamma; TM)$.
\end{corollary}

\begin{proof} The proof is similar to that of Corollary \ref{cor2.SFred},
  but using Theorem \ref{thm.cyl.K} instead of Theorem \ref{thm.cyl.S}.
  Indeed, the operator $\In(\frac12 + \bsK)$ is the double layer potential
  operator associated to $\In(\bsXi_{V, V_{0}}) = \bsXi_{\newtilde V, \newtilde V_{0}}$,
  by Theorem \ref{thm2.indicial.jump1}.
  The assumptions of the corollary give that the assumptions
  of Theorem \ref{thm.cyl.K} are therefore satisfied by $\In(\frac12 + \bsK)$, which
  is hence invertible. Because $\frac12 + \bsK$ is elliptic, Theorem \ref{thm.ADN.Fredholm}
  gives that $\frac12 + \bsK$ is Fredholm. The last point follows from
  Theorem \ref{thm.spectral.inv}.
\end{proof}

We are ready to prove the invertibility of the double layer potential
operator $\frac12 + \bsK$ on $\Gamma$.

\begin{theorem}\label{thm2.K}
  Let us assume that $M, V$, and $V_{0}$ satisfy Assumption
  \ref{assumpt.VV0} \lpar so, in particular $\newtilde V \succ 0$ on
  $M'$\rpar, that $\Omega$ is connected, and that $\Omega$ is
  the interior of $\overline{\Omega}$. We furthermore assume the following:
  \begin{enumerate}[\rm (i)]
    \item $\newtilde V_{0} \succ 0$ on $\Omega'$,
    \item $\newtilde V \succ 0$
    on $\Omega_{-}'$, and
    \item $V \succ 0$ on $\Omega_{-}$.
  \end{enumerate}
  Let $\bsK \in \ePS{0}(\Gamma; TM)$ be as in
  Theorem \ref{thm.jump.K1} \lpar see also Equation \eqref{eq.matching.ce2}\rpar.
  Then $\frac12 + \bsK $ is invertible on $L^2(\Gamma ; TM)$ and $\big [\frac12 + \bsK \big]^{-1}
  \in \Psi^{0}(\Gamma ; TM)$.
\end{theorem}

\begin{proof}
  We split our proof into five steps.
  \smallskip

  \noindent {\bf Step 1\ ($\frac12 + \bsK$ has index zero).}\
  Corollary \ref{cor2.KFred} implies that $\frac12 + \bsK$ is a Fredholm operator.
  Proposition \ref{prop.S-K} then gives that $\frac12 + \bsK$ has index zero.
  Indeed, for a Fredholm operator $T : X \to Y$, let
  $\operatorname{ind}_{X}(T)$ be its index. Proposition \ref{prop.S-K} implies the relation
  $\big(\frac12 + \bsK\big) \bsS \seq \bsS\big(\frac12+ \bsK^*\big) :
  L^{2}(M; TM) \to H^{1}(M; TM)$
  and the properties of the index give that
  \begin{equation*}
    \operatorname{ind}_{H^{1}} \Big(\frac12 + \bsK\Big)
    + \operatorname{ind}_{L^{2}}(\bsS) \seq
    \operatorname{ind}_{L^{2}} (\bsS)
    + \operatorname{ind}_{L^{2}}\Big(\frac12+ \bsK^*\Big)\,,
  \end{equation*}
  and hence $\operatorname{ind}_{H^{1}} \big(\frac12 + \bsK\big) =
  \operatorname{ind}_{L^{2}}\big(\frac12+ \bsK^*\big)$.
  On the other hand, Corollary 15.4.18 in \cite{KMNW-2025} implies that
  the indices above do not depend on the
  corresponding spaces. Thus, we have
  $\operatorname{ind}_{L^{2}}\big(\frac12+ \bsK^*\big) =
  - \operatorname{ind}_{L^{2}}\big(\frac12+ \bsK\big) =
  - \operatorname{ind}_{H^{1}}\big(\frac12+ \bsK\big)$, and hence
  $\operatorname{ind}_{H^{1}} \big(\frac12 + \bsK\big) =
  \operatorname{ind}_{L^{2}} \big(\frac12+ \bsK^*\big) =
  - \operatorname{ind}_{H^{1}} \big(\frac12 + \bsK\big)$.
  Therefore $\operatorname{ind}_{H^{1}} \big(\frac12 + \bsK\big) = \operatorname{ind}_{L^{2}}\big(\frac12+ \bsK\big) = 0$, as claimed.

  To complete the proof of our result, it is enough to determine the
  image of $\frac12 + \bsK$. This property will be obtained in the
  next \emph{four steps}.
  \smallskip

  \noindent {\bf Step 2 (Single layer $U := \maS_{\rm{ST}}(\bsh)$).}\
  To prove that $\frac12 + \bsK$ is onto, it suffices to show that 
  $\ker (\frac12 + \bsK^{*}) = 0$.
  To this end, let $\bsh \in L^{2}(\Gamma ; TM)$ be such that $(\frac12 + \bsK^{*})\bsh = 0$.
  The ellipticity of $\frac12 + \bsK^{*}$, the fact that $\Gamma$ has straight cylindrical ends, 
  and
  the elliptic regularity property (see Theorem \ref{thm.regularity} or 
  \cite[Theorem 15.3.30]{KMNW-2025}
  or \cite[Proposition 6.5]{KNW-22}) imply that $\bsh \in H^{s}(\Gamma; TM)$   
  for all $s \in \RR$. Let
  now the \emph{single layer potential} $\maS_{\rm{ST}}\bsh $. In view of Proposition 
  \ref{prop.H2estimates}
  we obtain that the restrictions of $U$ to $\Omega_{+} := \Omega$ and to $\Omega_{-}$ 
  satisfy
  \begin{equation}\label{eq.needed.est-3}
    U \ede (\, \bsu \ \ p\, )^{\top} \ede \maS_{\rm{ST}}(\bsh)
    \in H^{2}(\Omega_{\pm}; TM) \oplus H^{1}(\Omega_{\pm})\,.
  \end{equation}
  We next study the restrictions of $U$ on both domains $\Omega_{+}$ and $\Omega_{-}$.
  \smallskip

  \noindent {\bf Step 3 (Study of $U$ on $\Omega_{-}$).}\
  Let $U := \maS_{\rm{ST}}(\bsh)$ be as in the previous step.
  Proposition \ref{prop.compatibility} implies that $\bsXi U = 0$ in $\Omega_{-}$, 
  and Theorem \ref{thm.jump.rel} gives that
  $[\bop U]_{-} = (\frac12 + \bsK^{*})\bsh = 0$.
  (Recall that $[\bop U]_{-}$ is the trace at $\Gamma$ of $\bop U$ from the domain
  $\Omega_{-}$). Because $V$ does not vanish identically in any connected component of
  $\Omega_{-}$, we obtain $\bsu = 0$ in $\Omega_{-}$, by Corollary \ref{cor.e.est}.
  The same corollary gives that $p$ is constant
  in all connected components of $\Omega_{-}$.
  Moreover, the relation
  \begin{equation*}
    0 \seq [\bop U]_{-} \ede  [-2 \Dnu \bsu + p \bsnu]_{-} \seq p_{-} \bsnu
  \end{equation*}
  on $\Gamma$ gives that $p_{-} = 0$, and hence $p = 0$ on $\Omega_{-}$, since
  $p$ is constant in all connected components of $\Omega_{-}$.
  \smallskip

  \noindent {\bf Step 4 (Study of $U$ on $\Omega_{+}$).}\
  Let $U = (\, \bsu \ \ p\, )^{\top} := \maS_{\rm{ST}}(\bsh)$, as in the previous
  two steps. Recall that $\bsXi U = 0$ in $M \smallsetminus \Gamma$.
  Theorem \ref{thm.jump.rel} gives
  the ``no-jump relation'' $\bsu_{+} = \bsu_{-} = 0$ at the boundary
  $\Gamma := \partial \Omega$ of $\Omega$ (interior and exterior traces).
  The fact that $\bsu_{+} = 0$ at $\Gamma$ allows us to use
  Corollary \ref{cor.e.est} that gives then that $\bsu = 0$ in $\Omega$.
  Because $\Omega$ is connected and $\newtilde V_{0} \succ 0$ on $\Omega'$,
  we also have $V_{0} \succ 0$ on $\Omega$. Therefore, the same corollary
  implies that $p$ is constant on $\Omega$. Because $V_{0}$ does not vanish identically
  on $\Omega$ (as we have just noticed) we even obtain that
  $p = 0$ in $\Omega $ and $p_{+}=0$. Recalling that we have
  already proved that $\bsu = 0$ in $\Omega_{-}$, we see that
  $\bsu = 0$ in $M \smallsetminus \Gamma$ and hence
  $\Dnu \bsu = 0$ in $M \smallsetminus \Gamma$.
  \smallskip

  \noindent {\bf Step 5 (Injectivity).}\
  The definitions of $\bop$ and $U = (\, \bsu \ \ p\, )^{\top}
  := \maS_{\rm{ST}}(\bsh)$ then give
    $\bop U := -2 \Dnu \bsu + p \bsnu \seq p \bsnu$
  on $M \smallsetminus \Gamma$, and
  \begin{equation}\label{eq.p.limits}
    \bsh \seq [\bop \maS_{\rm{ST}}(\bsh)]_{-}
    - [\bop \maS_{\rm{ST}}(\bsh)]_{+} \seq (p_{-} - p_{+}) \bsnu \seq 0\,.
  \end{equation}
  Therefore the operator $\frac12 + \bsK^{*}:L^2(\Gamma ;TM)\to L^2(\Gamma ;TM)$
  is injective implying that the Fredholm operator of index zero
  $\frac12 + \bsK : L^{2}(\Gamma ; TM) \to L^{2}(\Gamma ; TM)$ is onto and hence
  an isomorphism, as asserted.

  The fact that $(\frac12 + \bsK)^{-1} \in \Psi^{0}(\Gamma; TM)$ follows from Theorem
  \ref{thm.spectral.inv} (see Theorem 7.4 in \cite{KNW-22} or Theorem 15.4.11 in
  \cite{KMNW-2025}). This completes the proof.
\end{proof}

In all the results above, the assumptions that $M$ and $\Omega$ are
connected do not really decrease the generality. Note, however, that we
are not assuming $\Omega_{-} := M \smallsetminus \overline{\Omega}$, $\Omega'$,
or $M'$ to be
connected.

\section{\cn The Dirichlet problem for the generalized Stokes system via
layer potentials}

In this section, we present two applications of the results of the previous section.
First, we obtain the well-posedness result of the Dirichlet problem for the generalized Stokes
system on a smooth domain $\Omega$ with straight cylindrical ends. Then, we use this
well-posedness result to show the no-jump property of the co-normal derivative
$\bop \maD_{\rm{ST}}(\bsh)$ across the boundary of the domain $\Omega$.

\subsection{The homogeneous Stokes system}
\label{ssec.WP}

The invertibility of each of the operators $\frac12 + \bsK$ and $\bsS$
gives the following well-posedness result for the modified
Stokes system with Dirichlet boundary condition, problem \eqref{eq.bvp.Brinkman}
(see Equation \eqref{eq.bvp.Brinkman.explicit} for a more explicit form
of this problem). This result is known for the usual Stokes operator (i.e. when
$V = 0$ and $V_{0} = 0$, see \cite[Proposition 10.5.1, Theorem 10.6.2]{M-W} in
the case of a bounded Lipschitz domain in $\RR^n$, $n\geq 2$, and
\cite[Theorem 5.1]{D-M} in the case of a $C^1$ domain on a compact manifold,
see also \cite{KMNW-2025} for further related applications).
We consider the \emph{homogeneous Dirichlet problem}
\begin{align}\label{eq.Dirichlet.pr}
  \bsXi U \ede \bsXi_{V, V_{0}} U \seq 0 \ \mbox{ in } \ \Omega  \quad \mbox{and}
  \quad \bsu \seq \bsf \ \mbox{ on } \ \Gamma \ede \pa \Omega\,.
\end{align}

The following result gives the well-posedness of this problem and the
representability of its solutions \emph{under global assumptions on $M$} for
$V$ and $V_{0}$. Recall Assumption \ref{assumpt.VV0}.

\begin{proposition}\label{prop.main.WP}
  Let $M$, $V$, and $V_{0}$ satisfy Assumption \ref{assumpt.VV0} \lpar so, in particular, $M$
  is has straight cylindrical ends\rpar\ and let $\Omega \subset M$ be a compatible smooth
  domain \lpar so $\Omega$ also has straight cylindrical ends, see Equation 
  \eqref{eq.def.Omega}\rpar.
  We assume that $V_{0} \succ 0$ on $\Omega$, that $\newtilde V_{0} \succ 0$ on $\Omega'$, and
  that $\Omega $ is on one side of its boundary $\Gamma := \partial {\Omega }$. We further assume
  one of the following conditions:
  \begin{enumerate}[\rm (a)]
    \item $V \succ 0$ on $\Omega_{-}$ and $\newtilde V \succ 0$ on $\Omega_{-}'$ or
    \item $V_{0} \succ 0$ on $\Omega_{-}$ and $\newtilde V_{0} \succ 0$ on $\Omega_{-}'$.
  \end{enumerate}
  Then, for every $m \in \ZZ_{+}$ and any $\bsf \in H^{m+1/2}(\Gamma; TM)$,
  the homogeneous Diriechlet problem \eqref{eq.Dirichlet.pr}
  has a unique solution $U = (\bsu \ \ p)^{\top} \in H^{m+1}(\Omega; TM) \oplus
  H^{m}(\Omega)$. Moreover, there exists a constant $C_{m} \ge 0$ such that,
  for any $\bsf \in H^{m+1/2}(\Gamma; TM)$, the unique solution $U$ of this problem satisfies
  \begin{equation*}
    \|\bsu\|_{H^{m+1}(\Omega; TM)} + \|p\|_{H^{m}(\Omega)} \le C_{m}
    \|\bsf\|_{H^{m+1/2}(\Gamma; TM)}\,.
  \end{equation*}
  Furthermore, according to which assumption is satisfied {\rm ((a) or (b))}, the solution
  $U$ can be represented by at least one of the formulas
  \begin{itemize}
    \item $U \ede \maD_{\rm{ST}}\Big(\Big(\frac{1}{2} + \bsK\Big)^{-1}\bsf\Big)$,
    if {\rm (a)} is satisfied, respectively
    \item $U \ede \maS_{\rm{ST}}\left(\bsS^{-1}\bsf\right)$, if {\rm (b)}
    is satisfied.
  \end{itemize}
\end{proposition}

\begin{proof}
  Let us check that the solution $U$ is unique. To this end, let $U_{0} = (\bsu_{0} \ \
  p_{0})^{\top}$ satisfies the homogeneous Stokes system in $\Omega$ (i.e., $\bsXi U_{0}
  = 0$) and the homogeneous Dirichlet condition on $\Gamma := \pa \Omega$ (i.e., $\bsu_{0}
  = 0$ on $\Gamma$). Because $V_{0} \not \equiv 0$ on $\Omega$ and $\bsu_{0} = 0$ on
  $\pa \Omega$, Corollary \ref{cor.e.est} then ensures that $U_{0} = 0$
  in $\Omega $, which implies the desired uniqueness.

  Our assumption {\rm (a)} implies that the operator $\frac12 + \bsK$ is invertible on
  $L^2(\Gamma ; TM)$ and $\big [\frac12 + \bsK \big]^{-1} \in
  \ePS{0}(\Gamma ; TM)$, by Theorem \ref{thm2.K}. The mapping properties of the
  ``$\ess$-calculus'' then gives that $\bsh := (\frac12 + \bsK)^{-1} \bsf \in
  H^{m+1/2}(\Gamma ;TM)$.
  Then $U := \maD_{\rm{ST}}(\bsh)$ satisfies the required
  properties, including its boundedness with respect to the given datum $\bsf$, with
  a constant $C_m$ given by $\| \maD_{\rm{ST}} (\frac12 + \bsK)^{-1}\|$.

  If, on the other hand, condition {\rm (b)} is satisfied, then the operator $\bsS$ is
  invertible on $L^2(\Gamma ; TM)$ and $\bsS^{-1} \in \ePS{1}(\Gamma ; TM)$,
  by Theorem \ref{thm2.S}. The mapping properties of the ``$\ess$-calculus'' then imply that
  $\bsh := \bsS^{-1}\bsf \in H^{m - 1/2}(\Gamma ; TM)$.
  Theorem \ref{thm.mapping.lp} then gives
  \begin{equation*}
    U \ede \maS_{\rm{ST}}(\bsh) \ede \maS_{\rm{ST}}(\bsS^{-1}\bsf)
    \in H^{m+1}(M; TM) \oplus H^{m}(M)\,,
  \end{equation*}
  and the estimate
  \begin{equation*}
    \|\bsu\|_{H^{m+1}(\Omega; TM)} + \|p\|_{H^{m}(\Omega)} \le
    \| \maS_{\rm{ST}} \bsS^{-1}\| \|\bsf\|_{H^{m+1/2}(\Gamma; TM)}\,.
  \end{equation*}
  Moreover, Proposition \ref{prop.compatibility} gives that $\bsXi U = 0$ in $\Omega$, and
  \begin{equation*}
    \bsu \vert_{\Gamma} \ede \big[ \maV_{\rm{ST}}(\bsh) \big]_{+} \seq
    \bsS\bsh \ede \bsS \bsS^{-1} \bsf \seq \bsf\,.
  \end{equation*}
  Hence $U$ satisfies all required properties. This completes the proof of the theorem.
\end{proof}

This proposition then allows us to prove the main result of this paper.

\begin{proof}[Proof of Theorem \ref{thm.main.WP}]
  Let us first prove the uniqueness result stated in Theorem \ref{thm.main.WP}.
  To this end, let $U_{0} = (\bsu_{0} \ \ p_{0})^{\top}$ be the difference of two solutions
  of our Dirichlet problem, i.e. $U_{0}$ satisfies the homogeneous Stokes system in $\Omega $
  and the homogeneous Dirichlet condition on $\pa \Omega$.
  Because $V_{0} \succ 0$ on $\Omega$
  and $\bsu_{0} = 0$ on $\pa \Omega$,
  Corollary \ref{cor.e.est} gives that $U_{0} = 0$ in $\Omega $, which implies the desired
  uniqueness.

  Next we prove the existence part of Theorem \ref{thm.main.WP}.
  Given the potential $V$ satisfying the assumptions of this theorem, we can modify it
  on $\Omega_{-}$ such that $\newtilde V \succ 0$ on $\Omega_{-}'$ and $V \succ 0$ on $\Omega_{-}$.
  Then our assumptions imply that the operator $\frac12 + \bsK $ is
  invertible on $L^2(\Gamma ; TM)$ and $\big [\frac12 + \bsK \big]^{-1}
  \in \Psi^{0}(\Gamma ; TM)$, by Theorem \ref{thm2.K}. Moreover, the elliptic
  regularity of the operator $\frac12 + \bsK $ (Theorem \ref{thm.jump.K1}) implies that
  $\bsh := (\frac12 + \bsK)^{-1} \bsf \in
  H^{m+1/2}(\Gamma ;TM)$.
  Then $U_1 := \maD_{\rm{ST}}(\bsh)$ satisfies the required
  properties, including its boundedness with respect to the given datum $\bsf$, with
  a constant $C_m$ given by $\| \maD_{\rm{ST}} (\frac12 + \bsK)^{-1}\|$.
  This completes the proof.
\end{proof}

The analogous result for the Laplace operator was proved in \cite{Mitrea-Nistor}, but
using a different method.

\begin{remark}
We notice that the solution $U$ of the above theorem can be represented
either by a single or as a double layer potential, depending on how
we modify $V$ and $V_{0}$ on $\Omega_{-}$.
Indeed, modifying $V$ on $\Omega_{-}$ such that $\newtilde V \succ 0$ on $\Omega_{-}'$ and
$V \succ 0$ on $\Omega_{-}$, then the assumptions of Theorem \ref{thm2.K} are satisfied and
the solution $U$ can be represented as a double layer potential
\begin{equation*}
    U \ede \maD_{\rm{ST}}\Big(\Big(\frac{1}{2} + \bsK\Big)^{-1}\bsf\Big)\,.
\end{equation*}
On the other hand, modifying $V$ and $V_{0}$ on $\Omega_{-}$ such that the conditions of
Theorem \ref{thm2.S} be satisfied, the solution $U$ can be represented as a single layer
potential
\begin{equation*}
    U \ede \maS_{\rm{ST}}\left(\bsS^{-1}\bsf\right)\,.
\end{equation*}
\end{remark}

\begin{remark}
The constant $\mathcal C_m$ in Theorem \ref{thm.main.WP} is independent of the given datum
$\bsf \in H^{m+1/2}(\Gamma; TM)$. Indeed, it depends only on the norms of the operators
$\maD_{\rm{ST}}$ and $\big [\frac12 + \bsK \big]^{-1}$ and can be chosen the product of them.
See also \cite{Al-Am-1} and \cite{Al-Am}
for other results related to the Stokes operator on unbounded domains.
\end{remark}

\subsection{The non-homogeneous Stokes system}
Let us now prove the well-posedness result for the non-homogeneous problem
using Theorems \ref{thm.main.WP} and \ref{thm2.main.invert}.


\begin{theorem}
\label{thm.main.WP-non-1}
  Let the assumptions of Theorem \ref{thm.main.WP} hold. Then, for every $m \in \ZZ_{+}$,
  there exists a constant ${\mathcal C}_{m} \in (0,\infty )$ such that,
  for all $\bsf \in H^{m+1/2}(\Gamma; TM)$, $r \in H^{m-1}(\Omega )$, and $\bsh \in H^{m-1}(\Omega ; TM)$, the
  \emph{non-homogeneous Dirichlet problem}
  \begin{equation}\label{eq.nhom.pr}
    \bsXi U \ede \bsXi_{V, V_{0}} U \seq \cvector{\bsh}{r} \ \mbox { in } \Omega\,, \
    \mbox{ and }\ \bsu \vert_{\Gamma} \seq \bsf\,,
  \end{equation}
  has a unique solution $U = (\bsu \ \ p)^{\top} \in H^{m+1}(\Omega; TM) \oplus H^{m}(\Omega)$
  and this solution satisfies
  \begin{equation*}
    \|\bsu\|_{H^{m+1}(\Omega; TM)} + \|p\|_{H^{m}(\Omega)} \le {\mathcal C}_{m}
    \left(\|\bsh\|_{H^{m-1}(\Omega; TM)}+\|r\|_{H^{m-1}(\Omega)}
    +\|\bsf\|_{H^{m+1/2}(\Gamma; TM)}\right)\,.
  \end{equation*}
\end{theorem}

\begin{proof}
  There exists $C_{1, m} \ge 1$ such that, for all $\bsh \in H^{m-1}(\Omega ; TM)$ and
  $\bsh \in H^{m}(\Omega )$
  there are ${\boldsymbol H}\in H^{m-1}(M; TM)$ and $R\in H^{m}(M)$ such that
  $\boldsymbol H|_{\Omega } \seq \bsh$, $R|_{\Omega }=r$, and
  \[\|\boldsymbol H\|_{H^{m-1}(M)} + \|R\|_{H^{m}(M)}\le C_{1, m}\big( \|\bsh\|_{H^{m-1}(\Omega; TM)}
  +\|r\|_{H^{m}(\Omega)}\big).\]
  Then the invertibility of the generalized Stokes operator
  $\bsXi \ede \bsXi_{V, V_{0}}$, Theorem \ref{thm2.main.invert}, implies that there exists a unique
  $U_1 =  (\bsu _1 \ \ p _1)^{\top} \in H^{m+1}(M; TM) \oplus H^{m}(M)$ such that
  $\bsXi U_1= (\boldsymbol H \ \ R)^{\top}$ in $M$
  and there exists a constant $C_{0;m}>0$ (which does not depend on $\bsh$ and $r$) such that
  \[\| U_1 \|_{H^{m+1}(M; TM) \oplus H^{m}(M)}\leq C_{0;m}\big(\|\boldsymbol H\|_{H^{m}(M; TM)}
  +\|R\|_{H^{m}(M)}\big).\]
  Consequently,
  \begin{equation*}
    \bsXi U_1 \seq \cvector{\bsh}{r} \ \mbox { in } \Omega\,.
  \end{equation*}

  Now let $\bsf_1 := \bsf - \bsu_1|_{\Gamma }$. Hence
  $\bsf_1\in H^{m+\frac{1}{2}}(\Gamma ;TM)$
  and Theorem \ref{thm.main.WP} implies that the {Dirichlet problem}
  \begin{equation*}
    \bsXi U_2 \ede \bsXi_{V, V_{0}} U_2 \seq 0 \ \mbox { in } \Omega \
    \mbox{ and }\ \bsu_2 \vert_{\Gamma} \seq \bsf_1
  \end{equation*}
  has a unique solution
  $U_2 = (\bsu _2 \ \ p_2)^{\top} \in H^{m+1}(\Omega; TM) \oplus H^{m}(\Omega)$
  and
  \begin{align*}
  \|\bsu _2\|_{H^{m+1}(\Omega; TM)} &+ \|p_2\|_{H^{m}(\Omega)}
  \le C_{m} \|\bsf_1\|_{H^{m+1/2}(\Gamma; TM)}\\
  &\leq C_{2;m}\left(\|\bsf\|_{H^{m+1/2}(\Gamma; TM)} + \|\bsh\|_{H^{m-1}(\Omega ; TM)}
  +\|r\|_{H^{m-1}(\Omega)}\right)\,,
  \end{align*}
  with some constants $C_m, C_{2;m} > 0$ independent of $\bsf$, $r$ and $\bsh$.

  Finally, let $U:=U_1+U_2$, $U =(\bsu  \ \ p)^{\top}$. The above properties of $U_1$
  and $U_2$ imply then that $U\in H^{m+1}(\Omega; TM) \oplus H^{m}(\Omega)$ satisfies
  the required properties. The uniqueness of $U$ with these properties follows from
  Theorem \ref{thm.main.WP}.
\end{proof}

Let us now consider the case $r=0$ in Equation \eqref{eq.nhom.pr}. 
Then a consequence of Theorem \ref{thm.main.WP-non-1} that will be used in 
what follows is the following corollary.

\begin{corollary}
\label{cor.main.WP-non}
  Let us further assume in Theorem \ref{thm.main.WP-non-1} that $r = 0$.
  Then, for every $m \in \ZZ_{+}$, there exists a constant ${\mathcal C}_{m} \in (0,\infty )$ 
  such that, for any $\bsf \in H^{m+1/2}(\Gamma; TM)$ and any $\bsh \in H^{m-1}(\Omega ; TM)$, the
  non-homogeneous Dirichlet problem \eqref{eq.nhom.pr}
  has a unique solution $U = (\bsu \ \ p)^{\top} \in H^{m+1}(\Omega; TM) \oplus H^{m}(\Omega)$
  and this solution satisfies
  \begin{equation*}
    \|\bsu\|_{H^{m+1}(\Omega; TM)} + \|p\|_{H^{m}(\Omega)} \le {\mathcal C}_{m}
    \left(\|\bsh\|_{H^{m-1}(\Omega; TM)}+\|\bsf\|_{H^{m+1/2}(\Gamma; TM)}\right)\,.
  \end{equation*}
  Let $U$ is the unique solution of Problem \eqref{eq.nhom.pr} for the given data
  $(\bsh \ \ \bsf)^{\top}$ and let ${\mathfrak A}(\bsh \ \ \bsf)^{\top} := U$.
  In particular, the map
  \begin{equation*}
    {\mathfrak A}:H^{m-1}(\Omega ; TM)\oplus H^{m+1/2}(\Gamma; TM)\to
    H^{m+1}(\Omega; TM) \oplus H^{m}(\Omega),
  \end{equation*}
  is linear and continuous.
\end{corollary}

\subsection{Dirichlet-to-Neumann operator and boundary behavior of $\bop \maD_{\rm{ST}}$}
Let assumptions of Theorem \ref{thm.main.WP} hold. Let
$\maN_{\rm{ST}}:H^{3/2}(\partial \Omega ;TM)
\to H^{1/2}(\partial \Omega ;TM)$
be the {\it Dirichlet-to-Neumann operator} defined as follows
(see \cite[p.37]{Taylor2} in the case of the Laplace operator on a compact manifold).
For $\bsf\in H^{3/2}(\partial \Omega ;TM)$ arbitrary, let
$U := (\bsu \ \ p)^{\top}\in H^2(\Omega ;TM)\oplus
H^1(\Omega)$ be the unique solution of the Dirichlet problem
\eqref{eq.Dirichlet.pr} (that is, $\bsXi_{V, V_{0}} U \seq 0$ in $\Omega$
and $\bsu \vert_{\Gamma} \seq \bsf$ on $\Gamma := \pa \Omega$),
see Theorem \ref{thm.main.WP}. Then we define
\begin{equation}
\label{Dirichlet-oper}
  \maN_{\rm{ST}}\bsf \ede [\bop U]_{+} \ \mbox{ on } \Gamma \,,
\end{equation}
the limit being evaluated from $\Omega $.

The next result shows that there is no jump of the conormal derivative
$\bop \maD_{\rm{ST}} (\bsf)$ across $\Gamma :=\pa \Omega $
(see \cite[Proposition 11.4]{Taylor2} in the case of the Laplace operator).

\begin{theorem} \label{jump-conormal-dl}
  Under the assumptions of Theorem \ref{thm.main.WP},
  there is no jump
  across $\Gamma $ of $\bop \maD_{\rm{ST}} (\bsf)$. More precisely, for
  $\bsf \in  H^{3/2}(\Gamma ;TM)$,
  \begin{equation*}
    [\bop \maD_{\rm{ST}} (\bsf)]_{+} \seq \bop [\maD_{\rm{ST}} (\bsf)]_{-}
    \seq \left(\frac12+ {\bsK}^{*}\right)
    \maN_{\rm{ST}}\bsf\,.
  \end{equation*}
\end{theorem}

\begin{proof}
  The second relation of Proposition \ref{prop.rep.formula}
  implies that
  \begin{align} \label{cons.rep.invert}
  \maD_{\rm{ST}} \bsf(x) - \maS_{\rm{ST}} (\maN_{\rm{ST}} \bsf )(x) \seq
  \begin{cases}
    \ U(x) & \mbox{ if } x \in \Omega\\
    \ \ \, 0    & \mbox{ if } x \in M \smallsetminus \overline{\Omega} \,,
  \end{cases}
  \end{align}
  where $U$ is the unique solution of the Dirichlet problem
  \eqref{eq.Dirichlet.pr} with $\bsf\in H^{3/2}(\Gamma ;TM)$.

  Now considering the vector part of identity \eqref{cons.rep.invert}, taking the limit of \eqref{cons.rep.invert} on $\Gamma $ from $\Omega $, and using Theorem
  \ref{thm.jump.K1} and Theorem \ref{thm.jump.rel}, we obtain
  \[
  \left(\frac12 + \bsK\right)\bsf - \bsS(\maN_{\rm{ST}} \bsf ) =
  \bsu \vert_{\Gamma} =\bsf\,,
  \]
  and hence
  \begin{align*}
    \bsS(\maN_{\rm{ST}} \bsf ) \seq \left(-\frac12 + \bsK\right)\bsf\,.
  \end{align*}
  Thus we obtain the identity
  \begin{align*}
    \bsS\maN_{\rm{ST}} \seq -\frac12 + \bsK \,.
  \end{align*}
  (The same identity follows if we take the limit of the vector part of \eqref{cons.rep.invert} 
  on $\Gamma $ from $\Omega_{-} := M\setminus \overline\Omega$ and use again Theorem
  \ref{thm.jump.K1} and Theorem \ref{thm.jump.rel}.)
  This identity and the ellipticity of the operators $\bsS$ and $-\frac12 + \bsK$ (see also 
  Theorem \ref{thm.jump.K1} and Theorem \ref{thm.jump.rel}) imply that $\maN_{\rm{ST}}$ is 
  elliptic as well.

  Next we apply the operator $\bop $ to both sides of identity \eqref{cons.rep.invert}.
  Evaluating it first from $\Omega $, we obtain
  \begin{align}\label{conormal-dl-plus}
    [\bop \maD_{\rm{ST}} (\bsf)]_{+} - [\bop \maS_{\rm{ST}} (\maN_{\rm{ST}}\bsf)]_{+} = [\bop U]_{+}
    = \maN_{\rm{ST}}\bsf\,,
  \end{align}
  while evaluating from $M\setminus \overline\Omega $ implies
  \begin{align}\label{conormal-dl-minus}
    [\bop \maD_{\rm{ST}} (\bsf)]_{-} - [\bop \maS_{\rm{ST}}(\maN_{\rm{ST}}\bsf)]_{-} = 0\,.
  \end{align}
  Since both limits $[\bop \maS_{\rm{ST}}(\maN_{\rm{ST}}\bsf)]_\pm $ exist,
  by Theorem \ref{thm.jump.rel}, formulas \eqref{conormal-dl-plus} and
  \eqref{conormal-dl-minus} show that the limits $\bop \maD_{\rm{ST}} (\bsf)_\pm $
  exist as well, and they are given by
  \begin{align*}
  &[\bop \maD_{\rm{ST}} (\bsf)]_{+} \seq \maN_{\rm{ST}}\bsf + \left(-\frac12+ {\bsK}^{*}\right)
  \maN_{\rm{ST}}\bsf\,,\\
  &[\bop \maD_{\rm{ST}} (\bsf)]_{-} \seq \left(\frac12 + {\bsK}^{*}\right)
  \maN_{\rm{ST}}\bsf \,,
  \end{align*}
  and hence
  \[
  [\bop \maD_{\rm{ST}} (\bsf)]_\pm \seq \left(\frac12+ {\bsK}^{*}\right)
  \maN_{\rm{ST}}\bsf\,.
  \]
  This completes the proof.
\end{proof}

\section{\cn Applications and extensions}
\label{sec.NS}

We now consider an application and an extension of our results so far.

\subsection{The Dirichlet problem for the generalized Navier-Stokes system}
\label{ssec.NS}

The first part is of this section is devoted to the Dirichlet problem for a Navier-Stokes
type system on a manifold with cylindrical ends. We analyze the Dirichlet problem when the norms
of the given data (the right hand side of the Navier-Stokes equation and the Dirichlet datum)
are restricted to satisfy a boundedness condition which ensures the well-posedness of the
Dirichlet problem.


A key role in the analysis of the nonlinear Dirichlet problem for the Navier-Stokes system
on a domain with cylindrical ends is played by estimates of the product
of two functions in suitable Sobolev spaces. An example is the following result,
which is well known in $\RR^n$ (see e.g. \cite[Theorem 7.3, Theorem 8.2]{Be},
\cite[\S 4.5.2, Theorem]{Runst-Sickel}, and \cite[Lemma 8.1]{medkova26}).

\begin{lemma}
  \label{lemma.product-spaces}
  Let $m \in \ZZ _+$ and $M$ be as in Assumption \ref{assumpt.VV0}.
  Assume that $2\leq \operatorname{dim}\, M =n < 2(m+2)$.
  Then there exists $C_{m} > 0$ with the following property: for all
  $u \in H^{m+1}(M)$ and $v \in H^{m}(M)$, we have
  $uv \in H^{m-1}(M)$ and
  \begin{equation*}
    \|uv\|_{H^{m-1}(M)} \le C_{m}\|u\|_{H^{m+1}(M)} \|v\|_{H^{m}(M)}\,.
  \end{equation*}
\end{lemma}

\begin{proof}
  This follows from the analogous result on $\RR^{n}$ just mentioned above using
  a dyadic partition of unity (see, for instance \cite{KMNW-2025}).
\end{proof}

We notice that the same argument works for manifolds with bounded geometry.
Also, notice that the result of Lemma \ref{lemma.product-spaces}
extends to Sobolev spaces on sections of vector bundles over $M$
(see e.g. \cite[\S 4.2.2, Proposition]{Runst-Sickel} in the Euclidean setting),
and the next lemma is a consequence of this result
(see \cite[Lemma 7.5]{D-M} in the case of a Lipschitz domain on a compact manifold).

\begin{lemma}
\label{product-Sobolev}
  Let $M$ be as in Assumption
  \ref{assumpt.VV0} and $\Omega \subset M$ be as in Theorem \ref{thm.main.WP},
  that is, $\Omega $ is a smooth domain with straight cylindrical ends in the connected
  manifold with straight cylindrical ends $M$.
  Assume that $m\in \ZZ _+$ and that ${\rm{dim}}\, M =n < 2(m+2)$.
  Then there exists a constant $C>0$, such that
\begin{align}
  \label{NS-estimate1}
  &\|\nabla _{\bsu}\bsv\|_{H^{m-1}(\Omega ;TM)}\le C\|\bsu \|_{H^{m+1}(\Omega ;TM)}
  \|\bsv \|_{H^{m+1}(\Omega ;TM)},\\
  \label{NS-estimate2}
  &\|\nabla _{\bsu}\bsu - \nabla _{\bsv}\bsv\|_{H^{m-1}(\Omega ;TM)}\\
  &\hspace{2em}\le C\left(\|\bsu \|_{H^{m+1}(\Omega ;TM)}+\|\bsv \|_{H^{m+1}(\Omega ;TM)}\right)
  \|\bsu - \bsv \|_{H^{m+1}(\Omega ;TM)},\nonumber
\end{align}
for all $\bsu ,\bsv\in H^{m+1}(\Omega ;TM)$.
\end{lemma}

\begin{proof}
  First, we notice that $\Omega $ is a manifold with bounded geometry. Indeed,
  $\Omega $ is a smooth domain with straight cylindrical ends in the smooth
  manifold with straight cylindrical ends $M$, and hence both $M$ and $\Omega $
  have bounded geometry.
  Hence the result of Lemma \ref{lemma.product-spaces} holds also on $\Omega $.
  Thus, for each $m\in \ZZ _+$ such that $n < 2(m+2)$, there exists a constant
  $c_{m}>0$ (depending only on $n$ and the bounded geometry of $M$) such that
  the estimate \eqref{NS-estimate1} holds for all $u \in H^{m+1}(\Omega)$ and
  $v \in H^{m}(\Omega )$. Thus, we have
  $uv \in H^{m-1}(\Omega )$ and
  \begin{equation*}
    \|uv\|_{H^{m-1}(\Omega )} \le c_{m}\|u\|_{H^{m+1}(\Omega )} \|v\|_{H^{m}(\Omega)}\,.
  \end{equation*}
  (see \cite[\S 4.2.2, Proposition]{Runst-Sickel} in $\RR^n$). This result extends also
  to vector bundles and, thus, there exists a constant $C>0$ (depending only on $n$, $m$
  and the bounded geometry of $M$) such that the estimate \eqref{NS-estimate1} holds for
  all $\bsu ,\bsv\in H^{m+1}(\Omega ;TM)$.

  Moreover, the linearity of the covariant derivative $\nabla _{\bsu}\bsv$ with respect
  to $\bsu$ and the additivity with respect to $\bsv$, and
  estimate \eqref{NS-estimate1} give
  \begin{align*}
    \|\nabla _{\bsu}\bsu - \nabla _{\bsv}\bsv\|_{H^{m-1}(\Omega ;TM)}
    &\le \|\nabla _{\bsu - \bsv}\bsu \|_{H^{m-1}(\Omega ;TM)}+
    \|\nabla _{\bsv}(\bsu -\bsv)\|_{H^{m-1}(\Omega ;TM)}\\
    &\le C\left(\|\bsu \|_{H^{m+1}(\Omega ;TM)}+\|\bsv \|_{H^{m+1}(\Omega ;TM)}\right)
    \|\bsu - \bsv \|_{H^{m+1}(\Omega ;TM)}
  \end{align*}
  for all $\bsu ,\bsv\in H^{m+1}(\Omega ;TM)$, that is, the estimate \eqref{NS-estimate2}.
\end{proof}

Let us now prove the well-posedness of the Dirichlet problem for the Navier-Stokes
system under the assumption of small data (see also \cite{Am-Ciarlet, D-M, KMNW-2025,
medkova26} for results in bounded Lipschitz domains).

\begin{theorem}\label{thm.main.NS}
  Let $M$, $V$, and $V_{0}$ satisfy Assumption
  \ref{assumpt.VV0} and also the assumptions in Theorem \ref{thm.main.WP}.
  Let $\Omega \subset M$ be a smooth domain with straight cylindrical ends in the smooth
  connected manifold with straight cylindrical ends $M$. Assume also that $\Omega $ is the
  interior of $\overline{\Omega }$.
  If $m\in \ZZ_{+}$ and $2\leq {\rm{dim}}\, M =n < 2(m+2)$, then
  there exist two constants $\zeta , \eta \in (0,\infty )$ such that for all given data
  $\bsf \in H^{m+1/2}(\Gamma; TM)$ and $\bsh \in H^{m-1}(\Omega ; TM)$ satisfying the condition
  \begin{align}
  \label{cond-small-f-h}
  \|\bsh \|_{H^{m-1}(\Omega ; TM)}+\|\bsf \|_{H^{m+1/2}(\Gamma; TM)}\leq \zeta ,
  \end{align}
  the \emph{Dirichlet problem} for the Navier-Stokes system
  \begin{equation}
  \label{NS-Dirichlet}
    \bsXi U + \cvector{\nabla _{\bsu}\bsu}{0}
    \seq \cvector{\bsh}{0} \ \mbox { in } \Omega \
    \mbox{ and }\ \bsu \vert_{\Gamma} \seq \bsf
  \end{equation}
  has a unique solution
  $U = (\bsu \ \ p)^{\top} \in H^{m+1}(\Omega; TM) \oplus H^{m}(\Omega)$
  such that
  \begin{equation}
  \label{NS-Dirichlet-cond}
  \|\bsu\|_{H^{m+1}(\Omega ;TM)}\leq \eta \,.
  \end{equation}
  Moreover, there exists a constant $C_{0}\in (0,\infty )$ such that
  \[\|\bsu\|_{H^{m+1}(\Omega; TM)} + \|p\|_{H^{m}(\Omega)} \le C_{0}
  \left(\|\bsh \|_{H^{m-1}(\Omega ; TM)}+\|\bsf \|_{H^{m+1/2}(\Gamma; TM)}\right)\,.\]
 \end{theorem}

\begin{proof}
Let
$\left(\bsh \ \ {\bsf}\right)^\top
\in H^{m-1}(\Omega ; TM)\oplus H^{m+1/2}(\Gamma; TM)$ be the given data that will be kept
fixed in all arguments of the proof.
Also, for a fixed $\bsv \in H^{m+1}(\Omega ;TM)$, we consider the linear Dirichlet problem
for the generalized Stokes system
\begin{equation}
\label{Newtonian-D-B-F-new2-D-Rn-Lp}
  \begin{cases}
  \bsXi U^0
  \seq \cvector{\bsh}{0} - \cvector{\nabla _{\bsv}\bsv}{0}\ \mbox { in } \Omega\,, \\
  \bsu^0 \vert_{\Gamma} \seq \bsf
\end{cases}
\end{equation}
with the unknown $U^0=(\bsu^0\ \ p^0)^{\top}\in H^{m+1}(\Omega; TM) \oplus H^{m}(\Omega)$.

Lemma \ref{product-Sobolev} gives that $\nabla _{\bsv}\bsv\in H^{m-1}(\Omega ; TM)$.
Corollary \ref{cor.main.WP-non} then implies that the problem \eqref{Newtonian-D-B-F-new2-D-Rn-Lp}
with the given data $\left(\bsh -\nabla _{\bsv}\bsv \ \ \bsf\right)^{\top}
\in H^{m-1}(\Omega ; TM)\oplus H^{m+1/2}(\Gamma; TM)$
has a unique solution $U^0 = (\bsu^0 \ \ p^0)^{\top} \in H^{m+1}(\Omega; TM) \oplus H^{m}(\Omega)$, which can be expressed
in the form
\begin{align}
\label{solution-v0-Lp}
  \left(\bsu^0\ \ p^0\right)^{\top}
  ={\mathfrak A}\left(\bsh -\nabla _{\bsv}\bsv\ \ \bsf\right)^{\top}
  =:\left({\mathcal U}_{\bsh,{\bsf}}(\bsv )\ \
  \maP_{\bsh,{\bsf}}(\bsv )\right)^{\top},
\end{align}
where
${\mathfrak A}:H^{m-1}(\Omega ; TM)\oplus H^{m+1/2}(\Gamma; TM)
\to H^{m+1}(\Omega; TM) \oplus H^{m}(\Omega)$ is the solution operator given by Corollary
\ref{cor.main.WP-non}. This is a linear and continuous operator.
Hence the nonlinear operator
\begin{align}
\label{nonlinear-sol-oper}
\left({\mathcal U}_{\bsh,{\bsf}}\ \ \maP_{\bsh,{\bsf}}\right)^{\top}:H^{m+1}(\Omega; TM)\to H^{m+1}(\Omega; TM)\oplus H^{m}(\Omega)
\end{align}
satisfies the estimate
\begin{align}
\label{estimate-NS-1}
&\|{\mathcal U}_{\bsh,{\bsf}}(\bsv)\|_{H^{m+1}(\Omega; TM)}+
\|\maP_{\bsh,{\bsf}}(\bsv)\|_{H^{m}(\Omega)}\nonumber\\
&\hspace{3em}\leq {\mathcal C}_m\|\left(\bsh -\nabla _{\bsv}\bsv \ \ 
\bsf\right)^{\top}\|_{H^{m-1}(\Omega; TM)\oplus H^{m+1/2}(\Gamma; TM)}\nonumber\\
&\hspace{3em}\leq {\mathcal C}_m\left(\|\left(\bsh \ \ 
\bsf\right)^{\top}\|_{H^{m-1}(\Omega; TM)\oplus H^{m+1/2}(\Gamma; TM)}
+\|\nabla _{\bsv}\bsv\|_{H^{m-1}(\Omega; TM)}\right)\nonumber\\
&\hspace{3em}\leq {\mathcal C}_m\|\left(\bsh \ \ 
\bsf\right)^{\top}\|_{H^{m-1}(\Omega; TM)\oplus H^{m+1/2}(\Gamma; TM)}
+{\mathcal C}_mC\|\bsv\|_{H^{m+1}(\Omega; TM)}^2\,,
\end{align}
where ${\mathcal C}_m$ and $C$ are the constants provided by Corollary 
\ref{cor.main.WP-non} and Lemma \ref{product-Sobolev}.
Thus, the operator $\left({\mathcal U}_{\bsh,{\bsf}}\ \ 
\maP_{\bsh,{\bsf}}\right)^{\top}:H^{m+1}(\Omega; TM)
\to H^{m+1}(\Omega; TM)\oplus H^{m}(\Omega)$ is continuous.
In addition, by \eqref{Newtonian-D-B-F-new2-D-Rn-Lp} and 
\eqref{solution-v0-Lp} we have
\begin{equation*}
  \begin{cases}
  \bsXi \cvector{{\mathcal U}_{\bsh,{\bsf}}(\bsv )}{\maP_{\bsh,{\bsf}}(\bsv )}
  \seq \cvector{\bsh}{0} - \cvector{\nabla _{\bsv}\bsv}{0}\ \mbox { in } \Omega\,, \\
  {\mathcal U}_{\bsh,{\bsf}}(\bsv ) \vert_{\Gamma} \seq \bsf \,.
\end{cases}
\end{equation*}
Therefore, if we show that the nonlinear operator
${\mathcal U}_{\bsh,{\bsf}}$ has a fixed
point $\bsu\in H^{m+1}(\Omega; TM)$, i.e.,
${\mathcal U}_{\bsh,{\bsf}}(\bsu)=\bsu$,
then $\bsu$ and the pressure $p=\maP_{\bsh,{\bsf}}(\bsu)$
will determine a solution of the nonlinear problem \eqref{NS-Dirichlet}
in $H^{m+1}(\Omega; TM)\oplus H^{m}(\Omega)$.
To prove this claim, we consider the constants
\begin{align}
\label{zeta-eta}
  &\zeta :=\frac{3}{16C{\mathcal C}_m^2}>0\ \ \mbox{ and } \ \ 
  \eta :=\frac{1}{4C{\mathcal C}_m}>0
\end{align}
and the closed ball of the space $H^{m+1}(\Omega; TM)$
\begin{align}
  \label{gamma-0n-Lp}
  {\bf B}_{\eta }:=\left\{\bsw \in H^{m+1}(\Omega; TM):
\|\bsw \|_{H^{m+1}(\Omega; TM)}\leq \eta \right\},
\end{align}
and we assume that the given data $\left(\bsh \ \ {\bsf}\right)^\top$ 
satisfy the inequality \eqref{cond-small-f-h}, that is,
\begin{align*}
\|\left(\bsh \ \ {\bsf}\right)^{\top}\|_{H^{m-1}(\Omega; TM)\oplus H^{m+1/2}(\Gamma; TM)}
=\|\bsh \|_{H^{m-1}(\Omega; TM)}+\|\bsf \|_{H^{m+1/2}(\Gamma; TM)}\leq \zeta \,.
\end{align*}
In view of \eqref{cond-small-f-h}, \eqref{estimate-NS-1},
\eqref{zeta-eta}, and \eqref{gamma-0n-Lp} it follows that
\begin{align}
\label{Newtonian-D-B-F-new9-new-crackn-Lp}
  \left\|\left({\mathcal U}_{\bsh,{\bsf}}(\bsw )\ \
  \maP_{\bsh,{\bsf}}(\bsw )\right)^{\top}\right\|_{H^{m+1}(\Omega; TM)\oplus H^{m}(\Omega)}
  \leq {\frac{1}{4C\mathcal C}_m}= \eta \,,\quad \forall \, \bsw  \in \mathbf{B}_{\eta }\,.
\end{align}
Hence the nonlinear operator ${\mathcal U}_{\bsh,{\bsf}}$ maps ${\bf B}_{\eta }$
into ${\bf B}_{\eta }$.

Next we prove that ${\mathcal U}_{\bsh,{\bsf}}$ is a contraction
on ${\bf B}_{\eta }$. Indeed, the expression of ${\mathcal U}_{\bsh,{\bsf}}$
in \eqref{solution-v0-Lp}, the linearity and continuity of the operator
${\mathfrak A}$, the estimate of Corollary \ref{cor.main.WP-non},
and inequality \eqref{NS-estimate1} of Lemma \ref{product-Sobolev} give
\begin{multline*}
  \left\|{\mathcal U}_{\bsh,{\bsf}}(\bsv )
  -{\mathcal U}_{\bsh,{\bsf}}(\bsw )
  \right\|_{H^{m+1}(\Omega; TM)}
  \leq \big\|{\mathfrak A}\left(\nabla_{\bsv }\bsv -\nabla _{\bsw}\bsw \ \ 0 \right)^{\top}
  \big\|_{H^{m+1}(\Omega; TM)\oplus H^{m}(\Omega)}\\
  \leq {\mathcal C}_m\|\nabla_{\bsv }\bsv -\nabla _{\bsw}\bsw \|_{H^{m-1}(\Omega; TM)}\\
  \leq C{\mathcal C}_m\Big(\|\bsv \|_{H^{m+1}(\Omega; TM)}
  + \|\bsw \|_{H^{m+1}(\Omega; TM)}\Big)
  \|\bsv -\bsw \|_{H^{m+1}(\Omega; TM)}\\
  \leq 2\eta C{\mathcal C}_m\|\bsv -\bsw \|_{H^{m+1}(\Omega; TM)}
  \seq \frac{1}{2}\|\bsv -\bsw \|_{H^{m+1}(\Omega; TM)}\,,
\end{multline*}
for all $\bsv ,\bsw  \in \mathbf{B}_{\eta }$.
The Banach-Picard Fixed Point Theorem (see e.g. Theorem A.1.17 in \cite{KMNW-2025})
gives that there exists a unique fixed point $\bsu\in \mathbf{B}_{\eta }$ of
 ${\mathcal U}_{\bsh,{\bsf}}$, i.e.,
 ${\mathcal U}_{\bsh,{\bsf}}(\bsu)=\bsu$.
Moreover, $\bsu$ and the pressure function
$p=\maP_{\bsh,{\bsf}}(\bsu)$ (where $\maP_{\bsh,{\bsf}}$ is given by
\eqref{solution-v0-Lp}) determine a solution of the nonlinear problem
\eqref{NS-Dirichlet} in the space $H^{m+1}(\Omega; TM)\oplus H^{m}(\Omega)$. The
Banach-Picard Fixed Point Theorem implies also the uniqueness of the solution
under the assumptions \eqref{cond-small-f-h} and \eqref{NS-Dirichlet-cond}.

The condition $\bsu\in {\bf B}_{\eta }$ and inequality \eqref{estimate-NS-1} imply that
\begin{align*}
\|\bsu\|_{H^{m+1}(\Omega; TM)}
&+\|p\|_{H^{m}(\Omega)}\\
&\leq {\mathcal C}_m\big\|\left(\bsh \ \ {\bsf}\right)^{\top}\big\|_{H^{m-1}(\Omega; TM)\oplus H^{m+1/2}(\Gamma; TM)}
+\frac{1}{4}\|\bsu\|_{H^{m+1}(\Omega; TM)}
\end{align*}
and hence the estimate
\begin{align*}
\|\bsu\|_{H^{m+1}(\Omega; TM)}+
\|p\|_{H^{m}(\Omega)}
\leq \frac{4}{3}{\mathcal C}_m\big\|\left(\bsh \ \
{\bsf}\right)^{\top}\big\|_{H^{m-1}(\Omega; TM)\oplus H^{m+1/2}(\Gamma; TM)}\,,
\end{align*}
as asserted. This completes the proof.
\end{proof}

\subsection{The case $M$ closed}\label{ssec.ToAdd}
We now consider the extension of our results to
the case of a closed manifold (unlike most of the other
sections of this paper, in which we considered, for the most part, manifolds with straight
cylindrical ends).
Recall that a closed manifold is one that is smooth, compact and does not
have a boundary. Almost all the results on manifolds with straight cylindrical
ends hold also for closed manifolds (when they make sense).

Here are the main modifications:
\begin{enumerate}
  \item There are no restrictions ``or conditions'' at infinity on our operators,
  functions, and vector bundles. For instance, $\CI_{\inv}$ is replaced
  with $\CI$, all open subsets and all vector bundles are compatible, and
  $\iPS{m}$ and $\ePS{m}$
  are replaced with $\Psi^{m}$, $m \in \ZZ \cup \{\pm \infty\}$.

  \item The results about the limit at infinity $\In$ do not make sense anymore
  (and are not needed). They have to be ignored in the closed case.

  \item The same is true for the conditions involving $\In$; they are to be removed,
  as they are no longer needed.
\end{enumerate}

The case $M$ closed was considered in great detail in \cite{KNW-2025}, so we will not
go through the exercice of performing all the modifications needed here. Instead, we recall
here only two results from that paper that do not follow directly from the results
of this paper. The first result describes the kernel of $\bsXi$.
The second result formulates the well-posendess of the Stokes operator.

\begin{theorem}\label{thm.form}
  Let us assume that $V, V_{0}$ are smooth and non-negative
  and that $M$ is a smooth, compact manifold without boundary (i.e., a closed manifold).
  Then $\bsXi := \bsXi_{V, V_{0}}: H^{1}(M; TM) \oplus L^{2}(M) \to H^{-1}(M; TM) \oplus L^{2}(M)$
  is a self-adjoint Fredholm operator.
  Let $\maN \subset \CI(M; TM \oplus \CC)$ be defined by
  \begin{enumerate}[\rm (1)]
    \item $\maN := \{(\bsu, p) \mid \Def \bsu = 0 \,, \ \nabla p = 0\}$
    if $V = 0$ and $V_{0}=0$ on $M$;
    \item $\maN := \{(\bsu, 0) \mid \Def \bsu = 0\}$,
    if $V =0$ and $V_{0} \not \equiv 0$  on $M$;
    \item $\maN := \{(0, p) \mid \nabla p = 0\}$,
    if $V_{0}= 0$ on $M$ and either $V \not \equiv 0$ on $M$ or $M$ does not have
    non-zero Killing vector fields;
    \item $\maN := \{0\}$, if $V_{0} \not \equiv 0$ on $M$ and either $V \not \equiv 0$
    on $M$ or $M$ does not have non-zero Killing vector fields.
  \end{enumerate}
  Then $\maN$ is the kernel of $\bsXi_{V, V_{0}}$ \lpar i.e.,
  $\ker \bsXi_{V, V_{0}} = \maN$\rpar. Moreover, $\bsXi$ has a \lpar unique\rpar\
  Moore-Penrose   pseudoinverse $\psdinv \in \Psi_{\cl}^{-\bss - \bst}(M; TM \oplus \CC)$,
  with $\bss =\bst =(1,0)$, whose image is $\maN^{\perp}$.
\end{theorem}

The well-posedness result of the Dirichlet problem becomes

\begin{theorem}\label{thm.main.WP.compact}
  Let $\overline{\Omega}$ be a smooth, compact, connected manifold
  with boundary and let $V : \overline{\Omega} \to [0, \infty)$ be a
  smooth function that can be extended to a smooth, non-negative
  function on the {\it double} of $\overline{\Omega}$. We consider the
  generalized Stokes operator $\bsXi := \bsXi_{V, 0}$ \lpar thus $V_{0} = 0$
  in $\Omega$\rpar. Let $\Gamma
  := \pa \overline{\Omega}$, as before. Then, for every $m \in \ZZ_{+}$ and
  for any $\bsf \in H^{m+1/2}(\Gamma; T\overline{\Omega})$ such that
  $(\bsf, \bsnu)_{\Gamma } = 0$, there exists a solution
  $U = (\bsu \ \ p)^{\top} \in H^{m+1}(\Omega; T\overline{\Omega})
  \oplus H^{m}(\Omega)$ of the \emph{Dirichlet problem}
  \begin{equation}\label{eq.Dirichlet.p}
    \bsXi U \ede \bsXi_{V, 0} U \seq 0 \ \mbox { in } \Omega \
    \mbox{ and }\ \bsu \vert_{\Gamma} \seq \bsf \ \mbox { on } \pa \Omega \,.
  \end{equation}
  Any two solutions $U_{1}$ and $U_{2}$ of this Dirichlet problem differ by a
  constant scalar field: $U_{2} - U_{1} = (0 \ \ \ c)^{\top}$, $c \in \CC$.
  Moreover, there exists a constant $C_{m} \ge 0$, independent of
  $\bsh$, such that all solutions
  $U = (\bsu \ \ p)^{\top}$ satisfy
  \begin{equation*}
    \|\bsu\|_{H^{m+1}(\Omega; T\overline{\Omega})} +
    \Big \|p - \frac1{\operatorname{vol}(\Omega)}\int_{\Omega} p \,
    \dvol \Big \|_{H^{m}(\Omega)} \le C_{m}
    \|\bsf\|_{H^{m+1/2}(\Gamma; T\overline{\Omega})}\,.
  \end{equation*}
\end{theorem}

Moreover, using layer potentials, we obtain explicit formulas for the
solution. Indeed, any of the following two formulas provides a solution of
the Dirichlet problem for $\bsXi_{V, 0}$ in the above theorem:
  \begin{equation*}
    U_{1} \ede \maD_{\rm{ST}}\Big(\Big(\frac{1}{2} + \bsK\Big)^{(-1)}\bsf\Big)
    \quad \mbox{or} \quad \mbox
    U_{2} \ede \maS_{\rm{ST}}\left(\bsS^{(-1)}\bsf\right)\,.
  \end{equation*}

If, however, $V_{0} \not \equiv 0$ on $\Omega$ in the above theorem,
we obtain the uniqueness of the solution $U$, as in this paper.

\def\cprime{$'$}

\end{document}